\documentclass[a4paper,10pt]{article}
\usepackage[utf8x]{inputenc}
\usepackage{amsthm, amssymb, amsmath, amscd}

\usepackage{hyperref}

\usepackage{tikz}
\usetikzlibrary{matrix,arrows,shapes,decorations.pathmorphing}
\usepackage{xifthen}
\usepackage{verbatim}
\usepackage[nottoc,numbib]{tocbibind}
\usepackage[numbers]{natbib} 
\usetikzlibrary{arrows,shapes}

\title{Pseudo-real multiplication and an application to Teichm\"uller curves}
\author{Kolja Hept}

\begin{document}

\maketitle

\newtheorem{defi}{Definition}[section]
\newtheorem{cor}[defi]{Corollary}
\newtheorem{lem}[defi]{Lemma}
\newtheorem{prop}[defi]{Proposition}
\newtheorem{rem}[defi]{Remark}
\newtheorem{ex}[defi]{Example}
\newtheorem{theo}[defi]{Theorem}

\newcommand{\NN}{\mathbb{N}}
\newcommand{\ZZ}{\mathbb{Z}}
\newcommand{\QQ}{\mathbb{Q}}
\newcommand{\RR}{\mathbb{R}}
\newcommand{\CC}{\mathbb{C}}
\newcommand{\HH}{\mathbb{H}}
\newcommand{\bM}{\mathbb{M}}
\newcommand{\PQ}{\mathbb{P}^1(\mathbb{Q})}
\newcommand{\PK}{\mathbb{P}^1(K)}
\newcommand{\PF}{\mathbb{P}(F)}
\newcommand{\PC}{\mathbb{P}^1(\mathbb{C})}
\newcommand{\EE}{\mathcal{E}}
\newcommand{\OO}{\mathcal{O}}
\newcommand{\OD}{\mathcal{O}_D}
\newcommand{\OK}{\mathcal{O}_K}
\newcommand{\ODual}{\mathcal{O}_D^{\vee}}
\newcommand{\OA}{\mathcal{O}_\mathfrak{a}}
\newcommand{\I}{\mathcal{I}}
\newcommand{\R}{\mathcal{R}}
\newcommand{\A}{\mathfrak{a}}
\newcommand{\ADual}{\mathfrak{a}^{\vee}}
\newcommand{\B}{\mathfrak{b}}
\newcommand{\G}{\mathfrak{g}}
\newcommand{\M}{\mathfrak{m}}
\newcommand{\PP}{\mathfrak{p}}
\newcommand{\GL}{{\rm GL}}
\newcommand{\SL}{{\rm SL}}
\newcommand{\PSL}{{\rm PSL}}
\newcommand{\SO}{{\rm SO}}
\newcommand{\Sp}{{\rm Sp}}
\newcommand{\rk}{{\rm rk}}
\newcommand{\tr}{{\rm tr}}
\newcommand{\coker}{{\rm coker}}
\newcommand{\Hol}{{\rm Hol}}
\newcommand{\Mer}{{\rm Mer}}
\newcommand{\Hom}{{\rm Hom}}
\newcommand{\End}{{\rm End}}
\newcommand{\Aut}{{\rm Aut}}
\newcommand{\Sym}{{\rm Sym}}
\newcommand{\Prym}{{\rm Prym}}
\newcommand{\bS}{{\rm \mathbf{S}}}
\newcommand{\Ann}{{\rm Ann}}
\newcommand{\II}{\mathcal{I}}
\newcommand{\Ag}{\mathcal{A}_g}
\newcommand{\MM}{\mathcal{M}}
\newcommand{\cMM}{\overline{\mathcal{M}}}
\newcommand{\PM}{\mathbb{P}\Omega\mathcal{M}}
\newcommand{\PcM}{\mathbb{P}\Omega\overline{\mathcal{M}}}
\newcommand{\RM}{\mathcal{R}\mathcal{M}}
\newcommand{\cRM}{\overline{\mathcal{R}\mathcal{M}}}
\newcommand{\cRMO}{\overline{\mathcal{R}\mathcal{M}_{\mathcal{O}}}}
\newcommand{\cRMF}{\overline{\mathcal{R}\mathcal{M}_{F}}}
\newcommand{\RA}{\mathcal{R}\mathcal{A}}
\newcommand{\TT}{\mathcal{T}}
\newcommand{\WW}{\mathcal{W}}
\newcommand{\cTT}{\overline{\mathcal{T}}}
\newcommand{\id}{{\rm Id}}
\newcommand{\Jac}{{\rm Jac}}
\newcommand{\diag}{{\rm diag}}
\newcommand{\Mod}{{\rm Mod}}
\newcommand{\res}{{\rm res}}
\newcommand{\Fix}{{\rm Fix}}

\begin{abstract}
 In this paper, we classify three-dimensional complex Abelian varieties isogenous to a product $A_1 \times A_2$, where one of the factors admits real multiplication by a real quadratic order $\OD$ of discriminant~$D$. We show that the moduli space $X_D^{(3)}$ of these varieties essentially is the disjoint union of certain Hilbert modular varieties $X_\A^{(3)}$,
 each component depending on the choice of an ideal $\A$ of $\OD$. We give an explicit construction of these varieties.

 We show that the boundary of the eigenform locus for pseudo-real multiplication by an order $\OO$ in $\QQ(\sqrt{D}) \oplus \QQ$ over geometric genus zero stable curves is contained in the union of subvarieties defined by equations involving cross-ratios of projective coordinates.
 Moreover, restricted to certain topological types of stable curves relevant to the classification of primitive Teichm\"uller curves, we show that the boundary of the eigenform locus coincides with the subspace cut out by these cross-ratio equations.
 We compute these equations for the example of genus three Prym Teich\-m\"uller curves.
\end{abstract}
\clearpage

\tableofcontents
\clearpage

\section{Introduction}

Let~$\MM_g$ be the moduli space of Riemann surfaces of a fixed genus~$g$. Teich\-m\"uller curves in~$\MM_g$ arise naturally in the study of billiard tables and combine such various topics as algebraic
geometry, ergodic theory, flat geometry and number theory.
Much progress was made in classifying Teichm\"uller curves in the last two decades. We summarize the state of the art in Section~2.
Teichm\"uller curves can roughly be distinguished by their trace field, a totally real number field of degree at most~$g$.
Consider the bundle $\Omega\MM_g \to \MM_g$ of non-zero holomorphic one-forms over~$\MM_g$. If any pair~$(X,\omega) \in \Omega\MM_g$ generates a Teichm\"uller curve~$C$, then by \cite{MartinVHS} the
Jacobian~$\Jac(X)$ of~$X$ lies in the locus of principally polarized Abelian varieties isogenous to a product, where the one factor admits real multiplication by the trace field
$K = \QQ(\tr(\SL(X,\omega)))$ of~$X$ (and is consequently of dimension~$[K:\QQ]$).
If the trace field of a Teichm\"uller curve is~$\QQ$, then the curve comes from certain covering constructions of a torus.

Now consider the case $g=3$. It is shown in \cite{MartinMattPhilipp}, that there are only finitely many Teichm\"uller curves with $[K:\QQ] = 3$.
Many partially results are known for the case $[K:\QQ] = 2$. In particular for the Prym eigenform locus, there are several classification statements in \cite{LanneauNguyen2} and
\cite{LanneauNguyen1} similar to those by McMullen in the genus two case. Furthermore, it is shown in \cite{WrightOdd} and \cite{WrightHyp}, that there are only finitely many Teichm\"uller curves with $[K:\QQ] = 2$ in the minimal stratum outside of the Prym eigenform locus.

\bigskip
In this paper, we are dealing with the three-dimensional case where the trace field of $(X,\omega) \in \Omega\MM_3$ has degree two over~$\QQ$.
Thus $A=\Jac(X)$ is isogenous to a product of an elliptic curve and a complex Abelian surface that admits real multiplication by a quadratic order~$\OD$ of some positive, non-square discriminant~$D$.
The choice of such an Abelian subsurface $S$ together with the choice of real multiplication $\QQ(\sqrt{D}) \hookrightarrow \End^+(S) \otimes_\ZZ \QQ$ is equivalent to the choice of an endomorphism
\[ \QQ(\sqrt{D}) \oplus \QQ \hookrightarrow \End^+(A) \otimes \QQ, \] what we will call \emph{pseudo-real multiplication}.

Since we do not want to restrict ourselves to the Prym eigenform locus, the subvariety of $\Jac(X)$ admitting real multiplication can be of any type~$(1,d)$.
Our first goal will be to classify all three-dimensional complex Abelian varieties admitting pseudo-real multiplication by a \emph{pseudo-cubic numberfield} $\QQ(\sqrt{D}) \oplus \QQ$
together with the choice of such an endomorphism.
We will introduce in Section~\ref{Tracepairing} a certain condition on $d$ depending on the discriminant~$D$, which we will call the \emph{prime factor condition}.
Let $d \in \NN$ be relatively prime to the conductor of~$D$ and satisfying the prime factor condition.
Then for each ideal $\A \subseteq \OD$ of norm $d$ we will construct explicitly a Hilbert modular variety $X_\A^{(3)}$.
Denoting by $X_{D,d}^{(3)}$ the moduli space of three-dimensional complex Abelian varieties together with some choice of pseudo-real multiplication by $\QQ(\sqrt{D}) \oplus \QQ$ of type~$(1,d)$,
we will prove the following theorem.

\begin{theo}\label{Theorem1}
 Let $D \in \NN$ be a non-square discriminant and let $d \in \NN$ be relatively prime to the conductor of $D$.
 Then $X_{D,d}^{(3)}$ is non-empty if and only if $d$ satisfies the prime factor condition in $\OD$.
 Moreover, in this case $X_{D,d}^{(3)}$ consists of the $2^s$ irreducible components $X_\A^{(3)}$, where $s$ is the number of splitting prime divisors of $d$.
\end{theo}

Let $\Omega\cMM_g \to \cMM_g$ be the bundle of stable forms over the Deligne-Mumford compactification of $\MM_g$.
If $\I$ is a lattice in a totally real number field $F$ of degree~$g$, then each boundary stratum $\mathcal{S} \subset \cMM_g(\I)$ of $\I$-weighted geometric genus zero curves can be
embedded into the boundary of $\Omega\cMM_g$.
Let $\Omega^\iota\RM_\OO$ be the locus of stable forms $(X,\omega) \in \Omega\cMM_g$, where $\omega$ is some choice of a $\iota$-eigenform for real multiplication on~$X$ by an order~$\OO \subset F$.
In \cite{MartinMatt}, Bainbridge and M\"oller gave a certain condition for $\I$-weighted boundary strata called \emph{admissibility}.
They showed that each stable form $(X,[\omega]) \in \mathbb{P}\Omega\MM_g$ in the boundary of $\mathbb{P}\Omega^\iota\RM_\OO$ lies in the image of some~$\mathcal{S}(h)$, a certain subvariety of an
admissible boundary stratum~$\mathcal{S}$ defined by some~$h \in \End_\QQ^+(F)$.
Using the techniques developed in this paper, we will show that an analogous statement also holds in the pseudo-cubic case.

Furthermore, Bainbridge and M\"oller defined in \cite{MartinMatt} for the genus three case the quadratic function $Q$ mapping each $x \in F$ to the number $\mathcal{N}_\QQ^F(x)/x \in F$.
They showed that an~$\I$-weighted boundary stratum with weights $w_1,...,w_n$ is admissible if and only if the set $\{ Q(w_1),...,Q(w_n) \}$ is not contained in a closed half-space of its~$\QQ$-span.

Using coordinates, we will also define a quadratic function \[ Q:~ \QQ(\sqrt{D}) \oplus \QQ \to \QQ^3. \]
After the preliminary work in the sections before, we can construct this function precisely in that way, that the analogous statement like in the cubic case holds.
A boundary stratum $\mathcal{S}$ of $\II$-weighted stable curves is isomorphic to a product of certain $\MM_{0,n_i}$, the moduli spaces of genus zero curves with $n_i$ marked points.
Thus each weighted stable curve in $\mathcal{S}$ can be represented by projective coordinates. Using cross-ratios of these coordinates, it is shown in \cite{MartinMatt}, that each subvariety
$\mathcal{S}(h)$ of an admissible stratum $\mathcal{S}$ is cut ot by a single equation in these cross-ratios.
We will show the analogous statement for the stratum of irreducible stable curves with cross-ratios~$p_{ij}$ as follows.
\begin{theo}
 Let $r=(r_1,r_2,r_3)$ be a $\ZZ$-basis of $\II$ and let $s=(s_1,s_2,s_3)$ be its dual basis of~$\II^\vee$ with respect to the pseudo-trace pairing.
 The subvariety $\mathcal{S}_r(h)$ of the stratum $\mathcal{S}_r$ of trinodal curves with weights $r$ is cut out by the equations
 \[ p_{23}^{a_1} \cdot p_{13}^{a_2} \cdot p_{12}^{a_3} = \exp(-2\pi i (a_1b_{23}+a_2b_{13}+a_3b_{12})), \]
 where the $b_{ij}$ are certain rational numbers defined by $h$ and where $(a_1,a_2,a_3)$ runs over the integral solutions of
 \begin{equation}\label{Loesungsraum} a_1s_2s_3+a_2s_1s_3+a_3s_1s_2=0. \end{equation}
\end{theo}
Contrary to the cubic case in~\cite{MartinMatt}, where the subvariety $\mathcal{S}_r(h)$ is cut out by a single equation, note that the $\ZZ$-module of solutions of Equation~\eqref{Loesungsraum}
may have rank two, see Example~\ref{ranktwo}. This is a consequence of the fact, that in the pseudo-cubic case the trinodal boundary stratum is in general not a maximal dimensional boundary stratum.
Furthermore, in~\cite{MartinMatt} it turns out, that for genus three, admissibility is also a sufficient condition, i.e. that the boundary of the real multiplication locus is precisely the union of all
images in $\mathbb{P}\Omega\cMM_3$ of the varieties~$\mathcal{S}(h)$, where~$\mathcal{S}$ runs over all admissible boundary strata.
We show that this statement holds in the pseudo-cubic case at least for those boundary strata relevant for classifying Teichm\"uller curves.

The quadratic function $Q$ and the cross-ratio equations could be used in future work for constructing an algorithm to classify Teichm\"uller curves by their cusps in a similar spirit as the algorithm
given in~\cite{MartinMatt}.
Fixing the discriminant~$D$ and the type~$(1,d)$, one first uses the function~$Q$ to list admissible bases for all lattices $\I \subset F$ defining admissible $\II$-weighted boundary strata
and then solve the cross-ratio equations.
Presumably, one can use the techniques from~\cite{MartinMattPhilipp}, to get a finiteness result for primitive (but not algebraically primitive) Teichm\"uller curves in~$\MM_3$.
\\[1em]
{\bf Acknowledgements.}
This paper is my PhD-thesis at the Goethe Universit\"at Frankfurt am Main.
First of all I would like thank my adivsor Martin M\"oller for making all this work possible and for his excellent guidance.
I am also grateful to my colleagues Matteo Constantini, Quentin Gendron, Thorsten J\"orgens, \c{S}evda Kurul, David Torres Teigell and Jonathan Zachhuber for proofreading this paper.
Finally I would like to thank my family, especially my parents, and my friends for their continuous support.
\clearpage

\section{Teichm\"uller curves: state of the art}

The moduli space $\MM_g$ of genus $g$ compact Riemann surfaces carries a natural metric induced by the Teichm\"uller metric on its universal covering orbifold, the Teichm\"uller space $\TT_g$.
There is a natural action of $\SL_2(\RR)$ on $\mathcal{Q}\TT_g$, the space of genus $g$ marked Riemann surfaces together with a non-zero quadratic differential.
This action is induced by linear transformations on the quadratic differential.
Fixing any point $(X,q) \in \mathcal{Q}\TT_g$, post-composition of $A \mapsto A.(X,q)$ with the forgetful map $\mathcal{Q}\TT_g \to \TT_g$ factors through the quotient $\SO_2(\RR) \backslash \SL_2(\RR)$
and yields an isometric embedding \[ \HH = \SO_2(\RR) \backslash \SL_2(\RR) \hookrightarrow \TT_g. \]
Post-composing with the covering map $\TT_g \to \MM_g$, we obtain a complex geodesic \[ f:~ \HH \to \MM_g, \] and we say that $(X,q)$ generates $f(\HH)$.
Conversely, by a consequence of Teich\-m\"uller's Theorem, every complex geodesic in $\MM_g$ is generated by a quadratic differential in this way.
Of special interests are the complex geodesics generated by squares of a holomorphic one-form.
One reason is, that by \cite{Kra}, Theorem~1, this happens if and only if the composed map \[ \HH \to \MM_g \to \mathcal{A}_g \] is a geodesic (here $\MM_g \to \mathcal{A}_g$ denotes the Torelli map),
where $\mathcal{A}_g$ denotes the moduli space of $g$-dimensional, principally polarized Abelian varieties.
Another reason is, that any complex geodesic is commensurable to a complex geodesic generated by the square of a holomorphic one-form.
There is also a natural action of $\SL_2(\RR)$ (respectively $\GL_2(\RR)$) on $\Omega\MM_g$, the space of Riemann surfaces of genus $g$ together with a non-zero holomorphic one-forms.
If $(X,q)=(X,\omega^2)$ for some non-zero $\omega \in \Omega(X)$, and setting $F(A)=A.(X,\omega)$ for all $A \in \SL(X,\omega)$, then the diagram
\[ \begin{CD}  \SL_2(\RR) @>F>> \Omega\MM_g \\  @VV V @VV V \\ \HH = {\rm SO}_2(\RR) \backslash \SL_2(\RR) @>f>> \MM_g \end{CD} \] commutes.
Thus one can look at holomorphic one-forms instead of quadratic differentials.
Any Riemann surface with a holomorphic one-form $(X,\omega)$ can also be seen as a translation surface with singularities.
The translation atlas is given by integrating $\omega$, which allows a geometric approach to this topic. For this reason, any such pair $(X,\omega) \in \Omega\MM_g$ is called a \emph{flat surface}.
\\[1em]
{\bf Teichm\"uller curves.}
A complex geodesic in~$\MM_g$, which is simultaneously an algebraic curve (which is usually not the case), is called a \emph{Teichm\"uller curve}.
It is well known that a flat surface $(X,\omega)$ generates a Teichm\"uller curve if and only if its $\GL_2(\RR)$-orbit is closed in~$\MM_g$.
Furthermore, $(X,\omega)$ generates a Teichm\"uller curve if and only if ${\rm Stab}(f)$ or equivalently its Veech group $\SL(X,\omega)$
(the image in $\SL_2(\RR)$ of orientation-preserving affine diffeomorphisms under the derivative map) is a lattice in $\SL_2(\RR)$. Such a flat surface is called a \emph{Veech surface}.
In every stratum there are infinitely many Teichm\"uller curves, each of them generated by a torus covering ramified only over one point (also known as \emph{square-tiled surfaces}).

In his famous work \cite{Veech89}, Veech constructed first examples of Teichm\"uller curves not of this type.
He showed that the Veech group of any translation surface is a discrete subgroup of $\SL_2(\RR)$. Moreover, if it is a lattice, then it satisfies the Veech dichotomy:
The geodesic flow in any direction is either periodic or uniquely ergodic.
These surfaces with a lattice Veech group became an object of fruitful research, as they are combining such various topics as algebraic geometry, ergodic theory, flat geometry, number theory and
Teichm\"uller theory.
\\[1em]
{\bf Primitivity.}
A compact Riemann surface together with a quadratic differential~$(X,q)$ (respectively a flat surface~$(X,\omega)$) is said to be \emph{geometrically primitive},
if it is not the pullback of a quadratic differential (respectively a flat surface) of lower genus.
A complex geodesic is said to be \emph{geometrically primitive}, if any one (and thus each) generating differential is geometrically primitive.
M\"oller showed in \cite{MartinPeriodic}, Theorem~2.6 that each flat surface $(X,\omega)$ is the pullback of a primitive flat surface $(X_0,\omega_0)$.
The primitive flat surface $(X_0,\omega_0)$ is also unique provided that the genus of~$X_0$ is at least two.
Moreover, he showed that $(X,\omega)$ and $(X_0,\omega_0)$ are commensurable if $(X,\omega)$ generates a Teichm\"uller curve.
Thus, for classification of Teichm\"uller curves it makes sense to classify the primitive ones. 
\\[1em]
{\bf Classification.}
During the last twelve years, much effort was made to classify all Teichm\"uller curves generated by flat surfaces.
A complete classification for primitive Teichm\"uller curves in $\MM_2$ was achieved by McMullen in a series of papers in 2003 and 2004
(\cite{McMBilliards}, \cite{McMSpin}, \cite{McMDecagon}, \cite{McMSines}).
The central role here plays the Weierstrass locus $\WW_D$, where $D$ is a positive discriminant (i.e. $D \equiv 0,1 \mod 4$).
It consists of those genus two Riemann surfaces~$X$ such that the Jacobian~$\Jac(X)$ of~$X$ admits real multiplication by the real quadratic order~$\OD$ and such that there exists an
eigenform~$\omega \in \Omega(X)$ for some choice of real multiplication which has a double zero. The classification can be summarized as follows.
\newtheorem*{theoA}{Theorem A}
\begin{theoA}[McMullen]\label{McMKlassifikation}
 Let $D \geq 5$ be a non-square discriminant.
 \begin{enumerate}
  \item If $D \not\equiv 1 \mod 8$, then $\WW_D$ is a single primitive Teichm\"uller curve, generated by a certain $L$-shaped translation surface.
  \item If $D \equiv 1 \mod 8$, then $\WW_D$ consists of two irreducible components, distinguished by their spin invariant.
        Each of the two components is a primitive Teichm\"uller curve generated by a certain $L$-shaped translation surface.
  \item The regular decagon is an eigenform in $\Omega\MM(1,1)$ for real multiplication, and its orbit projects to a primitive Teichm\"uller curve in $\MM_2$.
 \end{enumerate}
 These curves are all primitive Teichm\"uller curves in $\MM_2$, and they are pairwise different.
\end{theoA}

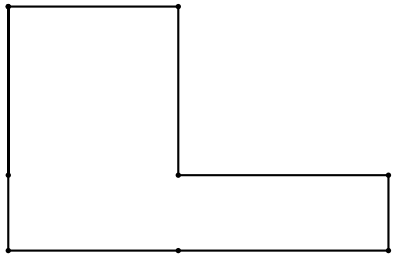
\begin{figure}[ht]
\centering
\begin{tikzpicture}[scale=1]
 \coordinate (A) at (0,0); 
 \coordinate (B) at (2.236,0);
 \coordinate (C) at (5,0); 
 \coordinate (D) at (5,1); 
 \coordinate (E) at (2.236,1); 
 \coordinate (F) at (2.236,3.236); 
 \coordinate (G) at (0,3.236);
 \coordinate (H) at (0,1);
 \coordinate (I) at (0,3.236);
 \fill (A) circle (1pt); \fill (B) circle (1pt); \fill (C) circle (1pt); \fill (D) circle (1pt); \fill (E) circle (1pt);
 \fill (F) circle (1pt); \fill (G) circle (1pt); \fill (H) circle (1pt); \fill (I) circle (1pt);
 \draw[thick] (A) -- (B) -- (C) -- (D) -- (E) -- (F) -- (G) -- (H) -- (I) -- (A);
\end{tikzpicture}
\label{L-Flaeche} \caption{L-shaped translation surface, opposite edges are identified}
\end{figure}

The curves in $\WW_D$ were indepently discovered by Calta in \cite{Calta}.
\\[1em]
{\bf The Prym-eigenform locus.}
The next step was to look for primitive Teichm\"uller curves in higher genera, using a generalization of $\WW_D$, namely the Prym-Weierstrass locus.
This locus, denoted by $\WW_D^g$, is defined to be the locus of Riemann surfaces $X \in \MM_g$, such that
\begin{enumerate}
 \item there is a holomorphic involution $\tau: X \to X$, such that the \emph{Prym variety} \[ \Jac(X)^- := (\Omega(X)^-)^* / H_1(X;\ZZ)^- \subset \Jac(X) \] has dimension two
       (where $\Omega(X)^-$ denotes the eigenspace of $\tau$ to the eigenvalue $-1$),
 \item $\Jac(X)^-$ admits real multiplication by $\OD$, and
 \item there exists an eigenform $\omega \in \Omega(X)^-$ for some choice of real multiplication which has only one zero.
\end{enumerate}
By the Riemann-Hurwitz formula, $\WW_D^g$ is empty if $g \notin \{ 2,3,4,5 \}$.
In genus two, this definition coincides with the old one, since each genus two Riemann surface is hyperelliptic.
In \cite{McMPrym}, McMullen constructed infinitely primitive Teichm\"uller curves contained in $\bigcup \WW_D^g \subset \MM_g$ for $g \in \{ 3,4,5 \}$.
He also showed that for $D$ not a square, the locus $\WW_D^g$ is a finite union of primitive Teichm\"uller curves.

Removing the 'only one zero'-condition in the definition of~$\WW_D^g$, we define~$\Omega \EE_D^g$ to be the locus of those $(X,\omega) \in \Omega\MM_g$, such that there is a holomorphic involution
$\tau: X \to X$, with $\Jac(X)^-$ a two-dimensional subvariety of $\Jac(X)$ that admits real multiplication by the quadratic order~$\OD$, and such that~$\omega$ is in $\Omega(X)^-$ and an eigenform
for some choice of real multiplication.
For any partition $\alpha = (\alpha_1,...,\alpha_k)$ of $2g-2$, there is the natural stratum $\Omega \MM_g(\alpha)$ in $\Omega\MM_g$ consisting of those flat surfaces~$(X,\omega)$, where~$\omega$
has exactly $k$ zeroes $z_1,...,z_k$ with $z_i$ of multiplicity $\alpha_i$.
We set $\Omega \EE_D^g(\alpha) = \Omega \EE_D^g \cap \Omega \MM_g(\alpha)$ and finally we denote by $\EE_D^g$ respectively $\EE_D^g(\alpha)$ the image of $\Omega \EE_D^g$ respectively
$\Omega \EE_D^g(\alpha)$ in~$\MM_g$ under the forgetful map.
Note that $\WW_D^g = \EE_D^g(2g-2)$ and that the decagon form in Theorem~\ref{McMKlassifikation} generates the only primitive Teichm\"uller curve in $\EE_D^2(1,1)$.
\\[1em]
{\bf The trace field and splitting varieties.}
The trace field $\QQ(\SL(X,\omega))$ of a holomorphic form $(X,\omega) \in \Omega\MM_g$ is defined to be the field extension of $\QQ$ generated by the traces of all the elements in $\SL(X,\omega)$.
By \cite{McMBilliards}, Theorem~5.1 we have \[ [\QQ(\SL(X,\omega)) : \QQ] \leq g. \]
Flat surfaces $(X,\omega)$ with $[\QQ(\SL(X,\omega)) : \QQ] = g$ are called \emph{algebraically primitive}, and the Teichm\"uller curve generated by such a form is also called algebraically primitive.
Since the trace field is preserved by translation coverings, each algebraically primitive Teichm\"uller curve is also geometrically primitive.
All the primitive Teichm\"uller curves mentioned above are generated by flat surfaces $(X,\omega)$, such that $\Jac(X)$ is isogenous to a product of Abelian varieties, and where the one
factor admits real multiplication by $\QQ(\SL(X,\omega))$, in these cases a real quadratic number field.

This had to happen, as M\"oller has shown in \cite{MartinVHS}:
\newtheorem*{theoB}{Theorem B}
\begin{theoB}[M\"oller]\label{MartinVHS}
 Let $(X,\omega) \in \Omega\MM_g$ be a flat surface generating a Teichm\"uller curve~$C$ in $\MM_g$. Then we have the following.
 \begin{enumerate} 
  \item The trace field $K = \QQ(\tr(\SL(X,\omega)))$ is totally real.
  \item The image of $C \to \mathcal{A}_g$ under the Torelli map is contained in the locus of Abelian varieties isogenous to $A_1 \times A_2$,
        where $A_1$ has dimension $[K:\QQ]$ and real multiplication by~$K$. The generating differential $\omega \in \Omega(X)$ is an eigenform for real multiplication by~$K$.
 \end{enumerate}
\end{theoB}
Thus, we have three possibilities for a flat surface $(X,\omega) \in \Omega\MM_3$ generating a Teichm\"uller curve $C$.
The first one is, that the trace field $K$ is $\QQ$ (these curves are called \emph{arithmetic}). Then $(X,\omega)$ covers a torus, in particular $C$ is not primitive.
The second one is, that $K$ is a totally real cubic number field, i.e. $C$ is algebraically primitive.
In~\cite{MartinMattPhilipp}, Bainbridge, Habegger and M\"oller have shown that there are only finitely many algebraically primitive Teichm\"uller curves in $\MM_3$. 
The remaining case, where $K$ is a real quadratic number field, is what we are dealing with in this paper.
Therefore, we are interested in three-dimensional principally polarized complex Abelian varieties isogenous to a product $A_1 \times A_2$, where $A_1$ is a one-dimensional polarized complex Abelian
variety of type~$d$ and $A_2$ is a two-dimensional polarized complex Abelian variety of type~$(1,d)$ admitting real multilplication by a real quadratic order~$\OD$.

In \cite{LanneauNguyen1}, Lanneau and Nguyen proved the following classification of the Prym-Weierstra{\ss} locus.
\newtheorem*{theoC}{Theorem C}
\begin{theoC}[Lanneau, Nguyen]
 For $D \geq 17$, $\Omega \EE_D^3(4) \subset \Omega\MM_3$ is non empty if and only if $D \equiv 0,1,$ or $4 \mod 8$. All the loci $\Omega \EE_D^3(4)$ are pairwise disjoint.
 Moreover, for the values $0$,$1$,$4$ of discriminants, the following dichotomy holds. Either
\begin{enumerate}
 \item $D$ is odd and then $\Omega \EE_D^3(4)$ has exactly two connected components,
 \item $D$ is even and $\Omega \EE_D^3(4)$ is connected.
\end{enumerate}
In addition, each component of $\Omega \EE_D^3(4)$ corresponds to a closed $\GL_2^+(\RR)$-orbit.
For $D<17$, $\Omega \EE_D^3(4)$ is non-empty if and only if $D \in \{ 8,12 \}$ and in this case it is connected.
\end{theoC}
In \cite{LanneauNguyen2}, they showed that the analogous satement holds for the loci $\Omega \EE_D^3(2,2)^{\rm odd}$ and $\Omega \EE_D^3(1,1,2)$ for $D\geq 8$.
Note that the Prym variety $\Jac(X)^-$ for a genus three Riemann surface always carries a polarization of type $(1,2)$.
Beyond the curves of Prym loci, there are currently no classification results known for Teichm\"uller curves in genus three that are primitive but not algebraically primitive.
\clearpage

\section{Complete statements of the results}

As mentioned at the end of the last section, we are interested in principally polarized complex Abelian varieties isogenous to a product $A_1 \times A_2$, where $A_1$ is a one-dimensional polarized
complex Abelian variety of type~$d$ and $A_2$ is a two-dimensional polarized complex Abelian variety of type~$(1,d)$ admitting real multilplication by a real quadratic order~$\OD$.
We will see in Section~\ref{pseudorealmult}, that this is equivalent to the concept of \emph{pseudo-real multiplication}.
\\[1em]
{\bf Moduli-space of Abelian varieties with pseudo-real multiplication.}
Our first goal is to parameterize all these polarized Abelian varieties together with a choice of some real multiplication.
With this goal in mind, we first consider the two-dimensional part. For a non-square discriminant $D$ and $d_1 \mid d_2 \in \NN$ we denote by $X_{D,(d_1,d_2)}$
the space of all isomorphism classes $[(A,H,\rho)]$, where $(A,H)$ is a polarized Abelian surface of type $(d_1,d_2)$ that admits real multiplication by $\OD$
and $\rho : \OD \hookrightarrow \End^+(A)$ is some choice of real multiplication.
If $d_1=d_2=1$, it is well known that $X_{D,(1,1)}$ is isomorphic to the \emph{Hilbert modular surface} \[ X_D = \PSL_2(\OD \oplus \ODual) \backslash \HH \times \HH. \]
The reason that $X_{D,(1,1)}$ consists only of one irreducible component, is that every rank two symplectic $\OD$-module (the lattice defining the Abelian surface) is isomorphic to
$\OD \oplus \ODual$, where $\OD \oplus \ODual$ is equipped with the natural trace pairing.
This does not hold in the non-principal case. It turns out that, provided that the quotient $d=\tfrac{d_2}{d_1}$ is relatively prime to the conductor of $D$, the lattice
defining the surface is as a symplectic $\OD$-module isomorphic to $\OD \oplus \tfrac{1}{\sqrt{D}}\A$ for some ideal $\A \subset \OD$ of norm~$d_1d_2$.
The locus of such Abelian surfaces is denoted by \[ X_\A \subset X_{D,(d_1,d_2)}. \]
To count the irreducible components of $X_{D,(d_1,d_2)}$, consider the prime factorization $\prod p_i^{k_i}$ of~$d$.
If $d$ is relatively prime to the conductor of $D$, then we say that the pair $(d_1,d_2)$ satisfies the \emph{prime factor condition in $\OD$}, if the following two conditions hold.
 \begin{enumerate}
  \item No $p_i$ is inert over $\QQ(\sqrt{D})$ and
  \item If $p_i$ is ramified over $\QQ(\sqrt{D})$, then $k_i=1$.
 \end{enumerate}

In Section~\ref{dim2}, we use the prime ideal factorization for $(d) \subseteq \OD$ to prove the following.

\newtheorem*{ModulraumDim2}{Theorem \ref{ModulraumDim2}}
\begin{ModulraumDim2}
 Let $D \in \NN$ be a non-square discriminant and let $d_1,d_2 \in \NN$ with $d_1 \mid d_2$ and $\tfrac{d_2}{d_1}$ relatively prime to the conductor of $D$.
 Then $X_{D,(d_1,d_2)}$ is non-empty if and only if $(d_1,d_2)$ satisfies the prime factor condition in $\OD$.
 In this case $X_{D,(d_1,d_2)}$ consists of $2^s$ irreducible components $X_\A$, where $s$ is the number of splitting prime divisors of $\tfrac{d_2}{d_1}$.
 For each component $X_\A$ defined by the ideal $\A$, we have an isomorphism \[ \Phi:~ {\rm PSL}_2(\OD \oplus \tfrac{1}{\sqrt{D}}\A) \backslash \HH^2 \cong X_\A,\quad [z] \mapsto [(A_z,H_z,\rho_z)] \]
 between the corresponding Hilbert modular surface and $X_\A$.
\end{ModulraumDim2}
Here we have seen why it makes sense to presuppose that~$d=\tfrac{d_2}{d_1}$ is relatively prime to the conductor of~$D$.
Otherwise we would not have a unique prime ideal factorization for the ideal~$(d)$.

\bigskip
Now let us go back to dimension three. Similar to the lower dimensional case above, we denote by $X_{D,d}^{(3)}$ the space of all isomorphism classes $[(A,H,S,\rho)]$, where~$(A,H)$ is a principally
polarized Abelian variety of dimension three, $S$ is a two dimensional subvariety of~$A$ of type $(1,d)$ and $\rho : \OD \to {\rm End}(S)$ is some choice of real multiplication by $\OD$ on $S$.
Again, we require that the degree~$d$ is relatively prime to the conductor of $D$.
We say that $(A,H,S,\rho)$ is of type $\A$ for an ideal $\A \subseteq \OD$, if $[(S,H|_S,\rho)]$ is in $X_\A$.
We denote the locus of such classes $[(A,H,S,\rho)]$ by \[ X_\A^{(3)} \subset X_{D,d}^{(3)}. \]
Analogously to the dimension two case, in Section~\ref{dim3} we construct for each $z \in \HH^3$ such a variety $(A_z,H_z,S_z,\rho_z) \in X_\A$ and get the following theorem.

\newtheorem*{modulispacedim3}{Theorem \ref{modulispacedim3}}
\begin{modulispacedim3}
 Let $D \in \NN$ be a non-square discriminant and let $d \in \NN$ be relatively prime to the conductor of~$D$.
 Then $X_{D,d}^{(3)}$ is non-empty if and only if~$d$ satisfies the prime factor condition in~$\OD$.
 In this case $X_{D,d}^{(3)}$ consists of $2^s$ irreducible components $X_\A^{(3)}$, where~$s$ is the number of splitting prime divisors of~$d$.
 For each component $X_\A^{(3)}$ defined by the ideal $\A$ we have an isomorphism \[ \Phi:~ \Gamma_{D,d}(\eta_1,\eta_2) \backslash \HH^3 \cong X_\A^{(3)},\quad [z] \mapsto [(A_z,H_z,s_z,\rho_z)] \]
 between the corresponding Hilbert modular variety and $X_\A^{(3)}$.
\end{modulispacedim3}
We will give an explicit description of the group $\Gamma = \Gamma_{D,d}(\eta_1,\eta_2)$, a finite-index subgroup of ${\rm SL}_2(\OD \oplus \tfrac{1}{\sqrt{D}}\A) \times {\rm SL}_2(\ZZ)$.
Unfortunately, $\Gamma$ depends not only on the ideal~$\A$, but also on the choice of a $\ZZ$-basis $(\eta_1,\eta_2)$ of~$\A$.
Furthermore, $\Gamma$ does not have the structure of a direct product.
We can construct a universal family \[ \Gamma_{D,d}(\eta_1,\eta_2) \ltimes \ZZ^6 \backslash (\HH^3 \times \CC^3) \to X^{(3)}_\A \] for those components,
but for many other problems (like counting the cusps of $\Gamma \backslash \HH^3$) the group is much too unwieldy.
Nevertheless, we can give a lower and an upper bound for $\Gamma$ (with both inclusions of finite index), which do not dependent on any choice of basis and which are direct products.
More precisely, we have the inclusion \[ \SL\begin{pmatrix} 1+\A & \sqrt{D} \OD \\ \tfrac{1}{\sqrt{D}}\A^2 & 1+\A \end{pmatrix} \times \Gamma(d) \subseteq \Gamma_{D,d}(\eta_1,\eta_2) \subseteq
{\rm SL}_2(\OD \oplus \tfrac{1}{\sqrt{D}}\A) \times {\rm SL}_2(\ZZ) \] and we interpret the quotients of $\HH^3$ by these groups as moduli spaces of polarized
complex Abelian varieties with real multiplication and certain level structures.
\\[1em]
{\bf Generalized period coordinates and cross-ratio equations.}
Before the final statement in \cite{MartinMattPhilipp} mentioned in the previous section, Bainbridge and M\"oller have proven in \cite{MartinMatt} the finiteness of algebraically primitive
curves in~$\MM_3$ generated by flat surfaces in the stratum $\Omega\MM_3(3,1)$, using a generalization of the classical period coordinates.

Let $\overline{\MM}_g(L)$ be the space of arithmetic genus~$g$ stable curves together with a Lagrangian marking (weighting) by a rank~$g$ free Abelian group~$L$. 
There is a natural stratification of~$\overline{\MM}_g(L)$ defined by weight preserving homeomorphisms. If  $\II$ is a lattice in a totally real number field of degree~$g$, then each boundary stratum
in~$\overline{\MM}_g(L)$ can be embedded into the boundary of~$\PcM_g$, where~$\Omega\cMM_g$ is the bundle of stable forms over the Deligne-Mumford compactification of~$\MM_g$.
Bainbridge and M\"oller showed that a necessary condition for a boundary point~$X$ for lying in the closure of the real multiplication locus $\mathcal{R}\MM_\OO \subset \MM_g$ in $\cMM_g$
is the following.
There is some subvariety $\bS(h)$ defined by $h \in \Sym_\QQ(F)$ of a so-called \emph{admissible} $\II$-weighted boundary stratum $\bS \subset \PcM_g$, such that~$X$ lies in the image of $\bS(h)$
under the forgetful map $\PcM_g \to \cMM_g$.
They also showed, that this condition is sufficient in genus three.

The key was a description of the closure of the eigenform locus $\Omega\mathcal{E}_\OO$ (where~$\OO$ is an order in a totally real number field~$F$ of degree~$g$) in $\Omega\overline{\MM}_g$.
For any free Abelian group $L$ of rank $g$, they constructed a homomorphism \[ \Psi:~ \bS_{\ZZ}(\Hom(L,\ZZ)) \to \Hol^*(\MM_g(L)), \]
where $\bS_\ZZ(\cdot)$ denotes the symmetric square and $\MM_g(L)$ the moduli space of Riemann surfaces together with a Lagrangian marking by~$L$.
This homomorphism can be seen as a coordinate free version of exponentials of the classical period coordinates. They proved that for each $a \in \bS_{\ZZ}(\Hom(L,\ZZ))$,
the function~$\Psi(a)$ extends to a meromorphic function on~$\overline{\MM}_g(L)$ and computed the orders of vanishing.

In genus three, they constructed a quadratic function $Q: F \to F$ and showed that a weighted boundary stratum is admissible if and only if the $Q$-images of the weights are not contained in a closed
half-space of their $\QQ$-span. This geometric condition is called the \emph{no-half-space condition}.
For each $n \in \NN$, let $\MM_{0,n}$ be the moduli space of~$n$ labeled points on $\PC$.
Each $\II$-weighted boundary stratum $\bS$ is isomorphic to a product of certain $\MM_{0,n_i}$, one factor for each irreducible component.
Thus, any point of $\bS$ can be represented by an ordered tuple of elements in $\PC$. They described the function~$\Psi(a)$ in terms of cross-ratios of these projective coordinates, and constructed
single explicit equations cutting out the subvarieties $\bS(h)$. They used these so-called \emph{cross-ratio equations} to prove the finiteness statement for algebraically primitive Teichm\"uller
curves in the stratum $\Omega\MM_3(3,1)$.
Moreover, they constructed an algorithm that searches for possible cusps (necessarily represented by irreducible stable forms) of algebraically primitive Teichm\"uller curves in the stratum
$\Omega\MM_3(4)^{\rm hyp}$. This algorithm works as follows.
Fixing an order $\OO$ in a totally real cubic number field, first one lists all admissible bases of ideals in $\OO$. For each admissible basis, there are just finitely many irreducible stable forms
having these residues and a single zero. Then one checks, if the cross-ratio equation holds.

Our goal in Section~\ref{boundstrat} is, to construct analogues for the locus of those genus three Riemann surfaces~$X$, where $\Jac(X)$ admits multiplication by a \emph{pseudo-cubic number field}
$F = \QQ(\sqrt{D}) \oplus \QQ$. We first show that an isomorphism class of $(A,\rho)$, where $\rho$ is such a choice of pseudo-real multiplication on a three-dimensional principally polarized Abelian
variety, is the same as a class $[(A,H,S,\rho)]$ in $X_{D,d}^{(3)}$.
We formulate admissibility for the pseudo-cubic case and give the following necessary condition for lying in the boundary of the pseudo-real multiplication locus.

\newtheorem*{admissible1}{Corollary \ref{admissible1}}
\begin{admissible1}
 Let $\OO$ be a pseudo-cubic order. Each geometric genus zero stable curve $X \in \overline{\mathcal{R}\MM_\OO}$ lies in the image under the forgetful map of some admissible
 $\II$-weighted boundary stratum $\mathcal{S} \subset \cMM_3(\II)$.
\end{admissible1}

We construct a quadratic function $Q: F \to \QQ^3$ and give a similar geometric reformulation of admissibility as in the cubic case.
To be more precise, we define \[ Q:~ F \to \QQ^3,\quad  (x,q) \mapsto \begin{pmatrix} \mathcal{N}(\sqrt{D}x) \\ \tr(x)q \\ \tr(\sqrt{D}x)q \end{pmatrix}. \]
Then we show the following criterion.

\newtheorem*{Q}{Theorem \ref{Q}}
\begin{Q}
 An $\mathcal{I}$-weighted boundary stratum $\mathcal{S}$ is admissible if and only if the set \[ Q(\mathcal{S}) := \{ Q(w) : w \in {\rm Weight}(\mathcal{S}) \} \]
 is not contained in a closed half-space of its $\QQ$-span.
\end{Q}
Then we give also explicit equations in terms of cross-ratios that cut out the subvarieties $\bS(h) \subset \bS$.
In the boundary stratum of irreducible stable curves of geometric genus zero, which we also call the stratum of trinodal type, this is the following.

\newtheorem*{CRequation}{Theorem \ref{CRequation}}
\begin{CRequation}
 Let $r=(r_1,r_2,r_3)$ be a $\ZZ$-basis of $\II$, let $s=(s_1,s_2,s_3)$ be its dual basis of $\II^\vee$ with respect to the pseudo-trace pairing and let $\mathcal{S}_r$ be the $\II$-weighted boundary
 stratum of trinodal curves with weigths~$r$. Moreover, if $h$ is an element of $\Sym_{\QQ}(F) / (\Lambda_1+\Sym_{\ZZ}(\II))$, then we write
 \[ h = \sum_{i,j = 1}^3 b_{ij}(r_i \otimes r_j)\] with $b_{ij} \in \QQ$ and $b_{ij}=b_{ji}$. We identify $\mathcal{S}_r$ with $\MM_{0,6}$, the moduli space of six labeled points in $\PC$,
 and denote the cross-ratios by \[ p_{jk} := (p_j,q_j;q_k,p_k) \in \CC \setminus \{ 0,1 \} \] for $p = (p_1,p_2,p_3,q_1,q_2,q_3)  \in \PC^6$ with pairwise distinct entries.
 Then, the subvariety $\mathcal{S}_r(h)$ of $\mathcal{S}_r$ is cut out by the equations \[ p_{23}^{a_1} \cdot p_{13}^{a_2} \cdot p_{12}^{a_3} = \exp(-2\pi i (a_1b_{23}+a_2b_{13}+a_3b_{12})), \]
 where $(a_1,a_2,a_3)$ runs over the integral solutions of \begin{equation*} a_1s_2s_3+a_2s_1s_3+a_3s_1s_2=0. \end{equation*}
\end{CRequation}
As mentioned in the introduction, note that contrary to the cubic case, in general we do not get a single equation anymore.
We get an analogous result for a second type of stratum, which we will call to be of \emph{nice non-trinodal type}.

A \emph{cusp packet} for a pseudo-cubic order~$\OO$ is a pair $(\II,E_h(\II))$ with a lattice~$\II$ in~$F$ and a certain symplectic extension class~$E_h(\II)$.
We will see that any boundary point of the pseudo-real multiplication locus corresponds to a cusp packet.
Hence, at least for the strata relevant for Teichm\"uller curves, we can use the cross-ratio equations to get the following description of the boundary of the eigenform locus for
pseudo-real multiplication.
\newtheorem*{hinreichend}{Theorem \ref{hinreichend}}
\begin{hinreichend}
 Let $\OO \subset F$ be a pseudo-cubic order of degree relatively prime to the conductor of $D$, and let $(\II,E_h(\II)) \in \mathcal{C}(\OO)$ be a cusp packet for $\OO$.
 Furthermore, let $\iota: F \hookrightarrow \RR$ be one of the two real quadratic pseudo-embeddings.
 
 Then the intersection of the closure of the cusp of $\mathbb{P}\overline{\Omega^\iota\RM_\OO}$ associated to $(\II,E_h(\II))$ with the union of the strata of trinodal and nice non-trinodal type is equal
 to the union of all $p_\iota(\mathcal{S}(h))$, the $\iota$-pseudo-embeddings of the subvarieties~$\mathcal{S}(h)$ of~$\mathcal{S}$, where~$\mathcal{S}$ runs over all admissible $\I$-weighted boundary
 strata of trinodal and nice non-trinodal type.
\end{hinreichend}

Finally, for doing a crosscheck, for every $n \in \NN$ we take the Teichm\"uller curve $C_n$ generated by a specific $S$-shaped surface $T_n$ (known as Thurston-Veech-construction) in the locus
$\Omega E_{4(2n+1)}$.
We show that a cusp indeed lies in the image of some admissible $\II$-weighted boundary stratum and we compute the (in this case single) defining cross-ratio equation for the subvariety $\bS(h)$
containing $C_n$.
\clearpage

\section{Complex Abelian varieties}

Compact Riemann surfaces of a fixed genus $g$ can be parameterized via the Torelli map by their Jacobians, principally polarized complex Abelian varieties of dimension $g$.
Although this is our main motivation, arbitrary (not necessarily principally) polarized complex Abelian varieties are objects also worth for studying them in detail.
In this paper non-principal polarized Abelian varieties appear as subvarieties of higher dimensional principally polarized ones.
Here we summarize standard notations for complex Abelian varieties and a section with technical lemmata on symplectic forms on lattices.

\subsection{Complex tori and Abelian varieties}\label{AbelianVarieties}

We recall some definitions and properties concerning complex Abelian varieties. If not other stated, we follow \cite{LangeBirke} and refer the reader there for more details.
\\[1em]
{\bf Complex tori and period matrices.}
A \emph{lattice} in a finite-dimensional complex vector space $V$ is a discrete subgroup $\Lambda$ of $V$ of maximal rank.
Thus, if~$V$ has complex dimension~$g$, then $\Lambda$ is a free Abelian group of rank~$2g$.
The quotient $A = V/\Lambda$, equipped with the Lie group structure inherited by $V$, is called a \emph{complex torus}, a connected compact complex Lie group of dimension~$g$.
Conversely, each connected compact complex Lie group $A$ of dimension $g$ is isomorphic to a complex torus of dimension~$g$, since the universal cover $V$ of $A$ is a complex vector space of
dimension~$g$ and the kernel of the universal covering map $\pi: V \to A$ is a lattice in $V$.

Any subset $A_1$ of a complex torus $A = V/\Lambda$ is a subtorus if and only if it is of the form $A_1 = V_1/\Lambda_1$,
where $V_1 \subseteq V$ is a subspace such that the intersection $\Lambda_1 = V_1 \cap \Lambda$ is a lattice in $V_1$.
A complex torus is said to be \emph{simple}, if it does not have any non-trivial subtorus.

A \emph{homomorphism} of complex tori is a holomorphic map compatible with the group structures.
Each homomorphism $f: V/\Lambda \to V'/\Lambda'$ of complex tori can be lifted to a unique $\CC$-linear map $F: V \to V'$. We call $\rho_a(f):=F$ the \emph{analytic representation} of $f$ and
its restriction to the lattices $\rho_r(f):=F|_{\Lambda} : \Lambda \to \Lambda'$ the \emph{rational representation of} $f$.
Conversely, a $\CC$-linear map $F: V \to V'$ descends to a homomorphism of complex tori if and only if $F(\Lambda) \subseteq \Lambda'$.
A homomorphism of complex tori is called an \emph{isogeny}, if it is surjective with finite kernel, and this holds if and only if its analytic representation is an isomorphism of vector spaces.
Consequently, an isomomorphism is an isogeny with trivial kernel, respectively $\rho_r(f)(\Lambda)=\Lambda'$. 

If $(e_1,...,e_g)$ is a $\CC$-basis of $V$ and $(\lambda_1,...,\lambda_{2g})$ is a $\ZZ$-basis of $\Lambda$, then we can write $\lambda_j = \sum_{i=1}^g \alpha_{ij} e_i$ and call the matrix
\[ \Pi = \begin{pmatrix}  \alpha_{1,1} & \cdots & \cdots & \alpha_{1,2g} \\ \vdots & & & \vdots \\ \alpha_{g,1} & \cdots & \cdots & \alpha_{g,2g} \end{pmatrix} \in \bM_{g,2g}(\CC) \]
a \emph{period matrix} for $A$. Let $\Pi$ and $\Pi'$ be period matrices for the tori $A=V/\Lambda$ and $A'=V'/\Lambda'$ with respect to certain bases and let $f: A \to A'$ be a homomorphism.
Then $\rho_a(f)$ is represented by a matrix $M \in \bM_{g,g}(\CC)$ with respect to the chosen $\CC$-bases and $\rho_r(f)$ is represented by a matrix $R \in \bM_{2g,2g}(\ZZ)$ with respect to the chosen
$\ZZ$-bases and the relation \[ M\Pi=\Pi'R \] holds.
\\[1em]
{\bf Complex Abelian varieties and polarizations.}
From the viewpoint of algebraic geometry, a polarization on a complex torus $A=V/\Lambda$ is by definition the first Chern class of a positive definite holomorphic line bundle on $A$.
For our purposes it is more convenient to take an equivalent definition in terms of Hermitian forms.
\begin{defi}
 Let $A=V/\Lambda$ be a complex torus. A \emph{polarization} on $A$ is a positive definite Hermitian form on $H: V \times V \to \CC$, such that ${\rm Im}(\Lambda \times \Lambda) \subseteq \ZZ$.
\end{defi}
If there exists a polarization $H$ on a complex torus $A$, then $A$ is called a \emph{complex Abelian variety} and the pair $(A,H)$ is called a \emph{polarized complex Abelian Variety}.
A \emph{homomorphism of polarized complex Abelian varieties} is a homomorphism of complex tori, such that the analytic representation preserves the Hermitian forms.
For simplicity, we will ommit the term 'complex' often.

Any positive definite Hermitian form $H$ on $V$ gives rise to a non-degenerated alternating $\RR$-bilinearform \[ E_H:~ V \times V \to \RR,\quad E_H(v,w):={\rm Im}(H(v,w)) \] with the property that
$E_H(iv,iw)=E_H(v,w)$ holds for all $v,w \in V$. Conversely, any non-degenerated alternating $\RR$-bilinearform $E$ on $V$ with this property defines a positive definite Hermitian form
\[ H_E:~ V \times V \to \CC,\quad H_E(v,w):=E(iv,w)+iE(v,w). \]
The assignment $H \mapsto H_E$ is a group isomorphism between the \emph{N\'{e}ron-Severi group} $NS(A)$, the group of Hermitian forms on $V$ with ${\rm Im}(\Lambda \times \Lambda) \subseteq \ZZ$,
and the group of alternating $\RR$-bilinearforms on $V$ with $E(\Lambda \times \Lambda) \subseteq \ZZ$ and $E_H(iv,iw)=E_H(v,w)$ for all $v,w \in V$.
Its inverse is $E \mapsto H_E$ and it induces a bijection between the positive definite Hermitian forms and the non-degenerated alternating forms.

A \emph{symplectic form} on a free $\ZZ$-module $\Lambda$ is a non-degenerated alternating bilinearform $E$ with integer values.
A \emph{symplectic module homomorphism} of modules with symplectic forms is a module homomorphism preserving the forms.
If the module with symplectic form $E$ is of finite rank, then it is of even rank $2g$ and there exists a $\ZZ$-basis $(\lambda_1,...,\lambda_g,\mu_1,...,\mu_g)$ of $\Lambda$, such that $E$ is given by
a matrix of the form \[ \begin{pmatrix} 0 & D \\ -D & 0 \end{pmatrix} \] with respect to this basis, where $D = \diag(d_1,...,d_g)$ with $d_i \in \NN$ and $d_i \mid d_{i+1}$ for all $i$.
Such a basis is called a \emph{symplectic basis} and the tuple $(d_1,...,d_g)$, as well as the diagonal matrix $D$, is called the \emph{type} of $E$.
The type does not depend on the choice of the basis, and so does the \emph{degree} of $E$, defined as $\deg(E) := \prod_{i=1}^g d_i$.
If $H$ is a polarization of an Abelian variety $V/\Lambda$ and if we speak of the type, respectively the degree of $H$, we mean the type, respectively the degree of $E_H|_{\Lambda \times \Lambda}$.
We also refer to a symplectic basis for $E_H$ as a symplectic basis for $H$.
A polarization $H$ is said to be \emph{principal}, if it is of type~$(1,...,1)$. That is, if and only if $E_H|_{\Lambda \times \Lambda}$ is unimodular.

Each subtorus $A_1$ of an Abelian variety $A$ is also an Abelian variety, because any polarization $H$ on $A=V/\Lambda$ induces the polarization $H_1=H|_{V_1 \times V_1}$ on~$A_1$.
Let $V_2$ be the orthogonal complement of $V_1$ in $V$ with respect to $H$. Then $\Lambda_2=V_2 \cap \Lambda$ is a lattice in $V_2$ and $A_2=V_2/\Lambda_2$ is called the complementary subvariety
of~$A_1$ (with respect to $H$). In the case that $H$ is principal, we have the following relation between the types of the induced polarizations on complementary subvarieties.
\begin{prop}\label{complementarytoritypes}
 Let $(A,H)$ be a principally polarized Abelian variety and let $A_1,A_2 \subseteq A$ be complementary subvarieties with $\dim(A_1) \leq \dim(A_2)$.
 If $(A_1,H_1)$ is of type $(d_1,...,d_h)$, then $(A_2,H_2)$ is of type $(1,...,1,d_1,...,d_h)$.
\end{prop}
This is Corollary~12.1.5. in \cite{LangeBirke}. In the next subsection, we will give a similar proof for the slightly more general situation, where we do not have a complex structure on our objects.
\\[1em]
{\bf The moduli space of polarized Abelian varieties.}
Given any such type $D=\diag(d_1,...,d_g)$, polarized Abelian varieties of type $D$ can be constructed in the following way. We denote by \[ \HH_g := \{ Z \in \bM_g(\CC) : Z^t=Z,~ {\rm Im}(Z)>0 \} \]
the \emph{upper Siegel half space}, i.e. the space of symmetric complex $g\times g$-matrices with positive definite imaginary part. For each $Z \in \HH_g$ we define $\Lambda_Z$ to be the lattice in
$\CC^g$ generated by the columns of~$\Pi_Z:=(Z,D) \in \bM_{g \times 2g}(\CC)$. Then the bilinearform $H_Z$ on $\CC^g$ defined by the matrix ${\rm Im}(Z)^{-1}$ with respect to the standard basis is a
polarization on the complex torus $A_Z := \CC^g/\Lambda_Z$ of type~$D$. A symplectic basis $B_Z$ is given by the columns of~$\Pi_Z$.

An isomorphism between two triples $(A_1,H_1,B_1)$ and $(A_2,H_2,B_2)$ where $(A_i,H_i)$ is a polarized Abelian variety and $B_i$ a symplectic basis for $H_i$, is an isomorphism between $(A_1,H_1)$
and $(A_2,H_2)$ such that the rational representation maps the ordered tuple $B_1$ onto $B_2$. It is a consequence of the Riemann Relations (see \cite{LangeBirke}, Section~4.2)
that any such triple is isomorphic to $(A_Z,H_Z,B_Z)$ for some $Z \in \HH_g$. Thus the upper Siegel half space $\HH_g$, a complex manifold of dimension $\tfrac{1}{2}g(g+1)$,
is a moduli space for polarized Abelian varieties of type $D$ together with the choice of a symplectic basis.
To get rid of the choice of a basis, let $\Lambda_D$ be the $\ZZ$-module generated by the columns of the matrix $\left( \begin{smallmatrix} I_g & 0 \\ 0 & D \end{smallmatrix} \right)$ and let
$G_D$ be the subgroup of the symplectic group $\Sp_{2g}(\QQ)$ consisting of those matrices $M$ with $M^t\Lambda_D \subset \Lambda_D$.
Then $G_D$ acts transitively, properly and discontinously on $\HH_g$ via
\[ \left( \begin{smallmatrix} \alpha & \beta \\ \gamma & \delta \end{smallmatrix} \right).Z := (\alpha Z+\beta)(\gamma Z+\delta)^{-1}. \]
Given any two $Z_1,Z_2 \in \HH_g$, the polarized Abelian varieties $(A_{Z_1},H_{Z_1})$ and $(A_{Z_2},H_{Z_2})$ are isomorphic if and only if $Z_1$ and $Z_2$ are in the same $G_D$-orbit.
\begin{theo}[\cite{LangeBirke}, Theorem~8.2.6.]
 The normal complex analytic space $\mathcal{A}_D := \HH_g/G_D$ is a moduli space for polarized Abelian varieties of type $D$.
\end{theo}
\noindent
{\bf Endomorphism structures and real multiplication.}
In this paper, we are interested in Abelian varieties with a specific endomorphism structure. In general, the endomorphism $\QQ$-algebra $\End_\QQ(A):=\End(A) \otimes_\ZZ \QQ$ of an Abelian variety~$A$
is determined by the endomorphism structure of its simple subvarieties. The reason is that for any two complementary subvarieties $A_1,A_2$, there is an isogeny between $A$ and $A_1 \times A_2$.
Then, by induction, we get the following decomposition.
\begin{theo}[Poincar\'e's Complete Reducibility Theorem]
 Every complex Abelian variety $A$ is isogenous to a product \[ A_1^{n_1} \times ... \times A_r^{n_r}, \] with pairwise not isogenous simple Abelian varieties $A_i$.
 Moreover, the natural numbers $n_i$ are uniquely determined by $A$ and the Abelian varieties $A_i$ are uniquely determined by $A$ up to isogeny.
\end{theo}
Given such an isogeny $A \to A_1^{n_1} \times ... \times A_r^{n_r}$ as in the theorem, there is an isomorphism of $\QQ$-algebras \[ \End_\QQ(X) \cong \bM_{n_1}(F_1) \oplus ... \oplus \bM_{n_r}(F_r), \]
where each $F_i=\End_\QQ(A_i)$ is a skew field of finite dimension over $\QQ$. Any choice of a polarization $H$ on an Abelian variety $A$ determines an anti-involution on $\End_\QQ(A)$,
the \emph{Rosati involution}. On the level of analytic representations, the Rosati involution just maps a $\CC$-linear map to its adjoint with respect to~$H$.
If $A$ is simple and the Rosati involution is the identity on the centre of the skew field $F=\End_\QQ(A)$, then there are three possibilities for~$F$.
It can be a totally real number field, a totally indefinite quaternion algebra, or a totally definite quaternion algebra.
In the first case, the endomorphism ring $\End(A)$ coincides with the subring \[ \End^+(A) := \{ f \in \End(A) :~ H(\rho_a(f)(\cdot),\cdot) = H(\cdot,\rho_a(f)(\cdot)) \} \]
of endomorphisms self-adjoint with respect to the polarization, and this leads to the definition of real multiplication.
\begin{defi}
 Let $F$ be a totally real number field of degree $g$ and let $(A,H)$ be a polarized Abelian variety of dimension $g$. \emph{Real multiplication} on $(A,H)$ by~$F$ is a $\QQ$-algebra monomorphism
 \[ \rho :~ F \hookrightarrow \End^+(A) \otimes \QQ. \]
 If such a monomorphism exists, we say that $A$ \emph{admits real multiplication} by~$F$.
\end{defi}
This is the definition from \cite{McMDynamics}, without the condition of the polarization being principal.
A \emph{lattice} in a finite dimensional $\QQ$-vector space is a free Abelian subgroup of rank equal to the dimension of the vector space.
An \emph{order} in a number field is a subring containing $1$ which is also a lattice.
If $\rho$ is real multiplication on $(A,H)$ by a totally real number field~$F$, then the preimage $\rho^{-1}(\End(A))$ is an order in~$F$.
\begin{defi}\label{RealMultO}
 Let $\OO$ be an order in a totally real number field $F$ of degree $g$ and let $(A,H)$ be a polarized Abelian variety of dimension $g$.
 \emph{Real multiplication} on $(A,H)$ by~$\OO$ is a ring-monomorphism \[ \rho :~ \OO \hookrightarrow \End^+(A) \]
 that is \emph{proper}, i.e. it does not extend to a larger order in $F$.
 If such a monomorphism exists, we say that $A$ \emph{admits real multiplication} by~$\OO$.
\end{defi}
It is well known that the moduli space of principally polarized Abelian varieties admitting real multiplication by $\OO$ together with a choice of such real multiplication is the disjoint union
of Hilbert modular varieties $\Gamma \backslash \HH^g$, see \cite{MartinMatt}.
\\[1em]
{\bf Jacobians and the Torelli map.}
For each $g \in \NN$, we denote by $\Ag$ the moduli space of principally polarized Abelian varieties of dimension $g$ and by $\MM_g$ the moduli space of compact Riemann surfaces of genus $g$.
Note that, throughout this paper, Riemann surfaces are always meant to be connected.
There is a natural embedding $\MM_g \hookrightarrow \Ag$ defined as follows. Given any compact Riemann surface $X$ of genus $g$, the $\CC$-vector space $\Omega(X)$ of holomorphic one-forms is of
dimension~$g$. The first singular homology group $H_1(X;\ZZ)$ can be embedded into the dual vector space $\Omega(X)^* := \Hom_\CC(\Omega(X);\CC)$ via
\[ H_1(X;\ZZ) \hookrightarrow \Omega(X)^*,\quad \gamma \mapsto \left( \omega \mapsto \int_\gamma \omega \right). \]
We identify $H_1(X;\ZZ)$ with its image under this map, a lattice in $\Omega(X)^*$.
The complex torus \[ \Jac(X) := \Omega(X)^* / H_1(X;\ZZ) \] is an Abelian variety and is called the \emph{Jacobian of $X$}. A polarization is given by the intersection form.
To be more precise, recall that for each homology class $\gamma \in H_1(X;\ZZ)$, there exists a unique class $\gamma^* \in H_{\rm dR}^1(X)$ in the first de Rham cohomology group satisfying
$\smallint_\gamma \omega = \smallint_X \gamma^* \wedge \omega$ for all $\omega \in H_{\rm dR}^1(X)$, the \emph{Poincar\'e dual} of~$\gamma$.
The \emph{intersection pairing} on $X$ is defined by
\[ \langle \cdot,\cdot \rangle:~ H_1(X;\ZZ) \times H_1(X;\ZZ) \to \ZZ,\quad \langle \gamma_1,\gamma_2 \rangle := \int_X \gamma_1^* \wedge \gamma_2^*, \]
an unimodular symplectic form on $H_1(X;\ZZ)$. Moreover, its $\RR$-linear extension to $\Omega(X)^*$ satisfies $\langle iv,iw \rangle = \langle v,w \rangle$ for all $v,w \in \Omega(X)^*$.
Thus it induces a principal polarization $\Theta_X := H_{\langle \cdot,\cdot \rangle}$ on $\Jac(X)$.
We have the following well known result.
\begin{theo}[Torelli]\label{Torelli}
 The \emph{Torelli map} \[ t:~ \MM_g \to \Ag,\quad X \mapsto (\Jac(X),\Theta_X) \] is injective.
\end{theo}

\subsection{Symplectic forms}

Here, we prove some technical lemmata about symplectic forms on lattices needed throughout this paper. Except for the last one, they are propably scattered throughout the literature.
For the readers convenience, we recall them at this point.

\bigskip
The first one concerns the type of the form induced by a sublattice.
\begin{lem}\label{Gradindex}
 Let $\Lambda$ be a free $\ZZ$-module of finite rank and let $E$ be a symplectic form on $\Lambda$.
 Then, for every submodule $\Lambda' \leq \Lambda$ of maximal rank and the induced form $E'$ on $\Lambda'$ we have \[ \deg E' = |\Lambda / \Lambda'| \cdot \deg E. \]
\end{lem}
\begin{proof}
 Let $(\lambda_1,...,\lambda_{2g})$ and $(\lambda_1',...,\lambda_{2g}')$ be bases of the free modules $\Lambda$ and $\Lambda'$.
 Then, for some $A \in \ZZ^{2g \times 2g}$ with $\det A \not= 0$ we have \[ (\lambda_1',...,\lambda_{2g}')^t = A (\lambda_1,...,\lambda_{2g})^t, \] in particular $|\Lambda / \Lambda'| = |\det A|$.
 Let $Z, Z' \in \ZZ^{2g \times 2g}$ be the matrices representing the symplectic forms $E, E'$ with respect to the chosen bases. Then $Z' = AZA^t$ and therefore
 \[ \deg E' = |\det Z'|^{\frac{1}{2}} = |\det A| \cdot |\det Z|^{\frac{1}{2}} = |\Lambda / \Lambda'| \deg E. \]
\end{proof}

\noindent
{\bf Rational tori.}
Analogous to the concept of complex tori, a \emph{rational torus} is just a quotient $V/\Lambda$, where $V$ is a finite dimensional vector space over $\QQ$ and $\Lambda$ a lattice in $V$, i.e.
a free Abelian subgroup with $\rk_\ZZ(\Lambda) = \dim_\QQ(V)$.
A \emph{homomorphism of complex tori} $T_1=V_1/\Lambda_1$, $T_2=V_2/\Lambda_2$ is a group homomorphism $f:~ T_1 \to T_2$ induced by a $\QQ$-linear map $f_a:~ V_1 \to V_2$ that maps $\Lambda_1$ into
$\Lambda_2$. Analogous to the complex tori, we call $f_a$ the \emph{analytic representation} of $f$ and $f_r:= f_a|_{\Lambda_1}:~ \Lambda_1 \to \Lambda_2$ the \emph{rational representation} of $f$.
Given any such rational torus $T=V/\Lambda$, its \emph{dual torus} is defined by $\widehat{T} = \widehat{V} / \widehat{\Lambda}$, with $\widehat{V} := \Hom_\QQ(V,\QQ)$ and
$\widehat{\Lambda} := \{ \alpha \in \widehat{V}: \alpha(\Lambda) \subseteq \ZZ \}$.
Any homomorphism $f:~ T_1 \to T_2$ induces a well-defined homomorphism $\widehat{f}:~ \widehat{T}_2 \to \widehat{T}_1$ on the dual tori by
$[\alpha]_{\widehat{\Lambda}_2} \mapsto [\alpha \circ f_a]_{\widehat{\Lambda}_1}$.
Thus, we get a contravariant functor from the category of rational tori to itself, and this functor turns out to be exact.
\begin{lem}
 Let \[ 0 \to T_1 \to T_2 \to T_3 \to 0  \] be an exact sequence of rational tori. Then, the dual sequence \[ 0 \to \widehat{T}_3 \to \widehat{T}_2 \to \widehat{T}_1 \to 0  \]
 is also exact.
\end{lem}
\begin{proof}
 The rational representation of the sequence \[ 0 \to T_1 \to T_2 \to T_3 \to 0  \] yields a sequence \[ 0 \to \Lambda_1 \to \Lambda_2 \to \Lambda_3 \to 0. \]
 By considering the $\QQ$-span of each point in $\Lambda_i$, one can see that this sequence of the lattices is exact.
 This also holds for the dual sequence \[ 0 \to \Hom(\Lambda_3;\ZZ) \to \Hom(\Lambda_2;\ZZ) \to \Hom(\Lambda_1;\ZZ) \to 0, \] where exactness at the last position follows from the fact that the sequence
 of the free Abelian $\ZZ$-modules $\Lambda_i$ splits.
 Identifying each $\Hom(\Lambda_i;\ZZ)$ with $\widehat{\Lambda}_i$, we get a commutative diagramm 
 \[ \begin{CD}
  0 @>>> \widehat{\Lambda}_3 @>>> \widehat{\Lambda}_2 @>>> \widehat{\Lambda}_1 @>>> 0 \\
  @. @VVV @VVV @VVV @. \\
  0 @>>> \widehat{V}_3 @>>> \widehat{V}_2 @>>> \widehat{V}_1 @>>> 0,
 \end{CD} \]
 where the vertical arrows are the natural inclusions. Since these inclusions are injective, the induced sequence of the cokernels
 \[ 0 \to \widehat{T}_1 \to \widehat{T}_2 \to \widehat{T}_3 \to 0 \] is exact.
\end{proof}
Now we can proof the more general version of Proposition~\ref{complementarytoritypes} which does not make use of a complex structure.
\begin{lem}\label{complementarytypes}
 Let $V$ be a $2g$-dimensional vector space over $\QQ$ with an alternating non-degenerated $\QQ$-bilinearform $E$. Furthermore, let $V_1$ be a subspace of $V$ of dimension at most $g$, such that $E$
 is non-degenerated on $V_1$. Furthermore, let $\Lambda$ be a lattice in $V$, such that $(\Lambda,E)$ is of type $(1,..,1)$.
 Then, denoting the orthogonal complement of $V_1$ in $V$ by $V_2$ and writing $\Lambda_i = V_i \cap \Lambda$ for $i \in \{ 1,2 \}$, we have:
 If $(\Lambda_1,E)$ is of type $(d_1,...,d_h)$, then $(\Lambda_2,E)$ is of type $(1,...,1,d_1,...,d_h)$.
\end{lem}
\begin{proof}
 A priori, $(\Lambda_2,E)$ is of some type $(e_1,...,e_{g-h})$.
 For each $i \in \{ 1,2 \}$, consider the group \[ K_i := \{ [v]_{\Lambda_i} \in V_i / \Lambda_i :~ E(v,\Lambda_i) \subseteq \ZZ \}. \]
 By choosing symplectic bases, we immediately see \[ K_1 \cong \bigoplus_{i=1}^h \ZZ/d_i\ZZ \quad {\rm and} \quad K_2 \cong \bigoplus_{i=1}^{g-h} \ZZ/e_i\ZZ. \]
 Hence, it suffices to show that $K_1$ is isomorphic to $K_2$.
 Let $T := V / \Lambda$ and $T_i := V_i / \Lambda_i$ be the resulting rational tori.
 Since $(\Lambda,E)$ is of type $(1,...,1)$, the map \[ \Phi :~ T \to \widehat{T},\quad [v]_\Lambda \mapsto [E(v,\cdot)]_{\widehat{\Lambda}} \] is an isomorphism.
 Denoting by $\iota_i:~ T_i \hookrightarrow T$ the natural inclusions, we get \[ K_i = \ker(\widehat{\iota}_i \circ \Phi \circ \iota_i) = \iota_i^{-1} ( \ker(\widehat{\iota}_i \circ \Phi) ). \]
 We claim that it suffices to show $\ker(\widehat{\iota}_1 \circ \Phi) = T_2$. Indeed, then $K_1 = \iota_1^{-1}(T_2) = T_1 \cap T_2$ would follow and thus by symmetry $K_1=K_2$.\\
 To prove $\ker(\widehat{\iota}_1 \circ \Phi) = T_2$, we first note that we have
 \[ \ker(\widehat{\iota}_1 \circ \Phi) = \{ [v]_\Lambda \in V / \Lambda :~ E(v,\Lambda_i) \subseteq \ZZ \}. \]
 Hence, we have $T_2 \subseteq \ker(\widehat{\iota}_1 \circ \Phi)$, as $V_2$ is the orthogonal complement of $V_1$. Denoting the connected component containing zero with the subscript~'$0$',
 we get by equality of the dimensions of the rational tori \[ T_2 =  (\ker(\widehat{\iota}_1 \circ \Phi))_0. \]
 Therefore, we have to show that $\ker(\widehat{\iota}_1 \circ \Phi) \cong  \ker(\widehat{\iota}_1)$ is connected.
 Consider the exact sequence \[ 0 \to T_1 \xrightarrow[]{\iota_1} T \to T/T_1 \to 0, \] where $T/T_1 = V /(\Lambda+V_1) \cong V_2/\pi_2(\Lambda)$ is indeed a rational torus
 (here, the map $\pi_2:~ V \to V_2$ denotes the natural projection).
 Then, by the preceeding lemma, the dual sequence \[ 0 \to \widehat{T/T_1} \to \widehat{T} \xrightarrow[]{\widehat{\iota}_1} \widehat{T}_1 \to 0 \] is also exact.
 It follows that $\ker(\widehat{\iota}_1) = \widehat{T/T_1}$ is a rational torus, in particular it is connected.
\end{proof}

\noindent
{\bf Lagrangian subgroups.}
Any symplectic basis of a symplectic $\ZZ$-module determines a decomposition into convenient submodules defined as follows.
\begin{defi}
 Let $\Lambda$ be a free $\ZZ$-module of rank $2g$ and let $E$ be a symplectic form on $\Lambda$.
 \begin{enumerate}
  \item A subgroup $\Lambda_1 \leq \Lambda$ is called a \emph{Lagrangian subgroup} of $\Lambda$, if $\Lambda / \Lambda_1 \cong \ZZ^g$ (or equivalently,
        $\Lambda_1 \cong \ZZ^g$ and $\Lambda / \Lambda_1$ is torsion-free) and $E|_{\Lambda_1 \times \Lambda_1} \equiv 0$.
  \item A pair $(\Lambda_1,\Lambda_2)$ of two Lagrangian subgroups of $\Lambda$ is called a \emph{decomposition} for $\Lambda$, if $\Lambda = \Lambda_1 \oplus \Lambda_2$.
 \end{enumerate}
\end{defi}
Obviously, any choice of a symplectic basis determines a decomposition. Conversely, the next lemma shows that each decomposition is induced by a suitable symplectic basis.
\begin{lem}\label{Lagrange1}
 If $(\Lambda_1,\Lambda_2)$ is a decomposition for $(\Lambda,E)$, then there exists a symplectic basis $(\lambda_1,...\lambda_g,\mu_1,...,\mu_g)$ for $(\Lambda,E)$
 with $\lambda_i \in \Lambda_1$ and $\mu_i \in \Lambda_2$ for all $i \in \{ 1,...,g \}$.
\end{lem}
\begin{proof}
 First, choose bases $(\lambda_1,...\lambda_g)$ and $(\mu_1,...,\mu_g)$ of $\Lambda_1$ and $\Lambda_2$. Then the matrix of $E$ with respect to the basis $(\lambda_1,...\lambda_g,\mu_1,...,\mu_g)$
 is of the form \[ \begin{pmatrix} 0 & Z \\ -Z^t & 0 \end{pmatrix} \] for some $Z \in \ZZ^{g \times g}$ with $\det(Z) \not= 0$. If $A_1 (\lambda_1,...\lambda_g)^t$ and $A_2 (\mu_1,...,\mu_g)^t$
 with $A_i \in GL_2(\ZZ)$ are new bases of $\Lambda_1$ and $\Lambda_2$, then the matrix of $E$ with respect to the new basis of $\Lambda$ is \[ \begin{pmatrix} A_1 & 0 \\ 0 & A_2 \end{pmatrix}
 \begin{pmatrix} 0 & Z \\ -Z^t & 0 \end{pmatrix} \begin{pmatrix} A_1 & 0 \\ 0 & A_2 \end{pmatrix}^t = \begin{pmatrix} 0 & A_1 Z A_2^t \\ -(A_1 Z A_2^t)^t & 0 \end{pmatrix}. \]
 By the elementary divisor theorem, one can choose $A_1,A_2$, such that $A_1 Z A_2^t$ is of the form diag$(d_1,...d_g)$ with $d_i \mid d_{i+1}$ for all $i \in \{ 1,...,g-1 \}$.
\end{proof}
If the form is unimodular and a submodule is given by a submodule of a Lagrangian subgroup, then the type of the induced form is given by the elementary divisors.
\begin{lem}\label{Lagrange2}
 Let $E$ be of type $(1,...,1)$ and let $(\Lambda_1,\Lambda_2)$ be a decomposition for $(\Lambda,E)$. Moreover, let $\Lambda_1'$ be a subgroup of $\Lambda_1$ of $\ZZ$-rank $g$ and let $(d_1,...,d_g)$
 be the type of $E|_{\Lambda' \times \Lambda'}$ with $\Lambda' = \Lambda_1' \oplus \Lambda_2$. Then, there is a basis $(\eta_1,...\eta_g)$ of $\Lambda_1$ such that $(d_1\eta_1,...,d_g\eta_g)$ is a basis
 of $\Lambda_1'$.
\end{lem}
\begin{proof}
 By the previous lemma, there is a symplectic basis $(\lambda_1,...,\lambda_g,\mu_1,...,\mu_g)$ of $\Lambda$ with $\lambda_i \in \Lambda_1, \mu_i \in \Lambda_2$ and a symplectic basis
 $(\lambda_1',...,\lambda_g',\mu_1',...,\mu_g')$ of $\Lambda'$ with $\lambda_i' \in \Lambda_1',\mu_i' \in \Lambda_2$ for all $i \in \{ 1,...,g \}$.
 By dualizing these two $\ZZ$-bases of~$\Lambda_2$, we get \[ {\rm Hom}_{\ZZ}(\Lambda_2,\ZZ) = \langle \mu_1^*,...,\mu_g^* \rangle_{\ZZ} = \langle (\mu_1')^*,...,(\mu_g')^* \rangle_{\ZZ}. \]
 Now, consider the homomorphism \[ \phi :~ \Lambda_1 \to {\rm Hom}_{\ZZ}(\Lambda_2,\ZZ),\quad \lambda \mapsto (\mu \mapsto E(\lambda, \mu)). \]
 Since $\phi(\lambda_i) = \mu_i^*$ for all $i$, it follows that~$\phi$ is an isomorphism, and together with $\phi(\lambda_i') = d_i (\mu_i')^*$ we get
 \[ \Lambda_1 / \Lambda_1' \cong \phi(\Lambda_1) / \phi(\Lambda_1') \cong \bigoplus_{i=1}^g (\ZZ / d_i\ZZ). \]
\end{proof}

We conclude with the proof of a very technical statement needed in Section~\ref{dim3}.
Roughly speaking, it states that for each sublattice $\Lambda \leq \underline{\Lambda}$ and each 'basis' of the quotient $\underline{\Lambda}/\Lambda$, the represantives of this basis can be chosen in a
way such that they differ only by rational multiples from a $\ZZ$-basis of $\Lambda$.

\begin{lem}\label{Rationalbasis}
 Let $\underline{\Lambda}$ be a free $\ZZ$-module of rank $n$, let $\Lambda \leq \underline{\Lambda}$ be a submodule of the same rank $n$ and let $d_1|...|d_n \in \NN$ be the elementary divisors of
 $\Lambda$ in $\underline{\Lambda}$. Moreover, let $(\underline{\lambda}_1,...,\underline{\lambda}_n)$ be a tuple in $\underline{\Lambda}$,
 such that \[ \underline{\Lambda}/\Lambda = \langle [\underline{\lambda}_1] \rangle \oplus ... \oplus \langle [\underline{\lambda}_n] \rangle \] with ${\rm ord}([\underline{\lambda}_i])=d_i$. \\
 Then, there is a basis $(\lambda_1,...,\lambda_n)$ of $\Lambda$ such we can choose for every $i \in \{ 1,...,n \}$ a representative
 $\underline{\lambda}_i' \in [\underline{\lambda}_i]$ with \[ d_i\underline{\lambda}_i'=a_i\lambda_i \] for some $a_i \in \NN$ with $a_i \leq d_i$ and $(a_i,d_i)=1$.
\end{lem}
\begin{proof}
 We can choose a basis $(\mu_1,...,\mu_n)$ of $\Lambda$ such that for all $i \in \{ 1,...,n \}$ we have \[ d_i\underline{\lambda}_i = \sum_{j=i}^na_{ij}\mu_j \] for some $a_{ij} \in \ZZ$. \\
 First, we will show by induction that for all $i \in \{ 1,...,n \}$ we have
 \[ \sum_{k=i+1}^n a_{ik}\mu_k \in d_{i+1} \Lambda + d_{i+1} \langle \underline{\lambda}_{i+1} \rangle + ... + d_n \langle \underline{\lambda}_n \rangle \mbox{ and } (d_i,a_{ii})=1. \]
 We will do this induction in reverse order, so starting with $i=n$. For the first statement, there is nothing to proof and $(d_n,a_{nn})=1$ holds, because the order of
 $[\underline{\lambda}_n]$ is $d_n$ and $d_n\underline{\lambda}_n=a_{nn}\mu_n$. \\
 Now, let $i \in \{ 1,...,n-1 \}$ and let us assume that for all $j>i$ we have
 \[ \nu_j := \sum_{k=j+1}^n a_{jk}\mu_k \in d_{j+1}\Lambda + d_{j+1} \langle \underline{\lambda}_{j+1} \rangle + ... + d_n \langle \underline{\lambda}_n \rangle \mbox{ and } (d_j,a_{jj})=1. \]
 In particular, for all $j>i$ there are $x_j,y_j \in \ZZ$ with $x_jd_j+y_ja_{jj}=1$ and this yields
 \begin{eqnarray*}
  \sum_{j=i+1}^n a_{ij}\mu_j & = & \sum_{j=i+1}^n a_{ij}x_j d_j \mu_j + \sum_{j=i+1}^n a_{ij} y_j a_{jj} \mu_j \\
  & = & \sum_{j=i+1}^n d_j a_{ij} x_j \mu_j + \sum_{j=i+1}^n a_{ij} y_j (d_j \underline{\lambda}_j - \sum_{k=j+1}^n a_{jk}\mu_k) \\
  & = & \sum_{j=i+1}^n d_j a_{ij} x_j \mu_j + \sum_{j=i+1}^n d_j a_{ij} y_j \underline{\lambda}_j - \sum_{j=i+1}^n a_{ij} y_j \nu_j.
 \end{eqnarray*}
 By the induction hypothesis and the fact that $d_k$ divides $d_l$ for $k<l$, it follows
 \[ \nu_i \in d_{i+1}\Lambda + d_{i+1} \langle \underline{\lambda}_{i+1} \rangle + ... + d_n \langle \underline{\lambda}_n \rangle. \]
 In particular, we have \[ \nu_i \in (d_i,a_{ii}) (\Lambda + \langle \underline{\lambda}_{i+1} \rangle + ... + \langle \underline{\lambda}_n \rangle) \] and therefore
 \[ d_i\underline{\lambda}_i = a_{ii}\mu_i + \nu_i \in (d_i,a_{ii}) (\Lambda + \langle \underline{\lambda}_{i+1} \rangle + ... + \langle \underline{\lambda}_n \rangle). \]
 Hence, in $\underline{\Lambda}/\Lambda$ we have \[ \tfrac{d_i}{(d_i,a_{ii})}[\underline{\lambda}_i] \in \langle [\underline{\lambda}_{i+1}] \rangle + ... + \langle [\underline{\lambda}_n] \rangle, \]
 and as the sum of $\langle [\underline{\lambda}_i] \rangle ,..., \langle [\underline{\lambda}_n] \rangle$ is direct in $\underline{\Lambda}/\Lambda$ and the order of $[\underline{\lambda}_i]$ is $d_i$,
 it follows $(d_i,a_{ii})=1$.
 \\[1em]
 Now we know that for all $i \in \{ 1,...,n \}$, there are $x_i,y_i \in \ZZ$ with $x_i d_i + y_i a_{ii} = 1$, so we can write \[ a_{ij}x_id_i = a_{ij} - y_ia_{ii}a_{ij} \] for all $j>i$.
 Thus, we can choose the basis $(\mu_1,...,\mu_n)$ of $\Lambda$ such that $d_i$ divides~$a_{ij}$ for all $i<j$. Changing the class representatives of $[\underline{\lambda}_i]$ allows us to add
 an element of $d_i\ZZ$ to each $a_{ij}$. Therefore, we can assume without loss of generality \[ d_i\underline{\lambda}_i = a_{ii}\mu_i \] with $0 < a_{ii} \leq d_i$.
\end{proof}
\clearpage

\section{Quadratic orders and the trace pairing}

The defining lattice of a two-dimensional principally polarized complex Abelian variety with real multiplication by a real quadratic order $\OO$ is a proper torsion-free $\OO$-module of rank two
together with a unimodular symplectic form compatible with the $\OO$-module structure.
It is well known that any such module is isomorphic to $\OO \oplus \OO^\vee$ as a symplectic $\OO$-module, where $\OO^\vee$ denotes the inverse different and $\OO \oplus \OO^\vee$ is equipped with a
certain symplectic form, the \emph{trace pairing}.
In preparation to parameterize two-dimensional polarized complex Abelian varieties of arbitrary type, we classify in this section such modules where the symplectic form is not unimodular anymore.

\subsection{Quadratic orders and their ideals}

First, we collect some facts about quadratic orders with non-square discriminants and their ideals.
Our focus will be on orders which are not maximal and we assume that the reader is familiar with the ring of integers in a quadratic number field.
If not, we refer to \cite{Schmidt} for all the results we need about quadratic number fields and their maximal orders.
\\[1em]
{\bf Two viewpoints of quadratic orders.}
The usual definition of a quadratic order is the following.
\begin{defi}\label{order}
 For $b,c \in \ZZ$, we call the ring \[ \OO_{b,c} := \ZZ[X] / (X^2+bX+c) \] a \emph{quadratic order}.
 The \emph{discriminant} of $\OO_{b,c}$ is defined to be the integer \[ D := b^2-4c. \]
\end{defi}

For our purposes it will be more convenient to work with quadratic orders as subrings of quadratic number fields instead of quotients of $\ZZ[X]$.
So, let $K$ be a quadratic number field and let \[ \sigma :~ K \to K,\quad x \mapsto x^{\sigma} \] be the Galois-conjugation, i.e. the unique non-trivial automorphism of $K$.
For all $x \in K$ we denote by $\mathcal{N}(x) := xx^{\sigma}$ the \emph{norm of $x$} and by ${\rm tr}(x) := x+x^{\sigma}$ the \emph{trace of $x$}.

Now let $\OO \subset K$ be a ring with $1 \in \OO$ and $\OO \cong \ZZ^2$ as a $\ZZ$-module.
Note that the last condition implies $\OO \otimes_{\ZZ} \QQ = K$. For a $\ZZ$-basis $(\omega_1,\omega_2)$, we call the number
\[ \left( \det \begin{pmatrix} \omega_1 & \omega_2 \\ \omega_1^{\sigma} & \omega_2^{\sigma} \end{pmatrix} \right)^2 \] the \emph{discriminant} of $\OO$,
which is independent of the choice of the basis. We denote by $\OK$ the ring of integers of $K$ and by $\Delta_K$ its discriminant.
Now let $D$ be the discriminant of $\OO$. Elementary computations show the following properties: 
\begin{enumerate}
 \item $K=\QQ(\sqrt{D})$
 \item $D \in \ZZ$, $D \equiv 0 \mod 4$ or $D \equiv 1 \mod 4$, $D$ is not a square
 \item $\OO = \langle 1,\tfrac{D+\sqrt{D}}{2} \rangle_{\ZZ}$, in particular $\OO$ is uniquely determined by $D$.
 \item If we write $K = \QQ(\sqrt{d})$ with $d \in \ZZ$ squarefree, then there is a natural number $f$, called the \emph{conductor of $\OO$}, with
 \[ D=f^2 \Delta_{K}=\left \{ \begin{array}{ll} f^2d & \mbox{ if } d \equiv 1 \mod 4 \\ f^2 4d & \mbox{ if } d \equiv 2,3 \mod 4 \end{array} \right. . \]
 \item If $f$ is the conductor of $\OO$, then for every $\omega \in K$ with $\OK = \langle 1,\omega \rangle_{\ZZ}$ we have
 \[ \OO = \ZZ + f\OK = \langle 1,f\omega \rangle_{\ZZ} \subset \OK, \] and therefore $|\OK : \OO| = f$ and $\OO^{\sigma}=\OO$.
\end{enumerate}

On the other hand, for any non-square integer $D \equiv 0,1 \mod 4$, we set \[ \gamma_D :=\tfrac{D+\sqrt{D}}{2} \] and see that the $\ZZ$-module \[ \OD := \langle 1, \gamma_D \rangle_{\ZZ} \]
is a subring of $K=\QQ(\sqrt{D})$ with $1 \in \OD$ and isomorphic to $\ZZ^2$ as a $\ZZ$-module.

\bigskip
Now let $\OO_{b,c}$ be a quadratic order in the sense of Definition~\ref{order} and let $D=b^2-4c$ be its discriminant. Again we can see by elementary computations:
\begin{enumerate}
 \item $D \equiv 0 \mod 4$ or  $D \equiv 1 \mod 4$
 \item If $D$ is any integer with $D \equiv 0 \mod 4$ or  $D \equiv 1 \mod 4$, then the discriminant of the quadratic order $\OO_{D,\frac{D^2-D}{4}}$ is $D$.
 \item $\OO_{b,c}$ is an integral domain if and only if $D$ is not a square.
 \item $\OO_{b,c} = \langle 1,X \rangle_{\ZZ} \cong \ZZ^2$ as a $\ZZ$-module
 \item If $D$ is not a square, then \[ \mathcal{O}_{b,c} \cong \OD \] as rings, with an isomorphism defined by \[ \quad X \mapsto \gamma_D - \tfrac{D+b}{2}. \]
 \item If $D$ is a square, say $D=d^2$, then we have an injective ring homomorphism \[ \OO_{b,c} \hookrightarrow \ZZ^2 \] defined by \[ \quad X \mapsto (\tfrac{-b+d}{2},\tfrac{-b-d}{2}). \]
       The image of $\OO_{b,c}$ is \[ \OO_{d^2} := \{ (x,y) \in \ZZ^2 : x \equiv y \mod d \}. \]
\end{enumerate}

So, for every $D \in \ZZ$ with $D \equiv 0,1 \mod 4$ it makes sense to call $D$ a \emph{discriminant} and $\OD$ a \emph{quadratic order}.
If $D$ is not a square, then $\OD$ is an order in a quadratic number field $K$, and if in addition $D$ is positive, then $\OD$ is called a \emph{real quadratic order}.
The ring of integers $\OK$ is called the \emph{maximal order in $K$} (and indeed it is the unique maximal order with respect to inclusion)
and finally a discriminant is called \emph{fundamental} if it is the discriminant of a maximal order.
For the rest of this paper, we deal with the case that $D$ is not a square. 
\\[1em]
{\bf Ideals.}
We are interested in ideals of non-maximal orders. Many properties, like the multiplicativity of the norm function or the factorization of ideals into prime ideals,
work quite different than in maximal orders. Our aim is to restrict ourselves to those ideals which 'behave' like ideals of maximal orders.
If not other stated, proofs for all the statements below about quadratic orders which are not maximal can be found in \cite{Cox}, \S7.
The statements in \cite{Cox} are mostly formulated for imaginary quadratic number fields, but note that all these proofs work also without the condition 'imaginary'.

\bigskip
The first big difference to maximal orders is that the set of fractional ideals with respect to the multiplication is in general not a group, but just a monoid.
The invertible ideals are precisely the proper ones, defined as follows.
\begin{defi}
 Let $\OO$ be an order in a quadratic number field $K$. A fractional ideal $\A$ of $\OO$ is said to be \emph{proper} if and only if \[ \OO = \{ x \in K : x\A \subset \A \}. \]
\end{defi}

\begin{lem}
 Let $\OO$ be an order in a quadratic number field $K$ and let $\A$ be a fractional ideal of $\OO$. Then the following conditions are equivalent.
 \begin{enumerate}
  \item $\A$ is proper.
  \item $\A$ is invertible.
  \item $\A$ is projective as an $\OO$-module.
 \end{enumerate}
 Moreover, in this case the inverse of $\A$ is given by \[ \A^{-1} = \{ x \in K : x\A \subset \OO \}. \]
\end{lem}
\noindent
The equivalence of the last condition to the previous ones can be found in \cite{Ribenboim}, Chapter~I.6~(e).

In maximal orders all fractional ideals are proper.
As an example of a non-proper ideal, let $\OO$ be the order of conductor $2$ in $K=\QQ(\sqrt{-3})$, i.e. $\OO = \ZZ[\sqrt{-3}]$.
Setting $\A := \langle 2,1+\sqrt{-3} \rangle_{\ZZ}$, we have \[ \{ x \in K : x\A \subset \A \} = \OK \not= \OO. \] Hence, $\A$ is an example for an ideal which is not proper and thus not invertible.

By ${\rm I}(\OO)$ we denote the group of all proper fractional ideals of $\OO$ and by~$I_K$ the group of all fractional ideals of $\OK$.
For an ideal $\A \trianglelefteq \OO$ we denote by $\mathfrak{N}(\A) := [\OO : \A]$ the \emph{norm} of $\A$.

\begin{lem}
 Let $\OO$ be an order in a quadratic number field. The norm function has the following properties.
 \begin{enumerate}
  \item $\mathfrak{N}(x\OO) = \mathcal{N}(x)$ for all $x \in \OO \setminus \{ 0 \}$.
  \item $\mathfrak{N}(\A \B) = \mathfrak{N}(\A) \mathfrak{N}(\B)$ if at least one of the ideals $\A,\B \trianglelefteq \OO$ is proper.
  \item $\A \A^{\sigma} = \mathfrak{N}(\A) \OO$ for all proper ideals $\A \trianglelefteq \OO$.
 \end{enumerate}
\end{lem}
The second statement can be found in \cite{Cohen}, Proposition~4.6.8..
As examples where the last two equations do not hold, let $\OO$ be an order in a quadratic number field $K$ with conductor $f \not= 1$, that is $\OO = \ZZ + f\OK$, and
let $\A$ be the ideal of $\OO$ defined as $\A := f\OK$, so $\mathfrak{N}(\A)=f$. Then, we have $\A^2 = f^2\OK$ and hence $\mathfrak{N}(\A^2) = f^3 \not= \mathfrak{N}(\A)^2$.
We also have $\A^{\sigma} = \A$ and therefore $\A \A^{\sigma} = f^2\OK \not= \mathfrak{N}(\A)\OO$.
By the previous lemma, it follows that $\A$ is not proper for $f \not= 1$. Therefore, we can always find non-proper ideals of non-maximal orders.

It turns out that the condition 'proper' is still to weak for the factorization into prime ideals. This leads to the following definition.

\begin{defi}
 Let $\OO$ be an order in a quadratic number field, $\A \trianglelefteq \OO$ an ideal which is not zero and $n \in \ZZ$.
 Then $\A$ is called \emph{relatively prime to $n$} if \[ \A + n\OO = \OO. \]
\end{defi}
This is a stronger condition than being proper, and it can be easily verfied by computing the norm.
\begin{lem}
 Let $\OO$ be an order in a quadratic number field with conductor $f$ and let $\A \trianglelefteq \OO$ be an ideal which is not zero.
 Then, $\A$ is relatively prime to $f$ if and only if $\mathfrak{N}(\A)$ is relatively prime to $f$. Furthermore, if $\A$ is relatively prime to $f$, then $\A$ is proper.
\end{lem}
Considering just ideals relatively prime to a given conductor, one can switch between the ideals in the corresponding order and the ideals in the maximal order in the following sense.
\begin{lem}\label{Idealkorrespondenz}
 Let $\OO$ be an order in a quadratic number field $K$ with conductor~$f$.
 \begin{enumerate}
  \item If $\A$ is an ideal of $\OO$ relatively prime to $f$, then $\A \OK$ is an ideal of $\OK$ relatively prime to $f$ of the same norm and $\A \OK \cap \OO = \A$.
  \item If $\A$ is an ideal of $\OK$ relatively prime to $f$, then $\A \cap \OO$ is an ideal of $\OO$ relatively prime to $f$ of the same norm and $(\A \cap \OO)\OK = \A$.
 \end{enumerate}
\end{lem}

For an order $\OO$ in a quadratic number field $K$ with conductor~$f$, we denote by ${\rm I}^+(\OO,f)$ the set of ideals in $\OO$ relatively prime to $f$ and
by ${\rm I}_K^+(f)$ the set of ideals in $\OK$ relatively prime to $f$.

For all $\A,\B \in I(\OO)$ we have $(\A\B)\OK = (\A\OK)(\B\OK)$ and for all $\A \in I^+(\OO,f)$ the map \[ \OO / \A \to \OK / \A \OK ,\quad [x] \mapsto [x] \] is a ring isomorphism.
Together with Lemma~\ref{Idealkorrespondenz}, this yields the following corollary.

\begin{cor}
 Let $\OO$ be an order in a quadratic number field $K$ with conductor~$f$.
 The map \[ {\rm I}^+(\OO,f) \to {\rm I}_K^+(f),\quad \A \mapsto \A \OK  \] is a norm preserving monoid isomorphism.
 Its inverse map is \[ {\rm I}_K^+(f) \to {\rm I}^+(\OO,f),\quad \A \mapsto \A \cap \OO. \]
 An ideal $\A$ of $\OO$ relatively prime to $f$ is a prime ideal if and only if $\A \OK$ is a prime ideal of $\OK$.
\end{cor}

For this paper, the most important consequences of this corollary are the following.
Every ideal of $\OO$ relatively prime to $f$ has a unique factorization into prime ideals, we get this factorization in a natural way by the prime factorization in $\OK$,
and we can use the ramification law for prime integers not dividing $f$.
Moreover, for two ideals $\A,\B$ of $\OO$ relatively prime to $f$, the inclusion $\A \subset \B$ holds if and only if $\B \mid \A$.

\bigskip
How do ideals of orders in a quadratic number field look like?
For any ideal $\A \trianglelefteq \OD$ we have $\A \cap \NN \not= \emptyset$, thus we can always find a $\ZZ$-basis for $\A$ of the form $(n,a+b\gamma_D)$
with $n,a,b \in \ZZ$. We can even choose $n,b>0$ and $0 \leq a < n$, and in this case the basis is unique.
Elementary computations lead to the following technical lemma needed in Section~\ref{dim2}.

\begin{lem}\label{Idealnormalform}
 Let $\OD$ be an order in a quadratic number field, let $n,a,b \in \ZZ$ with $n,b>0$ and let $\omega := a+b\gamma_D$. Then, we have for $\A := \langle n,\omega \rangle_{\ZZ}$ the following.
 \begin{enumerate}
  \item $\A$ is an ideal of $\OD$ if and only if $b \mid a$, $b \mid n$ and $n \mid \tfrac{\mathcal{N}(\omega)}{b}$.
  \item $\mathfrak{N}(\A) = nb$.
  \item If $\A$ is an ideal of $\OD$, then $\A$ is \emph{primitive} 
        (that is, if $m$ is an integer, then $\tfrac{1}{m}\A \subset \OD$ implies $m \in \{ -1,1 \}$) if and only if $b=1$.
  \item The primitive ideal $\A = \langle n,a+\gamma_D \rangle_{\ZZ}$ is invertible if and only if \[ \gcd(n,\tfrac{\mathcal{N}(a+\gamma_D)}{n},\tr(a+\gamma_D)) = 1. \]
        In this case, we have $\A^{-1} = \tfrac{1}{n}\A^{\sigma}$. \\
        More generally, an arbitrary ideal $\B = \langle \alpha,\beta \rangle_{\ZZ}$ of $\OD$ is invertible if and only if
        \[ \gcd(\mathcal{N}(\alpha),\mathcal{N}(\beta),\tr(a^\sigma\beta)) = \mathfrak{N}(\B). \] In this case, we have \[ \B^{-1} = \tfrac{1}{\mathfrak{N}(\B)}\B^{\sigma}. \]
 \end{enumerate}
\end{lem}

For many problems it suffices to consider ideals just up to multiples in $K$.
Two fractional ideals $I,J$ of a quadratic order $\OO$ are said to be \emph{equivalent}, if there is some $x \in K$, such that $xI = J$.
Furthermore, they are said to be \emph{strictly equivalent}, if there is some $x \in K$ with $\mathcal{N}(x)>0$, such that $xI = J$.
Then, we have by \cite{Halter-Koch}, Theorem~5.5.7.2(b) the following very useful lemma.

\begin{lem}\label{HalterKoch}
 Let $I$ be an invertible, fractional ideal of $\OO$. Then, for every positive integer $M$, there is an invertible, primitive ideal $\A$ of $\OD$ with
 $\gcd(\mathfrak{N}(\A),M) = 1$ and $\mathfrak{N}(\A) \geq 1$, such that $I$ and $\A$ are strictly equivalent.
\end{lem}
\noindent
The proof uses the continued fraction expansion of quadratic irrationals.
There are two immediate consequences we are going to use.
\begin{cor}\label{Idealaequiv}
 Let $\OO$ be a quadratic order.
 \begin{enumerate}
  \item Every invertible ideal $\OO$ is strictly equivalent to an ideal relatively prime to the conductor of $\OO$.
  \item Let $\A,\B$ be ideals relatively prime to the conductor of $\OO$. Then, $\A$ is strictly equivalent to an ideal $\A'$ of $\OO$ such that $\A'$ and $\B$ are relatively prime.
 \end{enumerate}
\end{cor}
\subsection{Symplectic modules over real quadratic orders}\label{Tracepairing}

We are interested in polarized complex Abelian surfaces with real multiplication (in the sense of Definition~\ref{RealMultO}).
Real multiplication by a real quadratic order $\OD$ on the surface $X/\Lambda$ makes the lattice $\Lambda$ to a proper, torsion-free $\OD$-module of rank two carrying the symplectic form induced by
the polarization. Here the rank of $\Lambda$ is the dimension of $\Lambda \otimes_\ZZ \QQ$ as a $K$-vector space.
This form is \emph{compatible with the $\OD$-module structure}, i.e. the action of $\OD$ is self-adjoint with respect to the form.

The goal of this section is, to classify symplectic $\OD$-modules arising from polarized complex Abelian surfaces with real multiplication.
We will see that for a given type and a given real quadratic order $\OD$, the isomorphism classes of these modules correspond to the ideals of $\OD$ with norm equal to the quotient of the type of the
polarization.

\begin{defi}\label{symplecticmodule}
 Let $D>0$ be a non-square discriminant. A \emph{symplectic $\OD$-module} is a pair $(\Lambda,E)$, where $\Lambda$ is a torsion-free rank two $\OD$-module and $E$ is a symplectic form on~$\Lambda$
 compatible with the $\OD$-module structure.\\
 An isomorphism of symplectic $\OD$-modules is an $\OD$-module isomorphism preserving the forms.
\end{defi}

Now we introduce the most important example for symplectic forms on rank two $\OD$-modules, those induced by the trace pairing in $\QQ(\sqrt{D}) \oplus \QQ(\sqrt{D})$.

\bigskip
Recall that, if $\OD$ is the real quadratic order of discriminant $D$, then we have $\OD \subset \QQ(\sqrt{D})$ and $\OD = \langle 1,\gamma \rangle_{\ZZ}$ with
$\gamma := \gamma_D = \tfrac{D+\sqrt{D}}{2}$. From now on, we fix $D$ to be a positive, non-square discriminant and set $K=\QQ(\sqrt{D})$.
Writing $K = \QQ (\sqrt{d})$ with $d$ a square-free integer and $x=a+b\sqrt{d} \in K$ with $a,b \in \QQ$, we call $\R(x):=a$ the \emph{invariant part of $x$} and $\I(x):=b$ the
\emph{antiinvariant part of $x$}. Again we denote for all $x \in K$ by $x^{\sigma}$ the Galois-conjugate, by $\mathcal{N}(x) := xx^{\sigma}$ the norm and by
${\rm tr}(x):=x+x^{\sigma} = 2 \R(x)$ the trace of $x$. For any ideal $\A$ of $\OD$, we define the \emph{inverse different of $\A$} as the fractional ideal
\[ \ADual := \{ x \in K : {\rm tr}(xy) \in \ZZ \mbox{ for all } y \in \A \}. \] Using Lemma~\ref{Idealnormalform}, an elementary computation yields
\[ \ADual = \tfrac{1}{\sqrt{D}} \tfrac{1}{\mathfrak{N}(\A)} \A^{\sigma} \] for any ideal $\A$. Thus, in the case of $\A = \OD$ we get $\ODual = \tfrac{1}{\sqrt{D}} \OD$.
We will now define a symplectic form on the $\OD$-module $\OD \oplus \ODual$.

\begin{defi}
 The \emph{trace pairing} on $\OD \oplus \ODual$ is defined by
 \[ \langle \cdot , \cdot \rangle :~ (\OD \oplus \ODual)^2 \to \ZZ,\quad \langle (x,y),(\tilde x,\tilde y) \rangle := {\rm tr}(\tilde x y-x \tilde y). \]
\end{defi}

One can easily check that the trace pairing is a symplectic form of type~$(1,1)$.
A symplectic basis is given by \[ ((1,0),(\gamma,0),(0,\tfrac{1}{\sqrt{D}}\gamma^{\sigma}),(0,\tfrac{-1}{\sqrt{D}})), \]
which we call the \emph{standard symplectic basis of $(\OD \oplus \ODual,\langle \cdot,\cdot \rangle)$}.
In particular, $(\OD \oplus \{ 0 \},\{ 0 \} \oplus \ODual)$ is a decomposition for $(\OD \oplus \ODual,\langle \cdot,\cdot \rangle)$.
Moreover, the action of $\OD$ on $\OD \oplus \ODual$ is self-adjoint with respect to the trace pairing, i.e.
\[ \langle x \lambda,\lambda' \rangle = \langle \lambda,x \lambda' \rangle \] for all $x \in \OD$ and $\lambda, \lambda' \in \OD \oplus \ODual$.

In order to get symplectic forms of arbitrary types $(d_1,d_2)$, we consider submodules of $\OD \oplus \ODual$.
We will show that not every type will appear, and this leads to the following definition.

\begin{defi}
 Let $d_1,d_2 \in \NN$ with $d_1 \mid d_2$ and $\frac{d_2}{d_1}$ relatively prime to the conductor of $D$.
 Let $\prod p_i^{k_i}$, with $k_i \in \NN$, be the prime factorization of the quotient $\frac{d_2}{d_1}$.
 We say that the pair $(d_1,d_2)$ satisfies the \emph{prime factor condition in~$\OD$}, if the following conditions hold.
 \begin{enumerate}
  \item No $p_i$ is inert over $K$ and
  \item If $p_i$ is ramified over $K$, then $k_i=1$.
 \end{enumerate}
 We say that $d \in \NN$ satisfies the \emph{prime factor condition in $\OD$}, if $(1,d)$ satisfies the prime factor condition in $\OD$. \\
 For a pair $(d_1,d_2) \in \NN^2$ with $d_1 \mid d_2$, we call $\tfrac{d_2}{d_1}$ the \emph{quotient of $(d_1,d_2)$}.
\end{defi}
This is precisely the condition for the type that guarantees the existence of an ideal inducing this type with respect to the trace pairing.
\begin{prop}\label{Primfaktorbedingung}
 Let $d_1,d_2 \in \NN$ with $d_1 \mid d_2$ and $\frac{d_2}{d_1}$ relatively prime to the conductor of $D$.
 Then, there exists an ideal $\A \trianglelefteq \OD$ such that the trace pairing on $\A \oplus \ODual$ is of type $(d_1,d_2)$ if and only if $(d_1,d_2)$ satisfies  the prime factor condition in $\OD$.
 The same holds for $\OD \oplus \tfrac{1}{\sqrt{D}}\A$ instead of $\A \oplus \ODual$.
\end{prop}
\begin{proof}
 Let $\prod p_i^{k_i}$ be the prime factorization of $\frac{d_2}{d_1}$.
 The prime ideal factorization of $p_i \OD$ is either  $p_i \OD$ (if $p_i$ is inert), $\PP_i^2$ (if $p_i$ is ramified) or $\PP_i \PP_i^{\sigma}$ with $\PP_i \not= \PP_i^{\sigma}$ (if $p_i$ splits). \\
 Let $\A \trianglelefteq \OD$ be an ideal such that the trace pairing on $\A \oplus \ODual$ is of type $(d_1,d_2)$.
 By Lemma~\ref{Lagrange2}, the $\OD$-module $\B := \tfrac{1}{d_1}\A$ is still an ideal of $\OD$, and on $\B \oplus \ODual$ the trace pairing is of type $(1,\tfrac{d_2}{d_1})$.
 Hence, the norm of $\B$ is relatively prime to the conductor of $D$. In particular, $\B$ is proper.
 Now let us assume that $(d_1,d_2)$ does not satisfy the prime factor condition in $\OD$.
 Then, there is an inert or a ramified prime number $p_i$ such that $p_i^{\varepsilon} \mid \tfrac{d_2}{d_1}$ with $\varepsilon = 1$ if $p_i$ is inert and $\varepsilon = 2$ if $p_i$ is ramified.
 Lemma~\ref{Gradindex} implies $\mathfrak{N}(\B) = \tfrac{d_2}{d_1}$, where $\mathfrak{N}$ is the norm function of ideals of $\OD$, thus $p_i^{\varepsilon} \mid \mathfrak{N}(\B)$.
 Since $\B$ is proper, we have $p_i^{\varepsilon} \OD \mid \B \B^{\sigma}$ and therefore, using the unique prime ideal factorization for ideals relatively prime to the conductor, $\B \subset p_i \OD$.
 But, as the trace pairing on $p_i \OD \oplus \ODual$ is of type $(p_i, p_i)$ and on $\B \oplus \ODual$ of type $(1,\tfrac{d_2}{d_1})$, we have a contradiction.

 To prove the converse, suppose that $(d_1,d_2)$ satisfies the prime factor condition in $\OD$.
 The ideal \[ \A := (\prod \PP_i^{k_i})(d_1 \OD) \] has norm $\mathfrak{N}(\A) =d_1d_2$. As $\A \subset d_1\OD$ and as the trace pairing on $d_1\OD \oplus \ODual$ is of type $(d_1,d_1)$,
 the trace pairing on $\A \oplus \ODual$ is of type $(d_1',d_2')$ with $d_1 \mid d_1'$ and, again by Lemma~\ref{Gradindex}, $d_1d_2=d_1'd_2'$. It remains to show that $d_1=d_1'$.
 Assume that there is a prime number $p$ with $p \mid \frac{d_1'}{d_1}$. Let $(\lambda_1,\lambda_2,\mu_1,\mu_2)$ be the standard symplectic basis of $(\OD \oplus \ODual,\langle \cdot,\cdot \rangle)$.
 Then, we have for every $a \in \A$, that $(a,0) = \alpha_1 \lambda_1 + \alpha_2 \lambda_2$ with $\alpha_i \in \ZZ$ and thus $pd_1 \mid \langle (a,0),\mu_i \rangle = \alpha_i$.
 It follows $\A \subset pd_1 \OD$ and therefore $p\OD \mid \prod \PP_i^{k_i}$, which is impossible as $(d_1,d_2)$ satisfies the prime factor condition in $\OD$.
\end{proof}

For any fractional ideal $I$ of $K$, we have $K = \QQ \otimes_{\ZZ} I$.
Therefore, if $I_1,I_2$ are fractional ideals of $K$, any symplectic form on $I_1 \oplus I_2$ extends naturally by tensoring with $\QQ$ to a unique alternating $\QQ$-bilinear form on $K^2$.
We denote this form also by~$E$.

\bigskip
The next step is the construction of symplectic bases.
\begin{lem}\label{Standardbasis}
 Let $\A$ be an ideal of $\OD$ and let $(\eta_1,\eta_2)$ be a $\ZZ$-basis of $\A$. Then, we have possibly after multiplying one of the $\eta_i$ with $-1$ the following identity.
 \[ \I(\eta_1 \eta_2^{\sigma}) = -\tfrac{1}{2}\I(\sqrt{D})\mathfrak{N}(\A). \] In particular, $(\lambda_1,\lambda_2,\mu_1,\mu_2)$ with
 \[ \lambda_1 := (\eta_1,0),~ \lambda_2 := (\eta_2,0),~ \mu_1 := (0,\tfrac{1}{\sqrt{D}}\tfrac{1}{\mathfrak{N}(\A)}\eta_2^{\sigma}),~
 \mu_2 := (0,\tfrac{-1}{\sqrt{D}}\tfrac{1}{\mathfrak{N}(\A)}\eta_1^{\sigma}) \]
 is a symplectic basis of $(\A \oplus \ADual,\langle \cdot,\cdot \rangle_{\QQ})$ and $\langle \cdot,\cdot \rangle_{\QQ}$ is on $\A \oplus \ADual$ of type $(1,1)$.
\end{lem}
\begin{proof}
 Because of $(\eta_1,\eta_2)^t = A (1,\gamma)^t$ with $A \in GL_2(\QQ) \cap \ZZ^{2 \times 2}$, we have
 \[ \I(\eta_1\eta_2^{\sigma}) = -\frac{1}{2} \I(\sqrt{D})\det(A), \]
 and thus
 \[ \langle \lambda_1,\mu_1 \rangle = \langle \lambda_2,\mu_2 \rangle = \det(A). \]
 Using $\mathfrak{N}(\A) = |\det(A)|$, an easy computation gives the second assertion.
\end{proof}
A direct consequence of Lemma~\ref{Lagrange2}, Lemma~\ref{Standardbasis} and the fact, that the Galois conjugation is an isomorphism of $\OD$, is the following corollary,
which allows us to choose a symplectic basis and a basis for the Smith normal form simultaneously.
\begin{cor}
 Let $\A \trianglelefteq \OD$ be an ideal such that the trace pairing on $\A \oplus \ODual$ (respectively $\OD \oplus \frac{1}{\sqrt{D}} \A$) is of type $(d_1,d_2)$.
 Then, there exists a $\ZZ$-basis $(\eta_1,\eta_2)$ of $\OD$ such that \[ ((\eta_1,0),(\eta_2,0),(0,\tfrac{1}{\sqrt{D}}\eta_2^{\sigma}),(0,\tfrac{-1}{\sqrt{D}}\eta_1^{\sigma})) \]
 is a symplectic basis of $(\OD \oplus \ODual,\langle \cdot,\cdot \rangle)$ and \[ ((d_1 \eta_1,0),(d_2 \eta_2,0),(0,\tfrac{1}{\sqrt{D}}\eta_2^{\sigma}),(0,\tfrac{-1}{\sqrt{D}}\eta_1^{\sigma})) \]
 \[ (\mbox{respectively } ((\eta_1,0),(\eta_2,0),(0,\tfrac{d_1}{\sqrt{D}}\eta_2^{\sigma}),(0,\tfrac{-d_2}{\sqrt{D}}\eta_1^{\sigma}))) \]
 is a symplectic basis of $(\A \oplus \ODual,\langle \cdot,\cdot \rangle)$ (respectively $(\OD \oplus \frac{1}{\sqrt{D}}\A,\langle \cdot,\cdot \rangle)$).
\end{cor}

The goal of this section is to show that each proper symplectic $\OD$-module, presupposing that the quotient of the type is relatively prime to the conductor, is isomorphic to one of those constructed
in Proposition~\ref{Primfaktorbedingung}. For that we need the following three technical lemmata.
\begin{lem}\label{Zerlegung}
 Let $I_1,I_2$ be fractional ideals of $K$ and let $E$ be a symplectic form on $I_1 \oplus I_2$ compatible with the $\OD$-module structure.
 Then, $(I_1 \oplus \{ 0 \},\{ 0 \} \oplus I_2)$ is a decomposition for $(I_1 \oplus I_2,E)$.
\end{lem}
\begin{proof}
 It remains to show that $E|_{I_i \oplus I_i} \equiv 0$ for $i \in \{ 1,2 \}$ (here, we identify $I_1$ with $I_1 \oplus \{ 0 \}$ and $I_2$ with $\{ 0 \} \oplus I_2$).
 We extend $E$ to the natural $\QQ$-bilinearform on $K^2$.
 Since the multiplication with elements of $K$ is self-adjoint with respect to $E$ and $E$ is alternating, we have for all $k_1,k_2 \in K$:
 \begin{eqnarray*}
  E((k_1,0),(k_2,0)) & = & E((1,0),(k_1 k_2,0)) \\
  & = & -E((k_1 k_2,0),(1,0)) \\
  & = & -E((k_1,0),(k_2,0))
 \end{eqnarray*}
 This yields $E_{I_1 \oplus I_1} \equiv 0$ and analogous~$E_{I_2 \oplus I_2} \equiv 0$.
\end{proof}

\begin{lem}\label{Prinzipalisierung}
 Let $I_1,I_2$ be fractional ideals of $K$ and let $E$ be a symplectic form on $I_1 \oplus I_2$ compatible with the $\OD$-module structure.
 Then, there exists a fractional ideal $I$ of $K$ with $I_1 \subset I$ (respectively $I_2 \subset I$) such that on $\Lambda := I \oplus I_2$
 (respectively $\Lambda := I_1 \oplus I$) the alternating form $E$ is symplectic of type $(1,1)$.
\end{lem}
\begin{proof}
 We define \[ I := \{ x \in K : \forall y \in I_2: E((x,0),(0,y)) \in \ZZ \}. \]
 Obviously, we have $I_1 \subset I \subset K$ and, since the action of $\OD$ is self-adjoint with respect to $E$, we have that $I$ is an $\OD$-module. \\
 Let $(d_1,d_2)$ be the type of $E$ and let $((\lambda_1,0),(\lambda_2,0),(0,\mu_1),(0,\mu_2))$ be a symplectic basis of $(I_1 \oplus I_2,E)$.
 We can write every $v \in K$ as \[ v = q_1 \tfrac{1}{d_1} \lambda_1 + q_2 \tfrac{1}{d_2} \lambda_2 \] with $q_1,q_2 \in \QQ$.
 This leads to \[ I = \langle \tfrac{1}{d_1} \lambda_1 , \tfrac{1}{d_2} \lambda_2 \rangle_{\ZZ}. \] In particular, $I$ is a fractional ideal and on $\Lambda := I \oplus I_2$ the form $E$ is clearly
 symplectic of type $(1,1)$ with symplectic basis \[ (\tfrac{1}{d_1}(\lambda_1,0),\tfrac{1}{d_2}(\lambda_2,0),(0,\mu_1),(0,\mu_2)) \]
\end{proof}

\begin{lem}\label{Hexenlemma}
 Let $\A$ be an ideal of $\OD$, let $I$ be a fractional ideal of $K$ and let~$E$ be a symplectic form on $\A \oplus I$ of type $(1,1)$ compatible with the $\OD$-module structure.
 Then, there exists an isomorphism of symplectic $\OD$-modules \[ \theta :~ (\A \oplus I,E) \to (\A \oplus \ADual,\langle \cdot , \cdot \rangle) \] with
 $\theta((\lambda,0))=(\lambda,0)$ for all $\lambda \in \A$ and $\theta(\{ 0 \} \oplus \ADual) = \{ 0 \} \oplus \ADual$.
\end{lem}
\begin{proof}
 By Lemma~\ref{Zerlegung} the pair $(\A \oplus \{ 0 \},\{ 0 \} \oplus I)$ is a decomposition for $(\Lambda ,E)$ with $\Lambda := \A \oplus I$. Hence, by Lemma~\ref{Lagrange1},
 there are $\ZZ$-bases $(\lambda_1,\lambda_2)$ of $\A$ and $(\mu_1,\mu_2)$ of $I$, such that \[ ((\lambda_1,0),(\lambda_2,0),(0,\mu_1),(0,\mu_2)) \]
 is a symplectic basis of $(\Lambda,E)$. By Lemma~\ref{Standardbasis}, there is a $\ZZ$-basis $(\tilde \mu_1,\tilde \mu_2)$ of~$\ADual$ such that
 \[ ((\lambda_1,0),(\lambda_2,0),(0,\tilde \mu_1),(0,\tilde \mu_2)) \] is a symplectic basis of $\A \oplus \ADual$. Then, the assignment
 \[ (\lambda_i,0) \mapsto (\lambda_i,0),~ (0,\mu_i) \mapsto (0,\tilde \mu_i) \mbox{ for all } i \in \{ 1,2 \} \]
 defines a symplectic isomorphism of groups \[ \theta :~ (\Lambda,E) \to (\A \oplus \ADual,\langle \cdot , \cdot \rangle) \]
 with $\theta((\lambda,0))=(\lambda,0)$ for all $\lambda \in \A$ and $\theta(\{ 0 \} \oplus I) = \{ 0 \} \oplus \ADual$.
 It remains to show that $\theta$ is $\OD$-linear. \\
 By construction, we clearly have $\theta(x(\lambda,0)) = x\theta((\lambda,0))$ for all $x \in \OD,\lambda \in \A$.
 Now, let $x \in \OD$ and $\mu \in I$. Then, we have
 \[ \theta (x(0,\mu)) = \alpha_1 \tilde \mu_1 + \alpha_2 \tilde \mu_2 \mbox{ and } x \theta ((0,\mu)) = \beta_1 \tilde \mu_1 + \beta_2 \tilde \mu_2 \]
 with $\alpha_i,\beta_i \in \ZZ$. Since multiplying by $x$ is self-adjoint, we get for all $i \in \{ 1,2 \}$
 \begin{eqnarray*}
  \alpha_i & = & \langle (\lambda_i,0),\theta (x(0,\mu)) \rangle \\
  & = & E((\lambda_i,0),x(0,\mu)) \\
  & = & E((x \lambda_i,0),(0,\mu)) \\
  & = & \langle (x \lambda_i,0),\theta (0,\mu) \rangle \\
  & = & \langle (\lambda_i,0),x \theta (0,\mu) \rangle \\
  & = & \beta_i
 \end{eqnarray*}
 and, therefore, \[ \theta (x(0,\mu)) = x \theta ((0,\mu)). \]
 Because $\theta$ is $\ZZ$-linear, this finishes the proof.
\end{proof}

A direct consequence of the previous two lemmata and Lemma~\ref{Primfaktorbedingung} is, that in this situation there do not exist symplectic forms of any arbitrary type.
\begin{cor}
 Let $I$ be a fractional ideal of $K$ and let $E$ be a symplectic form on $\OD \oplus I$ of type $(d_1,d_2)$ compatible with the $\OD$-module structure.
 If the quotient $\tfrac{d_2}{d_1}$ is relatively prime to the conductor of $D$, then $(d_1,d_2)$ satisfies the prime factor condition in $\OD$.
\end{cor}

For torsion-free $\OD$-modules of rank one (i.e. fractional ideals), properness is equivalent to invertibility and also equivalent to projectivity.
In the rank two case, projectivity is a strictly stronger condition if $\OD$ is not maximal, since $\Lambda$ is isomorphic to the direct sum of two fractional ideals (see the proof below).
Nevertheless, our lattices carry a symplectic form of a type relatively prime to the conductor of $D$.
The next theorem shows in particular that this additional condition is enough to guarantee $\Lambda$ to be projective.

\begin{theo}\label{Struktursatz}
 Let $(\Lambda,E)$ be a proper symplectic $\OD$-module of type $(d_1,d_2)$ with $\tfrac{d_2}{d_1}$ relatively prime to the conductor of $D$.
 Then, there is a unique ideal $\A$ of $\OD$ such that \[ (\Lambda,E) \cong (\OD \oplus \tfrac{1}{\sqrt{D}}\A, \langle \cdot , \cdot \rangle) \] as symplectic $\OD$-modules.
 In particular, $\Lambda$ is a projective $\OD$-module and $\tfrac{d_2}{d_1}$ satisfies the prime factor condition in~$\OD$.
\end{theo}
\begin{proof}
 First we show the existence. If we had $(\Lambda,\tfrac{1}{d_1}E) \cong (\OD \oplus \tfrac{1}{\sqrt{D}}\A, \langle \cdot , \cdot \rangle)$ for some ideal $\A$ of $\OD$, then also
 $(\Lambda,E) \cong (\OD \oplus \tfrac{1}{\sqrt{D}}(d_1\A), \langle \cdot , \cdot \rangle)$ would hold.
 Hence, we may assume without loss of generality that $(d_1,d_2) = (1,d)$.

 As an $\OD$-module, $\Lambda$ is torsion-free, every ideal of $\OD$ can be generated by at most two elements and
 the integral closure of $\OD$ is the maximal order $\OK$. Hence, it follows by \cite{Bass}, Theorem 1.7, that $\Lambda$ is isomorphic to $\A_0 \oplus \B_0$ as an $\OD$-module,
 where $\A_0$ and $\B_0$ are ideals of $\OD$. We can choose $\A_0$ and $\B_0$ to be primitive.
 We denote the induced alternating form on $\A_0 \oplus \B_0$ also by~$E$.

 The first step is, to show that $\A_0$ and $\B_0$ are proper as $\OD$-ideals. By Lemma~\ref{Prinzipalisierung} and Lemma~\ref{Hexenlemma}, we know that there is a symplectic $\OD$-module isomorphism
 \[ (\A_0 \oplus \B_0 , E) \cong (\A_0 \oplus I , \langle \cdot , \cdot \rangle) \] for some fractional ideal $I \subseteq \A_0^\vee$.
 The ideal $\A_0$ is not proper as an $\OD$-module if and only if there is a prime number $p$ dividing the conductor of $D$ such that for the order $\mathcal{O}_E \supsetneq \OD$ with $E=p^{-2}D$
 we have $\mathcal{O}_E \A_0 \subseteq \A_0$. We will show that such an order would also act on $I$.
 If we write $\A_0 = \langle n, a+\gamma_D \rangle_\ZZ$, we see that $\mathcal{O}_E \A_0$ is contained in $\A_0$ if and only if
 \begin{align*}
  \gamma_E \cdot n & = (\tfrac{E(1-p)}{2}-\tfrac{a}{p}) \cdot n + \tfrac{n}{p} \cdot (a+\gamma_D)\quad {\rm and} \\
  \gamma_E \cdot (a+\gamma_D) & = -\tfrac{\mathcal{N}(a+\gamma_D)}{pn} \cdot n + (\tfrac{a+D}{p}+\tfrac{E(1-p)}{2}) \cdot (a+\gamma_D)
 \end{align*}
 are both elements of $\A_0$. Using \[ n \mid \mathcal{N}(a+\gamma_D) = a^2 + aD + \tfrac{D(D-1)}{2} \] and considering the case $p=2$ separately (we have $p \mid a$ for $p$ odd), we deduce that this is equivalent to
 \begin{equation}\label{pn}
   p \mid n \quad {\rm and} \quad pn \mid \mathcal{N}(a+\gamma_D).
 \end{equation}
 We want to deduce that $\mathcal{O}_E \A_0 \subseteq \A_0$ implies $\mathcal{O}_E I \subseteq I$.
 Since $(\A_0 \oplus \A_0^{\vee} , \langle \cdot , \cdot \rangle)$ is of type $(1,1)$ and $\A_0^\vee = \tfrac{1}{\sqrt{D}}\tfrac{1}{\mathcal{N}(\A_0)}\A^\sigma$,
 the fractional ideal $I$ is of the form \[ I = \tfrac{1}{\sqrt{D}}\tfrac{1}{\mathcal{N}(\A_0)}J^\sigma \] for an $\OD$-ideal $J \subseteq \A_0$ with $[\A_0 : J]=d$.
 Thus, we can write \[ J = \langle m, b+c\gamma_D \rangle_\ZZ \] for some $m,b,c \in \NN_0$ with $n \mid m$ and $mc = nd$.
 We have to show that the conditions in \eqref{pn} hold for the $\OD$-ideal $\langle \tfrac{m}{c}, \tfrac{b}{c}+\gamma_D \rangle_\ZZ$.
 The first one is clear, as $p \mid n \mid m$, $c \mid d$ and $p \nmid d$.
 For the second one, we remark that \eqref{pn} implies that the norm of every element in $\A_0$ is a multiple of $pn$.
 Indeed, for $\lambda=kn+l(a+\gamma_D)$ with $k,l \in \ZZ$, we compute \[ \mathcal{N}(\lambda) = l^2\mathcal{N}(a+\gamma_D) + kn(kn+2la+lD) \] and use $p \mid a$ for $p$ odd.
 Since $J \subseteq \A_0$, we have $pn \mid \mathcal{N}(b+c\gamma_D)$. Together with $\tfrac{m}{c} \mid \mathcal{N}(\tfrac{b}{c}+\gamma_D)$ by Lemma~\ref{Idealnormalform}
 and $p \nmid d$, we get \[ p\tfrac{m}{c} \mid \mathcal{N}(\tfrac{b}{c}+\gamma_D). \]
 We have shown that $\mathcal{O}_E \A_0 \subseteq \A_0$ implies $\mathcal{O}_E I \subseteq I$ and thus $\mathcal{O}_E \Lambda \subseteq \Lambda$.
 In particular, if $\Lambda$ is proper as an $\OD$-module, then $\A_0$ and analogously $\B_0$ are also proper and $\Lambda$ is projective.

 Now, we are in the situation that $\Lambda$ is isomorphic to $\A_0 \oplus \B_0$, where~$\A_0$ and~$\B_0$ are proper ideals of~$\OD$.
 By Corollary~\ref{Idealaequiv}, we may choose $\A_0$ and $\B_0$ to be relatively prime.
 Hence, the $\OD$-module homomorphism \[ \A_0 \oplus \B_0 \to \OD ,\quad (x,y) \mapsto x+y \] is surjective, and as $\OD$ is free, this epimorphism splits.
 Thus, $\Lambda$ is isomorphic to $\OD \oplus I$ for some fractional ideal $I$ of $\OD$.
 Finally, by Lemma~\ref{Prinzipalisierung} and Lemma~\ref{Hexenlemma}, there is a symplectic $\OD$-module isomorphism
 \[ (\OD \oplus I, E) \cong (\OD \oplus \tfrac{1}{\sqrt{D}} \A, \langle \cdot,\cdot \rangle) \] for some ideal $\A \trianglelefteq \OD$.
 In particular, $(d_1,d_2)$ satisfies the prime factor condition in $\OD$.

 \bigskip
 It remains to show that the choice of the ideal $\A$ is unique. Let us assume that there is any ideal $\B$ with
 \[ (\OD \oplus \tfrac{1}{\sqrt{D}}\A, \langle \cdot,\cdot \rangle) \cong (\OD \oplus \tfrac{1}{\sqrt{D}} \B, \langle \cdot,\cdot \rangle) \] as symplectic $\OD$-modules.
 An isomorphism $\varphi$ between the underlying $\OD$-modules is given by multiplication with a matrix
 \[ A \in \GL \begin{pmatrix} \OD & \sqrt{D}\A^{-1} \\ \tfrac{1}{\sqrt{D}}\B & \B\A^{-1} \end{pmatrix}. \]
 Since \[ \langle \varphi(x,y),\varphi(\tilde x,\tilde y) \rangle = \tr(\det(A)(\tilde x y - x \tilde y)), \]
 it follows that $\varphi$ is symplectic if and only if $\det(A)=1$. Hence, we have $1 \in \B\A^{-1}$ and thus $\A \subseteq \B$.
 Since $\A$ and $\B$ have the same norm, they must coincide.
\end{proof}

\begin{cor}\label{2hochs}
 Let $D>0$ be a discriminant and let $d_1 \mid d_2 \in \NN$ such that $d=\tfrac{d_2}{d_1}$ is relatively prime to the conductor of $D$ and satisfies the prime factor condition in~$\OD$.
 Furthermore, we denote by $s$ the number of splitting prime divisors of $d$ in $\OD$.
 Then there are precisely $2^s$ isomorphism classes of proper symplectic $\OD$-modules of type $(d_1,d_2)$.
\end{cor}
\begin{proof}
 This follows immediately from the uniqueness statement in Theorem~\ref{Struktursatz} and
 the construction in Proposition~\ref{Primfaktorbedingung} using the prime ideal decomposition.
\end{proof}

\clearpage

\section{The real multiplication loci in dimension two and three}

The moduli space of principally polarized Abelian surfaces with some choice of real multiplication by $\OD$ is a single Hilbert modular surface $\Gamma \backslash \HH^2$.
In this section we construct the moduli space of principally polarized Abelian varieties of dimension three with some choice of real multiplication by $\OD$ on a two-dimensional subvariety,
such that the induced polarization on this subvariety is of a fixed type. This space turns out to be the disjoint union of Hilbert modular varieties $X_i = \Gamma_i \backslash \HH^3$,
where the $\Gamma_i$ are certain subgroups of $\SL_2(\QQ(\sqrt{D})) \times \SL_2(\QQ)$.
Furthermore, we will find lower and upper bounds for $\Gamma_i$ in $\SL_2(\QQ(\sqrt{D})) \times \SL_2(\QQ)$, both being direct products of a subgroup of $\SL_2(\QQ(\sqrt{D}))$ and a subgroup of 
$\SL_2(\QQ)$, and with both inclusions of finite index. Finally, we interpret the resulting quotients of $\HH^3$ by these groups as the moduli spaces of products of varieties with level structures.

\subsection{The moduli space in dimension two}\label{dim2}

First, we want to parameterize all polarized Abelian surfaces with real multiplication by~$\OD$, where the quotient of the type is relatively prime to the conductor of~$D$.
In accordance with the classification of symplectic~$\OD$-modules in the previous section, we will show that the moduli space of polarized Abelian surfaces of a fixed type together with the choice of
real multiplication by a fixed order~$\OD$ is the union of~$2^s$ Hilbert modular surfaces, where~$s$ is determined by the type and the discriminant.

\bigskip
Let us first specify the definition of real multiplication from Section~\ref{AbelianVarieties} to the two-dimensional case.
\begin{defi}
 Let $(A,H)$ with $A=V/\Lambda$ be a polarized Abelian surface.
 \emph{Real multiplication on $A$ by $\OD$} is an injective ring homomorphism \[ \rho :~ \OD \hookrightarrow {\rm End}(A), \] such that
 \begin{enumerate}
  \item the action of $\OD$ is self-adjoint with respect to $H$\!, i.e. \[ H(\rho (x)(v),w) = H(v,\rho (x)(w)) \] for all $x \in \OD$ and $v,w \in V$
        (here we identify $\rho (x)$ with its analytic representation).
  \item $\rho$ is proper, i.e. there is no injective ring homomorphism \[ \rho': \mathcal{O}_E \hookrightarrow {\rm End}(A) \]
        with $\rho'|_{\OD} = \rho$ for some order $\mathcal{O}_E \supsetneqq \OD$ in $K$.
 \end{enumerate}
 A \emph{polarized Abelian surface with real multiplication by $\OD$} is a triple \[ (A,H,\rho), \] where $(A,H)$ is a polarized Abelian surface and
 $\rho : \OD \hookrightarrow {\rm End}(A)$ a choice of real multiplication on $A$ by $\OD$.
\end{defi}
{\bf Construction of polarized Abelian surfaces with real multiplication.}
Let $(d_1,d_2) \in \NN^2$ with $d_1 \mid d_2$, such that $\tfrac{d_2}{d_1}$ is relatively prime to the conductor of $D$ and satisfies the prime factor condition in $\OD$.
Then by Lemma~\ref{Primfaktorbedingung}, there is an ideal $\A \trianglelefteq \OD$, such that the trace pairing on $\OD \oplus \frac{1}{\sqrt{D}} \A$ is of type $(d_1,d_2)$.
The first step is to construct polarized Abelian varieties with real multiplication by $\OD$ for any such ideal.

\begin{prop}\label{KonstruktionDim2}
 Let $(d_1,d_2) \in \NN^2$ and let $\A \subseteq \OD$ as above.
 Then for each pair~$z = (z_1,z_2) \in \HH^2$, the embedding of groups \[ \phi_z:~ \OD \oplus \tfrac{1}{\sqrt{D}}\A \hookrightarrow \CC^2,\quad (x,y) \mapsto (x+yz_1,x^{\sigma}+y^{\sigma}z_2)\]
 yields an Abelian variety $A_z := \CC^2/ \Lambda_z$ for $\Lambda_z := \phi_z(\OD \oplus \tfrac{1}{\sqrt{D}} \A)$. A polarization~$H_z$ of type $(d_1,d_2)$ is induced by the trace pairing.
 The polarized Abelian variety $(A_z,H_z)$ admits real multiplication by~$\OD$, given by the $\OD$-module structure on $\OD \oplus \tfrac{1}{\sqrt{D}}\A$.
\end{prop}
\begin{proof}
 By Lemma~\ref{Standardbasis} we can take a $\ZZ$-basis $(\eta_1,\eta_2)$ of $\OD$ satisfying the sign convention $\I (\eta_1 \eta_2^{\sigma}) = - \tfrac{1}{2}\I(\sqrt{D})$.
 Then defining  \[ \lambda_1 := (\eta_1,0),~ \lambda_2 := (\eta_2,0),~ \mu_1 := (0,\tfrac{d_1}{\sqrt{D}}\eta_2^{\sigma}),~ \mu_2 := (0,\tfrac{-d_2}{\sqrt{D}}\eta_1^{\sigma}), \]
 implies that $(\lambda_1,\lambda_2,\mu_1,\mu_2)$ is a symplectic basis of $(\OD \oplus \frac{1}{\sqrt{D}} \A, \langle \cdot,\cdot \rangle)$.

 The images of $\lambda_1,\lambda_2,\mu_1,\mu_2$ under $\phi_z$ are linearly independent over $\RR$, hence $A_z = \CC^2 / \Lambda_z$ is a complex torus.
 Let $E_z$ be the extension of the trace pairing to~$\CC^2$, i.e. the unique $\RR$-bilinear form on~$\CC^2$ with \[ E_z(\phi_z(\lambda),\phi_z(\lambda')) = \langle \lambda,\lambda' \rangle \]
 for all $\lambda,\lambda' \in \OD \oplus \frac{1}{\sqrt{D}} \A$.
 The period matrix for $A_z$ with respect to the $\ZZ$-basis \[ (\phi_z(\lambda_1),\phi_z(\lambda_2),\phi_z(\mu_1),\phi_z(\mu_2)) \] of the lattice $\Lambda_z$ and the standard $\CC$-basis
 $((1,0),(0,1))$ of $\CC^2$ is \[ \Pi = \begin{pmatrix} \eta_1 & \eta_2 & \frac{d_1}{\sqrt{D}}\eta_2^{\sigma} z_1 & \frac{-d_2}{\sqrt{D}}\eta_1^{\sigma} z_1 \\
 \eta_1^{\sigma} & \eta_2^{\sigma} & \frac{-d_1}{\sqrt{D}}\eta_2 z_2 & \frac{d_2}{\sqrt{D}}\eta_1 z_2 \end{pmatrix}. \]
 Setting \[ B := \begin{pmatrix} \eta_1 & \eta_2 \\ \eta_1^{\sigma} & \eta_2^{\sigma} \end{pmatrix} \] and then multplying $\Pi$ from the left with
 \[ B^t \begin{pmatrix} -z_1^{-1} & 0 \\ 0 & -z_2^{-1} \end{pmatrix} \] yields
 \[ \Pi' = \left( B^t \begin{pmatrix} -z_1^{-1} & 0 \\ 0 & -z_2^{-1} \end{pmatrix} B , \begin{pmatrix} d_1 & 0 \\ 0 & d_2 \end{pmatrix} \right)\]
 as a period matrix for $A_z$ with respect to the same $\ZZ$-basis of $\Lambda_z$. Let \[ Z := B^t \begin{pmatrix} -z_1^{-1} & 0 \\ 0 & -z_2^{-1} \end{pmatrix} B. \]
 Since our fixed $\ZZ$-basis of $\Lambda_z$ is symplectic with respect to the alternating form~$E_z$, the matrix ${\rm Im}(Z)^{-1}$ is the matrix of the Hermitian form $H_z$ corresponding to $E_z$,
 with respect to the new $\CC$-basis of $\CC^2$. Obviously, we have $Z \in \HH_2$, so $H_z$ is indeed a polarization on $A_z$.

 Real multiplication $\rho_z :~ \OD \hookrightarrow {\rm End}(A_z)$ on $A_z$ by $\OD$ is defined by
 \[ \rho :~ \OD \hookrightarrow {\rm End}_{\CC}(\CC^2),\quad \rho(k)(c_1,c_2) := (kc_1,k^{\sigma}c_2) \] for all $k \in \OD$ and $(c_1,c_2) \in \CC^2$.
 Then we obviously have \[ \phi_z(k(x,y)) = \rho(k) \phi_z((x,y)) \] for all $k \in \OD$ and $(x,y) \in \OD \oplus \frac{1}{\sqrt{D}} \A$.
 In particular, $\rho (k)$ acts on $\Lambda_z$ and is self-adjoint with respect to $E_z$ for all $k \in \OD$, and thus the triple $(A_z,H_z,\rho_z)$ is indeed a
 polarized Abelian surface with real multiplication by $\OD$.
\end{proof}

 An \emph{isomorphism of polarized Abelian surfaces with real multiplication by $\OD$}, say $(A,H,\rho)$ and $(A',H',\rho')$, is an isomorphism of polarized Abelian varieties
 \[ f:~ (A,H) \to (A',H') \] that commutes with the action of $\OD$, i.e. the diagram
 \[ \begin{CD}  A @>f>> A' \\
  @VV\rho(k)V @VV\rho'(k)V \\
  A @>f>> A' \end{CD} \]
 is commutative for all $k \in \OD$.

\begin{prop}\label{SurjektivDim2}
 The triples $(A_z,H_z,\rho_z)$ defined above are up to isomorphism all polarized Abelian surfaces with real multiplication by $\OD$,
 where the quotient of the type is relatively prime to the conductor of $D$.
\end{prop}
\begin{proof}
 Let $(A=V/\Lambda,H,\rho)$ be a polarized Abelian surface with real multiplication by~$\OD$, with the quotient of the type $(d_1,d_2)$ of $(A,H)$ relatively prime to the conductor of~$D$. 
 We identify $\rho$ with its analytic representation \[ \rho:~ \OD \hookrightarrow {\rm End}_{\CC}(V). \]
 Then $\rho$ defines a proper $\OD$-module structure on $\Lambda$. Let $E=E_H$ be the alternating form on $\Lambda$ induced by $H$.
 As an $\OD$-module, $\Lambda$ is torsion-free, and we have $\dim_\QQ(\Lambda \otimes_\ZZ \QQ)=4$.
 Hence, by Theorem~\ref{Struktursatz} there is a unique ideal $\A$ in $\OD$, such that there is an isomorphism
 \[ \phi:~ (\OD \oplus \tfrac{1}{\sqrt{D}} \A, \langle \cdot,\cdot \rangle) \cong (\Lambda, E). \]
 In particular, $(d_1,d_2)$ satisfies the prime factor condition in $\OD$.
 As in the proof of the previous proposition, let $(\eta_1,\eta_2)$ be a $\ZZ$-basis of $\OD$ satisfying the sign convention $\I (\eta_1 \eta_2^{\sigma}) = - \tfrac{1}{2}\I \sqrt{D}$.
 Again we define \[ \lambda_1 := (\eta_1,0),~ \lambda_2 := (\eta_2,0),~ \mu_1 := (0,\tfrac{d_1}{\sqrt{D}}\eta_2^{\sigma}),~ \mu_2 := (0,\tfrac{-d_2}{\sqrt{D}}\eta_1^{\sigma}) \] 
 and get $(\lambda_1,\lambda_2,\mu_1,\mu_2)$ as a symplectic basis of $(\OD \oplus \frac{1}{\sqrt{D}} \A, \langle \cdot,\cdot \rangle)$.
 The ring homomorphism $\rho$ extends uniquely to a homomorphism of $\QQ$-algebras \[ \rho_{\QQ}:~ K \hookrightarrow {\rm End}_{\CC}(V).\]
 
 We will now determine the eigenvalues of $\rho_{\QQ}(k)$. If $k$ is rational, we have $\rho_{\QQ}(k) = k {\rm id}_V$ and $k=k^{\sigma}$ is the only eigenvalue. \\
 Otherwise, if $k \in K \smallsetminus \QQ$, the minimal polynomial of $k$ with respect to the field extension $K / \QQ$ is \[ m(x) = (X-k)(X-k^{\sigma}) = X^2-\tr(k)X + \mathcal{N}(k) \in \QQ[X] \]
 and because of $m(\rho_{\QQ}(k))=0$, we have $\mu_{\rho_{\QQ}(k)}(X) \mid m(X)$, where $\mu_{\rho_{\QQ}(k)}(X)$ denotes the minimal polynomial of $\rho_{\QQ}(k) \in {\rm End}_{\CC}(V)$.
 As the characteristic polynomial of $\rho_{\QQ}(k)$ as an $\RR$-linear map of the 4-dimensional $\RR$-vector space $V$ is equal to the characteristic polynomial of the rational representation
 of $\rho_{\QQ}(k)$ and therefore a rational polynomial, both, $k$ and $k^{\sigma}$, must be eigenvalues of $\rho_{\QQ}(k)$.

 Summarizing, we have seen that for every $k \in K$ the $\CC$-linear map $\rho_{\QQ}(k)$ is diagonalizable and the set of eigenvalues is $\{ k,k^{\sigma} \}$.

 With $\eta_1,\eta_2$ from above, we fix $k:=\frac{\eta_2}{\eta_1} \in K \smallsetminus \QQ$ and choose a $\CC$-basis of~$V$, such that the matrix of $\rho_{\QQ}(k)$
 with respect to this basis is \[ D(k) := \begin{pmatrix} k & 0 \\ 0 & k^{\sigma} \end{pmatrix}. \]
 Since $\rho_{\QQ}$ is a $\QQ$-algebra homomorphism, the matrix of $\rho_{\QQ}(k^{\sigma})$ with respect to the chosen basis is
 \[ \mathcal{N}(k) D(k)^{-1} = \begin{pmatrix} k^{\sigma} & 0 \\ 0 & k \end{pmatrix}. \]
 With respect to this basis we write \[ \phi (\lambda_1) = \begin{pmatrix} r\eta_1 \\ s\eta_1^{\sigma} \end{pmatrix} \mbox{ and }
 \phi (\mu_2) = \begin{pmatrix} t\frac{d_2}{\sqrt{D}}\eta_1^{\sigma} \\ -u\frac{d_2}{\sqrt{D}}\eta_1 \end{pmatrix} \] with $r,s,t,u \in \CC$.
 Using the representation matrices of $\rho_{\QQ}(k)$ and $\rho_{\QQ}(k^{\sigma})$, we deduce that the period matrix of $A$ with respect to the symplectic basis
 $(\phi(\lambda_1),\phi(\lambda_2),\phi(\mu_1),\phi(\mu_2))$ is \[ \Pi = \begin{pmatrix} r\eta_1 & r\eta_2 & t\frac{-d_1}{\sqrt{D}}\eta_2^{\sigma} & t\frac{d_2}{\sqrt{D}}\eta_1^{\sigma} \\
 s\eta_1^{\sigma} & s\eta_2^{\sigma} & u\frac{d_1}{\sqrt{D}}\eta_2 & u\frac{-d_2}{\sqrt{D}}\eta_1 \end{pmatrix}. \]
 The columns have to be linearly independent over $\RR$, so $r,s,t,u \not= 0$.
 Again we set \[ B := \begin{pmatrix} \eta_1 & \eta_2 \\ \eta_1^{\sigma} & \eta_2^{\sigma} \end{pmatrix}, \] and changing the $\CC$-basis of $V$ while keeping the same $\ZZ$-basis of $\Lambda$ yields
 the period matrix \[ \Pi' = \left( B^t \begin{pmatrix} -z_1^{-1} & 0 \\ 0 & -z_2^{-1} \end{pmatrix} B , \begin{pmatrix} d_1 & 0 \\ 0 & d_2 \end{pmatrix} \right) \]
 with $z_1 := -\frac{t}{r}$ and $z_2 := -\frac{u}{s}$. The matrix \[ Z := B^t \begin{pmatrix} -z_1^{-1} & 0 \\ 0 & -z_2^{-1} \end{pmatrix} B \] has to be in the Siegel upper half space because
 ${\rm Im}(Z)^{-1}$ represents the Hermitian form $H$\!, and therefore $z_1,z_2 \in \HH$. Thus we get an isomorphism \[ (A,H) \cong (A_z,H_z) \] with $z:=(z_1,z_2) \in \HH^2$.
 Since for all $k \in \OD$ the matrix $D(k)$ commutes with diagonal matrices, this isomorphism is also compatible with the action of $\OD$, which yields \[ (A,H,\rho) \cong (A_z,H_z,\rho_z). \]
\end{proof}

We have seen that polarized Abelian surfaces of a fixed type with real multiplication by~$\OD$ are classified by certain ideals of $\OD$, and this leads to the following definition.

\begin{defi}
 Let $D \in \NN$ be a non-square discriminant and let $d_1,d_2 \in \NN$ with $d_1 \mid d_2$. Then we denote by~$X_{D,(d_1,d_2)}$ the moduli space of all isomorphism classes $[(A,H,\rho)]$,
 where $(A,H)$ is a polarized Abelian surface of type $(d_1,d_2)$ that admits real multiplication by $\OD$ and $\rho : \OD \hookrightarrow {\rm End}(A)$ is some choice of real multiplication.

 Furthermore, if $\tfrac{d_2}{d_1}$ is relatively prime to the conductor of~$D$, we denote by \[ X_\A \subseteq X_{D,(d_1,d_2)} \] the locus of those isomorphism classes coming from the ideal $\A$
 in the construction above. If $[(A,H,\rho)]$ is in $X_\A$, then we say that $(A,H,\rho)$ is \emph{of type~$\A$}. (Notice that the type $\A$ uniquely determines the type $(d_1,d_2)$.)
\end{defi}

We have seen that for $\tfrac{d_2}{d_1}$ relatively prime to the conductor, $X_{D,(d_1,d_2)}$ is non-empty if and only if $(d_1,d_2)$ satisfies the prime factor condition in $\OD$.
In this case the previous proposition, Theorem~\ref{Struktursatz} and the proof of Lemma~\ref{Primfaktorbedingung} tell us that $X_{D,(d_1,d_2)}$ is the disjoint union of components $X_\A$,
where $\A$ runs through the ideals of the form $\A = d_1\A'$ with $\A'$ a primitive ideal of $\OD$ of norm~$\tfrac{d_2}{d_1}$.
For~$\A'$ we have $2^s$ choices, where $s$ is the number of splitting prime factors of $\tfrac{d_2}{d_1}$.

For every such ideal $\A$, there is a surjective map \[ \widetilde \Phi:~ \HH^2 \twoheadrightarrow X_\A,\quad z \mapsto [(A_z,H_z,\rho_z)]. \]
We now want to make this map injective. To do this, consider the group
\[ {\rm SL}_2(\OD \oplus \tfrac{1}{\sqrt{D}}\A) := {\rm SL} \begin{pmatrix} \OD & (\tfrac{1}{\sqrt{D}}\A)^{-1} \\ \tfrac{1}{\sqrt{D}}\A & \OD \end{pmatrix}
:= {\rm SL}(K) \cap \begin{pmatrix} \OD & (\tfrac{1}{\sqrt{D}}\A)^{-1} \\ \tfrac{1}{\sqrt{D}}\A & \OD \end{pmatrix}. \]
Then $M \in K^{2 \times 2}$ is in ${\rm SL}_2(\OD \oplus \tfrac{1}{\sqrt{D}}\A)$ if and only if the map \[ (x,y) \mapsto (x,y) \cdot M^t \]
is a symplectic $\OD$-module automorphism of $\OD \oplus \tfrac{1}{\sqrt{D}}\A$.

There is also an action of ${\rm SL}_2(\OD \oplus \tfrac{1}{\sqrt{D}}\A)$ on $\HH^2$ via
\[ M.(z_1,z_2) := (Mz_1,M^{\sigma}z_2) = \left( \frac{az_1+b}{cz_1+d},\frac{a^{\sigma}z_2+b^{\sigma}}{c^{\sigma}z_2+d^{\sigma}} \right) \]
for all $M = \begin{pmatrix} a & b \\ c & d \end{pmatrix} \in {\rm SL}_2(\OD \oplus \tfrac{1}{\sqrt{D}}\A)$ and $(z_1,z_2) \in \HH^2$.
The kernel of this action is $\{ I_2,-I_2 \}$, thus the \emph{Hilbert modular group}
\[ \Gamma_D(\A) := {\rm PSL}_2(\OD \oplus \tfrac{1}{\sqrt{D}}\A) := {\rm SL}_2(\OD \oplus \tfrac{1}{\sqrt{D}}\A) / \{ I_2,-I_2 \} \] acts faithfully on $\HH^2$.

\begin{prop}\label{WohldefiniertheitDim2}
 For all $M \in {\rm SL}_2(\OD \oplus \tfrac{1}{\sqrt{D}}\A)$ and $z \in \HH^2$ we have \[ (A_z,H_z,\rho_z) \cong (A_{Mz},H_{Mz},\rho_{Mz}). \]
\end{prop}
\begin{proof}
 For \[ M = \begin{pmatrix} a & b \\ c & d \end{pmatrix} \in {\rm SL}_2(\OD \oplus \tfrac{1}{\sqrt{D}}\A) \] and $(z_1,z_2) \in \HH^2$ we define
 \[ M^* := \begin{pmatrix} a & -b \\ -c & d \end{pmatrix} \in {\rm SL}_2(\OD \oplus \tfrac{1}{\sqrt{D}}\A) \] and
 \[ D(M,z) := \begin{pmatrix} (cz_1+d)^{-1} & 0 \\ 0 & (c^{\sigma}z_2+d^{\sigma})^{-1} \end{pmatrix} \in {\rm GL}_2(\CC).\] Then we have
 \[ \phi_z ((x,y)) \cdot D(M,z)^t = \left( \frac{yz_1+x}{cz_1+d},\frac{y^{\sigma}z_2+x^{\sigma}}{c^{\sigma}z_2+d^{\sigma}} \right) = \phi_{Mz}((x,y) \cdot (M^*)^t) \]
 for all $(x,y) \in \OD \oplus \tfrac{1}{\sqrt{D}}\A$, i.e. the diagram
 \[ \begin{CD}
  \OD \oplus \tfrac{1}{\sqrt{D}}\A @>\phi_z>> \CC^2 \\
  @VVM^*V @VVD(M,z)V \\
  \OD \oplus \tfrac{1}{\sqrt{D}}\A @>\phi_{Mz}>> \CC^2
 \end{CD} \]
 is commutative. Therefore, \[ \CC^2 \to \CC^2,\quad (x,y) \mapsto (x,y) \cdot D(M,z)^t \] is the desired isomormorphism.
\end{proof}

Thus we have a well-defined surjective map \[ \Phi:~ {\rm PSL}_2(\OD \oplus \tfrac{1}{\sqrt{D}}\A) \backslash \HH^2 \twoheadrightarrow X_\A,\quad [z] \mapsto [(A_z,H_z,\rho_z)] \]
and we can state the main theorem of this subsection.

\begin{theo}\label{ModulraumDim2}
 Let $D \in \NN$ be a non-square discriminant and let $d_1,d_2 \in \NN$ with $d_1 \mid d_2$ and $\tfrac{d_2}{d_1}$ relatively prime to the conductor of $D$.
 Then $X_{D,(d_1,d_2)}$ is non-empty if and only if $(d_1,d_2)$ satisfies the prime factor condition in $\OD$.
 In this case $X_{D,(d_1,d_2)}$ consists of $2^s$ irreducible components $X_\A$, where $s$ is the number of splitting prime divisors of $\tfrac{d_2}{d_1}$.
 For each component $X_\A$ defined by the ideal $\A$, we have an isomorphism
 \[ \Phi:~ {\rm PSL}_2(\OD \oplus \tfrac{1}{\sqrt{D}}\A) \backslash \HH^2 \cong X_\A,\quad [z] \mapsto [(A_z,H_z,\rho_z)] \]
 between the corresponding Hilbert modular surface and $X_\A$.
\end{theo}
\begin{proof}
 It remains to show that $\Phi$ is injective. To do this, let $z=(z_1,z_2)$, $\tilde z=(\tilde z_1,\tilde z_2) \in \HH^2$ and let
 \[ f:~ (A_z,H_z,\rho_z) \to (A_{\tilde z},H_{\tilde z},\rho_{\tilde z}) \] be an isomorphism of polarized Abelian varieties with real multiplication by~$\OD$.
 The key idea of the proof is to compare the rational representation with the analytic representation.

 Let $\mathcal{B} = (\lambda_1,\lambda_2,\mu_1,\mu_2)$ be the symplectic basis of $(\OD \oplus \tfrac{1}{\sqrt{D}}\A,\langle \cdot,\cdot \rangle)$ as in the construction of $(A_z,H_z)$ above, i.e.
 \[ \lambda_1 = (\eta_1,0),~ \lambda_2 = (\eta_2,0),~ \mu_1 = (0,\tfrac{d_1 \eta_2^{\sigma}}{\sqrt{D}}),~ \mu_2 = (0,\tfrac{-d_2 \eta_1^{\sigma}}{\sqrt{D}}). \]
 Furthermore, let $R = (r_{ij})_{i,j} \in {\rm GL}_4(\ZZ)$ be the matrix of the rational representation \[ \rho_r(f):~ \Lambda_z \to \Lambda_{\tilde z} \] of $f$ with respect to the $\ZZ$-bases
 $\phi_z(\mathcal{B})$ and $\phi_{\tilde z}(\mathcal{B})$.
 Then $R$ is also the matrix of \[ F := \phi_{\tilde z}^{-1} \circ \rho_r(f) \circ \phi_z :~ \OD \oplus \tfrac{1}{\sqrt{D}}\A \to \OD \oplus \tfrac{1}{\sqrt{D}}\A \]
 with respect to the  $\ZZ$-basis $(\lambda_1,\lambda_2,\mu_1,\mu_2)$.
 Setting $k := \tfrac{\eta_2}{\eta_1}$ and $l:= -\tfrac{d_1}{d_2}k^{\sigma}$, we get $k\lambda_1 = \lambda_2$ and $l\mu_2 = \mu_1$.
 Let \[ \pi_1:~ K \oplus K \to K \oplus K,\quad (x,y) \mapsto (x,0) \] be the canonical projection onto the first factor and let $F_{\QQ} : K \oplus K \to K \oplus K$ be the extension of $F$ to
 $K \oplus K$ by tensoring with $\QQ$.
 For the composition $F_1 := \pi_1 \circ F_{\QQ}$ we have \[ F_1(\lambda_1) = (r_{11}+r_{21}k)\lambda_1 \mbox{ and } F_1(\lambda_2) = (r_{12}k^{-1}+r_{22})\lambda_2. \]
 As $F_1$ is $K$-linear and $\lambda_2 = k\lambda_1$, this yields \[ a := r_{11}+r_{21}k = r_{12}k^{-1}+r_{22}. \] Similarly, we get
 \[ \begin{array}{lclcl}
  b & := & r_{13}+r_{23}k & = & -\mathcal{N}(k)\frac{d_1}{d_2}(r_{14}k^{-1}+r_{24}) \\
  c & := & r_{31}+r_{41}l^{-1} & = & \frac{-1}{\mathcal{N}(k)}\frac{d_2}{d_1}(r_{32}l+r_{42}) \\
  d & := & r_{33}+r_{43}l^{-1} & = & r_{34}l+r_{44}.
 \end{array} \] Hence, we have 
 \[ F(\lambda_1) = r_{11}\lambda_1 + r_{21}\lambda_2 + r_{31}\mu_1 + r_{41}\mu_2 = a\lambda_1 + c\mu_1 \] and analogously
 \begin{eqnarray*}
  F(\lambda_2) & = & a\lambda_2 - \mathcal{N}(k)\tfrac{d_1}{d_2}c\mu_2 \\
  F(\mu_1) & = & b\lambda_1 + d \mu_1 \\
  F(\mu_2) & = & -\tfrac{1}{\mathcal{N}(k)}\tfrac{d_2}{d_1}b\lambda_2 + d\mu_2.
 \end{eqnarray*}
 If we define \[ M := \begin{pmatrix} a & -\frac{\sqrt{D}}{d_1}\frac{\eta_1}{\eta_2^{\sigma}}b \\
 -\frac{d_1}{\sqrt{D}}\frac{\eta_2^{\sigma}}{\eta_1}c & d \end{pmatrix} \] and $M^*$ in the same way as in the proof of Proposition~\ref{WohldefiniertheitDim2},
 one can easily compute that for all $(x,y) \in \OD \oplus \tfrac{1}{\sqrt{D}}\A$ we have \[ F((x,y)) = (x,y) (M^*)^t. \]
 Since $F$ is a symplectic $\OD$-module automorphism, $M^*$ lies in ${\rm SL}_2(\OD \oplus \tfrac{1}{\sqrt{D}}\A)$, hence $M \in {\rm SL}_2(\OD \oplus \tfrac{1}{\sqrt{D}}\A)$ and thus $ad-bc=1$.

 An elementary computation yields that the matrix of the analytic representation with respect to the standard $\CC$-basis $((1,0),(0,1))$ of $\CC^2$ is
 \[ A = \begin{pmatrix} a+\frac{d_1}{\sqrt{D}}\frac{\eta_2^{\sigma}}{\eta_1}c \tilde z_1 & 0 \\ 0 & a^{\sigma}-\frac{d_1}{\sqrt{D}}\frac{\eta_2}{\eta_1^{\sigma}}c^{\sigma} \tilde z_2 \end{pmatrix}. \]
 Now let $\Pi$ and $\widetilde \Pi$ be the period matrices of $A_z$ and $A_{\tilde z}$ with respect to the chosen bases, so we have $A \Pi = \widetilde \Pi R$.
 Comparing the matrix entries and using all the relations above, we see that this equation holds if and only if $\tilde z = Mz$.
 It follows that $\Phi$ is indeed injective.
\end{proof}
\subsection{Counting modular surfaces}

As varieties, two Hilbert modular surfaces $\Gamma_1 \backslash \HH^2$ and $\Gamma_2 \backslash \HH^2$ are obviously isomorphic, if the groups $\Gamma_1,\Gamma_2 \subset \SL_2(\QQ(\sqrt{D}))$
are conjugate. In this section we show that there is a bijection between the set of conjugacy classes of $\SL(\OD \oplus \tfrac{1}{\sqrt{D}}\A)$ and the genus class group of $\OD$.
\\[1em]
{\bf Genus of quadratic orders.}
Let $K \subset \RR$ be a real quadratic number field and let $K^+$ be the multiplicative group of all totally positive elements in $K$
(an element $x \in K$ is said to be \emph{totally positive, $x \gg 0$}, if both $x>0$ and $x^\sigma >0$).

Let $\OD$ be the order in a real quadratic number field $K$ of discriminant~$D>0$ and conductor~$f$.
Recall that ${\rm I}(\OD)$ is the group of invertible fractional ideals of~$\OD$.
We denote the ideal class group of~$\OD$ by \[ {\rm Cl}(\OD) := {\rm I}(\OD)/K^*. \]
Two ideals $\A_1, \A_2 \in {\rm I}(\OD)$ are defined to be \emph{equivalent in the narrow sense} or \emph{strictly equivalent} (denoted by $\A_1 \sim_+ \A_2$) if there is some $x \in K^+$ such that
$\A_1=x\A_2$. The equivalence classes form a group \[ {\rm Cl}^+(\OD) := {\rm I}(\OD)/K^+, \] the \emph{narrow ideal class group}.
Furthermore, the \emph{squaring map} is defined by \[ {\rm Sq}^+:~ {\rm Cl}^+(\OD) \to {\rm Cl}^+(\OD) ,\quad \A \mapsto \A^2. \]
\begin{defi}\label{genus}
 The factor group \[ {\rm G}(\OD) := {\rm Cl}^+(\OD)/{\rm Sq}^+({\rm Cl}^+(\OD)) \] is called the \emph{genus group of $\OD$},
 a coset of ${\rm Sq}^+({\rm Cl}^+(\OD))$ is called a \emph{genus} and ${\rm Sq}^+({\rm Cl}^+(\OD))$ is called the \emph{principal genus}.
\end{defi}

Now we want to count the number of conjugacy classes of groups of the form $\SL(\OD \oplus \tfrac{1}{\sqrt{D}}\A)$,
to get an upper bound for the number of isomorphy classes of the Hilbert modular varieties $X_\A = \SL_2(\OD \oplus \tfrac{1}{\sqrt{D}}\A) \backslash \HH^2$,
where $\A$ runs through the set ${\rm I}^+(\OO,f)$ of ideals relatively prime to the conductor~$f$ of~$D$ (without fixing the norm $d$).

\begin{lem}
Let $\A_1$ and $\A_2$ be two invertible, fractional ideals of $\OD$. Then the two matrix algebras 
 \[ \begin{pmatrix} \OD & \A_1^{-1} \\ \A_1 & \OD \end{pmatrix} \mbox{ and } \begin{pmatrix} \OD & \A_2^{-1} \\ \A_2 & \OD \end{pmatrix} \]
 are conjugate in $K^{2 \times 2}$ if and only if $\A_1$ and $\A_2$ belong to the same genus.
\end{lem}
\begin{proof}
 If $\A_1$ and $\A_2$ belong to the same genus, then we have by definition \[ \A_2 = x \B^2 \A_1 \] for an invertible ideal $\B$ and some totally positive $x \in K$.
 We claim that there is a matrix \[ M \in \begin{pmatrix} \B & \A_1^{-1}\B^{-1} \\ \A_1\B & \B^{-1} \end{pmatrix} \] with determinant $1$.
 Indeed, by Corollary~\ref{Idealaequiv} it suffices to show that for any two invertible, relatively prime $\OD$-ideals $\A$ and $\B$ there is a matrix in
 \[ \begin{pmatrix} \B & \A^{-1} \\ \A & \B^{-1} \end{pmatrix} \] with determinant $1$. But this is clear, since $1$ lies in each of $\A + \B$, $\A^{-1}$ and $\B^{-1}$.

 Now we can compute directly 
 \[ \begin{pmatrix} x^{-1} & 0 \\ 0 & 1 \end{pmatrix} M^{-1} \begin{pmatrix} \OD & \A_1^{-1} \\ \A_1 & \OD \end{pmatrix} M \begin{pmatrix} x & 0 \\ 0 & 1 \end{pmatrix}
 \subseteq \begin{pmatrix} \OD & \A_2^{-1} \\ \A_2 & \OD \end{pmatrix} \] respectively
 \[ M \begin{pmatrix} x & 0 \\ 0 & 1 \end{pmatrix} \begin{pmatrix} \OD & \A_2^{-1} \\ \A_2 & \OD \end{pmatrix} \begin{pmatrix} x^{-1} & 0 \\ 0 & 1 \end{pmatrix} M^{-1}
 \subseteq \begin{pmatrix} \OD & \A_1^{-1} \\ \A_1 & \OD \end{pmatrix}, \] which shows equality.

 To prove the converse, let $M = \begin{pmatrix} a & b \\ c & d \end{pmatrix} \in \GL_2(K)$ with
 \[ M \begin{pmatrix} \OD & \A_1^{-1} \\ \A_1 & \OD \end{pmatrix} M^{-1} = \begin{pmatrix} \OD & \A_2^{-1} \\ \A_2 & \OD \end{pmatrix}. \]
 Comparing the lower left entries gives \[ \A_2 = \tfrac{1}{\det(M)}(cd\OD + d^2 \A_1 + c^2\A_1^{-1}) = \tfrac{1}{\det(M)}(d\OD + c\A_1^{-1})^2 \A_1 \]
 and the invertibility of $\B = d\OD + c\A_1^{-1}$ follows from the invertibility of the fractional ideal $\B^2 = \det(M) \A_1^{-1}\A_2$.
\end{proof}

\begin{cor}
 For two invertible fractional ideals $\A_1$ and $\A_2$, the two groups ${\rm SL}(\OD \oplus \A_1)$ and ${\rm SL}(\OD \oplus \A_2)$ are conjugate in ${\rm GL}_2(K)$
 if and only if~$\A_1$ and~$\A_2$ belong to the same genus.
\end{cor}
\begin{proof}
 This follows from the same arguments as in the previous proof,
 together with the fact that $\left( \begin{smallmatrix} \OD & \A^{-1} \\ \A & \OD \end{smallmatrix} \right)$ is generated by ${\rm SL}_2(\OD \oplus \tfrac{1}{\sqrt{D}}\A)$ as a $\ZZ$-algebra.
\end{proof}
\subsection{The moduli space in dimension three}\label{dim3}

Now we will use the results of the previous section to parameterize all principally polarized Abelian varieties isogenous to a product of an Abelian surface with real multiplication and an elliptic
curve. We will show that the moduli space of principally polarized Abelian varieties together with the choice of a two-dimensional Abelian subvariety of a fixed type and a choice of real multiplication
by a fixed order $\OD$ is the union of $2^s$ three-dimensional Hilbert modular varieties, where the number $s$ is determined by the type and the discriminant.
For this, we will combine the techniques developed in the previous section and those in \cite{Xiao}, where Xiao parameterized principally polarized Abelian Surfaces isogenous to a product of
elliptic curves.

\begin{defi}
 Let $D \in \NN$ be a discriminant and let $d \in \NN$. A \emph{polarized Abelian variety of dimension three with real multiplication by $\OD$ of type $d$} is a quadruple \[ (A,H,S,\rho), \]
 where $(A,H)$ is a three-dimensional principally polarized Abelian variety, $S$ is a two-dimensional subvariety of $A$ of type $(1,d)$ and $\rho : \OD \to {\rm End}(S)$ is some choice of real
 multiplication by $\OD$ on $S$.
\end{defi}

Recall that the complementary subvariety of $S$ in $A$ in the definition above is of type $(d)$, as $(A,H)$ is principal.
\\[1em]
{\bf Construction of three-dimensional Abelian varieties with real multiplication on a subsurface.}
Again let $D$ be a squarefree discriminant and let~$d$ be relatively prime to the conductor of $D$.
We will see that there are different similar ways to construct polarized Abelian varieties of dimension three with real multiplication by $\OD$ of type $d$, corresponding to the choice of a $\ZZ$-basis
of certain ideals of $\OD$.

We already know that there is no such Abelian variety if $d$ does not satisfy the prime factor condition in $\OD$. In the other case, recall that there exists a primitive ideal~$\A$ in $\OD$
of norm $d$ and a $\ZZ$-basis $(\eta_1,\eta_2)$ of $\OD$ satisfying the sign convention $\I(\eta_1 \eta_2^{\sigma}) = - \tfrac{1}{2}\I(\sqrt{D})$, such that $(\eta_2^{\sigma},d\eta_1^{\sigma})$
is a $\ZZ$-basis of $\A$. We have seen that this means that \[ ((\eta_1,0),(\eta_2,0),(0,\tfrac{1}{\sqrt{D}}\eta_2^{\sigma}),(0,\tfrac{-d}{\sqrt{D}}\eta_1^{\sigma})) \]
is a symplectic basis of $(\OD \oplus \tfrac{1}{\sqrt{D}}\A,\langle \cdot,\cdot \rangle)$, where $\langle \cdot,\cdot \rangle$ is the trace pairing on $\OD \oplus \ODual$.
We call such a basis $(\eta_2^{\sigma},d\eta_1^{\sigma})$ a \emph{smart basis for $\A$}.

\begin{prop}\label{KonstruktionDim3}
 Let $d \in \NN$ be relatively prime to the conductor of $D$ and satisfying the prime factor condition in $\OD$. Furthermore, let $\A$ be a primitive ideal of $\OD$ of norm $d$.
 Then, fixing a smart basis $(\eta_2^{\sigma},d\eta_1^{\sigma})$ for~$\A$, each triple $z = (z_1,z_2,z_3) \in \HH^3$ determines a polarized Abelian variety $(A_z=\CC^3/\Lambda_z,H_z,S_z,\rho_z)$
 of dimension three with real multiplication by $\OD$ of type $d$. Identifying $\CC^2 \times \{ 0 \}$ with $\CC^2$, we have $S_z = \CC^2 / (\CC^2 \cap \Lambda_z)$, and the polarized Abelian surface
 with real multiplication $(S_z,H_z|_{\CC^2 \times \CC^2},\rho_z)$ is the one defined by
 \[ \phi_{(z_1,z_2)}:~ \OD \oplus \tfrac{1}{\sqrt{D}}\A \hookrightarrow \CC^2,\quad (x,y) \mapsto (x+yz_1,x^{\sigma}+y^{\sigma}z_2)\] as in Proposition~\ref{KonstruktionDim2}.
\end{prop}
\begin{proof}
 We start with constructing the three-dimensional principally polarized Abelian variety. In $\CC^3$ we define
 \[ \begin{array}{ll}
 \lambda_1 := (\eta_1,\eta_1^{\sigma},0), & \mu_1 := (\tfrac{1}{\sqrt{D}}\eta_2^{\sigma}z_1,\tfrac{-1}{\sqrt{D}}\eta_2z_2,0), \\
 \lambda_2 := (\eta_2,\eta_2^{\sigma},0), & \mu_2 := (\tfrac{-1}{\sqrt{D}}\eta_1^{\sigma}z_1,\tfrac{1}{\sqrt{D}}\eta_1z_2,\tfrac{1}{d}), \\
 \lambda_3 := (0,0,1), & \mu_3 := (\tfrac{1}{d}\eta_2,\tfrac{1}{d}\eta_2^{\sigma},-\tfrac{1}{d}z_3).
 \end{array} \]
 These vectors are linearly independent over $\RR$, hence \[ \Lambda_z := \langle \lambda_1,\lambda_2,\lambda_3,\mu_1,\mu_2,\mu_3 \rangle_{\ZZ} \] is a lattice in $\CC^3$
 and $A_z := \CC^3 / \Lambda_z$ is a complex torus.
 Let $E_z:~ \CC^3 \times \CC^3 \to \RR$ be the unique alternating $\RR$-bilinearform with $E_z(\lambda_i,\lambda_j)=E_z(\mu_i,\mu_j) = 0$ and $E_z(\lambda_i,\mu_j)= \delta_{ij}$ for all
 $i,j \in \{ 1,2,3 \}$, and let $H_z:~ \CC^3 \times \CC^3 \to \CC$ be the corresponding Hermitian form.
 Then $E_z$ is a symplectic form on $\Lambda_z$ of type $(1,1,1)$ with symplectic basis $(\lambda_1,\lambda_2,\lambda_3,\mu_1,\mu_2,\mu_3)$.
 The period matrix of $A_z$ with respect to this $\ZZ$-basis of $\Lambda_z$ and the standard $\CC$-basis $((1,0,0),(0,1,0),(0,0,1))$ of $\CC^3$ is
 \[ \Pi_z = \begin{pmatrix} \eta_1 & \eta_2 & 0 & \frac{1}{\sqrt{D}}\eta_2^{\sigma}z_1 & \frac{-1}{\sqrt{D}}\eta_1^{\sigma}z_1 & \tfrac{1}{d}\eta_2 \\
 \eta_1^{\sigma} & \eta_2^{\sigma} & 0 & \frac{-1}{\sqrt{D}}\eta_2z_2 & \frac{1}{\sqrt{D}}\eta_1z_2 & \tfrac{1}{d}\eta_2^{\sigma} \\
 0 & 0 & 1 & 0 & \tfrac{1}{d} & -\tfrac{1}{d}z_3 \end{pmatrix}. \]
 With \[ B := \begin{pmatrix} \eta_1 & \eta_2 & 0 \\ \eta_1^{\sigma} & \eta_2^{\sigma} & 0 \\ 0 & 0 & 1 \end{pmatrix} \mbox{ and } Z :=
 \begin{pmatrix} -z_1 & 0 & \tfrac{1}{d}\eta_2 \\ 0 & -z_2 & \tfrac{1}{d}\eta_2^{\sigma} \\ \tfrac{1}{d}\eta_2 & \tfrac{1}{d}\eta_2^{\sigma} & -\tfrac{1}{d}z_3 \end{pmatrix} \]
 we get \[ \Pi_z = (B,Z(B^{-1})^t), \] and as $-B^{-1}Z(B^{-1})^t$ is in the upper Siegel half space $\HH_2$, the form $H_z$ is indeed a polarization on $A_z$.

 With $V_2 := \CC^2 \times \{ 0 \}$ and \[ \hat\mu_2 := d\mu_2 - \lambda_3 = (\tfrac{-d}{\sqrt{D}}\eta_1^{\sigma}z_1,\tfrac{d}{\sqrt{D}}\eta_1z_2,0) \]
 we have \[ \Lambda_2 := V_2 \cap \Lambda_z = \langle \lambda_1,\lambda_2,\mu_1,\hat \mu_2 \rangle_{\ZZ}. \]
 Thus $\Lambda_2$ is a lattice in $V_2$ and $S_z := V_2 / \Lambda_2$ is a subvariety of $A_z$ of type $(1,d)$ with symplectic basis $(\lambda_1,\lambda_2,\mu_1,\hat \mu_2)$.
 Note that $\Lambda_2$ is just the image of the embedding \[ \phi_{(z_1,z_2)}:~ \OD \oplus \tfrac{1}{\sqrt{D}}\A \hookrightarrow \CC^3,\quad (x,y) \mapsto (x+yz_1,x^{\sigma}+y^{\sigma}z_2,0), \]
 which maps our symplectic basis $((\eta_1,0),(\eta_2,0),(0,\tfrac{1}{\sqrt{D}}\eta_2^{\sigma}),(0,\tfrac{-d}{\sqrt{D}}\eta_1^{\sigma}))$ of
 $(\OD \oplus \tfrac{1}{\sqrt{D}}\A,\langle \cdot,\cdot \rangle)$ to $(\lambda_1,\lambda_2,\mu_1,\hat \mu_2)$. \\
 As in the previous section, we define real multiplication $\rho_z:~ \OD \hookrightarrow {\rm End}(S_z)$ on~$S_z$ by~$\OD$ via
 \[ \rho:~ \OD \hookrightarrow {\rm End}_{\CC}(V_2),\quad \rho(k)(c_1,c_2,0) := (kc_1,k^{\sigma}c_2,0) \] for all $k \in \OD$ and $c_1,c_2 \in \CC$.

Collecting everything together, we have seen that $(A_z,H_z,S_z,\rho_z)$ is indeed a polarized Abelian variety of dimension three with real multiplication by $\OD$ of type $d$.
\end{proof}
We remark that in the proof above, the orthogonal complement of $V_2$ in $\CC^3$ with respect to $H_z$ is $V_1 := \{ 0 \}^2 \times \CC$, and that a symplectic basis of $\Lambda_1 := V_1 \cap \Lambda$
for $E_z|_{\Lambda_1 \times \Lambda_1}$ is given by $(\nu_1,\nu_2)$ with \[ \nu_1 := \lambda_3 = (0,0,1) \mbox{ and } \nu_2 := d \mu_3 - \lambda_2 = (0,0,-z_3). \]

An \emph{isomorphism of polarized Abelian varieties of dimension three with real multiplication by $\OD$}, say $(A,H,S,\rho)$ and $(A',H',S',\rho')$, is an isomorphism of polarized abelian varieties
\[ f:~ (A,H) \to (A',H') \] with $f(S) = S'$, such that $f|_{S}$ commutes with the action of $\OD$.

\begin{prop}\label{classifdim3}
 Let $d \in \NN$ be relatively prime to the conductor of $D$. Fixing a smart basis for each primitive ideal $\A \subseteq \OD$ of norm $d$,
 the quadruples $(A_z,H_z,S_z,\rho_z)$ defined above are up to isomorphism all polarized Abelian varieties of dimension three with real multiplication by $\OD$ of type $d$.
 More precisely, we define \[ A_z = \CC^3 / \Lambda_z \mbox{ with } \Lambda_z = \Pi_z \ZZ^6 \] and
 \[ \Pi_z = \begin{pmatrix} \eta_1 & \eta_2 & 0 & \frac{1}{\sqrt{D}}\eta_2^{\sigma}z_1 & \frac{-1}{\sqrt{D}}\eta_1^{\sigma}z_1 & \tfrac{1}{d}\eta_2 \\
 \eta_1^{\sigma} & \eta_2^{\sigma} & 0 & \frac{-1}{\sqrt{D}}\eta_2z_2 & \frac{1}{\sqrt{D}}\eta_1z_2 & \tfrac{1}{d}\eta_2^{\sigma} \\
 0 & 0 & 1 & 0 & \tfrac{1}{d} & -\tfrac{1}{d}z_3 \end{pmatrix} \] for a smart basis $(\eta_2^{\sigma},d \eta_1^{\sigma})$ of an ideal $\A$ in $\OD$
 and $z=(z_1,z_2,z_3) \in \HH^3$ and the polarization $H_z$ by symplectic pairing the columns of $\Pi_z$,
 \[ S_z = V_2 / (V_2 \cap \Lambda_z) \mbox{ with } V_2 = \CC^2 \times \{ 0 \} \] and the real multiplication $\rho_z$ on $S_z$ is defined by
 \[ \rho:~ \OD \hookrightarrow {\rm End}_{\CC}(V_2),\quad \rho(k)(c_1,c_2,0) := (kc_1,k^{\sigma}c_2,0) \]  for all $k \in \OD$ and $c_1,c_2 \in \CC$.
\end{prop}
\begin{proof}
 Let $d \in \NN$ be relatively prime to the conductor of $D$ and let $(A,H,S,\rho)$ be a polarized Abelian variety of dimension three with real multiplication by~$\OD$ of type $d$.
 We have $A=V/\Lambda$ for a three-dimensional $\CC$-vector space $V$ and a lattice $\Lambda$ in $V$ and $S = V_2 / \Lambda_2$ with $V_2 \leq V$, $\dim_{\CC}(V_2)=2$ and $\Lambda_2 := V_2 \cap \Lambda$,
 and we write $H_2$ for $H|_{V_2 \times V_2}$.

 We have seen in Proposition~\ref{SurjektivDim2}, that there is an isomorphism of polarized Abelian varieties with real multiplication by $\OD$,
 \[ f:~ (A_z,H_z,\rho_z) \to (S,H_2,\rho), \] where $(A_z=\CC^2/\Lambda_z,H_z,\rho_z)$ is a polarized Abelian surface with real multiplication by $\OD$ for a primitive ideal $\A$ of norm $d$ in $\OD$
 and for some pair $z=(z_1,z_2) \in \HH^2$.

 Let $(\eta_2^{\sigma},d\eta_1^{\sigma})$ be any smart basis for $\A$ and let $F:=\rho_a(f)$ be the analytic representation of $f$.
 Then \[ (F((\eta_1,\eta_1^{\sigma})),F((\eta_2,\eta_2^{\sigma})),F((\tfrac{1}{\sqrt{D}}\eta_2^{\sigma} z_1,\tfrac{-1}{\sqrt{D}}\eta_2 z_2)),
 F((\tfrac{-d}{\sqrt{D}}\eta_1^{\sigma} z_1,\tfrac{d}{\sqrt{D}}\eta_1 z_2))) \] is a symplectic basis of $(S,H_2)$, denoted by $(\lambda_1,\lambda_2,\mu_1,\hat \mu_2)$.
 Since the two vectors $v_1:=F((1,0))$ and $v_2:=F((0,1))$ form a $\CC$-basis of $V_2$, we get
 \[ \begin{array}{ll}
  \lambda_1 = \eta_1 v_1 + \eta_1^{\sigma} v_2, & \mu_1 = (\tfrac{1}{\sqrt{D}}\eta_2^{\sigma}z_1) v_1 + (\tfrac{-1}{\sqrt{D}}\eta_2z_2) v_2, \\
  \lambda_2 = \eta_2 v_1 + \eta_2^{\sigma} v_2, & \hat \mu_2 = (\tfrac{-d}{\sqrt{D}}\eta_1^{\sigma}z_1) v_1 + (\tfrac{d}{\sqrt{D}}\eta_1z_2) v_2, \\
 \end{array} \]
 and the real multiplication is given by \[ \rho(k)(c_1v_1+c_2v_2) = (kc_1)v_1+(k^{\sigma}c_2)v_2 \] for all $k \in \OD$ and $c_1,c_2 \in \CC$, as desired.

 It is important to keep in mind, that we will neither change the symplectic basis $(\lambda_1,\lambda_2,\mu_1,\hat \mu_2)$ nor the $\CC$-basis $(v_1,v_2)$ during the rest of the proof.

 We denote by $V_1$ the orthogonal complement of $V_2$ in $V$ with respect to $H$ and by $H_1 := H|_{V_1 \times V_1}$ the induced polarization on $A_1 := V_1/\Lambda_1$
 for $\Lambda_1 := V_1 \cap \Lambda$. Furthermore, for $\Lambda' := \Lambda_1 \oplus \Lambda_2$ let $H'$ be the induced polarization on $A':=V/\Lambda'$.
 Since $(A_2,H_2)$ is of type $(1,d)$, we know that $(A',H')$ is of type $(1,d,d)$, and so by Lemma~\ref{Gradindex} we have \[ |\Lambda/\Lambda'| = d^2. \]
 For $i \in \{ 1,2 \}$, let $\pi_i: V \to V_i$  be the canonical projection and let $\underline{\Lambda}_i := \pi_i(\Lambda)$ be the induced lattice in $V_i$.
 For the alternating form $E_2$ corresponding to $H_2$, we have $E_2(\underline{\Lambda}_2,\Lambda_2) \subseteq \ZZ$ and therefore
 $\underline{\Lambda}_2 \subseteq \langle \lambda_1,\tfrac{1}{d} \lambda_2,\mu_1,\tfrac{1}{d} \hat \mu_2 \rangle_{\ZZ}$.
 Because \[ \varphi_i:~ \Lambda/\Lambda' \to \underline{\Lambda}_i/\Lambda_i ,\quad [\lambda] \mapsto [\pi_i(\lambda)] \] is an isomorphism of groups
 and hence $|\underline{\Lambda}_i/\Lambda_i| = d^2$, we even have \[ \underline{\Lambda}_2 = \langle \lambda_1,\tfrac{1}{d} \lambda_2,\mu_1,\tfrac{1}{d} \hat \mu_2 \rangle_{\ZZ}, \]
 in particular \[ \Lambda/\Lambda' \cong \underline{\Lambda}_i/\Lambda_i \cong \ZZ_d^2. \]
 There are $\lambda',\mu' \in \underline{\Lambda}_1$, unique up to $\Lambda_1$, with $\tfrac{1}{d}\lambda_2+\lambda',\tfrac{1}{d}\mu_2+\mu' \in \Lambda$.
 As $\varphi_1$ and $\varphi_2$ are isomorphisms, we know that $([\lambda'],[\mu'])$ is a basis (in the sense of Lemma~\ref{Rationalbasis}) of $\underline{\Lambda}_1/\Lambda_1$ and so by
 Lemma~\ref{Rationalbasis} there is a basis $(\nu_1,\nu_2)$ of $\Lambda_1$ and $l,m \in \NN$ with $1 \leq l,m \leq d$ and $(l,d)=(m,d)=1$, such that $\lambda'=\tfrac{l}{d}\nu_2$ and
 $\mu'=\tfrac{m}{d}\nu_1$. We can write $\nu_2 = -z_3\nu_1$ with $z_3 \in \CC \setminus \RR$. If ${\rm Im}(z_3)<0$,
 we can choose $-\nu_2$ and $\lambda'-\nu_2 = \tfrac{d-l}{d}(-\nu_2)$ instead of $\nu_2$ and $\lambda'$, so without loss of generality we may assume $z_3 \in \HH$.
 Now we take $n \in \NN$ with $1 \leq n \leq d$ and $nl \equiv 1 \mod d$ and define
 \[ \mu_2 := \tfrac{1}{d}(\hat \mu_2 + m\nu_1) ,~ \mu_3 := \tfrac{1}{d}(n\lambda_2 + \nu_2) = n(\tfrac{1}{d}\lambda_2+\lambda')+\tfrac{1-nl}{d}\nu_2 \in \Lambda. \]
 Then we have $\langle \lambda_1, \lambda_2, \mu_1, \hat \mu_2, \nu_1, \nu_2 \rangle_{\ZZ} = \Lambda'$ and $\langle [\mu_2], [\mu_3] \rangle = \Lambda/\Lambda'$ and hence
 \[ \langle \lambda_1, \lambda_2, \mu_1, \hat \mu_2, \nu_1, \nu_2, \mu_2, \mu_3 \rangle_{\ZZ} = \Lambda. \]
 As ${\rm Im}(z_3) >0$, it follows \[ E_H(\nu_1,\nu_2) = {\rm Im}(z_3)H(\nu_1,\nu_1) = +d \] and \[ E_H(\mu_2,\mu_3) = \tfrac{1}{d}(m-n). \]
 Because $E_H(\mu_2,\mu_3)$ is an integer, we see by the choice of $m$ and $n$, that $m = n$ and $E_H(\mu_2,\mu_3) = 0$, so with $\lambda_3:=\nu_1$, we have that
 $(\lambda_1,\lambda_2,\lambda_3,\mu_1,\mu_2,\mu_3)$ is a symplectic basis of~$(A,H)$.

 To show that we can always choose $m=1$, let us assume $m\neq1$. Since $m$ and $d^2$ are relatively prime, there are $a,b \in \ZZ$ with $am-bd^2=1$. Hence
 \[ \mu'' := \tfrac{1}{d}(\lambda_2+a\nu_2) = a\mu_3-\tfrac{am-1}{d}\lambda_2 \] is in $\Lambda$ and we define \[ \mu_2':=\mu_2 +b\nu_2 ,~ \mu_3':=\mu''+\nu_1 \in \Lambda \]
 and \[ \begin{pmatrix} \nu_1' \\ \nu_2' \end{pmatrix} := \begin{pmatrix} m & bd \\ d & a \end{pmatrix} \begin{pmatrix} \nu_1 \\ \nu_2 \end{pmatrix}. \]
 We get a new symplectic basis $(\nu_1',\nu_2')$ for $(A_1,H_1)$ with $\nu_2' = -z_3' \nu_1'$ for some $z_3 \in \HH$ and define $\lambda_3':=\nu_1'$. Also note that we have now
 \[ \mu_2' = \tfrac{1}{d}(\hat \mu_2 + \nu_1') ,~ \mu_3' = \tfrac{1}{d}(\lambda_2 + \nu_2') \] and can replace the symplectic basis $(\nu_1,\nu_2)$ of $(A_1,H_1)$ by $(\nu_1',\nu_2')$
 and the symplectic basis $(\lambda_1,\lambda_2,\lambda_3,\mu_1,\mu_2,\mu_3)$ of $(A,H)$ by $(\lambda_1,\lambda_2,\lambda_3',\mu_1,\mu_2',\mu_3')$.
 This shows that we can assume without loss of generality $m=n=1$. Then with $v_1,v_2$ defined at the beginning of the proof and $v_3:=\nu_1$, our period matrix of $A$
 with respect to this symplectic basis and the $\CC$-basis $(v_1,v_2,v_3)$ of $V$ is indeed
 \[ \Pi_z = \begin{pmatrix} \eta_1 & \eta_2 & 0 & \frac{1}{\sqrt{D}}\eta_2^{\sigma}z_1 & \frac{-1}{\sqrt{D}}\eta_1^{\sigma}z_1 & \tfrac{1}{d}\eta_2 \\
 \eta_1^{\sigma} & \eta_2^{\sigma} & 0 & \frac{-1}{\sqrt{D}}\eta_2z_2 & \frac{1}{\sqrt{D}}\eta_1z_2 & \tfrac{1}{d}\eta_2^{\sigma} \\
 0 & 0 & 1 & 0 & \tfrac{1}{d} & -\tfrac{1}{d}z_3 \end{pmatrix}, \] which ends the proof.
\end{proof}

As in the two-dimensional case, we have seen that polarized Abelian varieties of dimension three with real multiplication by $\OD$ of a fixed type $d$ are classified by certain ideals of $\OD$,
and this leads to the following definition.
\begin{defi}
 Let $D \in \NN$ be a non-square discriminant and let $d \in \NN$. Then we denote by $X_{D,d}^{(3)}$ the set of all isomorphism classes $[(A,H,S,\rho)]$, where $(A,H,S,\rho)$ is a polarized
 Abelian variety of dimension three with real multiplication by $\OD$ of type $d$.

 Furthermore, if $d$ is relatively prime to the conductor of $D$, we denote by \[ X_\A^{(3)} \subseteq X_{D,d}^{(3)} \] the locus of those isomorphism classes $[(A,H,S,\rho)]$,
 where $(S,H|_S,\rho)$ is of type~$\A$. If $[(A,H,S,\rho)]$ is in $X_\A^{(3)}$, then we also say that $(A,H,S,\rho)$ is of type~$\A$.
\end{defi}

As in the two-dimensional case, it follows that for $d$ relatively prime to the conductor, $X_{D,d}^{(3)}$ is non-empty if and only if $d$ satisfies the prime factor condition in $\OD$.
In this case it is the disjoint union of $2^s$ components $X_\A^{(3)}$, where $\A$ runs through all primitive ideals of $\OD$ of norm $d$ and $s$ is the number of splitting prime factors of~$d$.

\bigskip
We fix an ideal $\A$ of $\OD$ of norm $d$ and a smart basis $(\eta_2^{\sigma},d\eta_1^{\sigma})$ for $\A$.
Again we want to make the surjective map \[ \widetilde \Phi:~ \HH^3 \twoheadrightarrow X_\A^{(3)},\quad z \mapsto [(A_z,H_z,S_z,\rho_z)] \] injective by factoring out the appropriate group action.

The group ${\rm SL}_2(\OD \oplus \tfrac{1}{\sqrt{D}}\A) \times {\rm SL}_2(\ZZ)$ acts on $\HH^3$ via \[ (A,B).(z_1,z_2,z_3) := (Az_1,A^{\sigma}z_2,Bz_3) \] for all
$(A,B) \in {\rm SL}_2(\OD \oplus \tfrac{1}{\sqrt{D}}\A) \times {\rm SL}_2(\ZZ)$ and $(z_1,z_2,z_3) \in \HH^3$, where the action on the components is given by M\"obius transformations.
We will see that this group is slightly to big, and that the correct group involves the choice of the smart basis.
With this in mind, we set \[ S:= S(D,d,\eta_1,\eta_2) := \begin{pmatrix} 0 & \tfrac{\sqrt{D}}{d}\tfrac{\eta_2}{\eta_1^{\sigma}} \\ 1 & 0 \end{pmatrix} \] and
\[ M_{D,d}(\eta_1,\eta_2) := \begin{pmatrix} \tfrac{d}{\eta_2}\OD & \tfrac{\sqrt{D}}{\eta_1^{\sigma}}\OD \\ \tfrac{d}{\sqrt{D}\eta_2}\A & \tfrac{1}{\eta_1^{\sigma}}\A \end{pmatrix} \subset \bM_2(K). \]
\begin{lem}
 \[ \Gamma_{D,d}(\eta_1,\eta_2) = \{ (A,B) \in {\rm SL}_2(\OD \oplus \tfrac{1}{\sqrt{D}}\A) \times {\rm SL}_2(\ZZ) : A-SBS^{-1} \in M_{D,d}(\eta_1,\eta_2) \} \]
 is a subgroup of ${\rm SL}_2(\OD \oplus \tfrac{1}{\sqrt{D}}\A) \times {\rm SL}_2(\ZZ)$.
\end{lem}
\begin{proof}
 For all $B \in {\rm SL}_2(\ZZ)$, we define $\overline{B} := SBS^{-1}$.

 First, let $(A_1,B_1),(A_2,B_2) \in \Gamma_{D,d}(\eta_2,\eta_2)$.
 From the inclusions \[ {\rm SL}_2(\OD \oplus \tfrac{1}{\sqrt{D}}\A) \cdot M_{D,d}(\eta_1,\eta_2) \subseteq M_{D,d}(\eta_1,\eta_2) \] and
 \[ M_{D,d}(\eta_1,\eta_2) \cdot S \cdot {\rm SL}_2(\ZZ) \cdot S^{-1} \subseteq M_{D,d}(\eta_1,\eta_2), \] we get
 \[ A_1 A_2 - \overline{B_1 B_2} = A_1 (A_2 - \overline{B_2}) + (A_1 - \overline{B_1}) \overline{B_2} \in M_{D,d}(\eta_1,\eta_2), \]
 and therefore $(A_1,B_1) \cdot (A_2,B_2)$ is in $\Gamma_{D,d}(\eta_1,\eta_2)$.

 It is obvious that $(I_2,I_2)$ lies in $\Gamma_{D,d}(\eta_1,\eta_2)$, because the zero matrix is in~$M_{D,d}(\eta_1,\eta_2)$.

 It remains to show that for any $(A,B) \in \Gamma_{D,d}(\eta_1,\eta_2)$, its inverse $(A^{-1},B^{-1})$ also lies in~$\Gamma_{D,d}(\eta_1,\eta_2)$.
 We write \[ A-\overline{B} = \begin{pmatrix} \alpha & \beta \\ \gamma & \delta \end{pmatrix} \] with
 \[ \alpha \in \tfrac{d}{\eta_2}\OD,~ \beta \in \tfrac{\sqrt{D}}{\eta_1^{\sigma}}\OD,~ \gamma \in \tfrac{d}{\sqrt{D}\eta_2}\A,~ \delta \in \tfrac{1}{\eta_1^{\sigma}}\A. \]
 As $A \in {\rm SL}_2(\OD \oplus \tfrac{1}{\sqrt{D}}\A)$ and $B \in {\rm SL}_2(\ZZ)$, we even get
 \[ \alpha \in \tfrac{d}{\eta_2}\OD \cap \OD \mbox{ and } \delta \in \tfrac{1}{\eta_1^{\sigma}}\A \cap \OD. \]
 Since \[ A^{-1}-\overline{B}^{-1} = \begin{pmatrix} \delta & -\beta \\ -\gamma & \alpha \end{pmatrix}, \] it is sufficient to show that
 \[ \tfrac{d}{\eta_2}\OD \cap \OD = \langle d, d\tfrac{\eta_1}{\eta_2} \rangle_{\ZZ} \cap \OD \] and
 \[ \tfrac{1}{\eta_1^{\sigma}}\A \cap \OD = \langle d, \tfrac{\eta_2^{\sigma}}{\eta_1^{\sigma}} \rangle_{\ZZ} \cap \OD \] are the same ideals of $\OD$.
 Recall that we can write \[ \begin{pmatrix} \eta_1 \\ \eta_2 \end{pmatrix} = \begin{pmatrix} a_1 & a_2 \\ a_3 & a_4 \end{pmatrix} \begin{pmatrix} 1 \\ \gamma \end{pmatrix} \]
 with $\gamma = \tfrac{D+\sqrt{D}}{2}$ and $a_1a_4-a_2a_3 = 1$. Then an elementary computation yields \[ \eta_1 \eta_2^{\sigma} = (a_1 a_3 + a_2 a_4 \mathcal{N}(\gamma) + a_1 a_4 D) - \gamma.\]
 In particular $\eta_1 \eta_2^{\sigma}$ is primitive in~$\OD$.
 Now we define \[ I := \{ x \in \QQ : x\tfrac{\eta_2^{\sigma}}{\eta_1^{\sigma}} \in \OD \} \subseteq \QQ. \]
 We have $\tfrac{\eta_2^{\sigma}}{\eta_1^{\sigma}} = \tfrac{1}{\mathcal{N}(\eta_1)}\eta_1\eta_2^{\sigma}$ and $\eta_1\eta_2^{\sigma}$ is primitive in $\OD$, hence $I = \mathcal{N}(\eta_1)\ZZ$.
 Now we take an arbitrary $z \in \langle d, d\tfrac{\eta_1}{\eta_2} \rangle_{\ZZ} \cap \OD$, i.e.
 \[ z = xd + yd\tfrac{\eta_1}{\eta_2} =  xd + yd\mathcal{N}(\tfrac{\eta_1}{\eta_2})\tfrac{\eta_2^{\sigma}}{\eta_1^{\sigma}} \] with $x,y \in \ZZ$ and the additional condition that $z$ lies in $\OD$.
 Then $yd\mathcal{N}(\tfrac{\eta_1}{\eta_2})$ is in $I \subseteq \ZZ$, and thus $z$ is in $\langle d, \tfrac{\eta_2^{\sigma}}{\eta_1^{\sigma}} \rangle_{\ZZ}$. \\
 For the reverse inclusion, note that $[\OD : \A] = d$ and $\eta_2^{\sigma} \in \A$ implies \[ \mathcal{N}(\eta_2) = [\OD : \eta_2^{\sigma}\OD] \in d\ZZ. \]
 For an arbitrary $z \in \langle d, \tfrac{\eta_2^{\sigma}}{\eta_1^{\sigma}} \rangle_{\ZZ} \cap \OD$,
 \[ z = xd + y\tfrac{\eta_2^{\sigma}}{\eta_1^{\sigma}} = xd + \tfrac{y}{\mathcal{N}(\eta_1)}\mathcal{N}(\eta_2)\tfrac{\eta_1}{\eta_2} \]
 with $x,y \in \ZZ$ and the additional condition that $z$ lies in $\OD$, the fact that $y \in I \mbox{ and } d \mid \mathcal{N}(\eta_2)$ implies
 that $z$ is in $\langle d, d\tfrac{\eta_1}{\eta_2} \rangle_{\ZZ}$.
\end{proof}

Thus we have a group action of $\Gamma_{D,d}(\eta_1,\eta_2)$ on $\HH^3$, and the following proposition states that the map $\widetilde \Phi:~ \HH^3 \twoheadrightarrow X_\A^{(3)}$
factorizes through the corresponding quotient map.

\begin{prop}\label{WohldefiniertheitDim3}
 For each $\gamma \in \Gamma_{D,d}(\eta_1,\eta_2)$ and $z \in \HH^3$, we have \[ (A_z,H_z,S_z,\rho_z) \cong (A_{\gamma.z},H_{\gamma.z},S_{\gamma.z},\rho_{\gamma.z}). \]
\end{prop}
\begin{proof}
 We write $z = (z_1,z_2,z_3)$, $\gamma.z = (\tilde z_1,\tilde z_2,\tilde z_3)$
 and \[ \gamma = (A,B) = \left( \begin{pmatrix} a_1 & a_2 \\ a_3 & a_4 \end{pmatrix} , \begin{pmatrix} b_1 & b_2 \\ b_3 & b_4 \end{pmatrix} \right). \]
 Furthermore, we define \[ D(\gamma,z) := \begin{pmatrix} (a_3z_1+a_4)^{-1} & 0 & 0 \\ 0 & (a_3^{\sigma}z_2+a_4^{\sigma})^{-1} & 0 \\ 0 & 0 & (b_3z_3+b_4)^{-1} \end{pmatrix} \]
 and \[ F:~ \CC^3 \to \CC^3 ,\quad v \mapsto D(\gamma,z) \cdot v. \]
 We want to show that $F$ is the analytic representation of an isomorphism \[ (A_z=\CC^3/\Lambda_z,H_z,S_z,\rho_z) \cong (A_{\gamma.z}=\CC^3/\Lambda_z,H_{\gamma.z},S_{\gamma.z},\rho_{\gamma.z}). \]
 Identifying $\CC^2 \times \{ 0 \}$ with $\CC^2$ and using the notations from Proposition~\ref{KonstruktionDim2}, we have $S_z=\CC^2/\Lambda_{(z_1,z_2)}$ and
 $S_{\gamma.z}=\CC^2/\Lambda_{(\tilde z_1,\tilde z_2)}$.
 As $A$ is in ${\rm SL}_2(\OD \oplus \tfrac{1}{\sqrt{D}}\A)$, we have seen in the proof of Proposition~\ref{WohldefiniertheitDim2}, that \[ F_2 := F|_{\CC^2}:~ \CC^2 \to \CC^2 \]
 defines an isomorphism of polarized Abelian surfaces with real multiplication \[ f_2:~ S_z \to S_{\gamma.z} ,\quad [x] \mapsto [F_2(x)]. \]
 We also identify $\CC$ with $\{ (0,0) \} \times \CC$. Denoting the complementary subvariety of $S_z$ respectively $S_{\gamma.z}$ by $C$ respectively $\widetilde C$, we get $C=\CC/\Lambda_{z_3}$
 and $\widetilde C=\CC/\Lambda_{\tilde z_3}$ with $\Lambda_{z_3} = \langle 1,z_3 \rangle_\ZZ$ and $\Lambda_{\tilde z_3} = \langle 1,\tilde z_3 \rangle_\ZZ$.
 Since $B$ is in ${\rm SL}_2(\ZZ)$, the $\CC$-linear map $F_1 := F|_{\CC}:~ \CC \to \CC$
 induces an isomorphism of polarized elliptic curves \[ f_1:~ C \to \widetilde C,\quad [x] \mapsto [F_1(x)]. \]
 Thus $F$ is an isomorphism of $\CC$-vectorspaces preserving the Hermitian forms and the $\OD$-module structure.

 It remains to prove that $F$ maps $\Lambda_z$ onto $\Lambda_{\gamma.z}$. As $F$ preserves the Hermitian forms, it suffices to show that $F(\Lambda_z)$ is contained in $\Lambda_{\gamma.z}$.

 Recall that $\ZZ$-bases of $\Lambda_{(z_1,z_2)}$ and $\Lambda_{(\tilde z_1,\tilde z_2)}$ are given by $(\lambda_1,\lambda_2,\mu_1,\mu_2')$ and $(\lambda_1,\lambda_2,\tilde \mu_1,\tilde \mu_2')$
 with \[ \begin{array}{ll}
 \lambda_1 := (\eta_1,\eta_1^{\sigma},0), & \lambda_2 := (\eta_2,\eta_2^{\sigma},0), \\
 \mu_1 := (\tfrac{1}{\sqrt{D}}\eta_2^{\sigma}z_1,\tfrac{-1}{\sqrt{D}}\eta_2z_2,0), &
 \tilde \mu_1 := (\tfrac{1}{\sqrt{D}}\eta_2^{\sigma}\tilde z_1,\tfrac{-1}{\sqrt{D}}\eta_2\tilde z_2,0), \\
 \mu_2' := (\tfrac{-d}{\sqrt{D}}\eta_1^{\sigma}z_1,\tfrac{d}{\sqrt{D}}\eta_1z_2,0) , &
 \tilde \mu_2' := (\tfrac{-d}{\sqrt{D}}\eta_1^{\sigma}\tilde z_1,\tfrac{d}{\sqrt{D}}\eta_1\tilde z_2,0). \end{array} \]
 Let $R=(r_{ij})_{i,j} \in \GL_4(\ZZ)$ be the matrix of the rational representation of~$f_2$ with respect to these bases.
 We have seen in the proof of Proposition~\ref{WohldefiniertheitDim2}, that
 \[ F_2' := \phi_{(\tilde z_1,\tilde z_2)}^{-1} \circ F_2|_{\Lambda_{(z_1,z_2)}} \circ \phi_{(z_1,z_2)}:~ \OD \oplus \tfrac{1}{\sqrt{D}}\A \to \OD \oplus \tfrac{1}{\sqrt{D}}\A \]
 is given by multiplication with the matrix $A^* = \left( \begin{smallmatrix} a_1 & -a_2 \\ -a_3 & a_4 \end{smallmatrix} \right) \in \SL_2(\OD \oplus \tfrac{1}{\sqrt{D}}\A)$.
 Moreover, we have seen in the proof of Theorem~\ref{ModulraumDim2} that
\[ A = \begin{pmatrix} r_{12}\tfrac{\eta_1}{\eta_2} + r_{22} &
 r_{14}\tfrac{\sqrt{D}}{d}\tfrac{\eta_1}{\eta_1^{\sigma}} + r_{24}\tfrac{\sqrt{D}}{d}\tfrac{\eta_2}{\eta_1^{\sigma}} \\
 -r_{32}\tfrac{1}{\sqrt{D}}\tfrac{\eta_2^{\sigma}}{\eta_2} + r_{42}\tfrac{d}{\sqrt{D}}\tfrac{\eta_1^{\sigma}}{\eta_2} &
 -r_{34}\tfrac{1}{d}\tfrac{\eta_2^{\sigma}}{\eta_1^{\sigma}} + r_{44} \end{pmatrix}. \]
 The matrix of the rational representation $f_1$ with respect to the symplectic bases $(\nu_1,\nu_2):=((0,0,1),(0,0,-z_3))$ of $\Lambda_{z_3}$ and
 $(\tilde \nu_1,\tilde \nu_2):=((0,0,1),(0,0,-\tilde z_3))$ of~$\tilde \Lambda_{\tilde z_3}$ is just~$B$.

 Now consider the $\ZZ$-basis $(\lambda_1,\lambda_2,\lambda_3,\mu_1,\mu_2,\mu_3)$ of $\Lambda_z$ with
 \[ \lambda_3 := \nu_1,~ \mu_2 := \tfrac{1}{d}(\mu_2'+\nu_1),~ \mu_3 := \tfrac{1}{d}(\lambda_2+\nu_2) \] and 
 the $\ZZ$-basis $(\lambda_1,\lambda_2,\lambda_3,\tilde \mu_1,\tilde \mu_2,\tilde \mu_3)$ of $\Lambda_{\tilde z}$ with
 \[ \tilde \mu_2 := \tfrac{1}{d}(\tilde \mu_2'+\nu_1),~ \tilde \mu_3 := \tfrac{1}{d}(\lambda_2+\tilde \nu_2). \]
 We already know that $F(\lambda_1)$, $F(\lambda_2)$, $F(\lambda_3)$ and $F(\mu_1)$ lie in $\Lambda_{(\tilde z_1,\tilde z_2)} + \Lambda_{\tilde z_3} \subseteq \Lambda_{\gamma.z}$.
 Therefore, it remains to show that both
 \begin{eqnarray*}
  F(\mu_2) & = & \tfrac{1}{d}(r_{14}\lambda_1+r_{24}\lambda_2+r_{34}\tilde \mu_1+r_{44}\tilde \mu_2'+b_1\nu_1+b_3\tilde \nu_2) \\
  & = & \tfrac{1}{d}(r_{14}\lambda_1+(r_{24}-b_3)\lambda_2+r_{34}\tilde \mu_1+(b_1-r_{44})\lambda_3) + r_{44}\tilde \mu_2 + b_3 \tilde \mu_3
 \end{eqnarray*} and
 \begin{eqnarray*}
  F(\mu_3) & = & \tfrac{1}{d}(r_{12}\lambda_1+r_{22}\lambda_2+r_{32}\tilde \mu_1+r_{42}\tilde \mu_2'+b_2\nu_1+b_4\tilde \nu_2) \\
  & = & \tfrac{1}{d}(r_{12}\lambda_1+(r_{22}-b_4)\lambda_2+r_{32}\tilde \mu_1+(b_2-r_{42})\lambda_3) + r_{42}\tilde \mu_2 + b_4 \tilde \mu_3
 \end{eqnarray*}
 are in $\Lambda_{\tilde z}$, and this holds if and only if
 \begin{equation}\label{Kongruenzbedingung} r_{14},r_{24}-b_3,r_{34},r_{44}-b_1,r_{12},r_{22}-b_4,r_{32},r_{42}-b_2 \in d\ZZ. \end{equation}
 Then, computing
 \[ A-SBS^{-1} = \begin{pmatrix} \tfrac{d}{\eta_2}(\tfrac{r_{12}}{d}\eta_1 + \tfrac{r_{22}-b_4}{d}\eta_2) & \tfrac{\sqrt{D}}{\eta_1^{\sigma}} (\tfrac{r_{14}}{d}\eta_1 + \tfrac{r_{24}-b_3}{d}\eta_2) \\
 \tfrac{d}{\sqrt{D}\eta_2}(\tfrac{-r_{32}}{d}\eta_2^{\sigma} + \tfrac{r_{42}-b_2}{d}(d\eta_1^{\sigma})) &
 \tfrac{1}{\eta_1^{\sigma}}(\tfrac{-r_{34}}{d}\eta_2^{\sigma} + \tfrac{r_{44}-b_1}{d}(d\eta_1^{\sigma})) \end{pmatrix}, \]
 we see that condition \eqref{Kongruenzbedingung} holds if and only if $A-SBS^{-1}$ is in $M_{D,d}(\eta_1,\eta_2)$.
 But this holds since $(A,B)$ is in $\Gamma_{D,d}(\eta_1,\eta_2)$, and thus \[ f:~ A_z \to A_{\tilde z},\quad [x] \mapsto [F(x)] \]
 is indeed the desired isomorphism of polarized Abelian varieties of dimension three with real multiplication by $\OD$ of type $d$.
\end{proof}

 Again it follows that we have a well-defined surjective map \[ \Phi:~ \Gamma_{D,d}(\eta_1,\eta_2) \backslash \HH^3 \twoheadrightarrow X_\A^{(3)},\quad [z] \mapsto [(A_z,H_z,S_z,\rho_z)], \]
 and we can state the main theorem of this subsection.

\begin{theo}\label{modulispacedim3}
 Let $D \in \NN$ be a non-square discriminant and let $d \in \NN$ be relatively prime to the conductor of~$D$.
 Then $X_{D,d}^{(3)}$ is non-empty if and only if~$d$ satisfies the prime factor condition in~$\OD$.
 In this case $X_{D,d}^{(3)}$ consists of $2^s$ irreducible components $X_\A^{(3)}$, where~$s$ is the number of splitting prime divisors of~$d$.
 For each component $X_\A^{(3)}$ defined by the ideal $\A$ we have an isomorphism \[ \Phi:~ \Gamma_{D,d}(\eta_1,\eta_2) \backslash \HH^3 \cong X_\A^{(3)},\quad [z] \mapsto [(A_z,H_z,s_z,\rho_z)] \]
 between the corresponding Hilbert modular variety and $X_\A^{(3)}$.
\end{theo}
\begin{proof}
 It remains to show that $\Phi$ is injective. To do this, let $z = (z_1,z_2,z_3)$, $\tilde z = (\tilde z_1,\tilde z_2,\tilde z_3) \in \HH^3$ such that there exists an isomorphism
 \[ f:~ (A_z,H_z,S_z,\rho_z) \to (A_{\tilde z},H_{\tilde z},S_{\tilde z},\rho_{\tilde z}) \] of polarized abelian varieties of dimension three with real multiplication by $\OD$ of type $d$.
 As in the proof of the previous proposition, we relate the representations of the induced isomorphisms on the complementary subvarieties $(S_z,C)$ and~$(\widetilde S_z,\widetilde C)$.

 Again we identify $\CC^2 \times \{ 0 \}$ with $\CC^2$ and $\{ (0,0) \} \times \CC$ with $\CC$, and thus we get $S_z = \CC^2/\Lambda_{(z_1,z_2)}$,
 $S_{\tilde z} = \CC^2/\tilde \Lambda_{(\tilde z_1,\tilde z_2)}$, $C=\CC/\Lambda_{z_3}$ and $\widetilde C=/\Lambda_{\tilde z_3}$.

 We denote by \[ F := \rho_a(f):~ \CC^3 \to \CC^3 \] the analytic representation of $f$, hence $F_2 := F|_{\CC^2}\!: \CC^2 \to \CC^2$ is an isomorphism of $\CC$-vector spaces.
 The symplectic $\OD$-module isomorphism
 \[ F_2' := \phi_{(\tilde z_1,\tilde z_2)}^{-1} \circ F|_{\Lambda_2} \circ \phi_{(z_1,z_2)}:~ \OD \oplus \tfrac{1}{\sqrt{D}}\A \to \OD \oplus \tfrac{1}{\sqrt{D}}\A \]
 is given by multiplication with $A^*$ for some matrix  $A \in {\rm SL}_2(\OD \oplus \tfrac{1}{\sqrt{D}}\A)$.
 Using the $\CC$-linearity of $F_2$, computations yield \[ (\tilde z_1,\tilde z_2) = (Az_1,A^{\sigma}z_2). \]

 Analogously, the $\CC$-linear isomorphism $F_1 := F|_{\CC}\!: \CC \to \CC$ induces a symplectic $\ZZ$-module isomorphism
 $F_1' := F_1|_{\Lambda_{z_3}}\!: \Lambda_{z_3} \to \tilde \Lambda_{\tilde z_3}$. Hence the matrix~$B$ of~$F_1'$ with respect to the symplectic bases
 $(\nu_1,\nu_2)=(1,-z_3)$ of~$\Lambda_{z_3}$ and $(\tilde \nu_1,\tilde \nu_2)=(1,-\tilde z_3)$ of~$\Lambda_{\tilde z_3}$ is in~$\SL_2(\ZZ)$.
 Again we can use $\CC$-linearity to compute \[ \tilde z_3 = B z_3. \]

 Now we have $\tilde z = (A,B).z$ for some $(A,B) \in {\rm SL}_2(\OD \oplus \tfrac{1}{\sqrt{D}}\A) \times {\rm SL}_2(\ZZ)$, and it remains to show that $(A,B)$ lies in $\Gamma_{D,d}(\eta_1,\eta_2)$.
 But this works exactly as in the proof of Proposition~\ref{WohldefiniertheitDim3}.
 We take the same bases $(\lambda_1,\lambda_2,\lambda_3,\mu_1,\mu_2,\mu_3)$ and $(\lambda_1,\lambda_2,\lambda_3,\tilde \mu_1,\tilde \mu_2,\tilde \mu_3)$ and denote the matrix of the rational
 representation with respect to these bases by $R=(r_{ij})_{i,j}$. Then we conclude with the same arguments, that $F$ maps $\Lambda_z$ into $\Lambda_{\tilde z}$ if and only if the congruence condition
 \eqref{Kongruenzbedingung} holds. Again, this is the case if and only if $A-SBS^{-1}$ lies in $M_{D,d}(\eta_1,\eta_2)$.
\end{proof}
\subsection{The universal family}\label{MatrixdesMonats}

Moduli spaces come with universal families, so the next natural step is to construct such a family for each connected component $X_\A^{(3)}$ of $X_{D,d}^{(3)}$.
\\[1em]
For a given ideal $\A$ of $\OD$, we choose a smart basis $(\eta_2^\sigma,d\eta_1^\sigma)$ for $\A$.
Over each $z \in \HH^3$, we have an Abelian variety $\CC^3 / \Pi_z \ZZ^6 \in X_\A^{(3)} \cong \Gamma_{D,d}(\eta_1,\eta_2) / \HH^3$, thus our aim is to construct a suitable semi-direct product structure
of the form $\Gamma_{D,d}(\eta_1,\eta_2) \ltimes \ZZ^6$ acting on $\HH^3 \times \CC^3$.

\begin{lem}
 For each $(A,B) \in \SL_2(\OD \oplus \tfrac{1}{\sqrt{D}}\A) \times \SL_2(\ZZ)$, there exists a unique matrix~$M(A,B) \in \bM_{6 \times 6}(\QQ(\sqrt{D}))$ independent of $z$ and which satisfies
 \begin{equation}\label{36Gleichungen} \Pi_z \cdot M(A,B) = D(A,B,z)^{-1} \cdot \Pi_{(A,B).z}. \end{equation}
\end{lem}
\begin{proof}
 Let $A = \left( \begin{smallmatrix} a_1 & a_2 \\ a_3 & a_4 \end{smallmatrix} \right)$ and $B = \left( \begin{smallmatrix} b_1 & b_2 \\ b_3 & b_4 \end{smallmatrix} \right)$.
 Recall that by the definition of the matrix $D(A,B,z)$ in the proof of Proposition~\ref{WohldefiniertheitDim3} we have
 \[ D(A,B,z)^{-1} = {\rm diag}(a_3 z_1 + a_4 , a_3^\sigma z_2 + a_4^\sigma , b_3 z_3 + b_4). \]
 Comparing the entries in \eqref{36Gleichungen}, we get 18 equations. But since 'independent of~$z$' just means that we can compare the coefficients of~$z$ separately, we have in fact 36 equations.
 Using the fact that the sign convention $\I(\eta_1 \eta_2^{\sigma}) = -\tfrac{1}{2}\I(\sqrt{D})$ for smart bases is equivalent to \[ \eta_1\eta_2^\sigma-\eta_1^\sigma\eta_2 = -\sqrt{D}, \]
 patiently solving these equations yield that the $(6 \times 6)$-matrix
\[ \left( \begin{smallmatrix}
 \tr \left( \frac{\eta_1\eta_2^\sigma a_4}{-\sqrt{D}} \right) &
 \tr \left( \frac{\mathcal{N}(\eta_2) a_4}{-\sqrt{D}} \right) &
 0 &
 \tr \left( \frac{(\eta_2^\sigma)^2 a_2}{-D} \right) &
 \tr \left( \frac{\eta_1^\sigma \eta_2^\sigma a_2}{D} \right) &
 \tr \left( \frac{\mathcal{N}(\eta_2) a_4}{-d\sqrt{D}} \right)
 \\
 \tr \left( \frac{\mathcal{N}(\eta_1) a_4}{\sqrt{D}} \right) &
 \tr \left( \frac{\eta_1^\sigma\eta_2 a_4}{\sqrt{D}} \right) &
 b_3 &
 \tr \left( \frac{\eta_1^\sigma \eta_2^\sigma a_2}{D} \right) &
 \tr \left( \frac{1}{2d}b_3 - \frac{(\eta_1^\sigma)^2 a_2}{D} \right) &
 \tr \left( -\frac{1}{2d}b_1 + \frac{\eta_1^\sigma \eta_2 a_4}{d\sqrt{D}} \right)
 \\
 \tr \left( \frac{\eta_1\eta_2 a_3}{d} \right) &
 \tr \left( \frac{\eta_2^2 a_3}{d} \right) &
 b_4 &
 \tr \left( \frac{\mathcal{N}(\eta_2) a_1}{d\sqrt{D}} \right) &
 \tr \left( \frac{1}{2d}b_4 - \frac{\eta_1^\sigma \eta_2 a_1}{d\sqrt{D}} \right) &
 \tr \left( -\frac{1}{2d}b_2 + \frac{\eta_2^2 a_3}{d^2} \right)
 \\
 \tr (-\eta_1^2 a_3) &
 \tr (-\eta_1\eta_2 a_3) &
 0 &
 \tr \left( \frac{\eta_1\eta_2^\sigma a_1}{-\sqrt{D}} \right) &
 \tr \left( \frac{\mathcal{N}(\eta_1) a_1}{\sqrt{D}} \right) &
 \tr \left( \frac{\eta_1\eta_2 a_3}{-d} \right)
 \\
 \tr (-\eta_1\eta_2 a_3) &
 \tr (-\eta_2^2 a_3) &
 0 &
 \tr \left( \frac{\mathcal{N}(\eta_2) a_1}{-\sqrt{D}} \right) &
 \tr \left( \frac{\eta_1^\sigma \eta_2 a_1}{\sqrt{D}} \right) &
 \tr \left( \frac{\eta_2^2 a_3}{-d} \right)
 \\
 0 & 0 & -db_3 & 0 & -b_3 & b_1
  \end{smallmatrix} \right) \]
 is the unique solution for $M(A,B)$.
\end{proof}

We want to define a group action of $\Gamma_{D,d}(\eta_1,\eta_2)$ on $\ZZ^6$ by \[ (A,B) \mapsto (r \mapsto M(A,B) \cdot r). \]
Therefore, we first check that $M(A,B)$ is indeed in $\bM_{6\times 6}(\ZZ)$.
\begin{lem}
 For all $(A,B) \in \Gamma_{D,d}(\eta_1,\eta_2)$, the entries of $M(A,B)$ are integers.
\end{lem}
\begin{proof}
 Writing out the condition $A-SBS^{-1} \in M_{D,d}(\eta_1,\eta_2)$ from the definition of $\Gamma_{D,d}(\eta_1,\eta_2)$ gives
\[ \begin{array}{ll} a_1 \in b_4 + \tfrac{d}{\eta_2}\OD, & a_2 \in \tfrac{\sqrt{D}}{d}\tfrac{\eta_2}{\eta_1^\sigma}b_3 + \tfrac{\sqrt{D}}{\eta_1^\sigma}\OD, \\
a_3 \in \tfrac{d}{\sqrt{D}}\tfrac{\eta_1^\sigma}{\eta_2}b_2 + \tfrac{d}{\sqrt{D}}\tfrac{1}{\eta_2} \A, & a_4 \in b_1 + \tfrac{1}{\eta_1^\sigma}\A. \end{array} \]
Writing
\[ \begin{array}{ll} a_1 = b_4 + \tfrac{d}{\eta_2}(x_1 \eta_1 + y_1 \eta_2) \\
                     a_2 = \tfrac{\sqrt{D}}{d}\tfrac{\eta_2}{\eta_1^\sigma}b_3 + \tfrac{\sqrt{D}}{\eta_1^\sigma}(x_2 \eta_1 + y_2 \eta_2), \\
                     a_3 = \tfrac{d}{\sqrt{D}}\tfrac{\eta_1^\sigma}{\eta_2}b_2 + \tfrac{d}{\sqrt{D}}\tfrac{1}{\eta_2} (x_3 \eta_2^\sigma + y_3 (d \eta_1^\sigma)), \\
                     a_4 = b_1 + \tfrac{1}{\eta_1^\sigma}(x_4 \eta_2^\sigma + y_4 (d \eta_1^\sigma))
\end{array} \]
with $x_i,y_i \in \ZZ$, we get
\[ M(A,B) = \left( \begin{smallmatrix}
 \tr \left( \frac{\eta_1\eta_2^\sigma a_4}{-\sqrt{D}} \right) 
 & \tr \left( \frac{\mathcal{N}(\eta_2) a_4}{-\sqrt{D}} \right) & 0
 & \tr \left( \frac{(\eta_2^\sigma)^2 a_2}{-D} \right) & -x_2 & \tr \left( \frac{\mathcal{N}(\eta_2) a_4}{-d\sqrt{D}} \right) 
 \\
 -x_4 & b_1 + dy_4 & b_3 & -x_2 & -y_2 & y_4
 \\
 -x_3 & b_2 + dy_3 & b_4 & -x_1 & -y_1 & y_3
 \\
 \tr (-\eta_1^2 a_3) & dx_3 & 0
 & \tr \left( \frac{\eta_1\eta_2^\sigma a_1}{-\sqrt{D}} \right) 
 & \tr \left( \frac{\mathcal{N}(\eta_1) a_1}{\sqrt{D}} \right) & x_3
 \\
 dx_3 & -db_2 - d^2y_3 & 0 & dx_1 & b_4 + dy_1 & -b_2 - dy_3
 \\
 0 & 0 & -db_3 & 0 & -b_3 & b_1
  \end{smallmatrix} \right). \]
From the description in Lemma~\ref{Idealnormalform} we see that $d$ divides the norm of every element in $\A$, in particular $d \mid \mathcal{N}(\eta_2)$.
Using the fact that $a_1, a_4 \in \OD$ and $a_3 \in \frac{1}{\sqrt{D}}\OD$ we see that the remaining entries containing $a_1$, $a_3$ or $a_4$ are integers.

Finally, as $a_2$ is in $(\frac{1}{\sqrt{D}}\A)^{-1} = \frac{\sqrt{D}}{d}\A^\sigma$, we can set $a_2 = \frac{\sqrt{D}}{d} (x\eta_2 + yd\eta_1)$ with $x,y \in \ZZ$ and get
\[ \tr \left( \tfrac{(\eta_2^\sigma)^2 a_2}{-D} \right) =
 -x \tfrac{\mathcal{N}(\eta_2)}{d} \tr \left( \tfrac{\eta_2^\sigma}{\sqrt{D}} \right) - y\tr \left( \tfrac{\eta_1(\eta_2^\sigma)^2}{\sqrt{D}} \right) \in \ZZ. \]
\end{proof}

\begin{lem}
 The assignment \[ ((\tilde A,\tilde B),\tilde r) \cdot ((A,B),r) := ((\tilde A A,\tilde B B),M(A,B) \cdot \tilde r+r)) \]
 defines a semi-direct group structure on the Cartesian product $\Gamma_{D,d}(\eta_1,\eta_2) \times \ZZ^6$.
\end{lem}
\begin{proof}
 It remains to show that $f_{(A,B)}\!: r \mapsto M(A,B) \cdot r$ defines an automorphism of~$\ZZ^6$ for each $(A,B) \in \Gamma_{D,d}(\eta_1,\eta_2)$.
 But as $M(I_2,I_2) = I_6$ holds, this would follow from compatibility of $(A,B) \mapsto f_{(A,B)}$ with the group structure.

 Now let $(A,B), (\tilde A,\tilde B) \in \Gamma_{D,d}(\eta_1,\eta_2)$ and write $A = \left( \begin{smallmatrix} a_1 & a_2 \\ a_3 & a_4 \end{smallmatrix} \right)$,
 $\tilde A = \left( \begin{smallmatrix} \tilde a_1 & \tilde a_2 \\ \tilde a_3 & \tilde a_4 \end{smallmatrix} \right)$, $B = \left( \begin{smallmatrix} b_1 & b_2 \\ b_3 & b_4 \end{smallmatrix} \right)$
 and $\tilde B = \left( \begin{smallmatrix} \tilde b_1 & \tilde b_2 \\ \tilde b_3 & \tilde b_4 \end{smallmatrix} \right)$.
 By using the relation \[ \tr(x)\tr(y) = \tr (x(y+y^\sigma)) \] for all $x,y \in \QQ(\sqrt{D})$, we see that the upper left entry of $M(\tilde A,\tilde B) \cdot M(A,B)$ is
\begin{eqnarray*}
 &  &   \tr \left( \frac{\eta_1\eta_2^\sigma \tilde a_4}{-\sqrt{D}} \right) \tr \left( \frac{\eta_1\eta_2^\sigma a_4}{-\sqrt{D}} \right)
      + \tr \left( \frac{\mathcal{N}(\eta_2) \tilde a_4}{-\sqrt{D}} \right) \tr \left( \frac{\mathcal{N}(\eta_1) a_4}{\sqrt{D}} \right)
      + 0 \\ 
 &  & + \tr \left( \frac{(\eta_2^\sigma)^2 \tilde a_2}{-D} \right) \tr (-\eta_1^2 a_3)
      + \tr \left( \frac{\eta_1^\sigma \eta_2^\sigma \tilde a_2}{D} \right) \tr (-\eta_1\eta_2 a_3)
      + 0 \\ 
 & =  & \tr \left( \frac{\eta_1\eta_2^\sigma \tilde a_4}{-\sqrt{D}} \left( \frac{\eta_1\eta_2^\sigma a_4}{-\sqrt{D}} + \frac{\eta_1^\sigma \eta_2 a_4^\sigma}{\sqrt{D}} \right) \right)
      + \tr \left( \frac{\eta_2 \eta_2^\sigma \tilde a_4}{-\sqrt{D}} \left( \frac{\eta_1 \eta_1^\sigma a_4}{\sqrt{D}} + \frac{\eta_1^\sigma \eta_1 a_4^\sigma}{-\sqrt{D}} \right) \right) \\
 &  & + \tr \left( \frac{\eta_2^\sigma \eta_2^\sigma \tilde a_2}{-D} (-\eta_1 \eta_1 a_3 - \eta_1^\sigma \eta_1^\sigma a_3^\sigma) \right)
      + \tr \left( \frac{\eta_1^\sigma \eta_2^\sigma \tilde a_2}{D} (-\eta_1\eta_2 a_3 - \eta_1^\sigma \eta_2^\sigma a_3^\sigma) \right) \\
 & =  & \tr \left( \frac{\eta_1\eta_2^\sigma \tilde a_4}{-\sqrt{D}} \frac{\eta_1\eta_2^\sigma a_4}{-\sqrt{D}}
      + \frac{\eta_2 \eta_2^\sigma \tilde a_4}{-\sqrt{D}} \frac{\eta_1 \eta_1^\sigma a_4}{\sqrt{D}}
      + \frac{\eta_2^\sigma \eta_2^\sigma \tilde a_2}{D} \eta_1 \eta_1 a_3
      + \frac{\eta_1^\sigma \eta_2^\sigma \tilde a_2}{-D} \eta_1\eta_2 a_3 \right) \\
 & =  & \tr \left( \frac{\eta_1 \eta_2^\sigma}{D} ((\eta_1\eta_2^\sigma - \eta_2 \eta_1^\sigma) \tilde a_4 a_4 + (\eta_1 \eta_2^\sigma - \eta_1^\sigma \eta_2) \tilde a_2 a_3) \right) \\
 & =  & \tr \left( \frac{\eta_1 \eta_2^\sigma (a_3 \tilde a_2 + a_4 \tilde a_4)}{-\sqrt{D}} \right).
\end{eqnarray*}
 Analogous computations for the other entries yield \begin{equation}\label{twist} M(\tilde A,\tilde B) \cdot M(A,B) = M (A \tilde A , B \tilde B). \end{equation}
\end{proof}

The twist in Equation~\eqref{twist} is no problem. It just means that $\Gamma_{D,d}(\eta_1,\eta_2)$ acts on $\ZZ^6$ from the right by multiplication from the left.
Nevertheless, we obtain a well-defined group structure and set \[ \widetilde \Gamma_{D,d}(\eta_1,\eta_2) := \Gamma_{D,d}(\eta_1,\eta_2) \ltimes \ZZ^6. \]
Moreover, for all $((A,B),r) \in \widetilde \Gamma_{D,d}(\eta_1,\eta_2)$ and $(z,u) \in \HH^3 \times \CC^3$ we define
\begin{equation}\label{UniverselleOperation} ((A,B),r).(z,u) := ((A,B).z,D(A,B,z) \cdot (u + \Pi_z r)). \end{equation}

\begin{lem}
 The assignment \eqref{UniverselleOperation} defines an action of $\widetilde \Gamma_{D,d}(\eta_1,\eta_2)$ on~$\HH^3 \times \CC^3$.
\end{lem}
\begin{proof}
 Clearly we have \[ ((I_2,I_2),0).(z,u) = (z,u) \] and using the two relations
 \[ \Pi_z \cdot M(A,B) = D(A,B,z)^{-1} \cdot \Pi_{(A,B).z} \] and \[ D(\tilde A,\tilde B,(A,B).z) \cdot D(A,B,z) = D(\tilde A A,\tilde B B,z), \] we get
 \begin{eqnarray*}
  &   & ((\tilde A, \tilde B),\tilde r).(((A,B),r).(z,u)) \\
  & = & ((\tilde A, \tilde B),\tilde r).((A,B).z,D(A,B,z) \cdot (u + \Pi_z r)) \\
  & = & ((\tilde A A,\tilde B B).z  ,D(\tilde A,\tilde B,(A,B).z) \cdot (D(A,B,z) \cdot (u + \Pi_z r) + \Pi_{(A,B).z} \tilde r)) \\
  & = & ((\tilde A A,\tilde B B).z , D(\tilde A A,\tilde B B,z) \cdot (u + \Pi_z r + D(A,B,z)^{-1} \cdot \Pi_{(A,B).z} \tilde r))) \\
  & = & ((\tilde A A,\tilde B B).z , D(\tilde A A,\tilde B B,z) \cdot (u + \Pi_z (r + M(A,B) \cdot \tilde r))) \\
  & = & ((\tilde A A, \tilde B B),M(A,B) \cdot \tilde r + r).(z,u).
 \end{eqnarray*}
\end{proof}

Now we can prove the final result of this subsection.

\begin{theo}
 The universal family (with singularities) for the irreducible component $X^{(3)}_\A \subseteq X^{(3)}_{D,d}$ defined by the ideal
 $\A = \langle \eta_2^\sigma,d\eta_1^\sigma \rangle_\ZZ$ is
 \[ \pi:~ \widetilde \Gamma_{D,d}(\eta_1,\eta_2) \backslash (\HH^3 \times \CC^3) \to X^{(3)}_\A,\quad [(z,u)] \mapsto (A_z,H_z,S_z,\rho_z). \]
\end{theo}
\begin{proof}
 Left multiplication with $D(A,B,z)$ is just the analytic representation of the isomorphism \[ (A_z,H_z,S_z,\rho_z) \cong (A_{(A,B).z},H_{(A,B).z},S_{(A,B).z},\rho_{(A,B).z}), \]
 and thus the fibre over each $(A_z,H_z,S_z,\rho_z)$ is isomorphic to $\CC^3/\Pi_z\ZZ^6$, equipped with the usual polarization and $\OD$-structure.
\end{proof}
\subsection{Lower and upper bounds}

The group $\Gamma_{D,d}(\eta_1,\eta_2)$ is much to unwieldy to use it for problems like counting cusps.
For this reason, we construct a lower and an upper bound for $\Gamma_{D,d}(\eta_1,\eta_2)$, both being direct products and independent of the choice of the smart basis
and with both inclusions of finite index.
Our focus will be on the lower bound, as this defines a covering space of $X_\A^{(3)}$, and we will show that this bound is optimal with respect to the requirements above.

\bigskip
Again, let $d$ be a natural number satisfying the prime factor condition in~$\OD$, let $\A$ be a primitive ideal in $\OD$ of norm $d$ and let $(\eta_2^\sigma,d\eta_1^\sigma)$ be a smart basis of $\A$.
Recall that our model for the irreducible component $X^{(3)}_\A$ defined by $\A$ of our moduli space $X^{(3)}_{D,d}$ is \[ \Gamma_{D,d}(\eta_2,\eta_2) \backslash \HH^3 \cong X^{(3)}_\A, \]
where $\Gamma_{D,d}(\eta_1,\eta_2)$ was defined to be the group \[ \{ (A,B) \in \SL_2(\OD \oplus \tfrac{1}{\sqrt{D}}\A) \times \SL_2(\ZZ) :~ A-SBS^{-1} \in M_{D,d}(\eta_1,\eta_2) \}. \]
Unfortunately, this group depends on the choice of the pair $(\eta_1,\eta_2)$.
We handle this problem by finding two groups independent of this choice, an upper and a lower bound for $\Gamma_{D,d}(\eta_1,\eta_2)$, such that both inclusions have finite indices.
Additionaly, for better applicability, we require these two groups to be direct products, so they would be much better to handle with. To be concrete, setting
\[ \widetilde \Gamma_{{\rm lb}}(D,\A) := \SL\begin{pmatrix} 1+\A & \sqrt{D} \OD \\ \tfrac{1}{\sqrt{D}}\A^2 & 1+\A \end{pmatrix} \] and denoting by
$\Gamma(d) = \{ B \in {\rm SL}_2(\ZZ) : B \equiv I_2 \mod d \}$ the principal congruence subgroup of level $d$, we define
\[ \Gamma_{{\rm lb}}(D,\A) := \widetilde \Gamma_{{\rm lb}}(D,\A) \times \Gamma(d), \]
and \[ \Gamma_{{\rm ub}}(D,\A) := \SL_2(\OD \oplus \tfrac{1}{\sqrt{D}}\A) \times \SL_2(\ZZ). \]
\begin{prop}
 There is a chain of group inclusions of finite indices \[ \Gamma_{{\rm lb}}(D,\A) \subseteq \Gamma_{D,d}(\eta_1,\eta_2) \subseteq \Gamma_{{\rm ub}}(D,\A). \]
 The lower bound $\Gamma_{{\rm lb}}(D,\A)$ is optimal in the sense that it is the largest group of the form $\Gamma_1 \times \Gamma_2 \subseteq \Gamma_{D,d}(\eta_1,\eta_2)$ with
 $\Gamma_1 \subseteq \SL_2(\OD \oplus \tfrac{1}{\sqrt{D}}\A)$ and $\Gamma_2 \subseteq \SL_2(\ZZ)$ independent of the choice of the smart basis $(\eta_2^\sigma,d\eta_1^\sigma)$.
\end{prop}
\begin{proof}
 It is clear that $\Gamma_{D,d}(\eta_1,\eta_2)$ is a finite index subgroup of $\Gamma_{{\rm ub}}(D,\A)$.

 For the lower bound, recall that we defined
 \[ M_{D,d}(\eta_1,\eta_2) = \begin{pmatrix} \tfrac{d}{\eta_2}\OD & \tfrac{\sqrt{D}}{\eta_1^{\sigma}}\OD \\ \tfrac{d}{\sqrt{D}}\tfrac{1}{\eta_2}\A & \tfrac{1}{\eta_1^{\sigma}}\A \end{pmatrix}
 \mbox{ and } S = S(D,d,\eta_1,\eta_2) = \begin{pmatrix} 0 & \tfrac{\sqrt{D}}{d}\tfrac{\eta_2}{\eta_1^{\sigma}} \\ 1 & 0 \end{pmatrix}. \]

 For $A = \left( \begin{smallmatrix} a_1 & a_2 \\ a_3 & a_4 \end{smallmatrix} \right) \in \SL_2(\OD \oplus \tfrac{1}{\sqrt{D}}\A)$ and
 $B = \left( \begin{smallmatrix} b_1 & b_2 \\ b_3 & b_4 \end{smallmatrix} \right) \in \SL_2(\ZZ)$ we get \[ A-SBS^{-1} =
 \begin{pmatrix} a_1-b_4 & a_2-\tfrac{\sqrt{D}}{d}\tfrac{\eta_2}{\eta_1^{\sigma}}b_3 \\ a_3-\tfrac{d}{\sqrt{D}}\tfrac{\eta_1^{\sigma}}{\eta_2}b_2  & a_4-b_1 \end{pmatrix}. \]
 Hence the largest subgroups $\Gamma_1 \subseteq \SL_2(\OD \oplus \tfrac{1}{\sqrt{D}}\A)$ and $\Gamma_2 \subseteq \SL_2(\ZZ)$
 with $\Gamma_1 \times \Gamma_2$ contained in $\Gamma_{D,d}(\eta_1,\eta_2)$ are
 \[ \Gamma_1 = {\rm SL}\begin{pmatrix} 1 + (\OD \cap \tfrac{d}{\eta_2}\OD) & \tfrac{\sqrt{D}}{d}\A^\sigma \cap \tfrac{\sqrt{D}}{\eta_1^{\sigma}}\OD \\ 
    \tfrac{1}{\sqrt{D}}\A \cap \tfrac{d}{\sqrt{D}}\tfrac{1}{\eta_2}\A & 1 + (\OD \cap \tfrac{1}{\eta_1^{\sigma}}\A) \end{pmatrix} \] and
 \[ \Gamma_2 =  {\rm SL}_2(\ZZ) \cap \begin{pmatrix} 1 + \tfrac{1}{\eta_1^{\sigma}}\A, & \tfrac{1}{\eta_1^\sigma}\A \\ \tfrac{d}{\eta_2}\OD, & 1 + \tfrac{d}{\eta_2}\OD \end{pmatrix} \]
 We will compare these groups with those in the definition of $\Gamma_{{\rm lb}}(D,\A)$.
 Let us first look at $\Gamma_2$. Because of
 \[ \begin{array}{lllll} \tfrac{1}{\eta_1^{\sigma}}\A \cap \ZZ & = & \langle \tfrac{\eta_2^\sigma}{\eta_1^\sigma},d \rangle_\ZZ \cap \ZZ & = & d\ZZ \quad \mbox{ and} \\
    \tfrac{d}{\eta_2}\OD \cap \ZZ & = & \langle d \tfrac{\eta_1}{\eta_2},d \rangle_\ZZ \cap \ZZ & = & d\ZZ \end{array} \]
 we get $\Gamma_2 = \Gamma(d)$.

 Now consider $\Gamma_1$. We claim that $(\eta_1\eta_2^\sigma,d)$ is also a smart basis for $\A$.
 To check this, just use $\I(\eta_1 \eta_2^{\sigma}) = -\tfrac{1}{2}\I(\sqrt{D})$ and the normal form $\A = \langle d,\omega \rangle_\ZZ$ of
 Lemma~\ref{Idealnormalform} to see that a matrix in $\ZZ^{2 \times 2}$ transforming $(d,\omega)$ to $(d,\eta_1\eta_2^\sigma)$ is necessarily in ${\rm GL}_2(\ZZ)$.

 Now we see that
 \begin{eqnarray*}
  \OD \cap \tfrac{d}{\eta_2}\OD & = & \OD \cap \tfrac{d}{\eta_2} \langle \eta_1, \eta_2 \rangle_\ZZ \\
   & = & \langle \eta_1\eta_2^\sigma, 1 \rangle_\ZZ \cap \langle \tfrac{1}{\mathcal{N}(\eta_2)/d}\eta_1\eta_2^\sigma, d \rangle_\ZZ \\
   & = & \langle \eta_1\eta_2^\sigma, d \rangle_\ZZ \\
   & = & \A,
 \end{eqnarray*}
 \begin{eqnarray*}
  \tfrac{\sqrt{D}}{d}\A^\sigma \cap \tfrac{\sqrt{D}}{\eta_1^{\sigma}}\OD
   & = & \tfrac{\sqrt{D}}{d} \langle \eta_2,d\eta_1 \rangle_\ZZ \cap \tfrac{\sqrt{D}}{\eta_1^{\sigma}} \langle 1,\eta_1^\sigma\eta_2 \rangle_\ZZ \\
   & = & \sqrt{D}~ \left( \langle \mathcal{N}(\eta_1) \tfrac{1}{\eta_1^{\sigma}}, \tfrac{1}{d}\eta_2 \rangle_\ZZ \cap \langle \tfrac{1}{\eta_1^{\sigma}},\eta_2 \rangle_\ZZ \right) \\
   & = & \sqrt{D}~ \langle \mathcal{N}(\eta_1) \tfrac{1}{\eta_1^{\sigma}}, \eta_2 \rangle_\ZZ \\
   & = & \sqrt{D}~ \OD
 \end{eqnarray*} and
 \begin{eqnarray*}
  \OD \cap \tfrac{1}{\eta_1^{\sigma}}\A & = & \OD \cap \tfrac{1}{\eta_1^{\sigma}} \langle \eta_2^\sigma,d\eta_1^\sigma \rangle_\ZZ \\
   & = & \langle 1,\eta_1\eta_2^\sigma \rangle_\ZZ \cap \langle d,\tfrac{1}{\mathcal{N}(\eta_1)} \eta_1 \eta_2^\sigma \rangle_\ZZ \\
   & = & \langle d,\eta_1\eta_2^\sigma \rangle_\ZZ \\
   & = & \A.
 \end{eqnarray*}
 The tricky part is to determine $\tfrac{1}{\sqrt{D}}\A \cap \tfrac{d}{\sqrt{D}}\tfrac{1}{\eta_2}\A$.
 Writing \[ \tfrac{d}{\eta_2}\A = \langle \tfrac{d}{\mathcal{N}(\eta_2)}(\eta_2^\sigma)^2 , \tfrac{d^2}{\mathcal{N}(\eta_2)}\eta_1^\sigma\eta_2^\sigma \rangle_\ZZ \] and
 \[ \A^2 = \langle (\eta_2^\sigma)^2 , d \eta_1^\sigma \eta_2^\sigma , d^2 (\eta_1^\sigma)^2 \rangle_\ZZ, \]
 we see, using $d \mid \mathcal{N}(\eta_2)$, that $\A^2$ is contained in $\tfrac{d}{\eta_2}\A$ and thus we have $\A^2 \subseteq \A \cap \tfrac{d}{\eta_2}\A$.

 Collecting everything together, so far we have shown that $\Gamma_{{\rm lb}}(D,\A)$ is contained in $\Gamma_{D,d}(\eta_1,\eta_2)$,
 and that this group is optimal in the sense of the proposition, if at least for one choice of the smart basis we have $\A^2 = \A \cap \tfrac{d}{\eta_2}\A$.
 Therefore, we choose a smart basis of the form $(\eta_2^\sigma,d\eta_1^\sigma) = (\omega^\sigma,d)$ with $\omega = a+\gamma_D$ for $\A$ and show that the index of
 $\A \cap \tfrac{d}{\omega}\A$ in $\A$ is $d$.
 We have
\[ \tfrac{d}{\omega} \A = \tfrac{d}{\mathcal{N}(\omega)}\omega^\sigma \A  = \langle \tfrac{d\tr(\omega)}{\mathcal{N}(\omega)}\omega^\sigma-d, \tfrac{d^2}{\mathcal{N}(\omega)}\omega^\sigma \rangle_\ZZ. \]
 Now, given any $x,y \in \ZZ$ we see that the linear combination
 \[ x \left( \tfrac{d\tr(\omega)}{\mathcal{N}(\omega)}\omega^\sigma-d \right) + y \left( \tfrac{d^2}{\mathcal{N}(\omega)}\omega^\sigma \right) \]
 is an element of $\A$ if and only if \[ x\tr(\omega) + yd \equiv 0 \mod \tfrac{\mathcal{N}(\omega)}{d}. \]
 Next we identify the $\ZZ$-module $\A$ with $\ZZ^2$ by the coordinate isomorphism given by the basis $(\omega^\sigma,d)$.
 Under this isomorphism, $\A \cap \tfrac{d}{\eta_2}\A$ is the image of \[ L := \{ (x,y) \in \ZZ^2 : x\tr(\omega) + yd \equiv 0 \mod \tfrac{\mathcal{N}(\omega)}{d} \} \]
 under \[ \varphi:~ L \to \ZZ^2 ,\quad (x,y) \mapsto \left( (x\tr(\omega) + yd)\tfrac{d}{\mathcal{N}(\omega)}, -x \right). \]
 This gives rise to the two chains of inclusions
 \[ \ZZ^2 \supseteq L \supseteq \tfrac{\mathcal{N}(\omega)}{d} \varphi(L)\quad \mbox{and}\quad \ZZ^2 \supseteq \varphi(L) \supseteq \tfrac{\mathcal{N}(\omega)}{d} \varphi(L). \]
 Three of the four occuring indices are known:
 We have $\gcd(d,\tr(\omega),\tfrac{\mathcal{N}(\omega)}{d}) = 1$ by Lemma~\ref{Idealnormalform}, and hence $[\ZZ^2 : L] = \tfrac{\mathcal{N}(\omega)}{d}$.
 For the second one, we mention that we have
 \[ \tfrac{\mathcal{N}(\omega)}{d} \varphi(L) = \left( \begin{smallmatrix} \tr(\omega) & d \\ -\tfrac{\mathcal{N}(\omega)}{d} & 0 \end{smallmatrix} \right) L \]
 and thus we get the index $[L:\tfrac{\mathcal{N}(\omega)}{d} \varphi(L)] = \mathcal{N}(\omega)$.
 Therefore, we have \[ [\ZZ^2:\tfrac{\mathcal{N}(\omega)}{d} \varphi(L)] = \tfrac{\mathcal{N}(\omega)^2}{d} \] and comparing with
 $[\varphi(L) : \tfrac{\mathcal{N}(\omega)}{d} \varphi(L)] = \left( \tfrac{\mathcal{N}(\omega)}{d} \right)^2$, we see that
 \[ [\A : \A \cap \tfrac{d}{\eta_2}\A] = [\ZZ^2 : \tfrac{\mathcal{N}(\omega)}{d} \varphi(L)] = d. \]

 It follows that $\Gamma_{{\rm lb}}(D,\A) = \widetilde \Gamma_{{\rm lb}}(D,\A) \times \Gamma(d)$ indeed has the desired properties.
\end{proof}

We can give at least an estimate for the indices.
\begin{prop}
 Let $\A$ be a primitive ideal of $\OD$ with norm $d$ relatively prime to the conductor of $D$.
 Then there is a group homomorphism \[ \varphi:~ \SL_2(\OD \oplus \tfrac{1}{\sqrt{D}}\A) \to \SL_2(\ZZ/d\ZZ) \] with $\ker(\varphi) = \widetilde \Gamma_{{\rm lb}}(D,\A)$.
 In particular, we have the index estimate \[ [\Gamma_{{\rm ub}}(D,\A) : \Gamma_{{\rm lb}}(D,\A)] \leq |\SL_2(\ZZ/d\ZZ)|^2 = ( d^3 \cdot \prod_{p \mid d} (1-p^{-2}) )^2. \]
\end{prop}
\begin{proof}
 The primitive ideal $\A$ is of the form $\A = \langle d,\omega \rangle_\ZZ$ for some $\omega = a+\gamma_D$ with~$a \in \ZZ$.
 Therefore, we can write \[ \SL_2(\OD \oplus \tfrac{1}{\sqrt{D}}\A) = \SL \begin{pmatrix} \langle 1,\omega \rangle_\ZZ & \sqrt{D} \langle 1, \tfrac{1}{d}\omega^\sigma \rangle_\ZZ \\
  \tfrac{1}{\sqrt{D}}\langle d,\omega \rangle_\ZZ & \langle 1,\omega \rangle_\ZZ \end{pmatrix}. \]
  We define the map $\varphi: \SL_2(\OD \oplus \tfrac{1}{\sqrt{D}}\A) \to \SL_2(\ZZ/d\ZZ)$ by
 \[ A = \begin{pmatrix} x_1+y_1\omega & \sqrt{D}(x_2+y_2\tfrac{1}{d}\omega^\sigma) \\ \tfrac{1}{\sqrt{D}}(x_3d+y_3\omega) & x_4+y_4\omega \end{pmatrix} \mapsto
  \begin{pmatrix} [x_1]_d & [y_2]_d \\ [x_3\tr(\omega)+y_3 \tfrac{\mathcal{N}(\omega)}{d}]_d & [x_4]_d \end{pmatrix} \] with $x_i,y_i \in \ZZ$.
 Since we have \begin{align*} 1 = \det(A) = & x_1x_4-y_2(x_3\tr(\omega)+y_3\tfrac{\mathcal{N}(\omega)}{d})-(y_1y_4\tfrac{\mathcal{N}(\omega)}{d}-x_2x_3)d \\
  & +(x_1y_4+y_1x_4+y_1y_4\tr(\omega)-x_2y_3+y_2x_3)\omega, \end{align*} we get $x_1x_4-y_2(x_3\tr(\omega)+y_3\tfrac{\mathcal{N}(\omega)}{d}) \equiv 1 \mod d$, hence $\varphi$ is well-defined.
 A straight-forward computation using $d \mid \mathcal{N}(\omega)$ yields that $\varphi$ is a homomorphism.

 It remains to show that the kernel is $\widetilde \Gamma_{{\rm lb}}(D,\A)$. We start with determining the ideal $\A^2$.
 Recall that $d$ satisfies the prime factor condition in $\OD$, as $\A$ is primitive.
 Therefore, we can write $d=PQ$, where~$P$ is the ramified part and~$Q$ is the splitting part in the prime factorization of $d$. Then we get \[ \A^2 = \langle PQ^2,P\omega+nPQ \rangle_\ZZ \]
 for some $n \in \NN_0$ with \[ Q^2 \mid \mathcal{N}(\omega+nPQ) = \mathcal{N}(\omega)+n^2Q^2+nQ\tr(\omega). \]
 Since $\A^2$ is also equal to $\langle P^2Q^2,PQ\omega, \tr(\omega)\omega-\mathcal{N}(\omega) \rangle_\ZZ$, we get the three additional conditions
 \[ P \mid \tr(\omega),\quad \gcd(Q,\tr(\omega))=1,\quad {\rm and}\quad \gcd(P^2,\mathcal{N}(\omega))=P. \]
 Now consider the lower left entry of $\varphi(A)$. Using these four conditions, we deduce
 \begin{align*}
  d \mid x_3\tr(\omega)+y_3 \tfrac{\mathcal{N}(\omega)}{d}  \Leftrightarrow & P^2Q^2 \mid x_3PQ\tr(\omega)+y_3\mathcal{N}(\omega) \\
   \Leftrightarrow & P^2 \mid y_3\mathcal{N}(\omega)~ \wedge~ Q^2 \mid x_3PQ\tr(\omega)-y_3(n^2Q^2+nQ\tr(\omega)) \\
   \Leftrightarrow & P \mid y_3~ \wedge~ Q \mid x_3P-y_3n,
 \end{align*}
 and this holds if and only if $x_3PQ+y_3\omega = (x_3P-y_3n)Q+y_3(\omega +nQ)$ is in $\A^2$.
 Hence, we have shown that the lower left entry of $\varphi(A)$ is zero in $\ZZ/d\ZZ$ if and only if the lower left entry of~$A$ is in~$\tfrac{1}{\sqrt{D}}\A^2$.
 The other three entries are easy to check, thus we get \[ \ker(\varphi) = \SL\begin{pmatrix} 1+\A & \sqrt{D} \OD \\ \tfrac{1}{\sqrt{D}}\A^2 & 1+\A \end{pmatrix} = \widetilde \Gamma_{{\rm lb}}(D,\A). \]
 Finally, the estimate follows from the well known fact \[ |\SL_2(\ZZ/d\ZZ)| = [\Gamma(1) : \Gamma(d)] = d^3 \prod_{p \mid d} (1-p^{-2}). \]
\end{proof}

A direct consequence is the following corollary, which could be used amongst other things for computating boundaries of the number of cusps of the Hilbert modular varieties $X_\A^{(3)}$.
\begin{cor}
  The identity on $\HH^3$ yields a chain of coverings of finite degrees
 \[ \Gamma_{{\rm lb}}(D,\A) \backslash \HH^3 \twoheadrightarrow \Gamma_{D,d}(\eta_2,\eta_2) \backslash \HH^3 \twoheadrightarrow \Gamma_{{\rm ub}}(D,\A) \backslash \HH^3. \]
 For the composition $\pi\!: \Gamma_{{\rm lb}}(D,\A) \backslash \HH^3 \twoheadrightarrow \Gamma_{{\rm ub}}(D,\A) \backslash \HH^3$, we can estimate the degree by
 \[ \deg(\pi) \leq \left \{ \begin{array}{ll} \tfrac{1}{4}( d^3 \prod_{p \mid d} (1-p^{-2}) )^2 & \mbox{ if } d>2 \\ 36 & \mbox{ if } d=2. \end{array} \right. \]
\end{cor}
\begin{proof}
 This is a direct consequence of the previous proposition, the fact that $\PSL_2(\OD \oplus \tfrac{1}{\sqrt{D}}\A) \times \PSL_2(\ZZ)$ acts faithfully on $\HH^3$ and that $-I_2$ is in
 $\widetilde \Gamma_{{\rm lb}}(D,\A)$ respectively $\Gamma(d)$ if and only if $d=2$.
\end{proof}
\subsection{Level structures}

By what we have seen in Section~\ref{dim2}, the Hilbert modular variety \[ \Gamma_{{\rm ub}}(D,\A) \backslash \HH^3 \cong X_\A \times \SL_2(\ZZ) \backslash \HH \] is the moduli space of isomorphism classes of products
of polarized complex Abelian surfaces with real multiplication by $\OD$ of type $\A$ and principally polarized elliptic curves.

Recall that we have $\Gamma_{{\rm lb}}(D,\A) = \widetilde \Gamma_{{\rm lb}}(D,\A) \times \Gamma(d)$.
To see what the space \[ \Gamma_{{\rm lb}}(D,\A) \backslash \HH^3 = \widetilde \Gamma_{{\rm lb}}(D,\A) \backslash \HH^2 \times \Gamma(d) \backslash \HH \] is parameterizing,
we define level structures for each of the two factors.

\bigskip
For the one-dimensional factor, we follow \cite{LangeBirke} Chapter~8.3.1. We equip~$\ZZ^2$ with the alternating form
\[ B:~ \ZZ^2 \times \ZZ^2 \to \ZZ,\quad ((x_1,y_1),(x_2,y_2)) \mapsto x_1y_2 - x_2y_1. \]
Then, for every $\ZZ$-module $\Lambda \cong \ZZ^2$ with an alternating form $E: \Lambda \times \Lambda \to \ZZ$ and every $d \in \NN$, we define the multiplicative alternating form
\[ \exp_d(E):~ (\Lambda/d\Lambda) \times (\Lambda/d\Lambda) \to \CC^*,\quad (\lambda_1,\lambda_2) \mapsto \exp \left(2\pi i \tfrac{E(\lambda_1,\lambda_2)}{d} \right). \]
\begin{defi}
 Let $(C=V/\Lambda,H)$ be a principally polarized elliptic curve and let $d \in \NN$. A \emph{level $d$-structure} on $(C,H)$ is a symplectic $\ZZ$-module isomorphism
 \[ \psi:~ (\Lambda/d\Lambda,\exp_d(E_H)) \to (\ZZ^2/d\ZZ^2,\exp_d(B)). \]
\end{defi}
 An \emph{isomorphism of principally polarized elliptic curves with level $d$-structure}, say $(C=V/\Lambda,H,\psi)$ and $(C'=V'/\Lambda',H',\psi')$, is an isomorphism \[ f:~ (C,H) \to (C',H') \]
 of polarized Abelian varieties, such that for the induced isomorphism  \[ [\rho_r(f)]:~ \Lambda/d\Lambda \to \Lambda'/d\Lambda' \] we have \[ \psi = \psi' \circ [\rho_r(f)]. \]
\begin{prop}
 The quotient $\Gamma(d) \backslash \HH$ is a moduli space for isomorphism classes of principally polarized elliptic curves with level $d$-structure.
\end{prop}
\begin{proof}
 We use the usual convention for ${\rm SL}_2(\ZZ) \backslash \HH$ as the moduli space of principally polarized elliptic curves as follows.
 We associate to every $z \in \HH$ the principally polarized elliptic curve $(C_z=\CC/\Lambda_z,H_z)$ with $\Lambda_z = \langle z,1 \rangle_\ZZ$ and the Hermitian form $H_z$ given by Im$(z)^{-1}$.
 It follows, that the rational representation of the corresponding isomorphism $C_z \cong C_{Az}$ for $A \in {\rm SL}_2(\ZZ)$ with respect to the latter bases $(z,1)$ and $(Az,1)$ is multiplication
 with $(A^t)^{-1}$ (\cite{LangeBirke} for example).

 Now we equip for every $z \in \HH$ the curve $(C_z=\CC/\Lambda_z,H_z)$ with the level $d$-structure \[ \psi_z:~ \Lambda_z/d\Lambda_z \to \ZZ^2/d\ZZ^2,\quad xz+y \mapsto (x,y). \]
 It is clear that every principally polarized elliptic curve with level $d$-structure is isomorphic to $(C_z,H_z,\psi)$ for some $z \in \HH$ and some $\psi: \Lambda_z/d\Lambda_z \to \ZZ^2/d\ZZ^2$.
 Because the isomorphism $\psi_z \circ \psi^{-1}: \ZZ^2/d\ZZ^2 \to \ZZ^2/d\ZZ^2$ is symplectic, it is given by a matrix $B \in {\rm SL}_2(\ZZ/d\ZZ)$. We lift $B$ to a matrix $A^t \in {\rm SL}_2(\ZZ)$,
 denote by $f: C_z \to C_{Az}$ the corresponding isomorphism and get a commutative diagram
 \[ \begin{CD} \Lambda_z/d\Lambda_z @> \id >> \Lambda_z/d\Lambda_z @>[\rho_r(f)]>> \Lambda_{Az}/d\Lambda_{Az} \\
  @VV \psi V @VV \psi_z V @VV \psi_{Az} V \\
  \ZZ^2/d\ZZ^2 @>B>> \ZZ^2/d\ZZ^2 @> B^{-1} >> \ZZ^2/d\ZZ^2 \end{CD} \]
 and deduce that $f: (C_z,H_z,\psi) \to (C_{Az},H_{Az},\psi_{Az})$ is an isomorphism of principally polarized elliptic curves with level $d$-structure.
 Thus every principally polarized elliptic curve with level $d$-structure is isomorphic to $(C_z,H_z,\psi_z)$ for some $z \in \HH$.
 Taking $A \in {\rm SL}_2(\ZZ)$, we see that $(C_z,H_z,\psi_z)$ and $(C_{Az},H_{Az},\psi_{Az})$ are isomorphic if and only if $(A^t)^{-1}$ induces the identity on $\ZZ^2/d\ZZ^2$,
 and this holds if and only if $A$ is in $\Gamma(d)$.
\end{proof}

Now we consider the two-dimensional factor $X_\A$, the moduli space of polarized Abelian surfaces $(A,H,\rho)$ with real multiplication by $\OD$ of type $\A$.
For abbreviation we just say that $(A,H,\rho)$ is a polarized Abelian surface of type $\A$.
\begin{defi}
 Let $(A=V/\Lambda,H,\rho)$ be a polarized Abelian surface of type~$\A$. A \emph{level $\A$-structure} on $(A,H,\rho)$ is a symplectic $\OD$-module isomorphism
 \[ \psi:~ (\Lambda,E_H) \to (\OD \oplus \tfrac{1}{\sqrt{D}} \A,\langle \cdot , \cdot \rangle). \] 
\end{defi}
We denote the induced isomorphism on the quotient $\OD$-modules by
\[ [\psi]:~ \Lambda / \A\Lambda \to (\OD \oplus \tfrac{1}{\sqrt{D}} \A) / (\A \oplus \tfrac{1}{\sqrt{D}}\A^2) ,\quad [\lambda] \mapsto [\psi(\lambda)]. \]
(Here $\A\Lambda$ is of course defined to be the $\ZZ$-algebra generated by the products $a\lambda$ with $a \in \A$ and $\lambda \in \Lambda$, which is also an $\OD$-module.)

An \emph{isomorphism of polarized Abelian surfaces of type $\A$ with level $\A$-structure}, say $(A=V/\Lambda,H,\rho,\psi)$ and $(A'=V'/\Lambda',H',\rho',\psi')$,
is an isomorphism \[ f:~ (A,H,\rho) \to (A',H',\rho') \] of polarized Abelian surfaces of type~$\A$, such that \[ [\psi] = [\psi' \circ \rho_r(f)]. \]

\begin{prop}
 The quotient $\widetilde \Gamma_{{\rm lb}}(D,\A) \backslash \HH^2$ is a moduli space for isomorphism classes of polarized Abelian surfaces of type~$\A$ with level $\A$-structure.
\end{prop}
\begin{proof}
 In Section~\ref{dim2} we associated to each $z=(z_1,z_2) \in \HH^2$ the polarized Abelian surface $(A_z=\CC^2/\Lambda_z,H_z,\rho_z)$ of type $\A$.
 Therefore we used the embedding \[ \phi_z:~ \OD \oplus \tfrac{1}{\sqrt{D}}\A \hookrightarrow \CC^2,\quad (x,y) \mapsto (x+yz_1,x^{\sigma}+y^{\sigma}z_2) \] to get the lattice
 $\Lambda_z = \phi_z(\OD \oplus \tfrac{1}{\sqrt{D}}\A)$.
 It turned out that two such polarized Abelian surfaces of type $\A$, say $(A_z,H_z,\rho_z)$ and $(A_{\tilde z},H_{\tilde z},\rho_{\tilde z})$, are isomorphic if and only if $\tilde z = Mz$
 for some matrix \[ M = \begin{pmatrix} a_1 & a_2 \\ a_3 & a_4 \end{pmatrix} \in {\rm SL}_2(\OD \oplus \tfrac{1}{\sqrt{D}}\A). \]
 Furthermore, we pointed out that the analytic representation of the corresponding isomorphism is given by the diagonal matrix $D(M,z)$, and that we have a commutative diagram
 \[ \begin{CD} \OD \oplus \tfrac{1}{\sqrt{D}}\A @>\phi_z>> \CC^2 \\
  @VVM^*V @VVD(M,z)V \\
  \OD \oplus \tfrac{1}{\sqrt{D}}\A @>\phi_{Mz}>> \CC^2 \end{CD} \]
 with \[ M^* = \begin{pmatrix} a_1 & -a_2 \\ -a_3 & a_4 \end{pmatrix} \in {\rm SL}_2(\OD \oplus \tfrac{1}{\sqrt{D}}\A). \]
 Now we define a level $\A$-structure on $(A_z,H_z,\rho_z)$ simply by setting \[ \psi_z := \phi_z^{-1}:~ \Lambda_z \to \OD \oplus \tfrac{1}{\sqrt{D}}\A. \]
 By taking a look at the diagram above it is clear that this isomorphism is also an isomorphism of polarized Abelian surfaces of type $\A$ with level $\A$-structure if and only if $M^*$ induces
 the identity on $(\OD \oplus \tfrac{1}{\sqrt{D}} \A) / (\A \oplus \tfrac{1}{\sqrt{D}}\A^2)$. Moreover, this holds if and only if $M^*$ and thus $M$ is in $\widetilde \Gamma_{{\rm lb}}(D,\A)$.

 It remains to show that every polarized Abelian surface $(A,H,\rho,\psi)$ of type~$\A$ with level $\A$-structure is isomorphic to one of these.
 We may assume without loss of generality that $(A,H,\rho,\psi) = (A_z,H_z,\rho_z,\psi)$ for some $z \in \HH^2$.
 Because~$\psi$ is symplectic, the isomorphism \[ \psi \circ \psi_z^{-1}:~ \OD \oplus \tfrac{1}{\sqrt{D}} \A \to \OD \oplus \tfrac{1}{\sqrt{D}} \A \] is given by multiplication with~$M^*$
 for some matrix $M \in {\rm SL}_2(\OD \oplus \tfrac{1}{\sqrt{D}}\A)$.
 Writing $f: A_z \to A_{Mz}$ for the corresponding isomorphism and denoting the multiplication with $M^*$ also with $M^*$, we get
 \begin{eqnarray*}
  (A_z,H_z,\rho_z,\psi) & \cong & (A_{Mz},H_{Mz},\rho_{Mz},\psi \circ \rho_r(f)^{-1}) \\
  & = & (A_{Mz},H_{Mz},\rho_{Mz},\psi \circ \phi_z \circ (M^*)^{-1} \circ \phi_{Mz}^{-1}) \\
  & = & (A_{Mz},H_{Mz},\rho_{Mz},\psi_{Mz}).
 \end{eqnarray*}
\end{proof}

\begin{rem}
 We defined level $\A$-structures as isomorphisms on the lattices instead of on the quotients as in the one-dimensional case.
 This was necessary because in general the submodule $\A\Lambda \subseteq \Lambda$ is not, unlike $d\Lambda \subseteq \Lambda$ in the one-dimensional case,
 the submodule defined by $\{ \lambda \in \Lambda :~ E_H(\lambda,\Lambda) \subseteq d\ZZ \}$ - this would be precisely $\A^\sigma\Lambda$.
\end{rem}

We finish this subsection with summing up both propositions.
\begin{theo}
 The space \[ \Gamma_{{\rm lb}}(D,\A) \backslash \HH^3 = \widetilde \Gamma_{{\rm lb}}(D,\A) \backslash \HH^2 \times \Gamma(d) \backslash \HH \]
 is a moduli space of products of polarized Abelian surfaces of type $\A$ with level~$\A$-structure and principally polarized elliptic curves with level~$d$-structure.
\end{theo}
\clearpage

\section{Pseudo-real multiplication, cusps and extension classes}

For an order $\OO$ in a totally real number field $F$ of degree $g$, let $X_\OO$ be the moduli space of $g$-dimensional principally polarized complex Abelian varieties with a choice of
real multiplication by $\OO$. It is the disjoint union of certain Hilbert modular varieties, one component for each isomorhism class of proper rank two symplectic $\OO$-modules.
Bainbridge and M\"oller have shown in \cite{MartinMatt}, that any such symplectic $\OO$-module can be completely described (up to isomorphism) by a pair formed by a lattice in $F$ and a
$\QQ$-endomorphism of $F$ compatible with the trace pairing. Such a pair modulo an action of $F^*$ is called a \emph{cusp packet for $\OO$}.
They proved that there is a natural bijection between the set of cusp packets for $\OO$, the set of cusps of the Baily-Borel compactification of $X_\OO$ and the set of isomorphism classes of
certain symplectic extensions over $\OO$.

In the case we are interested in, the analogue of real multiplication by a cubic number field on a three-dimensional Abelian variety is pseudo-real multiplication by a pseudo-cubic number field.
We will introduce this concept in the first subsection and show that it corresponds to a choice of real multiplication on an Abelian subsurface.
In the other two subsections, we define symplectic extensions over pseudo-cubic orders and cusp packets.
We will see that we defined all these objects exactly in this way, that the same bijections as for totally real cubic fields hold.

\subsection{Pseudo-real multiplication}\label{pseudorealmult}
 
We defined real multiplication on Abelian surfaces, but we are interested in three-dimensional Abelian varieties.
We have to consider endomorphisms on those three-dimensional objects, and this leads to the concept of pseudo-real multiplication.
We will see that that the choice of an Abelian subsurface with some choice of real multiplication on it is equivalent to the choice of some pseudo-real multiplication on the three-dimensional variety.

\bigskip
 
If $A$ is a polarized Abelian variety, recall that $\End^+(A)$ denotes the submodule of $\End(A)$ consisting of those endomorphisms which are self-adjoint with respect to the polarization of~$A$.
\begin{defi}
 Let $K$ be a real quadratic number field. Then the $\QQ$-Algebra \[ F = K \oplus \QQ \] is called a \emph{pseudo-cubic field}.
 If $A$ is a three-dimensional principally polarized Abelian variety, then a $\QQ$-Algebra monomorphism \[ \rho:~ F \hookrightarrow \End_\QQ^+(A) := \End^+(A) \otimes_{\ZZ} \QQ \]
 is called \emph{pseudo-real multiplication by $F$ on $A$}.
\end{defi}
The choice of such pseudo-real multiplication $\rho$ is not unique. If $\sigma$ denotes the Galois automorphism of~$K$,
then \[ \rho^{\sigma} := \rho \circ (\sigma \oplus \id_{\QQ}) \] is another choice of real multiplication by~$F$ on~$A$.

\begin{lem}
 Let $F=K \oplus \QQ$ be a pseudo-cubic field and let $\rho$ be pseudo-real multiplication by $F$ on a three-dimensional principally polarized Abelian variety~$A$.
 Then there is a unique pair $(S,\rho_2)$, where $S$ is a two-dimensional subvariety of $A$ with the induced polarization and $\rho_2: K \hookrightarrow \End_\QQ^+(S)$ is real
 multiplication by~$K$ induced by the embedding $K \hookrightarrow F$ and the restriction of $\rho$ to~$S$.
\end{lem}
\begin{proof}
 Consider such a pseudo-real multiplication $\rho:~ F \hookrightarrow \End_\QQ^+(A)$.
 The structure of $F = K \oplus \QQ$ as a direct product naturally induces two subvarieties of~$V$, namely the kernels of the two factors.
 To be more precise, choose $m,n \in \NN$ such that $\rho(m,0), \rho(0,n) \in \End^+(A)$ and set \[ C := \ker(\rho(m,0))_0,\quad S :=  \ker(\rho(0,n))_0, \]
 where the index $0$ denotes the connected component of $0$. We want to show that $C$ and $S$ are complementary subvarieties of $A$.
 If we write $A = V/\Lambda$ with a three dimensional $\CC$-vector space $V$, then for every $x \in F$ the image $\rho(x)$ acts on $V$ by linear transformations.
 Writing $C = V_1/\Lambda_1$ and $S = V_2/\Lambda_2$, we get
 \begin{align*}
  V_1 & = \{ v \in V : \rho(x)(v)=0 \mbox{ for all } x \in K \} = \ker(\rho(1,0)) \mbox{ and}\\ V_2 & = \{ v \in V : \rho(x)(v)=0 \mbox{ for all } x \in \QQ \} = \ker(\rho(0,1)).
 \end{align*}
 For all $x \in K$ and $q \in \QQ$ we have \begin{equation}\label{null} (x,0)(0,q)=0=(0,q)(x,0) \end{equation} and thus $\rho(K)(V_2) \subset V_2$ and $\rho(\QQ)(V_1) \subset V_1$ by the definition
 of~$V_1$ and~$V_2$. Hence, we have at least a homomorphism $\rho_2: K \to \End_\QQ^+(S)$ induced by the embedding of $K$ and the restriction of $\rho$.

 The $V_i$ are non-trivial ($V_1 = \{ 0 \}$ would imply $V_2=V$ by equation \eqref{null} and vice versa, in contradiction to the assumption that $\rho$ is injective). From $\rho(1,1) = \id$
 it follows that $V_1 \cap V_2 = \{ 0 \}$ and thus $\rho(x)(v) \not= 0$ for every $x \in K^*$, $v \in V_2 \setminus \{ 0 \}$ (respectively $x \in \QQ^*$, $v \in V_1 \setminus \{ 0 \}$).
 Hence we get monomorphisms \[ \rho_1:~ \QQ \hookrightarrow \End_\QQ^+(C),\quad q \mapsto \id_C \otimes q \] and \[ \rho_2:~ K \hookrightarrow\End_\QQ^+(S). \]
 As $K$ is a real quadratic field, the subvariety~$S$ cannot be an elliptic curve, hence $\dim(V_2)=2$, $\dim(V_1)=1$ and thus $V = V_1 \oplus V_2$.
 Since the action of $\rho$ is self-adjoint with respect to the polarization of $A$, the subvarieties $C$ and $S$ are indeed complementary by the definition of $V_1$ and $V_2$.

 It remains to prove the uniqueness statement. As $\rho_2(1)=\rho(1,0)$ and Equation~\eqref{null} holds in $F$, any such subvariety $S$ of $A$ must be contained in the kernel~$\ker(\rho(0,n))_0$,
 and for dimension reasons they coincide.
\end{proof}

Consider the monomorphisms $\rho_1,\rho_2$ from the proof above.
While we have $\rho_1^{-1}(\End(C)) = \ZZ$, the preimage $\rho_2^{-1}(\End(S))$ is an order in $K$ of some discriminant $D>0$, hence we get real multiplication
\[ \rho_2 :~ \OD \hookrightarrow \End^+(S). \]
Since $S$ and $\rho_2$ were uniquely determined, we also say that $\rho$ is \emph{pseudo-real multiplication with discriminant~$D$}.
Moreover, we get \[ \rho^{-1}(\End(A)) \subset \OD \oplus \ZZ. \] Hence $\rho^{-1}(\End(A))$ is a subring of $F$ with $1$, such that its additive group is free Abelian of rank three.

Now we can refine our definition of pseudo-real multiplication.
\begin{defi}
 Let $F$ be a pseudo-cubic field. A subring of $F$ with $1$, such that its additive group is free Abelian of rank three, is called a \emph{pseudo-cubic order} in~$F$.
 If $A$ is a three-dimensional principally polarized Abelian variety and $\OO \subset F$ is a pseudo-cubic order, then \emph{pseudo-real multiplication by $\OO$ on $A$} is a ring monomorphism
 \[ \rho:~ \OO \hookrightarrow \End^+(A) \] that is proper in the sense that it does not extend to a ring monomorphism $\rho:~ R \hookrightarrow \End^+(A)$ for any ring $R$ with
 $\OO \subsetneqq R \subset \QQ(\sqrt{D}) \oplus \QQ$.
 The \emph{degree of $\rho$} is defined to be the type of the polarization on the elliptic curve induced by~$\rho$.

 We denote by $X_\OO$ the moduli space of pairs $(A,\rho)$, where $A$ is a principally polarized Abelian variety and~$\rho$ is pseudo-real multiplication by $\OO$ on $A$.
\end{defi}

Each ideal $\A$ of $\OD$ gives rise to the pseudo-cubic order \[ \OO_\A := \A + \ZZ (1,1) = \{ (x,y) \in \OD \oplus \ZZ : x-y \in \A \}. \]
It turns out that these examples are all the pseudo-cubic orders we are interested in.

\begin{lem}
 Let  $\OO$ be a pseudo-cubic order and let $\rho:~ \OO \hookrightarrow \End^+(A)$ be pseudo-real multiplication by~$\OO$ on~$A$ with discriminant $D$.
 If the degree $d$ of~$\rho$ is relatively prime to the conductor of~$D$, then there is a primitive ideal $\A \subset \OD$ of norm $d$, such that~$\OO = \OO_\A$.
\end{lem}
\begin{proof}
 Let $(S,\rho_2)$ be the induced polarized Abelian subsurface with real multiplication by $\OD$.
 This polarized Abelian variety is of type $(1,d)$, hence $(S,\rho_2)$ is of type $\A$ for some primitive ideal $\A \subseteq \OD$ of norm $d$.
 Let $\rho_\QQ$ be the extension of $\rho$ to $\rho_\QQ: F \hookrightarrow \End^+(A) \otimes_\ZZ \QQ$.
 For any $(x,y) \in \OD \oplus \ZZ$, the image $\rho_\QQ(x,y)$ is an endomorphism of $A=V/\Lambda$ if and only if $\rho_\QQ(x,y)(\Lambda) \subset \Lambda$. Since $y$ is an integer,
 this inclusion holds if and only if $\rho_\QQ(x-y,0)(\Lambda) \subset \Lambda$. Using a smart basis $(\eta_2^\sigma,d)$ for $\A$ and the explicit description of $A$ in 
 Section~\ref{dim3}, one can easily compute that this holds if and only if $x-y$ is in $\A$.
 Since $\rho$ is proper, we have \[ \OO = \rho_\QQ^{-1}(\End^+(A)) = \{ (x,y) \in \OD \oplus \ZZ : x-y \in \A \}. \]
\end{proof}
We will also refer to the number $[\OD \oplus \ZZ : \OO]$ as the \emph{degree of $\OO$}.
We have just seen that it is equal to the degree~$d$ of~$\rho$, if~$d$ is relatively prime to the conductor of~$D$.
Furthermore, we say that pseudo-real multiplication by~$\OO$ is \emph{of type~$\A$}, if $\OO = \OO_\A$.
Thus, pseudo-real multiplication with discriminant~$D$ of type~$\A$ is the same as pseudo-real multiplication by~$\OO_\A$.

The final theorem of this subsection states that the choice of an Abelian subsurface with some choice of real multiplication on it is equivalent to the choice of pseudo-real multiplication on the
three-dimensional variety, and thus we will use these concepts interchangeably.
\begin{theo}
 Let $D>0$ be a non-square discriminant.
 \begin{enumerate}
  \item There are natural bijections between the set \[ X_D^{(3)} := \coprod_{d \in \NN} X_{D,d}^{(3)} \] of isomorphism classes of quadruples $(A,H,S,\rho)$ of principally polarized Abelian varieties
        of dimension three with real multiplication by $\OD$ described in Section~\ref{dim3} and the set of isomorphism classes of pairs $(A,\rho)$, where~$A$ is a principally polarized Abelian variety
        of dimension three and~$\rho$ is pseudo-real multiplication on~$A$ with discriminant~$D$.
  \item If $d$ is relatively prime to the conductor of $D$, then
        there are natural bijections between the set \[ X_\A^{(3)} \subset X_{D,d}^{(3)} \] of isomorphism classes of quadruples $(A,H,S,\rho)$ of principally polarized Abelian varieties of dimension
        three with real multiplication by~$\OD$ of type~$\A$ described in Section~\ref{dim3} and the set $X_{\OO_\A}$ of isomorphism classes of pairs~$(A,\rho)$, where~$A$ is a principally polarized
        Abelian variety of dimension three and $\rho$ is pseudo-real multiplication on~$A$ by $\OO_\A$. 
 \end{enumerate}
\end{theo}
\begin{proof}
 This follows immediately from the previous two lemmata and the fact that real multiplication $\rho\!: \QQ(\sqrt{D}) \hookrightarrow \End_\QQ^+(S)$ on an Abelian subsurface~$S$ of~$A$ can be naturally
 extended to pseudo-real multiplication \[ \rho'\!: \QQ(\sqrt{D}) \oplus \QQ \hookrightarrow \End_\QQ^+(A) = \End_\QQ^+(S) \oplus \End_\QQ^+(C), \] where~$C$ denotes the complementary subvariety
 of~$S$ in~$A$.
\end{proof}
\subsection{Cusps and symplectic extensions}

Three-dimensional complex Abelian varieties with a choice of pseudo-real multiplication are parameterized by certain Hilbert modular varieties.
In the following subsection we will introduce the Baily-Borel-compactification of these objects and show that the cusps correspond to symplectic extensions over pseudo-cubic orders.
\\[1em]
For preparation, we first treat the two-dimensional case.
\begin{defi}\label{Symplextension1}
 Let $D>0$ be a non-square discriminant. A \emph{symplectic $\OD$-extension} is a short exact sequence of $\OD$-modules \[ E:\quad 0 \to I_1 \to M \to I_2 \to 0, \]
 where $I_1,I_2$ are torsion-free $\OD$-modules of rank one and $M$ is a proper symplectic $\OD$-module in the sense of Definition~\ref{symplecticmodule}.
 An \emph{isomorphism of symplectic $\OD$-extensions} is an isomorphism of exact sequences such that the isomorphism on the rank two modules is symplectic.
\end{defi}
We want to parameterize these extension classes for a fixed symplectic module in the middle.

First we define for every $L \in \PK$ a symplectic $\OD$-extension. We can see $L$ as a line in $K^2$, so with the natural inclusion
$\OD \oplus \tfrac{1}{\sqrt{D}}\A \subset K^2$ we can take the intersection
\[ I_L := L \cap (\OD \oplus \tfrac{1}{\sqrt{D}}\A) \] and the $\OD$-factor module \[ \widetilde I_L := (\OD \oplus \tfrac{1}{\sqrt{D}}\A) / I_L. \]
The natural inclusion $\iota: I_L \hookrightarrow \OD \oplus \tfrac{1}{\sqrt{D}}\A$ and the natural projection
$ \pi: \OD \oplus \tfrac{1}{\sqrt{D}}\A \twoheadrightarrow \widetilde I_L$ yield the short exact sequence of $\OD$-modules
\[ 0 \to I_L \to \OD \oplus \tfrac{1}{\sqrt{D}}\A \to \widetilde I_L \to 0. \]
Finally we define $E_L$ to be this short exact sequence together with the trace pairing on $\OD \oplus \tfrac{1}{\sqrt{D}}\A$.

The matrix group $\GL_2(K)$ acts on $\PK$ by \[ A.[\alpha : \beta] := [a_1\alpha+a_2\beta : a_3\alpha+a_4\beta ] \] for all
$A = \left( \begin{smallmatrix} a_1 & a_2 \\ a_3 & a_4 \end{smallmatrix} \right) \in \GL_2(K)$ and $[\alpha : \beta] \in \PK$.

\begin{prop}\label{cusps1}
 Let $d \in \NN$ be relatively prime to the conductor of $D$.
 The map \[ L \mapsto\quad \left( E_L:~ 0 \to I_L \to (\OD \oplus \tfrac{1}{\sqrt{D}}\A,\langle \cdot , \cdot \rangle) \to \widetilde I_L \to 0 \right) \] induces a well-defined bijection between
 \[ \coprod_{[\OD:\A]=d} \PK / {\rm SL}_2(\OD \oplus \tfrac{1}{\sqrt{D}}\A) \] and the set of those isomorphism classes of symplectic $\OD$-extensions, where the symplectic form on the module in
 the middle is of type $(1,d)$.
\end{prop}
\begin{proof}
 Given $A \in {\rm SL}_2(\OD \oplus \tfrac{1}{\sqrt{D}}\A)$, multiplication by $A$ is a symplectic $\OD$-module automorphism of $\OD \oplus \tfrac{1}{\sqrt{D}}\A$ that maps $I_L$ onto $I_{A.L}$,
 so it defines an isomorphism of symplectic $\OD$-extensions between $E_L$ and $E_{A.L}$. Conversely, any symplectic $\OD$-module automorphism of $\OD \oplus \tfrac{1}{\sqrt{D}}\A$ is given by a
 matrix $A \in {\rm SL}_2(\OD \oplus \tfrac{1}{\sqrt{D}}\A)$, and this automorphism maps $I_L$ onto $I_{A.L}$. Thus we have seen that two such lines, say $L$ and $L'$, give isomorphic symplectic
 $\OD$-extensions~$E_L$ and~$E_{A.L}$ if and only if $L'=A.L$ for some $A \in \SL_2(\OD \oplus \tfrac{1}{\sqrt{D}}\A)$.

 It remains to show that the map is surjective. Given any symplectic $\OD$-extension $E:~ 0 \to I_1 \to M \to I_2 \to 0$, where the symplectic form on $M$ is of type $(1,d)$ with $d$ relatively prime
 to the conductor, the module $M$ is isomorphic to $(\OD \oplus \tfrac{1}{\sqrt{D}}\A,\langle \cdot , \cdot \rangle)$ as a symplectic $\OD$-module for some ideal $\A \trianglelefteq \OD$ of norm~$d$ and
 the trace pairing $\langle \cdot , \cdot \rangle$. So without loss of generality we can replace~$M$ by~$(\OD \oplus \tfrac{1}{\sqrt{D}}\A,\langle \cdot , \cdot \rangle)$ and~$I_1$ by its image in
 $\OD \oplus \tfrac{1}{\sqrt{D}}\A$. Because~$I_1$ has rank one, it is contained in a line $L \subset K^2$, so $I_1 \subseteq I_L = L \cap (\OD \oplus \tfrac{1}{\sqrt{D}}\A)$ and since~$I_2$
 is torsion-free, we even have $I_1 = I_L$ and thus $E \cong E_L$.
\end{proof}

\noindent
{\bf The symplectic pseudo-trace pairing}
We are interested in short exact sequences of~$\OO$-modules, where~$\OO$ is a pseudo-cubic order and the module in the middle comes from the lattice defining a complex Abelian variety with pseudo-real
multiplication by~$\OO$. We will see that these lattices are not of arbitrary form, but isomorphic to lattices in~$F^2$ for some pseudo-cubic number field~$F$, together with the symplectic form defined
as follows. \\
For all $x=(x_1,x_2) \in F$, the rational number $\tr_p(x) := \tr(x_1)+x_2$ is called the \emph{pseudo-trace} of $x$.
The pairing \[ F \times F \to \QQ,\quad (x,y) \mapsto \tr_p(xy) \] is called the \emph{pseudo-trace pairing} and the symplectic form
\[ \langle \cdot,\cdot \rangle_p:\quad  F^2 \times F^2 \to \QQ, \quad ((x_1,y_1),(x_2,y_2)) \mapsto \tr_p(x_2y_1-x_1y_2)  \] is called the \emph{symplectic pseudo-trace pairing}.
\begin{defi}\label{Symplmodule}
 Let $\OO$ be a pseudo-cubic order in some pseudo-cubic number field~$F$.
 A \emph{symplectic $\OO$-module} is an $\OO$-module together with a symplectic form of type $(1,1,1)$, isomorphic to a lattice in~$F^2$ together with the symplectic pseudo-trace pairing.
\end{defi}
\begin{prop}\label{latticeiso}
 Let $D>0$ be a non-square discriminant and let $d \in \NN$ be relatively prime to the conductor of~$D$. If $(X=V/\Lambda,H,\rho) \in X_{D,d}^{(3)}$ is a principally polarized complex Abelian variety
 with pseudo-real multiplication by a pseudo-cubic order $\OO \subset F$ of degree $d$, then the lattice $(\Lambda,E_H)$ is a proper symplectic $\OO$-module.
\end{prop}
\begin{proof}
 By the description in Proposition~\ref{classifdim3} we may assume without loss of generality, that $(X=V/\Lambda,H,\rho)$ is of the following form. There is a primitive ideal~$\A$ in~$\OD$ of norm~$d$,
 a smart basis $(\eta_2^\sigma,d\eta_1^\sigma)$ of $\A$ and some $z=(z_1,z_2,z_3) \in \HH^3$, such that $V=\CC^3$, and that a symplectic basis of $\Lambda$ for $E_H$ is given by the columns
 $\lambda_1,...,\lambda_6$ of \[ \Pi_z = \begin{pmatrix} \eta_1 & \eta_2 & 0 & \frac{1}{\sqrt{D}}\eta_2^{\sigma}z_1 & \frac{-1}{\sqrt{D}}\eta_1^{\sigma}z_1 & \tfrac{1}{d}\eta_2 \\
 \eta_1^{\sigma} & \eta_2^{\sigma} & 0 & \frac{-1}{\sqrt{D}}\eta_2z_2 & \frac{1}{\sqrt{D}}\eta_1z_2 & \tfrac{1}{d}\eta_2^{\sigma} \\ 0 & 0 & 1 & 0 & \tfrac{1}{d} & -\tfrac{1}{d}z_3 \end{pmatrix}. \]
 The pseudo-real multiplication $\rho: \OO \to \End^+(\Lambda)$ is given by \[ \rho(x,r).(c_1,c_2,c_3)^T = (xc_1,x^\sigma c_2,rc_3)^T. \]
 Consider the lattice $M$ in $F^2$ spanned by the columns $\alpha_1,...,\alpha_6$ of \[ \begin{pmatrix} (\eta_1,0) & (\eta_2,0) & (0,1) & (0,0) & (0,\tfrac{1}{d}) & (\tfrac{1}{d}\eta_2,0)
 \\ (0,0) & (0,0) & (0,0) & (\tfrac{1}{\sqrt{D}}\eta_2^\sigma,0) & (\tfrac{-1}{\sqrt{D}}\eta_1^\sigma,0) & (0,-1) \end{pmatrix}. \]
 Then $\alpha_i \mapsto \lambda_i$ defines a symplectic $\OO$-isomorphism $(M,\langle \cdot , \cdot \rangle_p) \cong (\Lambda,E_H)$.
\end{proof}

\noindent
{\bf The Baily-Borel-Satake-compactification and Cusps.}
Let~$D \in \NN$ be a non-square discriminant, $K=\QQ(\sqrt{D})$ and let $d \in \NN$ be relatively prime to the conductor of $D$.
We have seen in Theorem~\ref{modulispacedim3}, that the moduli space $X_{D,d}^{(3)}$ is the disjoint union of certain Hilbert modular varieties
\[ X_{D,d}^{(3)} = \coprod_\A X_\A^{(3)}, \] where $\A$ runs over all primitive ideals in $\OD$ of norm $d$.
The Hilbert modular varieties $X_\A^{(3)}$ are of the form \[ X_\A^{(3)} = X(M) := \SL(M) \backslash \HH^3, \] where $M \leq (F^2,\langle \cdot,\cdot \rangle_p)$ is a certain symplectic $\OO_\A$-module
and $\SL(M)$ is the subgroup of $\SL_2(F) = \SL_2(K) \times \SL_2(\QQ)$ preserving $M$, i.e. the group of symplectic automorphisms of $M$. We say that \emph{$M$ represents $\OO_\A$}.

More generally, given any pseudo-cubic order $\OO$ and any symplectic $\OO$-module~$M$, we define $\SL(M)$ and $X(M)=\SL(M) \backslash \HH^3$ analogously and set \[ X_\OO := \coprod_{M \leq F^2} X(M), \]
where $M$ runs over a set of representatives of all isomorphism classes of proper symplectic $\OO$-modules.

We compactify each $X(M)$ in the following way.
Consider the two embeddings \begin{align*} \PK \hookrightarrow (\RR \cup \{ \infty \})^2, & \quad [x:y] \mapsto (\tfrac{x}{y},\tfrac{x^\sigma}{y^\sigma}) \quad {\rm and} \\
\PQ \hookrightarrow (\RR \cup \{ \infty \}), & \quad [q:r] \mapsto \tfrac{q}{r}. \end{align*}
Each of the factors of \[ \HH_F^3 := (\HH^2 \cup \PK) \times (\HH \cup \PQ) \] carries a topology induced by the toplogy on~$\HH^2$ and~$\HH$ respectively, and certain neighbourhood
bases for each cusp in $\PK$ and $\PQ$ respectively. For more details, see \cite{vdGeer}. 
Then the action of $\GL_2(F) = \GL_2(K) \times \GL_2(\QQ)$ by M\"obius transformations extends naturally to an action on $\HH_F^3$.
The (Baily-Borel-)compactification of $X(M)$ is then defined by \[ \overline{X(M)}^{BB} := \SL(M) \backslash \HH_F^3, \] which is a complex analytic space.
Writing $\PF := \PK \times \PQ$, the embeddings above yield a well-defined embedding \[ \SL(M) \backslash \PF \hookrightarrow \overline{X(M)}^{BB}. \]
The image of $\PF$ in $\overline{X(M)}^{BB}$ is denoted by $\mathcal{C}(M)$, the set of \emph{cusps of $\overline{X(M)}^{BB}$}.
The (Baily-Borel-)compactification $\overline{X}_\OO^{BB}$ of $X_\OO$ is defined to be the disjoint union of the $\overline{X(M)}^{BB}$, and the union of the sets of cusps~$\mathcal{C}(M)$ is denoted
by~$\mathcal{C}_\OO$, the set of \emph{cusps of $\overline{X}_\OO^{BB}$}.

In the cases we are interested in, it turns out that $X_\OO$ has just one component.
\begin{prop}
 Let $\A$ be a primitive ideal in $\OD$ of norm $d$ relatively prime to the conductor of $D>0$. Then we have exactly one isomorphism class of proper symplectic $\OO_\A$-modules.
 In other words, we have \[ X_{\OO_\A} = X_\A^{(3)} = X(M) \] for some $M$ representing $\OO_\A$.
\end{prop}
\begin{proof}
 Let $d$ be the norm of $\A$ and let $(\eta_2^\sigma,d\eta_1^\sigma)$ be a smart basis of $\A$. We have to show that any symplectic $\OO_\A$-module $(\Lambda,E)$ is isomorphic to
 $(M,\langle \cdot , \cdot \rangle_p)$, where~$M$ is the lattice in $F^2$ described in the proof of Proposition~\ref{latticeiso}. To be more precise, the columns of
 \[ \begin{pmatrix} (\eta_1,0) & (\eta_2,0) & (0,1) & (0,0) & (0,\tfrac{1}{d}) & (\tfrac{1}{d}\eta_2,0)
    \\ (0,0) & (0,0) & (0,0) & (\tfrac{1}{\sqrt{D}}\eta_2^\sigma,0) & (\tfrac{-1}{\sqrt{D}}\eta_1^\sigma,0) & (0,-1) \end{pmatrix} \] form a symplectic basis for $M$.

 Consider some arbitrary proper symplectic $\OO_\A$-module $(\Lambda,E)$ with $\Lambda \subset F^2$.
 We denote by~$\pi_K$ (respectively~$\pi_\QQ$) the natural projections from~$F = K \oplus \QQ$ to $K=\QQ(\sqrt{D})$ (respectively~$\QQ$).
 Moreover, we define $\Lambda_1 := \Lambda \cap \QQ^2$ and $\Lambda_2 := \Lambda \cap K^2$.
 We have $\pi_K(\OO_\A)=\OD$, thus~$\OD$ acts on $\Lambda_2$, and we claim that~$\Lambda_2$ is proper as an~$\OD$-module.
 Otherwise, there would be some $e\in \NN_{>1}$ dividing the conductor~$f$ of~$D$, such that the order~$\OO_E$ with conductor $e$ acts on~$\Lambda_2$.
 Writing $\OK = \langle 1,\gamma \rangle_\ZZ$, we have $\OD = \langle 1,f\gamma \rangle_\ZZ$ and $\OO_E = \langle 1,e\gamma \rangle_\ZZ$.
 As $(d,0)$ maps $\Lambda$ to $\Lambda_2$, we deduce that $(de\gamma,0)$ acts on $\Lambda$. But $d$ and $f$ are relatively prime, hence $de\gamma$ is not in~$\OD$ and consequently
 $(de\gamma,0)$ cannot lie in~$\OO_\A$, in contradiction to the fact that $\Lambda$ is proper as an $\OO_\A$-module.

 Writing $E_i := E|_{\Lambda_i \times \Lambda_i}$ for $i \in \{ 1,2 \}$, the symplectic $\ZZ$-module $(\Lambda_1,E_1)$ is a priori of some type~$(c)$.
 As~$K^2$ is the orthogonal complement of~$\QQ^2$ in~$F^2$, the symplectic $\OO_D$-module $(\Lambda_2,E_2)$ is of type $(1,c)$ by Lemma~\ref{complementarytypes}.
 Since $(d,0)\Lambda \subseteq \Lambda_2$ and $(0,d)\Lambda \subseteq \Lambda_1$, we have $\Lambda \subseteq \tfrac{1}{d}(\Lambda_1 \oplus \Lambda_2)$.
 Considering the symplectic form $d^2E$, the multiplication formula for the degree of sublattices from Lemma~\ref{Gradindex} gives
 \[ d^8 = |\tfrac{1}{d}(\Lambda_1 \oplus \Lambda_2) / \Lambda| \cdot c^2. \]
 Thus $c \mid d^4$, in particular $c$ is relatively prime to the conductor of $D$.
 Therefore, by Theorem~\ref{Struktursatz}, $c$ satisfies the prime factor condition in $\OD$ and there is a unique ideal $\B$ of norm $c$ in $\OD$, such that $(\Lambda_2,E_2)$ is isomorphic to
 $(\OD \oplus \tfrac{1}{\sqrt{D}}\B, \langle \cdot , \cdot \rangle)$.
 We choose a smart basis $(\theta_2^\sigma,d\theta_1^\sigma)$ for $\B$.
 As \[ (\lambda_1,\lambda_2,\mu_1,\hat{\mu}_2) = \left( (\theta_1,0),(\theta_2,0),(0,\tfrac{1}{\sqrt{D}}\theta_2^\sigma),(0,\tfrac{-c}{\sqrt{D}}\theta_1^\sigma) \right) \]
 is a symplectic basis of $(\OD \oplus \tfrac{1}{\sqrt{D}}\B) \subset K^2$ and $E({(1,0)\Lambda \times \Lambda_2}) \subset \ZZ$, we see that $(1,0)\Lambda$ is contained in
 $\langle \lambda_1,\tfrac{1}{c}\lambda_2,\mu_1,\tfrac{1}{c}\hat{\mu}_2 \rangle_\ZZ$.
 Now we are exactly in the same situation as in the proof of Proposition~\ref{classifdim3}, where we have shown that we can choose a basis $(\nu_1,\nu_2)$ of $\Lambda_1$, such that
 \[ \left( \lambda_1,\lambda_2,\nu_1,\mu_1,\tfrac{1}{c}(\hat{\mu}_2+\nu_1),\tfrac{1}{c}(\lambda_2+\nu_2) \right) \] is a symplectic basis of $\Lambda$, which has the desired form.
 Hence $\Lambda$ is a proper $\OO_\B$-module. But then $\A$ und $\B$ must equal, and we can choose $(\eta_2^\sigma,d\eta_1^\sigma)$ for our smart basis, which finishes the proof.
\end{proof}

\noindent
{\bf Symplectic $\OO$-extensions}
Now we will classify the extension classes coming from complex Abelian varieties with pseudo-real multiplication by a pseudo-cubic number field
$F=\QQ(\sqrt{D}) \oplus \QQ$.

\begin{defi}\label{Symplextension2}
 Let $\OO$ be a pseudo-cubic order in a pseudo-cubic numberfield~$F$.
 A \emph{symplectic $\OO$-extension} is a short exact sequence of $\OO$-modules \[ E:\quad 0 \to \II_1 \to M \to \II_2 \to 0, \] where $\II_1,\II_2$ are $\OO$-modules isomorphic to lattices
 in~$F$ and~$M$ is a proper symplectic $\OO$-module. (In particular, $\II_1$ maps onto a Lagrangian subgroup of~$M$.)

 An \emph{isomorphism of symplectic $\OO$-extensions} is an isomorphism of exact sequences such that the isomorphism on the rank two modules is symplectic.
\end{defi}
Analogous to the quadratic case, we want to parameterize these extension classes for a fixed symplectic module in the middle.
We identify each point $L = ([x_1:x_2],[q_1:q_2]) \in \PF$ with the line $\langle (v_1,v_2) \rangle_F$ in $F^2$, where $v_i := (x_i,q_i)$.
Again, given any lattice $M$ in $F^2$, we set $\II_L := L \cap M$.
Finally, we get the symplectic $\OO$-extension \[ E_L:~ 0 \to \II_L \to (M,\langle \cdot,\cdot \rangle_p) \to M/\II_L \to 0. \]

\begin{prop}
 Let $\OO$ be a pseudo-cubic order and let $M_i$ be lattices in $F^2$, such that $\{ (M_i,\langle \cdot , \cdot \rangle_p) \}$ is a set of representatives
 of all isomorphism classes of proper symplectic $\OO$-modules.
 Then the maps \[ L \mapsto\quad \left( E_L:~ 0 \to \II_L \to (M_i, \langle \cdot , \cdot \rangle_p) \to M_i/\II_L \to 0 \right) \]
 induce a well-defined bijection between the set of cusps \[ \mathcal{C}_\OO = \coprod_i \SL(M_i) \backslash \PF \] and the set of isomorphism classes of symplectic
 $\OO$-extensions.
\end{prop}
\begin{proof}
 The proof works identically like the proof of Proposition~\ref{cusps1}.
 We just note that the elements of $F^2$, whose $F$-span is of $\QQ$-dimension three, are exactly the elements of the form $((x_1,q_1),(x_2,q_2))$ with $(x_1,x_2) \not= 0$ and
 $(q_1,q_2) \not= 0$.
\end{proof}

\subsection{Extension classes and cusp packets}\label{extensionclasses}

There is another viewpoint of symplectic extensions over pseudo-cubic orders, namely that of cusp packets.
Analogous to \cite{MartinMatt}, Section~2, we will show that extension classes correspond to certain classes of $\QQ$-linear endomorphisms of the pseudo-cubic field.

\bigskip
For a lattice $\II$ in a pseudo-cubic number field $F$ we denote by  \[ \OO(\II) := \{ x \in F : x \cdot \II \subset \II \} \] its \emph{coefficient ring} and by
\[ \II^{\vee} := \{ x \in F : \tr_p(x\II) \subset \ZZ \} \cong \Hom_\ZZ(\II;\ZZ) \] its dual with respect to the pseudo-trace pairing.
\begin{defi}\label{extension}
 Let $\II$ be a lattice in $F$ such that its coefficient ring contains a pseudo-cubic order $\OO$. A \emph{symplectic $(\OO,\II)$-extension} is a
 short exact sequence of $\OO$-modules \[ 0 \to \II \to M \to \II^\vee \to 0, \] where $M$ is a proper $\OO$-module together with a symplectic form of type $(1,1,1)$, such that
 \begin{enumerate}
  \item the $\OO$-module structure on $M$ is self-adjoint with respect to this form (in particular, $\II$ maps onto a Lagrangian subgroup of $M$),
  \item if $s: \II^\vee \to M$ is a $\ZZ$-splitting that maps $\II^\vee$ onto a Lagrangian complement of $\II$ in $M$, then the induced $\ZZ$-isomorphism
        $\II \oplus \II^\vee \cong M$ is symplectic with respect to the symplectic pseudo-trace pairing on $\II \oplus \II^\vee$ (this property does not depend on the choice of the splitting).
 \end{enumerate}
\end{defi}

We note that the coefficient ring of a lattice in $F$ is always a pseudo-cubic order. An isomorphism between two such symplectic extensions is an isomorphism of exact
sequences of $\OO$-modules that is symplectic on the module in the middle. Let $\EE(\II)$ be the set of all symplectic $(\OO,\II)$-extensions for some pseudo-cubic
order $\OO \subset \OO(\II)$ up to isomorphism which are the identity on $\II$ and $\II^\vee$.
With a slight expansion of the Baer sum of extensions, see \cite{Weibel}, we get the following.
\begin{lem}
 The set $\EE(\II)$ carries the structure of an Abelian group.
\end{lem}
\begin{proof}
 We recall the definition of the Baer sum of extensions, see \cite{Weibel}. Given any two symplectic $(\OO,\II)$-extensions
 \[ 0 \to \II \xrightarrow[]{\iota_i} M_i \xrightarrow[]{\pi_i} \II^\vee \to 0, \] then $M_1 \oplus M_2$ is a proper $\OO_1 \cap \OO_2$-module.
 We define $\iota: \II \to M_1 \oplus M_2$ by $\iota(x) := (\iota_1(x),-\iota_2(x))$ and $\pi : M_1 \oplus M_2 \to \II^\vee$ by $\pi(m_1,m_2) := \pi_1(m_1)-\pi_2(m_2)$.
 The sum of the two extensions is defined as \[ 0 \to \II \to \ker(\pi)/{\rm im}(\iota) \to \II^\vee \to 0 \] with \[ x \mapsto (\iota_1(x),0) = (0,\iota_2(x)) \] for all $x \in \II$ and
\[ (m_1,m_2) \mapsto \pi_1(m_1) = \pi_2(m_2) \] for all $(m_1,m_2) \in \ker(\pi)/{\rm im}(\iota)$.

 In our context we have to extend this definition by declaring a symplectic form on $\ker(\pi)/{\rm im}(\iota)$.
 If $\langle \cdot,\cdot \rangle_i$ are the forms on $M_i$, the symplectic form on $\ker(\pi)/{\rm im}(\iota)$ is simply defined as their sum,
 \[ \langle (m_1,m_2),(n_1,n_2) \rangle := \langle m_1,n_1 \rangle_1 + \langle m_2,n_2 \rangle_2. \]
 That this form satisfies the conditions in Definition~\ref{extension} can be easily seen as follows.
 Taking two splittings $s_i : \II^\vee \to M_i$ as in the definition, a splitting of the Baer sum is given by $s := s_1 \oplus s_2$.
 Thus the induced isomorphism $\II \oplus \II^\vee \cong \ker(\pi)/{\rm im}(\iota)$ is given by \[ (x,y) \mapsto (\iota_1(x)+s_1(y),s_2(y)) = (s_1(y),\iota_2(x)+s_2(y)). \]
 Here we can see that the pairing is well-defined and that it is just the form induced by the symplectic pseudo-trace pairing on $\II \oplus \II^\vee$.

 The Baer sum gives $\EE(\II)$ the structure of an Abelian group.
 The identity element of $\EE(\II)$ is the trivial extension \[ E_0(\II) :~ 0 \to \II \to (\II \oplus \II^\vee, \langle \cdot,\cdot \rangle_p) \to \II^\vee \to 0. \]
 The inverse of any symplectic $\OO$-extension \[ 0 \to \II \xrightarrow[]{\iota} (M,E) \xrightarrow[]{\pi} \II^\vee \to 0 \]
 is given by \[ 0 \to \II \xrightarrow[]{\iota} (M,-E) \xrightarrow[]{-\pi} \II^\vee \to 0. \]
\end{proof}

The next step is to classify the extension classes in $\EE(\II)$, and we start with the construction of certain extensions.
For this we define
\begin{align*} \Hom_\QQ^+(F,F) := \{ & h \in \Hom_\QQ(F,F) : \\ & \tr_p(h(\mu_1)\mu_2) = \tr_p(\mu_1h(\mu_2)) \mbox{ for all } \mu_1,\mu_2 \in F \}, \end{align*}
which is a $\QQ$-vector space of dimension six, and
\begin{align*} \Hom_\QQ^-(F,F) := \{ & h \in \Hom_\QQ(F,F) : \\ & \tr_p(h(\mu_1)\mu_2) = -\tr_p(\mu_1h(\mu_2)) \mbox{ for all } \mu_1,\mu_2 \in F \}, \end{align*}
which is a $\QQ$-vector space of dimension three. 
For each $x \in F$, we denote by $M_x \in \Hom_F(F,F)$ multiplication by~$x$.
Now given any $h \in \Hom_\QQ^+(F,F)$, we assign to each $x \in F$ the $\QQ$-endomorphism $[M_x,h] \in \Hom_\QQ^-(F,F)$, where $[X,Y]:= XY-YX$ denotes the commutator of~$X$ and~$Y$ wherever it makes
sense.
We have the pseudo-cubic order \[ \OO_h(\II) := \{ x \in \OO(\II) : [M_x,h](\II^\vee) \subset \II \} \] and can define a proper $\OO_h(\II)$-module structure on $\II \oplus \II^\vee$ that is
self-adjoint with respect to $\langle \cdot,\cdot \rangle_p$ by setting \[ x.(\lambda,\mu) := (x\lambda+[M_x,h](\mu),x\mu) \] for all $x \in \OO_h(\II)$ and $(\lambda,\mu) \in \II \oplus \II^\vee$.
We denote $\II \oplus \II^\vee$ together with this module structure and the symplectic pseudo-trace pairing by $(\II \oplus \II^\vee)_h$.
Together with the natural inclusion of $\II$ and the natural projection onto $\II^\vee$ this becomes a symplectic $(\OO_h(\II),\II)$-extension, which we will denote by~$E_h(\II)$.
\begin{lem}\label{HomExtSur}
 Each class in $\EE(\II)$ has a representative \[ E_h(\II):~ 0 \to \II \to (\II \oplus \II^\vee)_h \to \II^\vee \to 0 \] for some $h \in \Hom_\QQ^+(F,F)$.
\end{lem}
\begin{proof}
 Let $[E:~ 0 \to \II \to M \to \II^\vee \to 0] \in \EE(\II)$.
 Again we use a $\ZZ$-splitting $s: \II^\vee \to \ M$ mapping $\II^\vee$ onto a Lagrangian complement of $\II$ in $M$ and obtain a $\ZZ$-isomorphism of short exact sequences
 \[ \begin{CD} 0 @>>> \II @> \iota >> M @> \pi >> \II^\vee @>>> 0 \\
  @. @| @VV \phi V @| @. \\
  0 @>>> \II @>>> \II \oplus \II^\vee @>>> \II^\vee @>>> 0, \end{CD} \]
 where $\phi = (\iota + s)^{-1}$ is symplectic by the definition of symplectic extensions. Equipping $\II \oplus \II^\vee$ with the $\OO$-module structure of $M$ induced by $\phi$,
 the lower row becomes $\OO$-linear and we get indeed an isomorphism of symplectic $(\OO,\II)$-extensions.
 From the $\OO$-linearity of the lower row it also follows that the $\OO$-structure on $\II \oplus \II^\vee$ is of the form \[ x.(\lambda,\mu) = (x\lambda + f_x(\mu),x\mu) \] for all $x \in \OO$
 and $(\lambda,\mu) \in \II \oplus \II^\vee$.
 Moreover, we have \[ (x \mapsto f_x) \in \Hom_\ZZ(\OO,\Hom_\ZZ(\II^\vee,\II)) \] and \[ f_{xy} = M_xf_y + f_xM_y \] for all $x,y \in \OO$.
 The property that the $\OO$-structure is self-adjoint with respect to the pairing $\langle \cdot,\cdot \rangle_p$ on $\II \oplus \II^\vee$ is equivalent to
 \[ \tr_p(f_x(\mu_1)\mu_2) = -\tr_p(\mu_1f_x(\mu_2)) \] for all $x \in \OO$ and $\mu_1,\mu_2 \in \II^\vee$.
 Thus, after tensoring with $\QQ$, the map $x \mapsto f_x$ is an element of
\begin{align*} \Hom_\QQ^*(F,\Hom_\QQ^-(F,F)) := \{ & f \in \Hom_\QQ(F,\Hom_\QQ^-(F,F)): \\ & f(xy) = M_xf(y)+M_yf(x) \mbox{ for all } x,y \in F \}, \end{align*}
 Since the $\OO$-module structure is proper, it remains to show that there is some $h \in \Hom_\QQ^+(F,F)$ with $f_x = [M_x,h]$ for all $x \in F$.
 We choose some $x_0 \in F$, such that $F = \langle (1,1),x_0,x_0^2 \rangle_\QQ$. Then the $\QQ$-homomorphism
 \begin{align*} \Hom_\QQ^*(F,\Hom_\QQ^-(F,F)) & \to \Hom_\QQ^-(F,F) \\ f & \mapsto f(x_0) \end{align*}
 is injective by $f(x_0^2)=M_{x_0}f(x_0)+f(x_0)M_{x_0}$ and $f(1,1) \equiv 0$ and thus \[ \dim_\QQ(\Hom_\QQ^*(F,\Hom_\QQ^-(F,F))) \leq 3. \]
 Since \[ \dim_\QQ(\Hom_\QQ^+(F,F)/(\Hom_F(F,F)) = 6-3 = 3 \] and the $\QQ$-linear map
 \begin{align*} \Hom_\QQ^+(F,F)/(\Hom_F(F,F) & \to \Hom_\QQ^*(F,\Hom_\QQ^-(F,F)) \\ h & \mapsto (x \mapsto [M_x,h]) \end{align*} is injective, it is also an isomorphism and hence
 $x \mapsto f_x$ lies in its image.
\end{proof}

Analogously to \cite{MartinMatt}, Theorem 2.1., we get for the pseudo-cubic case the following identification.
\begin{lem}\label{HomExt}
 The map $h \mapsto E_h(\II)$ induces a group isomorphism
 \[ \Hom_\QQ^+(F,F) / (\Hom_F(F,F)+\Hom_\ZZ^+(\II^\vee,\II)) \to \EE(\II). \]
\end{lem}
\begin{proof}
 We have already seen that $E_h(\II)$ lies in $\EE(\II)$. It is easy to check that this map is well-defined, we just mention that we have $\OO_h(\II) = \OO_{h+\varphi}(\II)$ for all
 $h \in \Hom_\QQ^+(F,F)$ and $\varphi \in \Hom_\ZZ^+(\II^\vee,\II)$, and that in this case an isomorphism of symplectic extensions over $\OO_h(\II)$ is given by
 \begin{align*} (\II \oplus \II^\vee)_{h+\varphi} & \to (\II \oplus \II^\vee)_h \\ (\lambda,\mu) & \mapsto (\lambda+\varphi(\mu),\mu). \end{align*}

 To see that our map is a group homomorphism, let $h_1,h_2 \in \Hom_\QQ^+(F,F)$. Then with $\iota$ and $\pi$ defined as in the definition of the Baer sum of~$E_{h_1}(\II)$ and~$E_{h_2}(\II)$,
 an isomorphism of symplectic $\OO$-extensions is given by \begin{align*} (\II \oplus \II^\vee)_{h_1+h_2} & \to \ker(\pi)/{\rm im}(\iota) \\ (\lambda,\mu) & \mapsto ((\lambda,\mu),(0,\mu)) \end{align*}
 and thus $E_{h_1+h_2}(\II) = E_{h_1}(\II) + E_{h_2}(\II)$ in $\EE(\II)$.

 To show that our map is injective, let $h \in \Hom_\QQ^+(F,F)$ such that there is an isomorphism $\phi: E_h(\II) \to E_0(\II)$ to the trivial extension.
 This isomorphism must be of the form $(\lambda,\mu) \mapsto (\lambda+f(\mu),\mu)$ for some $f \in \Hom_\QQ^+(\II^\vee,\II)$ with $[M_x,h-f]=0$ for all $x \in \OO_h(\II) = \OO(\II)$.
 The latter condition comes from the $\OO_h(\II)$-linearity of $\phi$ and is equivalent to $h-f \in \Hom_F(F,F)$, and thus $h \in \Hom_F(F,F) + \Hom_\QQ^+(\II^\vee,\II)$.

 Since we have already proven surjectivity in the previous lemma, the result follows.
\end{proof}

\noindent
{\bf Cusp packets.} Given any $h \in \Hom_\QQ^+(F,F)$ and $a \in F^*$, we denote by $h^a \in \Hom_\QQ^+(F,F)$ the endomorphism $x \mapsto ah(ax)$.
For each lattice $\II$ in $F$, we obviously have $\OO(\II) = \OO(a\II)$ and $\OO_h(\II) = \OO_{h^a}(a\II)$.
For any pseudo-cubic order $\OO$ in $F$, let $\mathcal{L}_\OO$ be the set of lattices $\II$ in $F$ with $\OO \subseteq \OO(\II)$.
We denote by $\EE_\OO(\II) \subset \EE(\II)$ the set of classes of symplectic $(\OO,\II)$-extensions.
Moreover, we denote by $\widetilde{\mathcal{C}}(\OO)$ the set of pairs $(\II,E)$ with $\II \in \mathcal{L}_\OO$ and $E \in \EE_\OO(\II)$.
Using the identification from Lemma~\ref{HomExt}, we have seen that each $a \in F^*$ acts on the disjoint union $\EE_\OO := \coprod_{\II \in \mathcal{L}_\OO} \EE_\OO(\II)$,
by sending $E=E_h(\II) \in \EE_\OO(\II)$ to $E^a:=E_{h^a}(a\II) \in \EE_\OO(a\II)$, as well as on $\widetilde{\mathcal{C}}(\OO)$ by sending $(\II,E)$ to $(a\II,E^a)$.
Finally, a \emph{cusp packet for $\OO$} is defined to be an element of \[ \mathcal{C}(\OO) := \widetilde{\mathcal{C}}(\OO) / F^*. \]
\begin{prop}
 Let $\OO$ be a pseudo-cubic order. The assignment \[ (\II,E_h(\II)) \mapsto \left( 0 \to \II \to (\II \oplus \II^\vee)_h \to \II^\vee \to 0 \right) \] induces a well-defined
 bijection between the set of cusp packets $\mathcal{C}(\OO)$ and the set of isomorphism classes of symplectic $(\OO,\II)$-extensions,
 where $\II$ runs over all lattices in $\mathcal{L}_\OO$.
\end{prop}
\begin{proof}
 Given any $(\II,[E_h(\II)]) \in \widetilde{\mathcal{C}}(\OO)$ and any $a \in F^*$, there is a natural isomorphism $\varphi_a$ of symplectic $\OO_h(\II)$-extensions between the sequences
 $E_h(\II)$ and $E_{h^a}(\II)$. To be more precise, $\varphi_a$ is given by
 \[ \begin{CD} 0 @>>> \II @>>> (\II \oplus \II^\vee)_h @>>> \II^\vee @>>> 0 \\
  @. @VV M_a V @VV M_a \oplus M_{a^{-1}} V @VV M_{a^{-1}} V @. \\
  0 @>>> a\II @>>> (a\II \oplus (a\II)^\vee)_{h^a} @>>> (a\II)^\vee @>>> 0. \end{CD} \]
 Thus the map is well-defined. Surjectivity follows from Lemma~\ref{HomExt}, so it remains to show injectivity.
 Let $(\II_1,[E_1]), (\II_2,[E_2]) \in \widetilde{\mathcal{C}}(\OO)$ with $E_i = E_{h_i}(\II_i)$, such that there is an isomorphism of symplectic $\OO$-extensions $\varphi: E_1 \cong E_2$.
 On rank one modules, an isomorphism between two lattices in $F$ is just multiplication with an invertible element of $F$.
 Hence there are $a,d \in F^*$ and $b \in F$ with $\II_2 = a\II_1$, such that $\varphi$ is of the form
 \[ \begin{CD} 0 @>>> \II_1 @>>> (\II_1 \oplus \II_1^\vee)_{h_1} @>>> \II^\vee @>>> 0 \\
  @. @VV M_a V @VV \psi V @VV M_d V @. \\
  0 @>>> a\II_1 @>>> (a\II_1 \oplus (a\II_1)^\vee)_{h_2} @>>> (a\II_1)^\vee @>>> 0 \end{CD} \]
 with $\psi(\lambda,\mu)=(a\lambda+b\mu,d\lambda)$ for all $(\lambda,\mu) \in \II \oplus \II^\vee$.
 As $\psi$ is symplectic, we have $d=a^{-1}$ and thus the isomorphism of symplectic $\OO$-extensions \[ \varphi_a \circ \varphi^{-1}: E_2 = E_{h_2}(a\II_1) \cong E_1^a \]
 is the identity on $a\II_1$ and $(a\II_1)^\vee$. Hence $E_2$ and $E_1^a$ represent the same elements in $E_\OO(a\II)$ and we get \[ (\II_1,E_1) \sim (a\II_1,E_1^a) = (\II_2,E_2). \] 
\end{proof}

\begin{cor}\label{BijektionCuspPacketsCusps}
 Let $\OO$ be a pseudo-cubic order. The assignment \[ (\II,E_h(\II)) \mapsto \left( 0 \to \II \to (\II \oplus \II^\vee)_h \to \II^\vee \to 0 \right), \]
 where $\II$ runs over all lattices in $\mathcal{L}_\OO$, induces a well-defined bijection between the set of cusp packets $\mathcal{C}(\OO)$ and the set of isomorphism classes
 of symplectic $\OO$-extensions (or equivalently the set of cusps $\mathcal{C}_\OO$).
\end{cor}
\begin{proof}
 To show that the map is well-defined, we have to prove that each~$(\II \oplus \II)_h$ is isomorphic to some lattice in $F^2$ as a symplectic $\OO$-module.
 But $(\II \oplus \II)_h$ is a lattice in $(\II \oplus \II)_h \otimes_\ZZ \QQ = (F \oplus F)_h$, and a symplectic $F$-isomorphism $F^2 \cong (F \oplus F)_h$ is given by
 $(x,y) \mapsto (x-[M_y,h](1,1),y)$.

 Injectivity follows from the last proposition, so it remains to show that any symplectic $\OO$-extension \[ E:\quad 0 \to \II_1 \to M \to \II_2 \to 0 \] is isomorphic to one
 of the desired form. By definition, we may assume
 without loss of generality that $M$ is a lattice in $F^2$ equipped with the symplectic pseudo-trace pairing and that $\II_1 = \II$ is a lattice in $F$.
 Since the image of $\II$ is a Lagrangian subgroup of $M$, the symplectic pseudo-trace pairing on $M$ and a choice of a $\ZZ$-splitting $s: \II_2 \hookrightarrow M$ onto a
 Lagrangian complement of $\II$ induces a $\ZZ$-module isomorphism \[ \II_2 \cong \Hom_\ZZ(\II;\ZZ). \]
 The pseudo-trace pairing on $\II$ induces also a $\ZZ$-module isomorphism \[ \II^{\vee} \cong \Hom_\ZZ(\II;\ZZ). \]
 Moreover, since the $\OO$-module structure on $M$ is self-adjoint with respect to the symplectic pseudo-trace pairing and for each $x \in \OO$ the commutator $[s,M_x]$ maps to
 $\II \subset M$, the resulting isomorphism $\varphi: \II_2 \cong \II^\vee$ is in fact $\OO$-linear. It follows that the sequence $E$ of $\OO$-modules is isomorphic to the
 sequence \[ 0 \to \II \to M \to \II^\vee \to 0. \] Since the $\ZZ$-splitting $s \circ \varphi^{-1} : \II^\vee \hookrightarrow M$ induces a symplectic isomorphism
 $\II \oplus \II^\vee \cong M$, the result follows from Lemma~\ref{HomExtSur}.
\end{proof}

{\bf Remark.} We have seen in the last proof, that in Definition~\ref{Symplextension2} we can replace the condition that the $\OO$-module $\II_2$ in the sequence
\[ E:\quad 0 \to \II_1 \to M \to \II_2 \to 0 \] is a lattice in $F$, with the condition that $\II_1$ maps onto a Lagrangian subgroup of~$M$.
\clearpage

\section{Boundary strata}\label{boundstrat}

In \cite{MartinMatt}, Bainbridge and M\"oller introduced a generalization of the classical period matrices on the moduli space $\cMM_g(L)$ of $L$-weighted stable curves via a meromorphic function on
$\cMM_g(L)$. They used this function to give a necessary and, for $g=3$, sufficient condition for a stable curve for lying in the boundary of the real multiplication locus.
Furthermore, they described how the corresponding weighted stable curves can be cut out by equations in terms of cross-ratios of projective coordinates.

In the following section, we transfer these techniques to our pseudo-cubic case.
We will see that one has to be very careful about the subtleties and that not everything works as fine as in the cubic case.
Nevertheless, we get the same necessary condition and very similar cross-ratio equations.

\subsection{Eigenforms for pseudo-real multiplication}

Now we introduce eigenforms for pseudo-real multiplication on Jacobians of Riemann surfaces.
We will see that, once we have fixed a quadratic pseudo-embedding, the choice of pseudo-real multiplication corresponds to the choice of an eigenform.
In fact, we will do this for a larger class then just Riemann surfaces, namely for stable curves with compact Jacobians.
\\[1em]
{\bf Deligne-Mumford compactification.}
A \emph{stable Riemann surface} or \emph{stable curve} is by definition a connected, compact, complex analytic variety of dimension one with only finitely many singularities,
such that each singularity is a node - i.e. a singularity of the form~$xy=0$ - and such that each connected component of the complement of the nodes  is hyperbolic.
The \emph{geometric genus} of a stable curve~$X$ is the sum of the genera of the connected components of the normalization of~$X$.
The \emph{arithmetic genus} of~$X$ is the genus of the Riemann surface obtained by thickening each node to an annulus.
If we speak just of the genus, we mean the arithmetic genus.
We give now a brief description, how the moduli space~$\cMM_g$ of genus~$g$ stable curves can be constructed.
For more details concerning this construction of~$\cMM_g$ as an analytic space, see~\cite{HubbardKoch}.
\begin{figure}[ht]
\centering
\begin{tikzpicture}[scale=1]
 \coordinate (A) at (0,0); \coordinate (B) at (8,0); \coordinate (C) at (4,1); \coordinate (D) at (4,-1);
 \draw[thick] (A) .. controls (0,2) and (3.5,1.5) .. (C);
 \draw[thick] (A) .. controls (0,-2) and (3.5,-1.5) .. (D);
 \draw[thick] (B) .. controls (8,2) and (4.5,1.5) .. (C);
 \draw[thick] (B) .. controls (8,-2) and (4.5,-1.5) .. (D);
 \draw[thick] (A) .. controls (0.5,0.5) and (1.5,0.5) .. (2,0);
 \draw[thick] (A) .. controls (0.5,-0.5) and (1.7,-0.25) .. (2.2,0.1);
 \draw[thick] (5.2,0) .. controls (5.6,0.3) and (6.4,0.3) .. (6.8,0);
 \draw[thick] (5,0.1) .. controls (5.4,-0.3) and (6.6,-0.3) .. (7,0.1);
 \draw[thick] (C) .. controls (3.5,0.5) and (3.5,-0.5) .. (D);
 \draw[thick] (C) .. controls (4.5,0.5) and (4.5,-0.5) .. (D);
\end{tikzpicture}
\setlength{\abovecaptionskip}{0cm}
\caption{A stable curve of arithmetic genus three and geometric genus one}\label{StableCurveg2}
\end{figure}
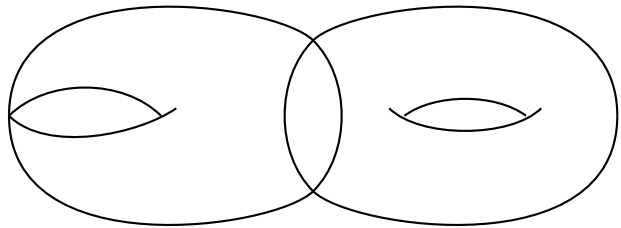

Fix a compact, oriented surface $\Sigma_g$ of genus $g$.
For any genus~$g$ stable curve~$X$, a \emph{marking by $\Sigma_g$} or \emph{collapse} is a continuous surjective map $f\!: \Sigma_g \to X$, such that the preimage of each node is a simple closed curve
in $\Sigma_g$ and such that $f$ induces an orientation preserving homeomorphism on the complement of these curves. We denote the set of preimages of the nodes under~$f$ by~$\Gamma_f$.
A pair $(X,f)$, where $X$ is a stable curve and $f: \Sigma_g \to X$ is a collapse, is called a \emph{$\Sigma_g$-marked} stable curve.
Two $\Sigma_g$-marked stable curves, say $(X_1,f_1)$ and $(X_2,f_2)$, are defined to be equivalent, if there is an isomorphism $\phi\!: X_1 \to X_2$, such that there is an homeomorphism
$h\!: \Sigma_g \to \Sigma_g$ isotopic to the identity on~$\Sigma_g$ with $\phi \circ f_1 = f_2 \circ h$.
The \emph{augmented Teichm\"uller space} $\cTT_g(\Sigma_g)$ is defined to be the set of all $\Sigma_g$-marked stable curves up to this equivalence relation
and it carries a topology defined as follows.
For each simple closed curve $\gamma$ in $\Sigma_g$, there is the well-defined function $l_\gamma\!: \cTT_g \to \RR_{\geq 0} \cup \{ \infty \}$ assigning to each $[(X,f)] \in \cTT_g(\Sigma_g)$
the hyperbolic length of the unique geodesic in~$X$ homotopic to~$f(\gamma)$.
Note that $l(\gamma)=0$, if $\gamma$ is homotopic to a curve in $\Gamma_f$, and that $l(\gamma)= \infty$, if $\gamma$ is transversal to a curve in~$\Gamma_f$.
Then $\cTT_g(\Sigma_g)$ is equipped with the smallest topology, such that each $f_\gamma$ is continuous.
It contains the \emph{Teichm\"uller space} $\TT_g(\Sigma_g)$ of classes of $\Sigma_g$-marked Riemann surfaces as an open, dense subspace.

The \emph{mapping class group} $\Mod(\Sigma_g)$ is defined to be the group of orientation preserving homeomorphisms of $\Sigma_g$ up to isotopy.
It acts on $\cTT_g(\Sigma_g)$ via \[ h.[(X,f)]:=[(X,f\circ h^{-1})] \] and we get $\cMM_g = \cTT_g(\Sigma_g)/\Mod(\Sigma_g)$ for the moduli space of genus~$g$ stable curves.
It compactifies the moduli space $\MM_g = \TT_g(\Sigma_g)/\Mod(\Sigma_g)$ of genus~$g$ Riemann surfaces and is called the \emph{Deligne-Mumford compactification} of $\MM_g$.
\\[1em]
{\bf Stable Abelian differentials and Jacobians.}
For any Riemann surface $X$ of genus ~$g$, we denote by $\Omega(X)$ the $g$-dimensional $\CC$-vector space of holomorphic one-forms (or Abelian differentials) and by $\Omega\MM_g \to \MM_g$
the vector bundle of holomorphic one-forms.
We extend this bundle to a vector bundle $\Omega\cMM_g \to \cMM_g$ over the moduli space of stable forms as in~\cite{MartinMatt} as follows.
For any stable curve $X$, we denote by $X_0$ the complement of the nodes in $X$. To a pair of cusps of $X_0$ coming from the same node of $X$ we will refer as \emph{opposite cusps}.
A \emph{stable Abelian differential} or for short just \emph{stable form} on $X$ is defined to be a holomorphic one-form $\omega$ on $X_0$ with at most simple poles at each cusp of $X_0$,
such that for any pair $(c_1,c_2)$ of opposite cusps, for the residues we have \[ \res_{c_1}(\omega) = -\res_{c_2}(\omega). \]
If~$X$ has arithmetic genus~$g$, then the space~$\Omega(X)$ of stable Abelian differentials on~$X$ is a $\CC$-vector space of dimension~$g$.
An isomorphism between two pairs $(X_1,\omega_1)$, $(X_2,\omega_2)$, where each $\omega_i$ is a stable form on the stable curve $X_i$, is an isomorphism $f\!: X_1 \to X_2$ with $f^*(\omega_2)=\omega_1$.
Defining~$\Omega\cMM_g$ to be the space of isomorphism classes of stable curves with stable forms, we get a vector bundle \[ \Omega\cMM_g \to \cMM_g, \quad [(X,\omega)] \to [X], \]
the \emph{bundle of stable forms}, that induces the usual bundle $\Omega\MM_g \to \MM_g$ of holomorphic one-forms.

The homology group $H_1(X_0;\ZZ)$, a free Abelian group of rank at most~$2g$, can be embedded into the $\CC$-vector space $\Omega(X)^* := \Hom_\CC(\Omega(X);\CC)$ via
\[ H_1(X_0;\ZZ) \hookrightarrow \Omega(X)^*,\quad \gamma \mapsto \left( \omega \mapsto \int_\gamma \omega \right). \]
The \emph{Jacobian} of $X$ is then defined as the quotient \[ \Jac(X):=\Omega(X)^*/H_1(X_0;\ZZ). \]
It is an Abelian variety if and only if the geometric genus of~$X$ is~$g$, and this holds if and only if every node is separating.
Otherwise it is not compact, since the $\ZZ$-rank of $H_1(X_0;\ZZ)$ is strictly smaller than~$2g$.
We denote by $\widetilde \MM_g \subset \cMM_g$ the locus of stable curves with compact Jacobians.
For $X \in \widetilde \MM_g$, the intersection pairing on the irreducible components induces a principal polarization~$\Theta_X$ on~$\Jac(X)$, and thus the Torelli map extends to a map
\[ t:~ \widetilde \MM_g \to \Ag,\quad X \mapsto (\Jac(X),\Theta_X), \] which is surjective for $g \leq 3$.

We will often just write $\Jac(X)$ instead of the pair $(\Jac(X),\Theta_X)$.
\\[1em]
{\bf Eigenforms for pseudo-real multiplication.}
Consider a genus three stable curve~$X$ with compact Jacobian $\Jac(X)$.
The embedding \[ H_1(X_0;\ZZ) \hookrightarrow \Omega(X)^*,\quad \gamma \mapsto (\omega \mapsto \smallint_{\gamma}\omega) \] extends $\QQ$-linear to an embedding
$H_1(X_0;\QQ) \hookrightarrow \Omega(X)^*$, therefore the action of~$\End(\Jac(X))$ on~$H_1(X_0;\ZZ)$ extends naturally to an action of $\End_\QQ(\Jac(X))$ on~$H_1(X_0;\QQ)$.
As always, let $\OO$ be a pseudo-cubic order in a pseudo-cubic number field $F = K \oplus \QQ$.
We say that $X$ has \emph{pseudo-real multiplication} by~$\OO$ (respectively~$F$), if~$\Jac(X)$ has pseudo-real multiplication by~$\OO$ (respectively~$F$).
We denote by \[ \RM_{\OO} \subset \MM_3\quad (\mbox{respectively } \widetilde{\RM}_\OO \subset \widetilde \MM_3) \] the locus of Riemann surfaces (respectively stable curves with compact Jacobian)
having pseudo-real multiplication by $\OO$ and by \[ \cRMO \subset \cMM_3\quad \] its closure in the Deligne-Mumford compactification of $\MM_3$.
We define $\RM_F$, $\widetilde{\RM}_F$ and $\cRMF$ analogously.

Now we want to discuss eigenforms for pseudo-real multiplication. In the case of a real cubic number field, eigenforms are defined by using the three field embeddings into
$\RR$. The analogue to these three maps in the pseudo-cubic case is the following.
Let $\iota_1,\iota_2: K \hookrightarrow \RR$ be the two real embeddings of $K$ into $\RR$. The two maps \[ F \to \RR,\quad (x,q) \mapsto \iota_i(x) \] will also be denoted
by $\iota_i$ for $i \in \{ 1,2 \}$. The third map is simply \[ \iota_3:~ F \to \RR,\quad (x,q) \mapsto q. \]
We call these three (non-injective) maps $\iota_1,\iota_2,\iota_3 : F \to \RR$ the three \emph{real pseudo-embeddings} of $F$ to $\RR$.
We also refer to the first two of these maps,~$\iota_1$ and~$\iota_2$, as the two \emph{real quadratic pseudo-embeddings}.
Now consider some choice of pseudo-real multiplication \[ \rho:~ F \hookrightarrow \End_\QQ^+(\Jac(X)). \]
This choice induces a monomorphism $F \hookrightarrow \End_\CC(\Omega(X)^*)$ and by duality a monomorphism $F \hookrightarrow \End_\CC(\Omega(X))$,
both also denoted by $\rho$. Once we have fixed the choice of pseudo-real multiplication, we will also write $x.\cdot$ instead of $\rho(x)(\cdot)$ wherever it makes sense.
\begin{defi}
 Let $X \in \widetilde{\RM}_F$ for a pseudo-cubic number field $F$ and let $\rho$ be some choice of pseudo-real multiplication by~$F$.
 Furthermore, let $\iota : F \to \RR$ be some real pseudo-embedding of $F$.
 Then a non-zero holomorphic one-form $\omega \in \Omega(X)$ is called a \emph{$\iota$-eigenform for $\rho$}, if \[ x.\omega = \iota(x)\omega \] for all $x \in F$.
 The form $\omega$ is called an \emph{eigenform for $\rho$} if it is a $\iota$-eigenform for some real pseudo-embedding $\iota: F \to \RR$
 and it is just called an \emph{eigenform} if it is an eigenform for some choice of pseudo-real multiplication.
\end{defi}
By the definition of the embedding of $H_1(X_0;\ZZ)$ into $\Omega(X)^*$ by integrating forms and the definition of the induced action of $F$ on $\Omega(X)$ by duality,
one sees immediately that $\omega \in \Omega(X)$ is a $\iota$-eigenform for $\rho$ if and only if we have \[ \smallint_{x.\gamma}\omega = \iota(x) \smallint_{\gamma}\omega \]
for all $x \in \OO$ and $\gamma \in H_1(X_0;\ZZ)$.

For every real pseudo-embedding $\iota$ we denote by $\Omega^\iota(X) \leq \Omega(X)$ the one-dimensional subspace of $\iota$-eigenforms for $\rho$.
Analogous to the case of totally real number fields, we have the decomposition of $\Omega(X)$ into eigenspaces,
\begin{equation}\label{Eigenraumzerlegung} \Omega(X) = \bigoplus_{i=1}^3 \Omega^{\iota_i}(X), \end{equation}
as can be easily seen by the description of principally polarized Abelian varieties with pseudo-real multiplication in Section~\ref{dim3}.
On $\Omega(X)$ there is the Hermitian form \[ \langle \omega_1,\omega_2 \rangle = \frac{i}{2} \int_X \omega_1 \wedge \overline{\omega_2}, \]
and the action of $F$ on $\Omega(X)$ is also self-adjoint with respect to this form.
Therefore, the decomposition (\ref{Eigenraumzerlegung}) is in fact orthogonal.

We denote by \[ \Omega^\iota\RM_{\OO} \subset \Omega\MM_3\quad (\mbox{respectively } \Omega^\iota\widetilde{\RM}_\OO \subset \Omega \widetilde \MM_3) \] the locus of Riemann surfaces
(respectively stable curves with compact Jacobian) admitting pseudo-real multiplication by $\OO$ together with the choice of a~$\iota$-eigenform for some choice of such a pseudo-real multiplication,
and by \[ \Omega^\iota\cRMO \subset \Omega \cMM_3\quad \] we denote its closure in the space of stable Abelian differentials.
We define $\Omega^\iota\RM_F$, $\Omega^\iota\widetilde{\RM}_F$ and $\Omega^\iota\cRMF$ analogously.

\begin{lem}\label{eigenrealmult}
 Let $\OO \subset F$ be a pseudo-cubic order and let $\iota:~ F \to \RR$ be one of the two real quadratic pseudo-embeddings. Moreover, let~$X \in \widetilde{\RM}_\OO$.
 Then, for each choice of pseudo-real multiplication $\rho$ by $\OO$, there is, up to scalar multiples, a unique $\iota$-eigenform $\omega$ for $\rho$.
 Conversely, if $\omega \in \Omega(X)$ is any~$\iota$-eigenform, then there is a unique choice of pseudo-real multiplication~$\rho$ realizing~$\omega$ as a~$\iota$-eigenform.
\end{lem}
\begin{proof}
 Without loss of generality, we may assume that $\iota=\iota_1$.
 The first statement follows from the decomposition~\eqref{Eigenraumzerlegung}. To prove the second statement, let $\omega \in \Omega(X)$ be a $\iota$-eigenform for two choices $\rho_1, \rho_2$
 of pseudo-real multiplication. We have to show that they induce the same decomposition~\eqref{Eigenraumzerlegung}.
 A priori, $\Omega^{\iota_3}(X) = \Omega_{\rho_i}^{\iota_3}(X)$ depends on $\rho_i$.
 We have seen in Section~\ref{pseudorealmult}, that each $\rho_i$ defines a one-dimensional subvariety $C_i$ of $X$, namely the kernel of $\rho_i(\QQ{\sqrt{D}})$.
 These kernels are just $C_i = \Omega_{\rho_i}^{\iota_3}(X) / \Lambda_i$ with $\Lambda_i = \Omega_{\rho_i}^{\iota_3}(X) \cap H_1(X;\ZZ)$.
 On the other hand, as $\omega$ is a $\iota_1$-eigenform for both $\rho_1$ and $\rho_2$, each lattice $\Lambda_i$ is generated by those cycles $\gamma \in H_1(X;\ZZ)$,
 where $\smallint_\gamma \omega = 0$. Thus $\Omega_{\rho_i}^{\iota_3}(X) = \Lambda_i \otimes_\ZZ \QQ$ depends only on $\omega$.
 Finally, $\Omega^{\iota_2}(X)$ is the orthogonal complement of $\Omega^{\iota_1}(X) \oplus \Omega^{\iota_3}(X)$ in $\Omega(X)$.
\end{proof}
 Recall that $X_\OO$ was defined to be the moduli space of pairs~$(A,\rho)$, where~$A$ is a principally polarized Abelian variety and~$\rho$ is pseudo-real multiplication by~$\OO$ on~$A$.
 Since the extended Torelli-map $t:~ \widetilde\MM_3 \to \mathcal{A}_3$ is surjective, we immediately get the following corollary. 
\begin{cor}\label{eigenrealmultmap}
 The map \[ \mathbb{P} \Omega^\iota \widetilde{\RM}_\OO \to X_\OO,\quad (X,[\omega]) \mapsto (\Jac(X),\rho), \] where $\rho$ denotes the choice of pseudo-real multiplication by $\OO$ realizing
 $\omega$ as a~$\iota$-eigenform form, is surjective. Furthermore, restricted to the dense subset $\mathbb{P} \Omega^\iota {\RM}_{\OO}$, it is injective.
\end{cor}
\subsection{Lagrangian boundary strata}

Here we recall the concept of Lagrangian boundary strata, see~\cite{MartinMatt}, Section~3.
\\[1em]
For any stable curve~$X$ of arithmetic genus~$g$, let~$N(X)$ be the set of nodes of~$X$, let $C(X)$ be the set of cusps of $X$ (so $|C(X)|=2|N(X)|$), and let~$X_0$ again be the complement of the nodes.
For each cusp $c$ of $X$ we denote by $c'$ the opposite cusp and by $\alpha_c \in H_1(X_0;\ZZ)$ the class of a positively oriented simple closed curve around~$c$.
We define the subgroup \[ I := \langle \alpha_c + \alpha_{c'} : c \in C(X) \rangle \leq H_1(X_0;\ZZ), \] the factor group \[ \widehat H_1(X;\ZZ) := H_1(X_0;\ZZ)/I \] and its subgroup
\[ \widehat C(X) := \langle [\alpha_c] : c \in C(X) \rangle \leq \widehat H_1(X;\ZZ). \]
If $\widetilde X$ denotes any smooth curve obtained by thickening each node of $X$ to an annulus, we get the natural homomorphism \[ h: \widehat H(X;\ZZ) \to H_1(\widetilde X;\ZZ). \]

\begin{defi}
 Let $X$ be a stable curve of arithmetic genus $g$.
 \begin{enumerate}
  \item A subgroup $L \leq \widehat H_1(X;\ZZ)$ is called a \emph{Lagrangian subgroup}, if the following three conditions hold.
        \begin{enumerate}
         \item The group $L$ is free Abelian of rank $g$.
         \item The factor group $\widehat H_1(X;\ZZ)/L$ is torsion-free.
         \item The intersection form of $H_1(\widetilde X;\ZZ)$ is trivial on $h(L)$. (This is independent from the choice of the thickening.)
        \end{enumerate}
  \item A \emph{Lagrangian marking} of $X$ by $L$, where $L$ is any free Abelian group of rank~$g$,
        is a group monomorphism \[ \varrho : L \hookrightarrow \widehat H_1(X;\ZZ), \] such that the image $\varrho(L)$ is a Lagrangian subgroup.
 \end{enumerate}
\end{defi}
Remark that $\varrho(L)$ necessarily contains $\widehat C(X)$.

\bigskip
An \emph{isomorphism} (respectively \emph{topological equivalence}) of pairs $(X_1,\varrho_1)$ and $(X_2,\varrho_2)$ of stable curves with Lagrangian markings by $L$ is by definition an isomorphism
(resp. homeomorphism) $\varphi\!: X_1 \to X_2$, such that the induced group isomorphism $\varphi_*\!: \widehat H_1(X_1;\ZZ) \to \widehat H_1(X_2;\ZZ)$ is compatible with the Lagrangian markings,
i.e. $\varphi_* \circ \varrho_1 = \varrho_2$. We denote by $\cMM_g(L)$ the space of isomorphism classes of arithmetic genus $g$ stable curves with Lagrangian markings by $L$ and by
$\MM_g(L) \subset \cMM_g(L)$ the subspace of isomorphism classes of genus $g$ Riemann surfaces with Lagrangian markings by $L$.

We can realize those spaces by taking quotients of the Teichm\"uller spaces in the following way.
Fix a genus $g$ surface $\Sigma_g$ and a Lagrangian subgroup $L \leq H_1(\Sigma_g;\ZZ)$. Let ${\rm Mod}(\Sigma_g)$ be the mapping class group of $\Sigma_g$. We define
\begin{align*}
 {\rm Mod}(\Sigma_g,L) & := \{ \sigma \in {\rm Mod}(\Sigma_g) : \sigma_*|_L = \id_L \},\\
 \cTT(\Sigma_g,L) & := \{ (X,f) \in \cTT(\Sigma_g): f \mbox{ collapses only curves } \gamma \mbox{ with } [\gamma] \in L \}
\end{align*}
and get
\begin{align*}
 \MM_g(L) & = \TT(\Sigma_g)/{\rm Mod}(\Sigma_g,L),\\
 \cMM_g(L) & = \cTT(\Sigma_g,L)/{\rm Mod}(\Sigma_g,L).
\end{align*}
A topological equivalence class in $\partial \cMM_g(L) = \cMM_g(L) \setminus \MM_g(L)$ is called a \emph{Lagrangian boundary stratum}.

If $\gamma \in L \leq H_1(\Sigma_g;\ZZ)$ is any non-zero simple closed curve, then we denote by $D_{\gamma} \subset \cMM_g(L)$ the divisor of stable curves obtained by pinching a curve homologous
to $\gamma$.
If the geometric genus of $X$ is zero, then there is another viewpoint of the boundary strata which we will describe next.
\begin{defi}
 Let $L$ be a free Abelian group of rank $g$ and let $X$ be a stable curve of arithmetic genus $g$ and geometric genus $0$.
 A \emph{$L$-weighting} of $X$ is a map $w: C(X) \to L$ satisfying the following three conditions.
 \begin{enumerate}
  \item $w(c)=-w(c')$ for all $c \in C(X)$.
  \item $\sum\limits_{c \in C(Y)}w(c)=0$ for all irreducible components $Y \subseteq X$.
  \item $\langle w(C(X)) \rangle_\ZZ$ = L
 \end{enumerate}
 The image $w(c)$ is called the \emph{weight of the cusp $c$}.
\end{defi}
It follows that cusps of separating nodes have weight zero.
There is an obvious correspondence between $L$-weighted stable curves $(X,w)$ and genus zero stable curves with Lagrangian markings $(X,\varrho)$.
If $c_1,c_1',...,c_n,c_n'$ are the cusps of $X$ coming from non-separating nodes, then the weights~$w(c_i)$ generate~$L$.
Hence the $L$-weighting~$w$ yields the well-defined marking $\varrho$ mapping each $w(c_i)$ to $\alpha_{c_i} \in \widehat H_1(X;\ZZ)$.

Given a stable curve~$X$ of arbitrary geometric genus with a Lagrangian marking~$\varrho$, and given a cusp~$c \in C(X)$ we will also call the element~$\varrho^{-1}(c) \in L$ the weight of the cusp~$c$.

\bigskip
An \emph{isomorphism} (respectively \emph{topological equivalence}) of pairs $(X_1,w_1)$ and $(X_2,w_2)$ of $L$-weighted stable curves is by definition a weight preserving isomorphism
(respectively homeomorphism). These correspond to the isomorphisms (respectively topological equivalences) of genus zero stable curves with Lagrangian markings defined above.
In terms of $L$-weighted stable curves we call a geometric genus zero Lagrangian boundary stratum a \emph{$L$-weighted boundary stratum}.

For an arbitrary Lagrangian boundary stratum (not necessarily consisting of geometric genus zero curves) or $L$-weighted boundary stratum $\mathcal{S}$ we denote by
${\rm Weight}(\mathcal{S})$ the set of weights of $\mathcal{S}$.
\\[1em]
{\bf Pseudo-embeddings of weighted boundary strata.}
In the case of a totally real cubic number field, we have natural embeddings of weighted boundary strata into $\Omega \cMM_3$ by using the embeddings of the field
into the real numbers. In our pseudo-cubic case we define in the same way a (not necessarily injective) map as follows.
Let $\II$ be a lattice in $F$ and let $\mathcal{S} \subset \cMM_3(\II)$ be an $\II$-weighted boundary stratum.
Furthermore, let $\iota: F \to \RR$ be one of the three real pseudo-embeddings into the real numbers.
For every $X \in \mathcal{S}$ and for each cusp $c$ of $X$, we denote by~$w(c)$ the weight of the cusp. 
Then we define \[ \pi_\iota:~ \mathcal{S} \to \Omega\cMM_3,\quad (X,\varrho) \mapsto (X,\omega), \]
where $\omega \in \Omega(X)$ is the unique stable form which satisfies ${\rm res}_c(\omega) = \iota(w(c))$ for each cusp~$c$.
We will denote by $p_\iota$ the composition of $\pi_\iota$ with the natural projection $\Omega(X) \to \mathbb{P}\Omega(X)$ and refer to~$p_\iota$ as the \emph{$\iota$-pseudo-embedding of $\mathcal{S}$}.
\subsection{Families with pseudo-real multiplication}\label{pseudorealmultfamilies}

We are interested in the boundary in $\cMM_3$ of the pseudo-real multiplication locus in $\MM_3$. We will use the concept of vanishing cycles, described in \cite{MartinMatt}, Section~5,
and will first give a short repetition. Then we will show that a smooth family of stable curves at a geometric zero stable curve with pseudo-real multiplication corresponds to a cusp packet.

\bigskip
We call a family $\mathcal{X} \to \Delta = \{ z \in \CC : |z|<1 \}$ of stable curves of arithmetic genus three \emph{smooth}, if for every
$z \in \Delta^* = \Delta \setminus \{ 0 \}$ the fiber $X_z$ is smooth.
Such a family defines via $z \mapsto X_z$ a holomorphic map $\Delta \to \cMM_3$. Conversely, any holomorphic map $\Delta \to \cMM_3$ sending $\Delta^*$ to $\MM_3$,
after possibly taking a cover of~$\Delta$ ramified only over~$0$, arises from such a family.

For any two points $z_1,z_2 \in \Delta^*$ and any path $\gamma$ joining $z_1$ with $z_2$, we denote by $f_\gamma: X_{z_1} \to X_{z_2}$ the lifted homeomorphism.
If $\gamma$ is a closed path in $\Delta^*$, starting and ending at $z$, then $f_\gamma$ is just a product of Dehn twists.
There is a collection of isotopy classes of simple closed curves in $X_{z_1}$, called \emph{vanishing curves}, such that the images under $f_\gamma$ are pinched as $z_2 \to 0$.
This collection does not depend on the choice of the path. The homology classes in the subgroup $V_z \leq H_1(X_z;\ZZ)$ generated by the vanishing curves are called \emph{vanishing cycles}.

Let $\OO$ be a pseudo-cubic order. A smooth family $\mathcal{X} \to \Delta$ has \emph{pseudo-real multiplication by $\OO$}, if we can equip each smooth fiber $X_z$ with pseudo-real multiplication
\[ \rho_z : \OO \hookrightarrow \End^+(\Jac(X_z)) \] for each $z \in \Delta^*$, such that the isomorphisms $H_1(X_{z_1};\ZZ) \to H_1(X_{z_2};\ZZ)$ arising from the Gauss-Manin connection commute
with the action of $\OO$. This is equivalent to the condition that for each $z_1,z_2 \in \Delta^*$ and for each path $\gamma$ in $\Delta^*$ joining $z_1$ with $z_2$, the action of $\OO$ on the
homology groups $H_1(X_{z_i};\ZZ)$ commutes with the induced isomorphism $(f_\gamma)_* : H_1(X_{z_1};\ZZ) \to H_1(X_{z_2};\ZZ)$.
Analogous to~\cite{MartinMatt}, Proposition~5.5, we have the following lemma.
\begin{lem}\label{vanishingcurves}
 Let $\mathcal{X} \to \Delta$ be a smooth family of stable curves of arithmetic genus three with pseudo-real multiplication by $\OO$.
 Then for each $z \in \Delta^*$, the action of $\OO$ on $H_1(X_z;\ZZ)$ preserves the subgroup $V_z \leq H_1(X_z;\ZZ)$ of vanishing cycles.
\end{lem}
\begin{proof}
 By the definition of pseudo-real multiplication on a family, it suffices to show that there is at least one $z \in \Delta^*$, such that $\OO$ preserves $V_z$.

 Given any Riemann surface $X$, the \emph{Hodge-norm} of $\gamma \in H_1(X;\RR)$ is defined as
 \[ ||\gamma||_X = \sup_{\substack{\omega \in \Omega(X) \\ ||\omega||=1}} \left| \int_\gamma \omega \right|, \] where $||\omega|| = (\int_X|\omega|^2)^{\frac{1}{2}}$ is the norm induced by the
 form $\langle \omega_1,\omega_2 \rangle = \frac{i}{2} \smallint_X \omega_1 \wedge \overline{\omega_2}$.
 The \emph{extremal length} of the family of curves homotopic to a fixed closed curve $\gamma$ in $X$ is defined as \[ {\rm Ext}(\gamma) := \sup_\rho \frac{L(\rho)^2}{A(\rho)}, \]
 where the supremum is taken over all conformal metrics $\rho(z)dz$ with $\rho$ nonnegative and measurable,
 \[ L(\rho) := \inf_{\delta \simeq \gamma} \int_\delta \rho(z)|dz| \] and \[ A(\rho) := \int_X \rho(z)^2|dz|^2. \]
 Due to \cite{MartinMatt}, Proposition~3.1, We have \[ ||\gamma||_X^2 \leq {\rm Ext}(\gamma) \] for every closed curve $\gamma$ in $X$.
 Furthermore, by \cite{MartinMatt}, Proposition~3.2, there is a constant $C>0$ independent of $X$, such that any $\gamma \in H_1(X;\ZZ)$ is represented by a sum of simple closed curves
 $\delta_1,...,\delta_n$ with \[ {\rm Ext}(\delta_i) \leq C||\gamma||_X^2 \] for each $i$.

 For each $x \in F$ we define \[ ||x||_\infty := |\iota_1(x)|+|\iota_2(x)|+|\iota_3(x)|, \] where $\iota_1,\iota_2,\iota_3$ again denote the three real pseudo-embeddings.
 Now let $x \in \OO$ be primitive. We choose some $\omega_0 \in \Omega(X)$ with $||\omega_0||=1$ and $||x.\gamma||_{X_z}=|\int_{x.\gamma} \omega_0|$.
 By the decomposition (\ref{Eigenraumzerlegung}) of $\Omega(X)$, we can write $\omega_0 = \omega_1 + \omega_2 + \omega_3$ with $\omega_i \in \Omega^{\iota_i}(X)$.
 Since this decomposition is orthogonal, we have $||\omega_i|| \leq ||\omega_0|| = 1$ and get the estimate
 \begin{align*} ||x.\gamma||_{X_z} & \leq \sum_{i=1}^3 \left| \int_{x.\gamma}\omega_i \right| \\
                 & = \sum_{i=1}^3 \left( |\iota_i(x)| \cdot ||\omega_i|| \cdot \left|\int_\gamma \frac{\omega_i}{||\omega_i||}\right| \right) \\
                 & \leq \sum_{i=1}^3 \left( |\iota_i(x)| \cdot ||\gamma||_{X_z} \right) \\
                 & = ||x||_\infty ||\gamma||_{X_z} \end{align*}
 for every $\gamma \in H_1(X_z;\ZZ)$. For any $\varepsilon > 0$, we may choose $z$ sufficiently small, such that for each vanishing curve $\gamma_i$ in $X_z$ we have ${\rm Ext}(\gamma_i) < \varepsilon$
 and thus \[ ||x.\gamma_i|| < ||x||_\infty \varepsilon^{1/2}. \]
 So in $H_1(X_z;\ZZ)$ each $x.\gamma_i$ is represented by a sum of simple closed curves $\delta_1,...,\delta_n$ with \[ {\rm Ext}(\delta_i) < C||x||_\infty^2 \varepsilon. \]
 There is a constant $D>0$, such that ${\rm Ext}(\gamma)>D$ for every $z \in \Delta^*$ and every curve $\gamma$ in $X_z$ that is not a vanishing curve.
 Thus, the $\delta_i$ must be vanishing curves if we choose $\varepsilon$ sufficiently small, and the claim follows.
\end{proof}
Recall that in $\cMM_3(L)$ only curves in $L$ are collapsed. Therefore, if the stable curve $X_0$ from the family in the lemma above is of geometric genus zero, then in this case the Lagrangian
subgroup of $H_1(X_z;\ZZ)$ marked by $L$ is exactly the subgroup of vanishing cycles, and we will identify these two groups.

\begin{prop}\label{FamilyCuspPacket}
 Let~$\OO$ be a pseudo-cubic order with discriminant~$D$ of degree~$d$ relatively prime to the conductor of~$D$. Then each smooth family of genus three stable curves $\mathcal{X} \to \Delta$
 with pseudo-real multiplication by $\OO$ and with the stable curve $X_0$ of geometric genus zero, determines a unique cusp packet in~$\mathcal{C}(\OO)$.
\end{prop}
\begin{proof}
 Since $X_0$ has geometric genus zero, each subgroup of vanishing cycles $V_z \leq H_1(X_z;\ZZ)$ is free Abelian of rank three.
 The action of $\OO$ on the groups $H_1(X_z;\ZZ)$ is compatible with the Gauss-Mannin connection and it preserves the subgroup of vanishing cycles.
 Hence, we obtain pairwise isomorphic short exact sequences of $\OO$-modules \begin{equation}\label{Sequenz} 0 \to L \to H_1(X_z;\ZZ) \to H_1(X_z;\ZZ)/L \to 0 \end{equation} mapping $L$ onto $V_z$,
 where $L$ is as a $\ZZ$-module of rank three.
 Because the intersection pairing on $H_1(X_z;\ZZ) \otimes_\ZZ \QQ$ restricted to $L \otimes_\ZZ \QQ$ is identically zero, real multiplication is self-adjoint with respect to the intersection pairing
 and $F \setminus \{ 0 \}$ has no nilpotent elements, there is no $x \in F \setminus \{ 0 \}$ with $\rho_z(x)|_{L \otimes_\ZZ \QQ} \equiv 0$.
 Together with the fact that $L  \otimes_\ZZ \QQ$ and $F$ both have $\QQ$-dimension three, it follows that there exists an $F$-torsion-free element $\lambda \in L$ and thus $\lambda$ generates
 $L \otimes_\ZZ \QQ$ as an $F$-module. In particular, there is an $\OO$-module isomorphism $L \cong \II$,
 where~$\II$ is a lattice in~$F$ such that the coefficient ring~$\OO(\II)$ of~$\II$ contains~$\OO$.
 Since $H_1(X_z;\ZZ)$ is isomorphic to a lattice in $F^2$ by Proposition~\ref{latticeiso}, it follows from the remark after Corollary~\ref{BijektionCuspPacketsCusps},
 that the sequence \eqref{Sequenz} is a symplectic~$\OO$-extension in the sense of Definition~\ref{Symplextension2}.
 Hence the choice of pseudo-real multiplication by~$\OO$ on the family $\mathcal{X} \to \Delta$ determines a unique cusp packet $(\II,E) \in \mathcal{C}(\OO)$.
\end{proof}
\subsection{Generalized period coordinates}\label{periodnonmatrices}

Here we follow \cite{MartinMatt}, Section~4, and introduce certain meromorphic functions on the closure $\cMM_g(L)$, generalizing the classical period matrices.

\bigskip
First we will introduce some notations and identifications. Let $R$ be a commutative ring with $1$ and $M$ a $R$-module.
Then the symmetric group $S_2 = \{ {\rm id}, \theta \}$ acts on the tensor product $M \otimes_R M$ by $\theta(x \otimes y) := y \otimes x$ for all $x,y \in M$,
so we can define \[ {\Sym}_R(M) := \{ \lambda \in M \otimes_R M : \theta(\lambda) = \lambda \} \] and \[ \bS_R(M) := (M \otimes_R M) / S_2. \]
Furthermore, we denote for a $\ZZ$-module $M$ by \[ \langle \cdot , \cdot \rangle:~ \bS_{\ZZ}(\Hom(M,\ZZ)) \times \Sym(M) \to \ZZ \] the natural pairing defined by
\begin{equation}\label{pairing} \langle [f_1 \otimes f_2],\sum_{i,j}x_i \otimes x_j \rangle := \sum_{i,j}f_1(x_i)f_2(x_j). \end{equation}

We fix a genus $g$ surface $\Sigma_g$ and a Lagrangian subgroup $L \leq H_1(\Sigma_g;\ZZ)$.
We choose a Lagrangian complement of $L$, which is by definition a Lagrangian subgroup $M \leq H_1(\Sigma_g;\ZZ)$, such that \[ H_1(\Sigma_g;\ZZ) = L \oplus M. \]
For every $(X,f) \in \TT(\Sigma_g)$, this splitting yields the two isomorphisms \[ P_L^X:~ \Omega(X) \to \Hom_{\ZZ}(L,\CC),\quad P_M^X:~ \Omega(X) \to \Hom_{\ZZ}(M,\CC), \]
both given by $\omega \mapsto (\gamma \mapsto \smallint_{f(\gamma)}\omega)$.
It follows that we get a holomorphic map
\begin{align*} \TT(\Sigma_g) & \to \Hom_{\CC}(\Hom_{\ZZ}(L,\CC),\Hom_{\ZZ}(M,\CC)),\\ (X,f) & \mapsto P_M^X \circ (P_L^X)^{-1}. \end{align*}
There is the natural isomorphism
\begin{align*} \Hom_{\CC}(L^* \otimes_{\ZZ} \CC,L \otimes_{\ZZ} \CC) & \cong L \otimes_{\ZZ} L \otimes_{\ZZ} \CC \\ \psi & \mapsto \sum_{i,j} c_{ij} \lambda_i \otimes \lambda_j, \end{align*}
where $(c_{ij})_{i,j} \in \CC^{g \times g}$ is the representation matrix of $\psi$ with respect to the basis $(\lambda_1,...,\lambda_g)$ of $L$ and its dual basis $(\lambda_1^*,...,\lambda_g^*)$
of $L^* = \Hom_{\ZZ}(L,\ZZ)$ (this isomorphism does not depend on the choice of the basis).

Using the isomorphism $M^* = \Hom_{\ZZ}(M,\ZZ) \cong L$ given by the intersection form on $\Sigma_g$, we obtain a holomorphic map \[ \Phi:~ \TT(\Sigma_g) \to L \otimes_{\ZZ} L \otimes_{\ZZ} \CC. \]
The embedding $H_1(X;\ZZ) \hookrightarrow \Omega(X)^*$ maps a given $\ZZ$-basis $(\lambda_1,...,\lambda_g)$ of $L$ to a $\CC$-basis of $\Omega(X)^*$, and we denote by $(\omega_1,...,\omega_g)$ its
dual basis of $\Omega(X)$. Moreover, let $(\mu_1,...,\mu_g)$ be its dual basis of $M$ with respect to the intersection form. With these notations and the identification
$\Hom_{\ZZ}(M,\CC) \cong L \otimes_{\ZZ} \CC$, we get \[ (P_M^X \circ (P_L^X)^{-1})(\lambda_j^*) = (\mu \mapsto \smallint_{f(\mu)}\omega_j) = \sum_i(\smallint_{f(\mu_i)}\omega_j)\lambda_i \] and thus
\[ \Phi(X,f) = \sum_{i,j} (\smallint_{f(\mu_i)}\omega_j) \lambda_i \otimes \lambda_j. \]
Therefore, the coefficients are the period coordinates of $\Jac(X)$ with respect to our chosen bases. To be more precise, the period matrix with respect to the symplectic basis
$(f(\lambda_1),...,f(\lambda_g),f(\mu_1),...,f(\mu_g))$ of $H_1(X,\ZZ)$ and the $\CC$-basis $(\omega_1^*,...,\omega_g^*)$ of $\Omega(X)^*$ is $(I_g,Z)$ with \[ Z=(\smallint_{f(\mu_j)}\omega_i)_{i,j}.\]
In particular, $Z^{-1}$ lies in the upper Siegel half space and thus $Z$ is symmetric.
Hence $\Phi$ maps in fact just to the symmetric tensors, \[ \Phi:~ \TT(\Sigma_g) \to \Sym_{\ZZ}(L) \otimes_{\ZZ} \CC. \] We get the dual group homomorphism
\begin{align*} \Phi^*:~ \bS_{\ZZ}(L^*) \otimes_{\ZZ} \CC & \to \Hol(\TT(\Sigma_g),\CC),\\ [f_1 \otimes f_2] & \mapsto ((X,f) \mapsto [f_1 \otimes f_2](\Phi(X,f))), \end{align*}
which is just given by \[ \Phi^*(\lambda_i^* \otimes \lambda_j^*)(X,f) = \smallint_{f(\mu_j)}\omega_i. \]

The map $\Phi^*$ depends on the choice of the Lagrangian complement $M$. Another choice, say $M'$, yields another basis $(\mu_1',...,\mu_g')$ of $M'$ dual to $(\lambda_1,...,\lambda_g)$ with respect
to the intersection form. Since $L$ is Lagrangian, the duality condition gives that the new basis elements must be of the form $\mu_i' = \mu_i + l_i$ for some $l_i \in L$.
It follows that the restriction to $\bS_{\ZZ}(L^*)$ of the new map differs from the restriction of the old one just by values in $\ZZ$. Thus, we get a well-defined homomorphism
\begin{align*} \Psi:~ \bS_{\ZZ}(L^*) & \to \Hol^*(\MM_g(L)),\\ a & \mapsto \exp(2\pi i \Phi^*(a)(\cdot)). \end{align*}

Recall that for any non-zero simple closed curve $\gamma \in L$, we defined $D_{\gamma} \subset \cMM_g(L)$ to be the divisor of stable curves obtained by pinching a curve
homologous to $\gamma$.
Using the notation of the pairing (\ref{pairing}), we have the following theorem.

\begin{theo}[Bainbridge, M\"oller]\label{Martin4.1}
 For each $a \in \bS_{\ZZ}(L^*)$, the function $\Psi(a)$ is meromorphic on $\cMM_g(L)$. For each $\gamma \in L \setminus \{ 0 \}$, the order of vanishing of $\Psi(a)$ along $D_{\gamma}$ is
 \[ {\rm ord}_{D_{\gamma}}(\Psi(a)) = \langle a,\gamma \otimes \gamma \rangle. \]
 The function $\Psi(a)$ is holomorphic and nowhere vanishing along any Lagrangian boundary stratum obtained by pinching a curve homologous to zero.
 If $\mathcal{S}$ is a Lagrangian boundary stratum with \[ \langle a,\gamma \otimes \gamma \rangle \geq 0 \] for all $\gamma \in {\rm Weight}(\mathcal{S})$, then $\Psi(a)$
 is holomorphic on $\mathcal{S}$. If $\langle a,\gamma \otimes \gamma \rangle = 0$ for all $\gamma \in {\rm Weight}(\mathcal{S})$, then $\Psi(a)$ is nowhere vanishing on $\mathcal{S}$.
 Otherwise, if $\langle a,\gamma \otimes \gamma \rangle \geq 0$ for all $\gamma \in {\rm Weight}(\mathcal{S})$ and $\langle a,\gamma_0 \otimes \gamma_0 \rangle > 0$ for at least one $\gamma_0$,
 we have $\Psi(a) \equiv 0$ on $\mathcal{S}$.
\end{theo}
This is Theorem 4.1. in \cite{MartinMatt} and a proof using plumbing coordinates can be found there.
\subsection{Admissibility}\label{admissibility}

Our aim is, analogous to \cite{MartinMatt}, Section~5, to formulate a necessary condition for geometric genus zero stable curves for lying in the boundary of the pseudo-real multiplication locus.

\bigskip
In all that follows, $F = K \oplus \QQ$ is again a pseudo-cubic number field.
As $\QQ$-vector spaces, we can identify $F$ with its dual $\Hom_{\QQ}(F,\QQ)$ via the isomorphism given by the pseudo-trace pairing
\begin{eqnarray*}
 F & \cong & \Hom_{\QQ}(F,\QQ),\\ x & \mapsto & \left(y \mapsto \tr_p(xy) \right) 
\end{eqnarray*}
and the tensor product $F \otimes_{\QQ} F$ with $\Hom_{\QQ}(F,F)$ via the isomorphism
\begin{eqnarray*}
 F \otimes_{\QQ} F & \cong & \Hom_{\QQ}(F,F),\\
 x_1 \otimes x_2 & \mapsto & \left( \psi_{x_1 \otimes x_2} : y \mapsto \tr_p(x_1y) \cdot x_2 \right) 
\end{eqnarray*}
We call a homomorphism $f \in \Hom_{\QQ}(F,F)$ \emph{self-adjoint}, if \[ \tr_p(f(x)y) = \tr_p(xf(y)) \] holds for all $x,y \in F$.
Thus, $\Hom_{\QQ}^+(F,F) \leq \Hom_{\QQ}(F,F)$ defined in Section~\ref{extensionclasses} is the subspace of all self-adjoint homomorphisms.

The tensor product $F \otimes_{\QQ} F$ carries the structure of an $F$-bimodule, and we can define the submodule
\[ \Lambda_1 := \{ \lambda \in F \otimes_{\QQ} F : \forall x \in F : x\lambda = \lambda x \}. \]
Under the identification of $ F \otimes_{\QQ} F$ with $\Hom_{\QQ}(F,F)$ above, ${\Sym}_\QQ(F)$ corresponds to $\Hom_{\QQ}^+(F,F)$ and~$\Lambda_1$ corresponds to $\Hom_F(F,F)$.
In particular, we have $\Lambda_1 \subset \Sym_{\QQ}(F)$.

Using the identification $F \cong F^* = \Hom_{\QQ}(F,\QQ)$, we have a pairing \[ \langle \cdot , \cdot \rangle :~ \bS_{\QQ}(F) \times \Sym_{\QQ}(F) \to \QQ, \]
defined by \[ \langle [\varphi_1 \otimes \varphi_2] , \sum_{i,j}q_{ij}x_i \otimes x_j \rangle := \sum_{i,j}q_{ij}\varphi_1(x_i) \varphi_2(x_j) \] on the simple tensors of
$\bS_{\QQ}(F)$. Since this pairing is perfect, we can identify the dual of $\Sym_{\QQ}(F)$ with $\bS_{\QQ}(F)$.
Finally, we denote by $\Ann(\Lambda_1) \subset \bS_{\QQ}(F)$ the annihilator of $\Lambda_1$ with respect to this pairing.

\bigskip
Now let~$\II$ be a lattice in~$F$. For each $\II$-weighted boundary stratum $\mathcal{S}$ we define
\begin{align*}
 \mathcal{C}(\mathcal{S}) & := \{ a \in \bS_{\QQ}(F) : \forall w \in {\rm Weight}(\bS) : \langle a,w \otimes w \rangle \geq 0 \}\quad {\rm and} \\
 \mathcal{N}(\mathcal{S}) & := \{ a \in \bS_{\QQ}(F) : \forall w \in {\rm Weight}(\bS) : \langle a,w \otimes w \rangle = 0 \}.
\end{align*}

\begin{defi}
 An $\II$-weighted boundary stratum $\mathcal{S}$ is called \emph{admissible} if \[ \mathcal{C}(\mathcal{S}) \cap \Ann(\Lambda_1) \subset \mathcal{N}(\mathcal{S}). \]
\end{defi}

Restriction gives a surjective map of algebraic tori
\[ p:~ \Hom(\mathcal{N}(\mathcal{S}) \cap \bS_{\ZZ}(\II^{\vee}),\CC^*) \to \Hom(\mathcal{N}(\mathcal{S}) \cap \Ann(\Lambda_1) \cap \bS_{\ZZ}(\II^{\vee}),\CC^*). \]
For each non-zero $a \in \mathcal{N}(\mathcal{S}) \cap \bS_{\ZZ}(\II^{\vee})$, the function $\Psi(a)$ is holomorphic and nowhere vanishing on $\mathcal{S}$ by Theorem~\ref{Martin4.1},
hence we get a morphism \[ {\rm CR}:~ \mathcal{S} \to \Hom(\mathcal{N}(\mathcal{S}) \cap \bS_{\ZZ}(\II^{\vee}),\CC^*),\quad X \mapsto \left( a \mapsto \Psi(a)(X) \right) \]
For every $a \in \Ann(\Lambda_1) \cap \bS_{\ZZ}(\II^{\vee})$ and $h \in \Sym_{\QQ}(F)$, we define \[ q(h)(a) := \exp(-2\pi i \langle a,h \rangle). \]
As $q(h)(a)=1$ for $h \in \Lambda_1$ or $h \in \Sym_{\ZZ}(\II)$, we get a well-defined homomorphism
\[ q:~ \Sym_\QQ(F) / (\Lambda_1+\Sym_{\ZZ}(\II)) \to \Hom(\mathcal{N}(\mathcal{S}) \cap \Ann(\Lambda_1) \cap \bS_{\ZZ}(\II^{\vee}),\CC^*). \]
Using the three maps above, every $h \in \Sym_{\QQ}(F) / (\Lambda_1+\Sym_{\ZZ}(\II))$ yields a subvariety of $\mathcal{S}$,
namely \[ \mathcal{S}(h) := {\rm CR}^{-1}(p^{-1}(q(h))), \] consisting of those $X \in \mathcal{S}$,
where $\Psi(\cdotp)(X)$ and $\exp(-2\pi i \langle \cdotp ,h \rangle)$ coincide on $\mathcal{N}(\mathcal{S}) \cap \Ann(\Lambda_1) \cap \bS_{\ZZ}(\II^{\vee})$.

\bigskip
Now we fix a genus three surface $\Sigma_3$. Given a lattice~$\II$ in~$F$, a homomorphism $h \in \Hom_\QQ^+(F,F) \cong \Sym_\QQ(F)$ and a symplectic
$\ZZ$-isomorphism $\eta: \II \oplus \II^\vee \to H_1(\Sigma_3;\ZZ)$, we define \[ \mathcal{R}\TT(\II,h,\eta) \subset \TT(\Sigma_3)\] to be the locus of those marked Riemann
surfaces $(X,f) \in \TT(\Sigma_3)$, such that the Jacobian of~$X$ has pseudo-real multiplication by $\OO_h(\II)$ and the symplectic $\ZZ$-module isomorphism
\[ f_* \circ \eta : (\II \oplus \II^\vee)_h \to H_1(X;\ZZ) \] is $\OO_h(\II)$-linear.
The Lagrangian decomposition induced by $\eta$ determines the holomorphic function \[ \Phi:~ \TT(\Sigma_3) \to \Sym_{\ZZ}(\II) \otimes_{\ZZ} \CC \]
from Section~\ref{periodnonmatrices}, and analogous to \cite{MartinMatt} we get the following proposition.

\begin{lem}
 The locus $\mathcal{R}\TT(\II,h,\eta)$ is the preimage under $\Phi$ of the coset of $\Lambda_1 \otimes_\ZZ \CC$ represented by~$-h$,
 \[ \mathcal{R}\TT(\II,h,\eta) = \Phi^{-1}(\Lambda_1 \otimes_\QQ \CC - h). \]
\end{lem}
\begin{proof}
 Again we will identify $F \otimes_\QQ F$ with $\Hom_\QQ(F,F)$ via the pseudo-trace pairing, so in particular we may regard $\phi := \Phi(X,f)$ from
 Section~\ref{periodnonmatrices} as an element of $\Hom_\CC^+(\II^\vee \otimes_\ZZ \CC,\II \otimes_\ZZ \CC)$.
 Consider the graph \[ {\rm Graph}(\phi) = \{ (\phi(\mu),\mu) : \mu \in \II^\vee \otimes_\ZZ \CC \} \subset (\II \otimes_\ZZ \CC) \oplus (\II^\vee \otimes_\ZZ \CC). \]
 By the Hodge decomposition, the $\CC$-linear map \[ {\rm Graph}(\phi) \to \Hom_\CC(\Omega(X);\CC), \] defined as
 \[ (\phi(\mu),\mu) \mapsto (\omega \mapsto \smallint_\mu\omega) = (\omega \mapsto \smallint_{\phi(\mu)}\omega), \] is an isomorphism.
 The last identity follows from the construction of $\Phi$, and therefore $\phi(\mu) \in \II \otimes_\ZZ \CC$ is the same element as $\mu \in \II^\vee \otimes_\ZZ \CC$ in $H_1(X;\RR)$.
 For every $x \in F$ we define $f_x \in \End_\CC((\II \otimes_\ZZ \CC) \oplus (\II^\vee \otimes_\ZZ \CC))$ by \[ f_x(\lambda,\mu) := (x\lambda+[M_x,h](\mu),x\mu), \]
 which is just the $\CC$-linear continuation of the action of $F$ on \[ H_1(X;\RR) \cong (\II \otimes_\ZZ \RR) \oplus (\II^\vee \otimes_\ZZ \RR). \] The marked surface $(X,f)$
 is in $\mathcal{R}\TT(\II,h,\eta)$ if and only if the $\OO_h(\II)$-module structure on $H_1(X;\ZZ)$ induced by $f_* \circ \eta$ extends to pseudo-real multiplication on $\Jac(X)$.
 This happens if and only if for each $x \in \OO_h(\II)$ the $\CC$-linear map $f_x$ preserves ${\rm Graph}(\phi)$, i.e. \[ f_x(\phi(\mu),\mu) = (\phi(x\mu),x\mu) \] for all $\mu \in \II^\vee$.
 This holds if and only if $\phi(x\mu) = x\phi(\mu)+[M_x,h](\mu)$, or equivalently $(\phi+h)(x\mu) = x(\phi+h)(\mu)$ for all $\mu \in \II^\vee$.
 Finally, this condition is equivalent to \[ \Phi(X,f)+h \in \Lambda_1 \otimes_\QQ \CC. \]
\end{proof}

Given a lattice $\II$ in $F$ and a homomorphism $h \in \Hom_\QQ^+(F,F) \cong \Sym_\QQ(F)$, we define \[ \mathcal{R}\MM(\II,h) \subset \MM_3(\II) \] to be the locus of pairs $(X,\varrho)$
of a genus three Riemann surface $X$ and a Lagrangian marking $\varrho: \II \hookrightarrow H_1(X;\ZZ)$, such that $\Jac(X)$ has pseudo-real multiplication by a pseudo-cubic order
$\OO \subseteq \OO_h(\II)$, the marking $\varrho$ is $\OO$-linear and the symplectic $(\OO,\II)$-extension \[ 0 \to \II \to H_1(X;\ZZ) \to \II^\vee \to 0 \] induced by $\varrho$ is isomorphic to
$E_h(\II)$ (in particular $\OO = \OO_h(\II)$).

For every $(X,\varrho) \in \mathcal{R}\MM(\II,h)$, we can choose a symplectic $\ZZ$-isomorphism $\eta: \II \oplus \II^\vee \to H_1(\Sigma_3;\ZZ)$ and a marking $f: \Sigma_3 \to X$,
such that $(X,\varrho)$ can be represented by some $(X,f) \in \mathcal{R}\TT(\II,h,\eta)$  in the Teichm\"uller space $\TT(\Sigma_3)$.
\begin{cor}\label{Gleichung}
 Let $\II$ be a lattice in $F$ and $h \in \Sym_\QQ(F)$.
 \begin{enumerate}
  \item  For every $a \in \Ann(\Lambda_1) \cap \bS_{\ZZ}(\II^{\vee})$ and $h \in \Sym_{\QQ}(F)$ we have \[ \Psi(a)|_{\mathcal{R}\MM(\II,h)} \equiv \exp(-2\pi i \langle a,h \rangle) = q(h)(a). \]
  \item If $\mathcal{S} \subset \cMM_3(\II)$ is an $\II$-weighted boundary stratum with $\mathcal{S} \cap \overline{\mathcal{R}\MM(\II,h)} \not= \emptyset$,
        then $\mathcal{S}$ is admissible and we have $\mathcal{S} \cap \overline{\mathcal{R}\MM(\II,h)} \subseteq \mathcal{S}(h)$.
 \end{enumerate}
\end{cor}
\begin{proof}
 The first statement follows immediately from the last proposition and the definitions of $\Phi$ and $\Psi$.
 To prove the second one, let $\mathcal{S}$ be an $\II$-weighted boundary stratum with $\mathcal{S} \cap \overline{\mathcal{R}\MM(\II,h)} \not= \emptyset$.
 If $\mathcal{S}$ were non-admissible, there would be some $a \in \Ann(\Lambda_1) \cap \bS_{\ZZ}(\II^{\vee})$ with $\langle a, w \otimes w \rangle \geq 0$ for all $w \in {\rm Weight}(\mathcal{S})$
 and $\langle a, w_0 \otimes w_0 \rangle > 0$ for some specific $w_0 \in {\rm Weight}(\mathcal{S})$.
 But then by Theorem~\ref{Martin4.1}, we would have $\Psi(a)|_{\mathcal{S}} \equiv 0$, hence a contradiction as $\Psi(a) \equiv q(h)(a) \not= 0$ on $\overline{\mathcal{R}\MM(\II,h)}$.
\end{proof}

Let $\widetilde \MM_g \subset \cMM_g$ be the locus of stable curves with compact Jacobians. For $g \in \{ 2,3 \}$, the Torelli-map $t: \widetilde \MM_g \to \mathcal{A}_g$ is surjective.
We denote by $\RA_\OO \subset \mathcal{A}_3$ the locus of principally polarized complex Abelian varieties with pseudo-real multiplication by a pseudo-cubic order $\OO$,
hence $\widetilde \RM_\OO = t^{-1}(\RA_\OO) \subset \cMM_3$. Now fix one of the two real quadratic pseudo-embeddings $\iota: F \hookrightarrow \RR$.
By Corollary~\ref{eigenrealmultmap}, we have a surjective map \[ \mathbb{P}\Omega^\iota\widetilde{\RM}_\OO \to X_\OO,\quad (X,[\omega]) \mapsto (\Jac(X),\rho), \] that is injective on
$\mathbb{P}\Omega^\iota{\RM}_\OO$.
It is a consequence of~\cite{Namikawa2}, Corollary~(18.9), that this map extends continuously to \[ \mathbb{P}\overline{\Omega^\iota\RM_\OO} \to \overline{X}_\OO^{BB}. \]
If the geometric genus of a stable curve $X$ is zero, then $(X,[\omega])$ is mapped to a cusp of $\overline{X}_\OO^{BB}$.
If the geometric genus is one (respectively two), then it is mapped to an element of $\overline{X}_\OO^{BB}$ with representatives in $\PK \times \HH$ (respectively $\HH^2 \times \PQ$).
Each cusp of $\overline{X}_\OO^{BB}$ corresponds to a unique cusp packet~$(\II,E)$ in~$\mathcal{C}(\OO)$, and we call the preimage of $(\II,E)$ in $\mathbb{P}\overline{\Omega^\iota\RM_\OO}$
the \emph{cusp of $\mathbb{P}\overline{\Omega^\iota\RM_\OO}$ associated to $(\II,E)$}.

\begin{theo}\label{admissible}
 Let $\OO \subset F$ be a pseudo-cubic order of degree relatively prime to the conductor of $D$, and let $(\II,E_h(\II)) \in \mathcal{C}(\OO)$ be a cusp packet for $\OO$.
 Furthermore, let $\iota: F \hookrightarrow \RR$ be one of the two real quadratic pseudo-embeddings.
 Then the closure of the cusp of $\mathbb{P}\overline{\Omega^\iota\RM_\OO}$ associated to $(\II,E_h(\II))$ is contained in the union of all $p_\iota(\mathcal{S}(h))$, the $\iota$-pseudo-embeddings
 of the subvarieties~$\mathcal{S}(h)$ of $\mathcal{S}$, where~$\mathcal{S}$ runs over all admissible $\II$-weighted boundary strata.
\end{theo}
\begin{proof}
 Let $(X,[\omega])$ be in the closure of the cusp of $\mathbb{P}\overline{\Omega^\iota\RM_\OO}$ associated to $(\II,E_h(\II))$.
 Since the locus $\mathbb{P}\overline{\Omega^\iota\RM_\OO}$ is a variety, we can choose a holomorphic map \[ \Delta \to \mathbb{P}\overline{\Omega^\iota\RM_\OO},\quad z \mapsto (X_z,[\omega_z]) \]
 with $(X_0,\omega_0)=(X,\omega)$ and which sends $\Delta^*$ to $\mathbb{P}\Omega^\iota\RM_\OO$. Let $f$ be the composition of this map with the forgetful map to $\cMM_3$.
 By what we have said in Section~\ref{pseudorealmultfamilies}, the map $f$ can be chosen such that it arises from a smooth family $\mathcal{X} \to \Delta$ with pseudo-real multiplication by~$\OO$.
 In the proof of Proposition~\ref{FamilyCuspPacket}, we have seen that $f$ factorizes through a map $\Delta \to \cMM_3(\II)$, mapping $\Delta^*$ to $\RM(\II,h)$ and~$0$ to some $\II$-weighted boundary
 stratum~$\mathcal{S}$ with $\mathcal{S} \cap \overline{\mathcal{R}\MM(\II,h)} \not= \emptyset$.
 Thus by Corollary~\ref{Gleichung}, the stratum~$\mathcal{S}$ is admissible and $(X_0,\varrho_0)$ lies in~$\mathcal{S}(h)$.

 It remains to show that the residues of the stable form~$\omega$ are common multiples of the images under~$\iota$ of the weights of the corresponding cusps.
 For every~$z \in \Delta^*$, the Lagrangian marking~$\varrho_z$ maps~$\II$ onto the group of vanishing cycles of~$X_z$. Therefore, for each vanishing curve $\lambda \in \II$, we have
 \[ \iota(\lambda)\smallint_{(1,1)}\omega_z = \smallint_{\varrho_z(\lambda)}\omega_z \equiv r(\lambda) , \] where $r(\lambda)$ is the residuum of~$\omega$ at the cusp represented by $\lambda$,
 and the statement follows.
\end{proof}
What we have just proven, is a refinement of the following corollary.
\begin{cor}\label{admissible1}
 Let $\OO$ be a pseudo-cubic order of degree relatively prime to the conductor of $D$.
 Then each geometric genus zero stable curve $X \in \overline{\mathcal{R}\MM_\OO}$ lies in the image under the forgetful map of some admissible $\II$-weighted boundary stratum
 $\mathcal{S} \subset \cMM_3(\II)$.
\end{cor}
\subsection{A reformulation of admissibility}

For a totally real cubic field, Bainbridge and M\"oller constructed in \cite{MartinMatt}, Section~6, a quadratic map $Q$ from the field to itself, which allows to test explicitly if a weighted boundary
stratum is admissible or not, by checking a geometric condition for the $Q$-values of the weights of the stratum.
In our pseudo-cubic case, this map map turns out to be a little more unwieldy, as it just maps to $\QQ^3$. However, we can construct it in such a way, that we get the same geometric condition.

\bigskip
Recall that we identified $F$ with $F^*=\Hom_\QQ(F,\QQ)$ using the pseudo-trace pairing and that we defined a perfect pairing on \[ \bS_\QQ(F) \times \Sym_\QQ(F) \to \QQ \] by
\[ \langle [\varphi_1 \otimes \varphi_2] , \sum_{i,j}q_{ij}x_i \otimes x_j \rangle = \sum_{i,j}q_{ij}\varphi_1(x_i) \varphi_2(x_j). \]
There is another perfect pairing \[ \langle \cdot,\cdot \rangle:~ (F \otimes_\QQ F) \times (F \otimes_\QQ F) \to \QQ \]
defined by \[ \langle x_1 \otimes x_2,y_1 \otimes y_2 \rangle := \tr_p(x_1y_1) \tr_p(x_2y_2) \] on the simple tensors.
The latter is in fact also symmetric and positive definite, and its restriction induces a symmetric and positive definite pairing \[ \Sym_\QQ(F) \times \Sym_\QQ(F) \to \QQ. \]
Therefore, we get isomorphisms \[ \bS_\QQ(F) \cong \Hom_{\QQ}(\Sym_\QQ(F),\QQ) \cong \Sym_\QQ(F). \]
These isomorphisms identify $\Ann(\Lambda_1) \subset \bS_\QQ(F)$ with $\Lambda_1^\perp \subset \Sym_\QQ(F)$.
Since $\Lambda_1$ is isomorphic to $\Hom_F(F,F)$, it has $\QQ$-dimension three, and thus \[ \dim_\QQ(\Lambda_1^\perp) = \dim_\QQ(\Sym_\QQ(F))-\dim_\QQ(\Lambda_1) = 3. \]
By using suitable coordinates we can identify $\Lambda_1^\perp$ with $\QQ^3$ and state a reformulation of admissibility.
\begin{theo}\label{Q}
 Let $Q$ be the function defined by
 \[ Q:~ F \to \QQ^3,\quad  (x,q) \mapsto \begin{pmatrix} \mathcal{N}(\sqrt{D}x) \\ \tr(x)q \\ \tr(\sqrt{D}x) q \end{pmatrix} \]
 and let $\langle \cdot , \cdot \rangle_s$ be the standard scalar product in $\QQ^3$.
 Then an $\mathcal{I}$-weighted boundary stratum $\mathcal{S}$ is admissible if and only if for every  $v \in \QQ^3$ the following condition holds.
 If we have \[ \langle v,Q(w) \rangle_s \geq 0 \] for all $w \in {\rm Weight}(\mathcal{S})$, then we also have \[ \langle v,Q(w) \rangle_s = 0 \] for all $w \in {\rm Weight}(\mathcal{S})$.

 Equivalently, an $\mathcal{I}$-weighted boundary stratum $\mathcal{S}$ is admissible if and only if the set \[ Q(\mathcal{S}) := \{ Q(w) : w \in {\rm Weight}(\mathcal{S}) \} \]
 is not contained in a closed half-space of its $\QQ$-span, and this holds if and only if~$0$ lies in the interior of the convex hull of~$Q(\mathcal{S})$.
\end{theo}
\begin{proof}
 A $\QQ$-basis of $\Lambda_1$ is given by $\mathcal{B}_1 = (\lambda_1,\lambda_2,\lambda_3)$ with
 \begin{align*}
  \lambda_1 & = (1,0) \otimes (1,0) + \tfrac{1}{D} [(\sqrt{D},0) \otimes (\sqrt{D},0)], \\
  \lambda_2 & = (\sqrt{D},0) \otimes (1,0) + (1,0) \otimes (\sqrt{D},0), \\
  \lambda_3 & = (0,1) \otimes (0,1).
 \end{align*}
 A $\QQ$-basis of the orthogonal complement $\Lambda_1^\perp$ is given by $\mathcal{B}_\perp = (\mu_1,\mu_2,\mu_3)$ with
 \begin{align*}
  \mu_1 & = \tfrac{1}{4}((\sqrt{D},0) \otimes (\sqrt{D},0) - D((1,0) \otimes (1,0))), \\
  \mu_2 & = \tfrac{1}{2} ((1,0) \otimes (0,1) + (0,1) \otimes (1,0)), \\
  \mu_3 & = \tfrac{1}{2} ((\sqrt{D},0) \otimes (0,1) + (0,1) \otimes (\sqrt{D},0)).
 \end{align*}
 We want to see what admissibility means in terms of our new spaces. Let~$a \in \Lambda_1^\perp$ and let $w \in F$.
 We write $a = \sum_{i=1}^3 q_i \mu_i$ with $q=(q_1,q_2,q_3) \in \QQ^3$ and we write $w = (x,q)$ with $x =r+s\sqrt{D}$, $q,r,s \in \QQ$.
 An elementary computation gives
 \begin{align*}
  \langle a, w \otimes w \rangle & = q_1 (s^2D^2 - r^2D) + q_2 r q + q_3 sD q \\
  & = \left\langle \begin{pmatrix} q_1 \\ q_2 \\ q_3 \end{pmatrix} ,
  \begin{pmatrix} \mathcal{N}(\sqrt{D}x) \\ \tr(x) q \\ \tr(\sqrt{D}x) q \end{pmatrix} \right\rangle_s \\
  & = \langle q,Q(w) \rangle_s
 \end{align*}
 and this finishes the proof.
\end{proof}

\begin{rem}
 In \cite{MartinMatt}, the pseudo-cubic field $F$ is replaced by a totally real number field. The Galois closure of $F$ is denoted by $K$.
 The orthogonal decomposition \[ \Sym_\QQ(F) = \Lambda_1 \oplus \Lambda_1^\perp \] is induced by the decomposition
 \[ K \otimes_\QQ K = \bigoplus_{\sigma \in {\rm Gal}(K/\QQ)}\Lambda_\sigma \] with
 \[ \Lambda_\sigma := \{ \lambda \in K \otimes_\QQ K : x \cdot \lambda = \lambda \cdot \sigma(x) \mbox{ for all } x \in K \}. \]
 This does not work in the pseudo-cubic case, because for our non-trivial automorphism $\sigma : \QQ(\sqrt{D}) \oplus \QQ \to \QQ(\sqrt{D}) \oplus \QQ$, one can easy compute
 that \[ \Lambda_1 \cap \Lambda_\sigma = \langle (0,1) \otimes (0,1) \rangle_\QQ, \] hence the intersection is non-trivial.
 The reason is that $\sigma$ acts on $\{ 0 \} \oplus \QQ$ as the identity.
\end{rem}
\subsection{Period coordinates in terms of cross-ratios}

Each $\II$-weighted boundary stratum $\mathcal{S}$ can be described in projective coordinates.
Analogously to \cite{MartinMatt}, we construct explicit equations in terms of cross-ratios of these coordinates which cut out the subvarieties $\mathcal{S}(h)$.
\\[1em]
In Section~\ref{periodnonmatrices} we introduced the group homomorphism
\[ \Psi:~ \bS_\ZZ(\II^\vee) \to \Mer(\cMM_3(\II)),\quad a \mapsto \exp(2\pi i \Phi^*(a)(\cdot)), \]
where $\Mer(\cMM_3(\II))$ denotes the multiplicative group of meromorphic functions on $\cMM_3(\II)$.
Furthermore, in Section~\ref{admissibility} we defined on a given $\II$-weighted boundary stratum $\mathcal{S}$ the map
\[ {\rm CR}:~ \mathcal{S} \to \Hom(\mathcal{N}(\mathcal{S}) \cap \bS_{\ZZ}(\II^{\vee}),\CC^*),\quad X \mapsto \left( a \mapsto \Psi(a)(X) \right). \]
For each $h \in \Sym_{\QQ}(F) / (\Lambda_1+\Sym_{\ZZ}(\II))$, the subvariety $\mathcal{S}(h) \subset \mathcal{S}$ was defined to be the locus of those $X \in \mathcal{S}$ satisfying the equations
\[ \Psi(a)(X) = \exp(-2\pi i \langle a,h \rangle) \] for all $a \in \mathcal{N}(\mathcal{S}) \cap \Ann(\Lambda_1) \cap \bS_{\ZZ}(\II^{\vee})$.
We want to describe these functions in terms of cross-ratios for certain types of Lagrangian boundary strata to give a more explicit version of Corollary~\ref{Gleichung}.

\bigskip
The type of a boundary stratum $\mathcal{S}$ is given by the combinatoric of its dual graph defined as follows.
Take an arbitrary $X \in \mathcal{S}$. Then the dual graph~$\Gamma(\mathcal{S})$ of~$\mathcal{S}$ consists of a vertex for each irreducible component of $X$ and an edge for each node of $X$.
The endpoints of an edge coming from a node $n \in N(X)$ are just the vertices coming from the irreducible components connected by $n$ (particularly $\Gamma(\mathcal{S})$ may have loops as edges).

We are interested in strata of geometric genus zero stable curves. The dual graphs of these strata must have Euler characteristic $-2$ (roughly spoken, one needs three nodes to
kill the holes and one node for producing an additional irreducible component) and the degree of each vertex is at least three.
It follows that such a dual graph has at most four vertices and that there exist exactly 15 types of strata.
\\[1em]
{\bf The irreducible stratum.}
First we will discuss the irreducible strata, i.e. the strata where the dual graph is the unique graph with just one vertex (and thus three loops as edges), as can be seen in Figure~\ref{Irreduzibel}.
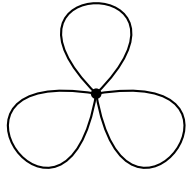
\begin{figure}[ht]
\centering
\begin{tikzpicture}[scale=2]
 \coordinate (A) at (0,0);
 \fill (A) circle (1pt);
 \draw[thick] (A) .. controls (-0.8,0.8) and (0.8,0.8) .. (A);
 \draw[thick] (A) .. controls (-1.2,0.2) and (-0.2,-1.2) .. (A);
 \draw[thick] (A) .. controls (1.2,0.2) and (0.2,-1.2) .. (A);
\end{tikzpicture}
\setlength{\abovecaptionskip}{-1cm}
\caption{The dual graph of an irreducible stable curve of geometric genus zero}\label{Irreduzibel}
\end{figure}

Given a $\ZZ$-basis $r=(r_1,r_2,r_3)$ of $\II$, let $\mathcal{S}_r \subset \partial\cMM_3(\II)$ be the $\II$-weighted irreducible boundary stratum with weights $r_1,r_2,r_3$.
Every $\II$-weighted stable curve in $\mathcal{S}_r$ can be realized in the following way. Choose a tuple of six distinct points $p = (p_1,p_2,p_3,q_1,q_2,q_3) \in \PC^6$
and identify each~$p_i$ with~$q_i$. Then define an $\II$-weighting of the resulting stable curve $X_p$ by setting $w(p_i)=r_i$ and $w(q_i)=-r_i$.
Two such $\II$-weighted stable curves represent the same element in~$\cMM_3(\II)$ if and only if there is an isomorphism between~$X_p$ and~$X_{p'}$ preserving the ordered sextuples of chosen points.
Thus $\mathcal{S}_r$ is isomorphic to $\MM_{0,6}$, the moduli space of six ordered marked points on~$\PC$.
Our aim is to determine $\Psi(a)(X_p)$ for arbitrary $X_p \in \mathcal{S}_r$ and $a \in \bS_\ZZ(\II^\vee)$, and we do this in terms of a basis for $\bS_\ZZ(\II^\vee)$.
\begin{lem}\label{N(S)-basis}
 Let $r=(r_1,r_2,r_3)$ be a $\ZZ$-basis of $\II$ and let $s=(s_1,s_2,s_3)$ be its dual basis of $\II^\vee$ with respect to the pseudo-trace pairing.
 If $\mathcal{S}$ is any $\II$-weighted boundary stratum with weigths \[ {\rm Weight}(\mathcal{S}) = \{ \pm r_1, \pm r_2, \pm r_3 \}, \] then $(s_1 \otimes s_2,s_1 \otimes s_3,s_2 \otimes s_3)$
 is a $\QQ$-basis of $\mathcal{N}(\mathcal{S})$ and thus a $\ZZ$-basis of the intersection $\mathcal{N}(\mathcal{S}) \cap \bS_\ZZ(\II^\vee)$.
\end{lem}
\begin{proof}
 A $\QQ$-basis of $\bS_\QQ(F)$ is given by the simple tensors $s_i \otimes s_j$ with $1 \leq i \leq j \leq 3$ and for any $a = \sum_{i \leq j}a_{ij}  s_i \otimes s_j$ we have
 $\langle a,r_i \otimes r_i \rangle = a_{ii}$.
\end{proof}
For $z_1,z_2,z_3,z_4 \in \PC$ we denote by \[ (z_1,z_2;z_3,z_4) = \frac{(z_1-z_3)(z_2-z_4)}{(z_1-z_4)(z_2-z_3)} \in \PC \] their cross-ratio
and for $p = (p_1,p_2,p_3,q_1,q_2,q_3) \in \PC^6$ with pairwise distinct entries we set \[ p_{jk} := (p_j,q_j;q_k,p_k) \in \CC \setminus \{ 0,1 \} \] for all $j \not= k$.
\begin{lem}\label{Psi1}
 Let $r=(r_1,r_2,r_3)$ be a $\ZZ$-basis of $\II$ and let $s=(s_1,s_2,s_3)$ be its dual basis of $\II^\vee$ with respect to the pseudo-trace pairing.
 Then we have \[ \Psi(s_j \otimes s_k)(X_p) = p_{jk} \] for all $X_p \in \mathcal{S}_r$ and $j \not= k$.
\end{lem}
\begin{proof}
 M\"obius transformations preserve cross-ratios, so without loss of generality we may assume that $p_j=0$, $q_j=\infty$ and~$q_k=1$.
 By the description in Section~\ref{periodnonmatrices} of~$\Phi^*$ in terms of period coordinates we see that $\Psi(s_j \otimes s_k)(X_p)$ is computed in the following way.

 For each cusp~$c$, let $\gamma(c)$ be a positively oriented, simple closed curve winding once around~$c$.
 Then let~$\omega_j$ be the stable Abelian differential with $\smallint_{\gamma(p_i)}\omega_j = \delta_{ij}$ and $\smallint_{\gamma(q_i)}\omega_j = -\delta_{ij}$ for every~$i$.
 Since~$p_j=0$ and~$q_j=\infty$, we get $\omega_j = \tfrac{1}{2\pi i z}dz$. Now we have to integrate~$\omega_j$ over a path~$\gamma_k$ starting at~$p_k$ and ending at $q_k=1$.
 This yields \[ \Psi(s_j \otimes s_k)(X_p) = \exp(2\pi i\smallint_{\gamma_k}\omega_j) = \exp(\smallint_{\gamma_k}\tfrac{1}{z}dz) = p_k^{-1} = p_{jk}. \]
\end{proof}

\begin{theo}\label{CRequation}
 Let $r=(r_1,r_2,r_3)$ be a $\ZZ$-basis of $\II$, let $s=(s_1,s_2,s_3)$ be its dual basis of $\II^\vee$ with respect to the pseudo-trace pairing and
 let~$h$ be an element of $\Sym_{\QQ}(F) / (\Lambda_1+\Sym_{\ZZ}(\II))$.
 Writing \[ h = \sum_{i,j = 1}^3 b_{ij}(r_i \otimes r_j)\] with $b_{ij} \in \QQ$ and $b_{ij}=b_{ji}$, and identifying the $\II$-weighted boundary stratum $\mathcal{S}_r$ with $\MM_{0,6}$ as above,
 the subvariety $\mathcal{S}_r(h) \subset \mathcal{S}_r$ is cut out by the equations \[ p_{23}^{a_1} \cdot p_{13}^{a_2} \cdot p_{12}^{a_3} = \exp(-2\pi i (a_1b_{23}+a_2b_{13}+a_3b_{12})), \]
 where $(a_1,a_2,a_3)$ runs over the integral solutions of \begin{equation}\label{Ann} a_1s_2s_3+a_2s_1s_3+a_3s_1s_2=0. \end{equation}
\end{theo}
\begin{proof}
 By Lemma~\ref{N(S)-basis}, each $a \in \mathcal{N}(\mathcal{S}_r) \cap \bS_\ZZ(\II^\vee)$ can be written as \[ a = a_1 s_2 \otimes s_3+a_2s_1 \otimes s_3+a_3s_1 \otimes s_2 \]
 with $a_i \in \ZZ$. Then the term $a_1b_{23}+a_2b_{13}+a_3b_{12}$ is just $\langle a,h \rangle$, so by the last lemma it remains to show that Equation~(\ref{Ann}) describes
 the intersection of $\mathcal{N}(\mathcal{S}_r) \cap \bS_\ZZ(\II^\vee)$ with $\Ann(\Lambda_1)$.
 We define $\lambda = r_1 \otimes s_1 + r_2 \otimes s_2 + r_3 \otimes s_3$. Under the identification $F \otimes_\QQ F \cong \Hom_\QQ(F,F)$ this element corresponds to the
 identity, hence it generates $\Lambda_1$ as a $F$-module. Thus $a$ lies in $\Ann(\Lambda_1)$ if and only if $\langle a,x\lambda \rangle =0$ for all $x \in F$.
 Since \[ \langle a,x\lambda \rangle = \langle a,\lambda x \rangle = \tr_p(x(a_1s_2s_3+a_2s_1s_3+a_3s_1s_2)), \] this holds if and only if
 \[ a_1s_2s_3+a_2s_1s_3+a_3s_1s_2 = 0. \]
\end{proof}
In particular, we have shown that $\mathcal{S}_r(h) \subset \mathcal{S}_r$ is cut out by~$r$ cross-ratio equations,
where $r = \rk_\ZZ\left( \mathcal{N}(\mathcal{S}_r) \cap \bS_\ZZ(\II^\vee) \cap \Ann(\Lambda_1) \right)$ is also the rank of the $\ZZ$-module of the solutions of Equation~(\ref{Ann}).

\begin{ex}\label{ranktwo}
 In \cite{MartinMatt} it is shown that for a cubic number field and an admissible stratum $\mathcal{S}_r$, the $\ZZ$-module $\mathcal{N}(\mathcal{S}_r) \cap \bS_\ZZ(\II^\vee) \cap \Ann(\Lambda_1)$
 has rank one. Here in the pseudo-cubic case, there is more than one possibility for this rank.
 \begin{enumerate}
  \item Set $r_1=(1,0)$, $r_2=(\sqrt{D},0)$ and $r_3=(0,1)$. Then $Q(r_1)=(-D,0,0)^t$, $Q(r_2)=(D^2,0,0)^t$ and $Q(r_3)=0$, thus $S_r$ is admissible.
        Moreover, we have $s_1s_2=(\tfrac{1}{4\sqrt{D}},0)$ and $s_2s_3=s_3s_1=0$, hence \[ \rk_\ZZ(\mathcal{N}(\mathcal{S}_r) \cap \bS_\ZZ(\II^\vee) \cap \Ann(\Lambda_1))=1.\]
  \item Set $r_1=(1,\tfrac{1}{2}D)$, $r_2=(1,-\tfrac{1}{2}D)$ and $r_3=(\sqrt{D},0)$. Then we get $Q(r_1)=(-D,D,0)^t$, $Q(r_2)=(-D,-D,0)^t$ and $Q(r_3)=(D^2,0,0)$, thus $S_r$ is admissible.
        Moreover, we have $s_1s_2=(\tfrac{1}{16},-\tfrac{1}{D^2})$ and $s_2s_3 = s_3s_1 = (\tfrac{1}{8D}\sqrt{D},0)$,
        hence \[ \rk_\ZZ(\mathcal{N}(\mathcal{S}_r) \cap \bS_\ZZ(\II^\vee) \cap \Ann(\Lambda_1))=2.\]
 \end{enumerate}
\end{ex}

\noindent
{\bf A non-irreducible stratum.}
Another type of stratum is the one where the dual graph has two vertices, one loop at each vertex and two edges connecting the two vertices, as can bee seen in Figure~\ref{NichtIrreduzibel}.
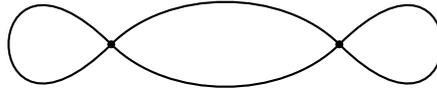
\begin{figure}[ht]
\centering
\begin{tikzpicture}[scale=1.5]
 \coordinate (A) at (0,0); \coordinate (B) at (2,0);
 \fill (A) circle (1pt); \fill (B) circle (1pt);
 \draw[thick] (A) .. controls (-1.2,1.2) and (-1.2,-1.2) .. (A);
 \draw[thick] (A) .. controls (0.5,0.5) and (1.5,0.5) .. (B);
 \draw[thick] (A) .. controls (0.5,-0.5) and (1.5,-0.5) .. (B);
 \draw[thick] (B) .. controls (3.2,1.2) and (3.2,-1.2) .. (B);
\end{tikzpicture}
\setlength{\abovecaptionskip}{-1cm}
\caption{The dual graph of a certain non-irreducible stable curve}\label{NichtIrreduzibel}
\end{figure}

For a given $\ZZ$-basis $(r_1,r_2,r_3)$ of $\II$, let $\bar{r} = ((r_1,r_3),(r_2,r_3))$ and let $\mathcal{S}_{\bar{r}}$ be the $\II$-weighted boundary stratum of this type
with weights $\pm r_1,\pm r_3$ on one irreducible component and weights $\pm r_2,\pm r_3$ on the other.
Every $\II$-weighted stable curve in $\mathcal{S}_{\bar{r}}$ can be realized in the following way.
Take two copies of $\PC$, then choose a tuple $(p_1,q_1,p_3^+,q_3^-)$ of four distinct points in the first copy and a tuple $(p_2,q_2,q_3^+,p_3^-)$ of four distinct points in the second copy.
Now identify $p_1$ with $q_1$, $p_2$ with $q_2$, $p_3^+$ with $q_3^+$ and $p_3^-$ with $q_3^-$.
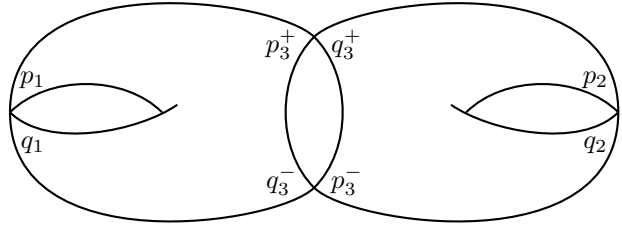
\begin{figure}[ht]
\centering
\begin{tikzpicture}[scale=1]
 \coordinate (A) at (0,0); \coordinate (B) at (8,0); \coordinate (C) at (4,1); \coordinate (D) at (4,-1);
 \draw[thick] (A) .. controls (0,2) and (3.5,1.5) .. (C);
 \draw[thick] (A) .. controls (0,-2) and (3.5,-1.5) .. (D);
 \draw[thick] (B) .. controls (8,2) and (4.5,1.5) .. (C);
 \draw[thick] (B) .. controls (8,-2) and (4.5,-1.5) .. (D);
 \draw[thick] (A) .. controls (0.5,0.5) and (1.5,0.5) .. (2,0);
 \draw[thick] (A) .. controls (0.5,-0.5) and (1.7,-0.25) .. (2.2,0.1);
 \draw[thick] (B) .. controls (7.5,0.5) and (6.5,0.5) .. (6,0);
 \draw[thick] (B) .. controls (7.5,-0.5) and (6.3,-0.25) .. (5.8,0.1);
 \draw[thick] (C) .. controls (3.5,0.5) and (3.5,-0.5) .. (D);
 \draw[thick] (C) .. controls (4.5,0.5) and (4.5,-0.5) .. (D);
 \coordinate[label=above:$p_1$] (p1) at (0.3,0.2);
 \coordinate[label=below:$q_1$] (q1) at (0.3,-0.2);
 \coordinate[label=above:$p_2$] (p2) at (7.7,0.2);
 \coordinate[label=below:$q_2$] (q1) at (7.7,-0.2);
 \coordinate[label=left:$p_3^+$] (p3+) at (3.9,0.9);
 \coordinate[label=left:$q_3^-$] (q3-) at (3.9,-0.9);
 \coordinate[label=right:$q_3^+$] (q3+) at (4.1,0.9);
 \coordinate[label=right:$p_3^-$] (p3-) at (4.1,-0.9);
\end{tikzpicture}
\setlength{\abovecaptionskip}{0cm}
\caption{The stable curve $X_{\bar{p}}$}\label{StableCurve}
\end{figure}

Then define an $\II$-weighting on the resulting stable curve $X_{\bar{p}}$ by $w(p_i^{(\pm)})=r_i$ and $w(q_i^{(\pm)})=-r_i$.
Therefore, the stratum $\mathcal{S}_{\bar{r}}$ is isomorphic to $\MM_{0,4} \times \MM_{0,4}$.
Analogous to the irreducible stratum we define the cross-ratios \[ \bar{p}_{13} = (p_1,q_1;q_3^-,p_3^+),\quad \bar{p}_{23} = (p_2,q_2;q_3^+,p_3^-) \]
and get the following lemma.
\begin{lem}\label{Psi2}
 Let $r=(r_1,r_2,r_3)$ be a $\ZZ$-basis of $\II$ and let $s=(s_1,s_2,s_3)$ be its dual basis of $\II^\vee$ with respect to the pseudo-trace pairing.
 Then we have \begin{align*}  \Psi(s_1 \otimes s_2)(X_{\bar{p}}) & = 1,\\ \Psi(s_1 \otimes s_3)(X_{\bar{p}}) & = \bar{p}_{13}, \\
 \Psi(s_2 \otimes s_3)(X_{\bar{p}}) & = \bar{p}_{23} \end{align*} for all $X_{\bar{p}} \in \mathcal{S}_{\bar{r}}$.
\end{lem}
\begin{proof}
 To compute $\Psi(s_j \otimes s_k)$ with $j \in \{ 1,2 \}$, we have to integrate a stable Abelian differential $\omega_j$ that has two simple poles at $p_j,q_j$
 and is holomorphic everywhere else. Thus $\omega_j$ must vanish on the component not containing $p_j$. In the case $k=2$ we integrate over a path $\gamma_2$ joining $p_2$ with
 $q_2$ in the component where $\omega_1$ is vanishing, so \[ \Psi(s_1 \otimes s_2)(X_{\bar{p}}) = \exp(2\pi i \smallint_{\gamma_2}\omega_1) = 1. \]
 In the case $i=1,k=3$ we have to integrate over a path starting at $p_3^+$, running over $q_3^-=p_3^-$ and ending at $q_3^+$. The second part of this path lies in the
 component where $\omega_1$ is vanishing, so we just have to integrate from $p_3^+$ to $q_3^-$. Again by using M\"obius transformations we may assume that $p_1=0$, $q_1=\infty$
 and $q_3^-=1$ and get \[ \Psi(s_1 \otimes s_3)(X_p) = \exp(2\pi i\smallint_{\gamma_3}\omega_1) = \exp(\smallint_{p_3^+}^1\tfrac{1}{z}dz) = (p_3^+)^{-1} = \bar{p}_{13}. \]
 Finally, by replacing $1$ with $2$, we get $\Psi(s_2 \otimes s_3)(X_p) =\bar{p}_{23}$.
\end{proof}

\begin{prop}
 Let $r=(r_1,r_2,r_3)$ be a $\ZZ$-basis of $\II$, let $s=(s_1,s_2,s_3)$ be its dual basis of $\II^\vee$ with respect to the pseudo-trace pairing and let
 ~$h$ be an element of $\Sym_{\QQ}(F) / (\Lambda_1+\Sym_{\ZZ}(\II))$.
 Writing \[ h = \sum_{i,j = 1}^3 b_{ij}(r_i \otimes r_j)\] with $b_{ij} \in \QQ$ and $b_{ij}=b_{ji}$ and identifying the $\II$-weighted boundary stratum $\mathcal{S}_{\bar{r}}$ with
 $\MM_{0,4} \times \MM_{0,4}$ as above, the subvariety $\mathcal{S}_{\bar{r}}(h) \subset \mathcal{S}_{\bar{r}}$ is cut out by the equations
 \[ (\bar{p}_{23})^{a_1} \cdot (\bar{p}_{13})^{a_2} = \exp(-2\pi i (a_1b_{23}+a_2b_{13}+a_3b_{12})), \] where $(a_1,a_2,a_3)$ runs over the integral solutions of
 \begin{equation*} a_1s_2s_3+a_2s_1s_3+a_3s_1s_2=0. \end{equation*}
\end{prop}
\begin{proof}
 Using the $\Psi$-values from the previous lemma, this proof works precisely as the proof of Theorem~\ref{CRequation}.
\end{proof}

We call the type of the boundary stratum given by the graph in Figure~\ref{Irreduzibel} the \emph{trinodal type} and the type of the boundary stratum given by the graph in Figure~\ref{NichtIrreduzibel}
the \emph{nice non-trinodal type}. We denote by~$\mathcal{S}_{nice} \subset \PcM_3$ the locus of stable forms~$(X,[\omega])$, where~$X$ is a stable curve in a stratum of one of these types.
By the cross-ratio equations and the techniques developed in~\cite{MartinMatt}, it follows that the necessary condition in Theorem~\ref{admissible} is also sufficient for stable forms
in~$\mathcal{S}_{nice}$.
\begin{theo}\label{hinreichend}
 Let $\OO \subset F$ be a pseudo-cubic order of degree relatively prime to the conductor of $D$, and let $(\II,E_h(\II)) \in \mathcal{C}(\OO)$ be a cusp packet for $\OO$.
 Furthermore, let $\iota: F \hookrightarrow \RR$ be one of the two real quadratic pseudo-embeddings.
 
 Then the intersection of the closure of the cusp of $\mathbb{P}\overline{\Omega^\iota\RM_\OO}$ associated to $(\II,E_h(\II))$ with~$\mathcal{S}_{nice}$ is equal to the union of all
 $p_\iota(\mathcal{S}(h))$, the $\iota$-pseudo-embeddings of the subvarieties~$\mathcal{S}(h)$ of~$\mathcal{S}$, where~$\mathcal{S}$ runs over all admissible $\II$-weighted boundary strata of trinodal
 and nice non-trinodal type.
\end{theo}
\begin{proof}
 By Corollary~\ref{Gleichung} and Theorem~\ref{admissible}, it remains to show that the converse inclusion $\mathcal{S}(h) \subseteq \mathcal{S} \cap \overline{\mathcal{R}\MM(\II,h)}$ holds for all
 admissible boundary strata~$\mathcal{S}$ of trinodal or nice non-trinodal type.
 But this proof is already done in~\cite{MartinMatt}, Section~8. A sketch of this proof is as follows.

 The stratum $\mathcal{S} \subset \cMM_3(\II) = \cTT(\Sigma_3,\II)/\Mod(\Sigma_3,\II)$ is obtained by pinching curves $\gamma_1,...,\gamma_m$ in~$\Sigma_3$.
 Let $U(\mathcal{S})$ be the union of $\MM_3(\II)$, $\mathcal{S}$ and all boundary strata obtained by pinching only curves $\gamma_k$.
 Thus $U(\mathcal{S})$ is an open subset of~$\cMM_3(\II)$. Then one can find convenient $\tau_1,...,\tau_n \in \bS_\ZZ(\II^{\vee})$, such that after possibly switching to a degree two quotient
 $\mathcal{S}'$ of~$\mathcal{S}$, the maps $\Psi(\tau_1),...,\Psi(\tau_n)$ form a system of local coordinates on~$\mathcal{S}$ respectively~$\mathcal{S}'$. 
 Here the description of~$\Psi$ in terms of cross-ratios was used, in our cases given by Lemma~\ref{Psi1} and Lemma~\ref{Psi2}.
 Finally, one can prove the analogous statement for algebraic tori, and the claim follows. We refer the reader to~\cite{MartinMatt} for more details. 
\end{proof}
\clearpage

\section{Prym Teichm\"uller curves}

In his famous work \cite{McMBilliards}, McMullen showed that in genus two, the Riemann surfaces (with a suitable choice of a holomorphic one-form) generating primitive Teichm\"uller curves are
precisely those Riemann surfaces admitting real multiplication on their Jacobian.
A few years later, he discovered in~\cite{McMPrym} infinite many examples of primitive Teichm\"uller curves in genus three.
Indeed, he constructed infinitely many flat surfaces with a holomorphic involution, such that the resulting Prym variety admits real multiplication.
For a crosscheck, we check the admissibility-condition from Theorem~\ref{admissible} on a class of these surfaces obtained by the Thurston-Veech construction.
Moreover, we will determine the cross-ratio equations for these examples.

\subsection{Prym varieties}

In this subsection we give a brief introduction to Prym varieties. We will use the definition of McMullen in~\cite{McMPrym} and present some of his results.
\\[1em]
{\bf Prym varieties.}
Fix a Riemann surface $X \in \MM_g$ and a non-trivial involution~$\tau$ of~$X$, i.e. some $\tau \in \Aut(X) \setminus \{ \id \}$ with $\tau^2 = \id$. We will call such a pair $(X,\tau)$ a
\emph{Prym surface}. The pullback $\tau^* \in \Aut_\CC(\Omega(X)) \setminus \{ \id \}$ is an involution of $\CC$-vector spaces and hence it defines a decomposition
\[ \Omega(X) = \Omega(X)^+ \oplus \Omega(X)^-, \] where $\Omega(X)^\pm$ is the eigenspace of $\tau^*$ to the eigenvalue $\pm 1$.
We will call the elements of $\Omega(X)^-$ the \emph{Prym forms of $X$ with respect to $\tau$}.
This yields a decomposition of the dual vector spaces \[ \Omega(X)^* = (\Omega(X)^+)^* \oplus (\Omega(X)^-)^*, \] where $(\Omega(X)^+)^*$ is the space of linear forms
vanishing on $\Omega(X)^-$ and vice versa.

Now consider the eigenspaces $H_1(X;\ZZ)^\pm \subset H_1(X;\ZZ)$ to the eigenvalues $\pm 1$ of the induced $\ZZ$-module automorphism $\tau_* \in \Aut_\ZZ(H_1(X;\ZZ))$.
We have the following well known decomposition of $\Jac(X)$.
\begin{lem}
 Let $(X,\tau)$ be a Prym surface.
 Then the quotients \[ \Jac(X)^+ := (\Omega(X)^+)^* / H_1(X;\ZZ)^+\quad {\rm and}\quad  \Jac(X)^- := (\Omega(X)^-)^* / H_1(X;\ZZ)^- \] are complementary subvarieties of $\Jac(X)$.
\end{lem}
\begin{proof}
The involution~$\tau_*$ is self-adjoint with respect to the intersection pairing and thus $H_1(X;\ZZ)^+$ and $H_1(X;\ZZ)^-$ are orthogonal submodules.
From the formula $\smallint_{\tau_*(\gamma)}\omega = \smallint_\gamma \tau^*(\omega)$ for all $\gamma \in H_1(X;\ZZ)$ and $\omega \in \Omega(X)$,
it follows that each group $H_1(X;\ZZ)^\pm$ is mapped into $(\Omega(X)^\pm)^*$ under the usual embedding $H_1(X;\ZZ) \hookrightarrow \Omega(X)^*$.
Since \[ H_1(X;\QQ) = H_1(X;\QQ)^+ \oplus H_1(X;\QQ)^-, \] we have that $H_1(X;\ZZ)^\pm$ is a lattice in $(\Omega(X)^\pm)^*$, and that $(\Omega(X)^+)^*$ and $(\Omega(X)^-)^*$ are
orthogonal with respect to the intersection form.
It follows that $(\Jac(X)^+,\Jac(X)^-)$ is indeed a pair of complementary subvarieties of~$\Jac(X)$.
\end{proof}
\begin{defi}
 Let $(X,\tau)$ be a Prym surface. The subvariety \[ \Prym(X,\tau) := \Jac(X)^- \] of $\Jac(X)$ is called the \emph{Prym variety of $(X,\tau)$}.
\end{defi}

The pullback of the canonical projection $\pi: X \to X/\tau$ is an isomorphism $\pi^*: \Omega(X/\tau) \cong \Omega(X)^+$, and thus \[ p(X,\tau) := \dim(\Prym(X,\tau)) = g(X)-g(X/\tau). \]
The projection $\pi$ is of degree two, so by the Riemann-Hurwitz formula we have \[ g(X)=2p(X,\tau)+1-\tfrac{1}{2} |\Fix(\tau)|, \] where $\Fix(\tau) \subset X$ denotes the set of fixpoints of $\tau$.
Therefore, we get \begin{align}\label{dimprym} p(X,\tau) \leq g(X) \leq 2p(X,\tau)+1. \end{align}
These inequalities are sharp. This can be seen by defining $\tau$ to be the canonical involution of a hyperelliptic surface $X$ which yields $\Omega(X)=\Omega(X)^-$,
respectively by defining $\tau$ to be the non-trivial deck transformation of a degree-two unramified double covering $X \to Y$, which yields $\Fix(\tau)=\emptyset$.
If $p(X,\tau)=2$, then according to \eqref{dimprym} the genus of~$X$ lies in $\{ 2,3,4,5 \}$.
\\[1em]
{\bf Prym varieties with real multiplication.}
We are interested in genus three Prym surfaces~$(X,\tau)$, such that~$\Prym(X,\tau)$ is of dimension two and has real multiplication by a real quadratic order~$\OD$.
Given such a Prym surface, we call a Prym form $\omega \in \Omega(X)^-$ a \emph{Prym eigenform}, if it is an eigenform for some choice of real multiplication on~$\Prym(X,\tau)$ by~$\OD$
(i.e. an eigenform for the corresponding pseudo-real multiplication).
We denote by $\Omega \EE_D^g \subseteq \Omega \MM_g$ the locus of all pairs~$(X,\omega)$, such that there exists a non-trivial involution $\tau$ of $X$, the corresponding Prym variety $\Prym(X,\tau)$
is of dimension two and has real multiplication by the real quadratic order~$\OD$, and such that $\omega \not= 0$ is a Prym eigenform.
Furthermore, we denote by $\Omega\WW_D^g \subseteq \Omega \EE_D^g$ the locus of those $(X,\omega)$, where $\omega$ has a unique zero (of degree $2g-2$), and by $\EE_D^g$ respectively $\WW_D^g$ the image
of~$\Omega\EE_D^g$ respectively $\Omega\WW_D^g$ in~$\MM_g$.

McMullen showed in \cite{McMPrym}, that the locus $\WW_D^g$ is the union of finitely many Teichm\"uller curves, and that for non-square $D$, each such curve is primitive.
He used the following theorem we will need in the next subsection.
\begin{theo}[McMullen]\label{closedinvariantsubset}
 The locus $\Omega \EE_D^g$ of Prym eigenforms for real multiplication by $\OD$ is a closed, ${\rm SL}_2(\RR)$-invariant subset of $\Omega \MM_g$.
\end{theo}

Furthermore, he developed the following very useful sufficient condition for holomorphic one-forms for lying in~$\Omega \EE_D^g$.
\begin{theo}[McMullen]\label{McM1}
 Let $(X,\tau)$ be a Prym surface, such that the dimension of $\Prym(X,\tau)$ is two and let $\omega \in \Omega(X)^- \setminus \{ 0 \}$.
 If ${\rm SL}(X,\omega)$ contains a hyperbolic element $A$, then we have $(X,\omega) \in \Omega \EE_D^g$ for some $D \in \NN_{>1}$ satisfying $\QQ(\sqrt{D})=\QQ(\tr(A))$.
\end{theo}
Note that in this theorem $D$ may be a square, or equivalently $\OD$ may be an order in $\QQ \oplus \QQ$.
\subsection{Cusps of Prym Teichm\"uller curves - an example}

In \cite{Thurston} Thurston constructed a large class of diffeomorphisms on a given surface by using generalized Dehn twists about multicurves.
In \cite{McMPrym}, McMullen used this construction to obtain infinitely many points in $\bigcup_D \WW_D \subset \MM_g$ for each $g \in \{ 2,3,4 \}$. This leads to infinitely many primitive Teichm\"uller
curves in $\MM_g$.
Since our focus is on the genus three case, we give a more detailed discussion of this construction for~$g=3$, including a computation of the discriminant of these examples.
Furthermore, we will use these examples for doing a crosscheck of Theorem~\ref{admissible} and determine the corresponding cross-ratio equations. 
\\[1em]
{\bf $S$-shaped polygons.}
We first explain the construction, established by McMullen in \cite{McMPrym}.
For each $n\in \NN$, we define the matrix
\[ A_n := \begin{pmatrix} 0 & 0 & 0 & 0 & 1 & 1 \\ 0 & 0 & 0 & 0 & n & 0 \\ 0 & 0 & 0 & 1 & 1 & 0 \\ 0 & 0 & n & 0 & 0 & 0 \\ 1 & 1 & 1 & 0 & 0 & 0 \\ n & 0 & 0 & 0 & 0 & 0\end{pmatrix}. \]
The characteristic polynomial of $A_n$ is \[ \chi_{A_n}(x) = (x^2-n)(x^4-2(n+1)x^2+n^2), \] and thus $A_n$ has six pairwise different real eigenvalues.
The largest one is \[ \mu = \sqrt{n+1+\sqrt{2n+1}}, \] and a generator for the one-dimensional eigenspace of $A_n$ to $\mu$ is
\[ \begin{pmatrix} h_1 \\ h_2 \\ h_3 \\ h_4 \\ h_5 \\ h_6 \end{pmatrix} = \begin{pmatrix} \mu/n \\ \mu-n/\mu \\ \mu/n \\ 1 \\ \mu^2/n -1 \\ 1 \end{pmatrix}
 = \begin{pmatrix} \mu/n \\ \sqrt{2} \\ \mu/n \\ 1 \\ \sqrt{2}\mu/n \\ 1 \end{pmatrix}. \]
Here we have used the beautiful identity \begin{equation} \frac{1+\sqrt{2n+1}}{\sqrt{n+1+\sqrt{2n+1}}} = \sqrt{2} \end{equation} for all $n \in \NN$.
Now let $S_n$ be the $S$-shaped polygon with sides $s_i,s_i'$ of lengths~$h_i$ as shown in Figure~\ref{S-Flaeche} below
and let $T_n$ be the translation surface obtained from $S_n$ by identifying $s_i$ with $s_i'$ by translations. The numbers on the right picture denote the lenghts of the edges.

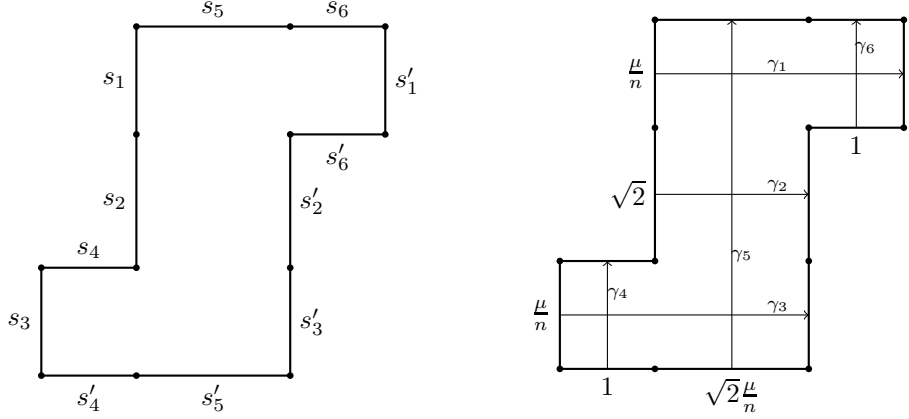
\begin{figure}[ht]
\centering
\begin{tikzpicture}[scale=1.25]
 \coordinate (A) at (0,0); \coordinate[label=below:$s_4'$] (AB) at (0.5,0);
 \coordinate (B) at (1,0); \coordinate[label=below:$s_5'$] (BC) at (1.809,0);
 \coordinate (C) at (2.618,0); \coordinate[label=right:$s_3'$] (CD) at (2.618,0.572);
 \coordinate (D) at (2.618,1.144); \coordinate[label=right:$s_2'$] (DE) at (2.618,1.851);
 \coordinate (E) at (2.618,2.558); \coordinate[label=below:$s_6'$] (EF) at (3.118,2.558);
 \coordinate (F) at (3.618,2.558); \coordinate[label=right:$s_1'$] (FG) at (3.618,3.130);
 \coordinate (G) at (3.618,3.702); \coordinate[label=above:$s_6$] (GH) at (3.118,3.702);
 \coordinate (H) at (2.618,3.702); \coordinate[label=above:$s_5$] (HI) at (1.809,3.702);
 \coordinate (I) at (1,3.702); \coordinate[label=left:$s_1$] (IJ) at (1,3.130);
 \coordinate (J) at (1,2.558); \coordinate[label=left:$s_2$] (JK) at (1,1.851);
 \coordinate (K) at (1,1.144); \coordinate[label=above:$s_4$] (KL) at (0.5,1.144);
 \coordinate (L) at (0,1.144); \coordinate[label=left:$s_3$] (LA) at (0,0.572);
 \fill (A) circle (1pt); \fill (B) circle (1pt); \fill (C) circle (1pt); \fill (D) circle (1pt); \fill (E) circle (1pt); \fill (F) circle (1pt);
 \fill (G) circle (1pt); \fill (H) circle (1pt); \fill (I) circle (1pt); \fill (J) circle (1pt); \fill (K) circle (1pt); \fill (L) circle (1pt);
 \draw[thick] (A) -- (B) -- (C) -- (D) -- (E) -- (F) -- (G) -- (H) -- (I) -- (J) -- (K) -- (L) -- (A);
\end{tikzpicture}
\hspace{1cm}
\centering
\begin{tikzpicture}[scale=1.25]
 \coordinate (A) at (0,0); \coordinate[label=below:$1$] (AB) at (0.5,0);
 \coordinate (B) at (1,0); \coordinate[label=below:$\sqrt{2}\tfrac{\mu}{n}$] (BC) at (1.809,0);
 \coordinate (C) at (2.618,0); \coordinate (CD) at (2.618,0.572);
 \coordinate (D) at (2.618,1.144); \coordinate (DE) at (2.618,1.851);
 \coordinate (E) at (2.618,2.558); \coordinate[label=below:$1$] (EF) at (3.118,2.558);
 \coordinate (F) at (3.618,2.558); \coordinate (FG) at (3.618,3.130);
 \coordinate (G) at (3.618,3.702); \coordinate (GH) at (3.118,3.702);
 \coordinate (H) at (2.618,3.702); \coordinate (HI) at (1.809,3.702);
 \coordinate (I) at (1,3.702); \coordinate[label=left:$\tfrac{\mu}{n}$] (IJ) at (1,3.130);
 \coordinate (J) at (1,2.558); \coordinate[label=left:$\sqrt{2}$] (JK) at (1,1.851);
 \coordinate (K) at (1,1.144); \coordinate (KL) at (0.5,1.144);
 \coordinate (L) at (0,1.144); \coordinate[label=left:$\tfrac{\mu}{n}$] (LA) at (0,0.572);
 \fill (A) circle (1pt); \fill (B) circle (1pt); \fill (C) circle (1pt); \fill (D) circle (1pt); \fill (E) circle (1pt); \fill (F) circle (1pt);
 \fill (G) circle (1pt); \fill (H) circle (1pt); \fill (I) circle (1pt); \fill (J) circle (1pt); \fill (K) circle (1pt); \fill (L) circle (1pt);
 \draw[thick] (A) -- (B) -- (C) -- (D) -- (E) -- (F) -- (G) -- (H) -- (I) -- (J) -- (K) -- (L) -- (A);
 \draw[->] (IJ) -- (FG); \draw[->] (JK) -- (DE); \draw[->] (LA) -- (CD);
 \draw[->] (AB) -- (KL); \draw[->] (BC) -- (HI); \draw[->] (EF) -- (GH);
 \coordinate[label=above:$_{\gamma_1}$] (g1) at (2.3,3.05);
 \coordinate[label=above:$_{\gamma_2}$] (g2) at (2.3,1.78);
 \coordinate[label=above:$_{\gamma_3}$] (g3) at (2.3,0.5);
 \coordinate[label=right:$_{\gamma_4}$] (g4) at (0.4,0.8);
 \coordinate[label=right:$_{\gamma_5}$] (g5) at (1.7,1.2);
 \coordinate[label=right:$_{\gamma_6}$] (g6) at (3,3.4);
\end{tikzpicture}
\caption{S-shaped surface}\label{S-Flaeche}
\end{figure}

The rotation around the centre of this polygon by the angle $\pi$ induces an involution  $\tau \in \Aut(S_n) \setminus \{ \id \}$.
Since the genus of $S_n/\tau$ is one, we have $\dim(\Prym(S_n,\tau))=2$.
Let $\omega \in \Omega(S_n)$ be the holomorphic form represented by~$dz$ on the charts of the translation atlas.
Then clearly~$\omega$ is a Prym form of~$T_n$ with respect to~$\tau$ having a single zero, and we want to show that $(T_n,\omega)$ is in $\Omega \WW_D^3$.
\begin{prop}[McMullen]
 For each $n \in \NN$ with $2n+1$ not a square, let~$(T_n,\omega)$ be the translation surface constructed above.
 Then we have \[ (T_n,\omega) \in \Omega \WW_D^3 \] for some non-square discriminant $D>0$ and $\Prym(T_n,\tau)$ is of type $(1,2)$.
\end{prop}
\begin{proof}
 First we verify that the Prym surface $(T_n,\tau)$ satisfies the conditions in Theorem~\ref{McM1}.
 Consider the natural foliation of the polygon $S_n$ into the three horizontal cylinders $C_1,C_2,C_3$, where $C_i$ is the cylinder with the sides~$s_i$ and~$s_i'$ identified.
 Let~$\tau_i$ be the standard Dehn twist of $C_i$ along the cylinder midth, i.e. the affine automorphisms with derivatives given by
 \[ D\tau_1 = \begin{pmatrix} 1 & \tfrac{h_5+h_6}{h_1} \\ 0 & 1 \end{pmatrix} = \begin{pmatrix} 1 & \mu \\ 0 & 1 \end{pmatrix}, \]
 \[ D\tau_2 = \begin{pmatrix} 1 & \tfrac{h_5}{h_2} \\ 0 & 1 \end{pmatrix} = \begin{pmatrix} 1 & \mu/n \\ 0 & 1 \end{pmatrix} \] and
 \[ D\tau_3 = \begin{pmatrix} 1 & \tfrac{h_4+h_5}{h_3} \\ 0 & 1 \end{pmatrix} = \begin{pmatrix} 1 & \mu \\ 0 & 1 \end{pmatrix}.\]
 Therefore, the composition \[ \tau_h := \tau_1 \circ \tau_2^n \circ \tau_3 \] is an affine automorphism of $S_n$ with derivative \[ A_h := D\tau_h = \begin{pmatrix} 1 & \mu \\ 0 & 1 \end{pmatrix}. \]
 Moreover, there is also a natural foliation of this polygon into three vertical cylinders~$C_4,C_5,C_6$. We can define the standard Dehn twists $\tau_i$ for these vertical cylinders in an analogous way
 and obtain the derivatives \[ D\tau_4 = \begin{pmatrix} 1 & 0 \\ -\mu/n & 1 \end{pmatrix},\quad
 D\tau_5 = \begin{pmatrix} 1 & 0 \\ -\mu & 1 \end{pmatrix}\quad \mbox{and}\quad D\tau_6 = \begin{pmatrix} 1 & 0 \\ -\mu/n & 1 \end{pmatrix}. \]
 Therefore, the composition \[ \tau_v := \tau_4^n \circ \tau_5 \circ \tau_6^n \] is an affine automorphism of $S_n$ with derivative \[ A_v := D\tau_v = \begin{pmatrix} 1 & 0 \\ -\mu & 1 \end{pmatrix}. \]
 From now on we will use that $n \in \NN$ with $2n+1$ is not a square.
 We compute that \[ \tr(A_h^kA_v^l) = 2 - kl\mu^2 = 2-kl(n+1+\sqrt{2n+1}) \] for all $k,l \in \ZZ$.
 By choosing $k$ and $l$ sufficiently large, we deduce that ${\rm SL}(T_n,\omega)$ contains infinitely many hyperbolic elements.
 Thus, by Theorem~\ref{McM1} we get that $\Prym(T_n,\tau)$ has real multiplication by $\OD$ for some discriminant $D$ with $\QQ(\sqrt{D})=\QQ(\sqrt{2n+1})$ and that $\omega$ is a Prym eigenform.

 A $\ZZ$-basis of $H_1(T_n,\ZZ)$ is given by the homology classes of the simple closed curves $\gamma_i$ given by the cylinder midths as in Figure~\ref{S-Flaeche}.
 We have \[ H_1(T_n;\ZZ)^+ = \langle \gamma_1-\gamma_3,\gamma_4-\gamma_6 \rangle_\ZZ \] and \[ H_1(T_n;\ZZ)^- = \langle \gamma_2, \gamma_5, \gamma_1+\gamma_3, \gamma_4+\gamma_6 \rangle_\ZZ. \]
 It follows that the induced polarization on $\Prym(T_n,\tau)$ is of type $(1,2)$.
\end{proof}
 As McMullen didn't determine the discriminant of the order~$\OD$, we will do this here.
\begin{prop}
 For each $n \in \NN$ with $2n+1$ not a square, let $(T_n,\omega) \in \Omega \EE_D^3$ be the translation surface constructed above.
 Then we have $D=4(2n+1)$.
\end{prop}
\begin{proof}
 We use the same notations as in the previous proof. Since \[ K=\QQ(\sqrt{D})=\QQ(\sqrt{2n+1}), \] we deduce that $D=q^2(2n+1)$ for some $q \in \QQ_+$.
 We already know that $K$ acts on $H_1(T_n;\QQ)^-$ by $\QQ$-linear maps, and the order $\OD \subset K$ is just the coefficient ring of $H_1(T_n;\ZZ)^- \subset H_1(T_n;\QQ)^-$.
 The order $\OD = \langle 1,\gamma_D \rangle_\ZZ$ acts on $H_1(T_n;\ZZ)^-$ if and only if $\gamma_D$ preserves $H_1(T_n;\ZZ)^-$. So our aim is to choose a suitable $\ZZ$-basis
 for $H_1(T_n;\ZZ)$ and to verify for which choices of $q$ the matrix of $\gamma_D$ with respect to this basis has integral entries.
 A symplectic basis of $H_1(T_n;\ZZ)^-$ is given by \[ (\lambda_1,\lambda_2,\lambda_3,\lambda_4) = (\gamma_5-\gamma_4-\gamma_6,\gamma_4+\gamma_6,\gamma_2,\gamma_1+\gamma_3). \]
 Let $A=(a_{ij})_{i,j} \in {\rm GL}_4(\QQ)$ be the matrix of $\gamma_D$ with respect to this basis, so \[ \gamma_D.\lambda_j = \sum_{i=1}^4 a_{ij}\lambda_i. \]
 The form $\omega$ is a $\iota$-eigenform for one of the two embeddings $K \hookrightarrow \RR$ and hence
 \begin{equation}\label{Integral} \iota(\gamma_D)\int_{\lambda_j}\omega = \sum_{i=1}^4 a_{ij}\int_{\lambda_i}\omega.\end{equation}
 Since $\omega = dz$ on the charts of the translation atlas, the integrals over the $\lambda_i$ are given by the complex lengths of the cylindermidths defining the curves $\gamma_i$.
 More concretely, we have \[ \int_{\lambda_1}\omega = \sqrt{2}i,\quad \int_{\lambda_2}\omega = 2(\mu/n) i,\quad \int_{\lambda_3}\omega = \sqrt{2}\mu/n,\quad \int_{\lambda_4}\omega = 2\mu^2/n. \]
 Writing $\iota(\gamma_D) = \tfrac{1}{2}(D \pm \sqrt{D})$, multiplying with $n/\mu = \tfrac{1}{2}\sqrt{2}(\sqrt{2n+1}-1)$ and using the identities
 $\mu = \tfrac{1}{2}\sqrt{2}(\sqrt{2n+1}+1)$ and $D=q^2(2n+1)$, the four equations in \eqref{Integral} become
 \begin{eqnarray*}
  \tfrac{q^2(2n+1) \pm q\sqrt{2n+1}}{2}(\sqrt{2n+1}-1)i & = & a_{11}(\sqrt{2n+1}-1)i + a_{21}2i \\ & & + a_{31}\sqrt{2} + a_{41}\sqrt{2}(\sqrt{2n+1}+1) \\
  (q^2(2n+1) \pm q\sqrt{2n+1})i & = & a_{12}(\sqrt{2n+1}-1)i + a_{22}2i \\ & & + a_{32}\sqrt{2} + a_{42}\sqrt{2}(\sqrt{2n+1}+1) \\
  \tfrac{q^2(2n+1) \pm q\sqrt{2n+1}}{2}\sqrt{2} & = & a_{13}(\sqrt{2n+1}-1)i + a_{23}2i \\ & & + a_{33}\sqrt{2} + a_{43}\sqrt{2}(\sqrt{2n+1}+1) \\
  \tfrac{q^2(2n+1) \pm q\sqrt{2n+1}}{2}\sqrt{2}(\sqrt{2n+1}+1) & = & a_{14}(\sqrt{2n+1}-1)i + a_{24}2i \\ & & + a_{34}\sqrt{2} + a_{44}\sqrt{2}(\sqrt{2n+1}+1)
 \end{eqnarray*}
 We assumed that $2n+1$ is not a square, hence $\sqrt{2}$ and $\sqrt{2}\sqrt{2n+1}$ are $\QQ$-linearly independent. By equating the coefficients, we obtain
 \[ A = \begin{pmatrix} \tfrac{1}{2}(D \mp q) & \pm q & 0 & 0 \\ \pm \tfrac{1}{2}qn & \tfrac{1}{2}(D \pm q) & 0 & 0 \\
                        0 & 0 & \tfrac{1}{2}(D \mp q) & \pm qn \\ 0 & 0 & \pm \tfrac{1}{2}q & \tfrac{1}{2}(D \pm q) \end{pmatrix} \]
 All the entries of $A$ are integers if and only if $q$ is in $2\ZZ$, and therfore it follows~$D=4(2n+1)$.
 \end{proof}

{\bf Crosscheck for admissibility.}
We denote by $\Gamma T_n \subset \Omega \MM_3$ the orbit of~$(T_n,\omega)$ under $\Gamma={\rm SL}_2(\RR)$, and by $\mathbb{P}\Gamma T_n \subset \PM_3$ its image under the natural projection.
By Theorem~\ref{closedinvariantsubset}, we have $\Gamma T_n \subset \Omega \EE_D^3$.

It follows by Theorem~\ref{admissible}, that any geometric genus zero curve in the closure of~$\mathbb{P}\Gamma T_n$ in~$\PcM_3$ lies in the image $p_\iota(\mathcal{S})$ of some admissible
$\II$-weighted boundary stratum under the $\iota$-pseudo-embedding.
To obtain such a stable curve, notice that for any $r \in \RR_+$, the diagonal matrix ${\rm diag}(r^{-1/2},r^{1/2})$ is in $\Gamma$, and the two translation surfaces
$T_{n,r} := {\rm diag}(1,r).(T_n,\omega)$ and ${\rm diag}(r^{-1/2},r^{1/2}).(T_n,\omega)$ define the same points in $\PM_3$.
The surface $T_{n,r}$ is obtained by $T_n$ by remaining the horizontal lengths constant and by stretching the vertical length by the factor~$r$. As~$r$ tends to infinity, topologically the curves
$\gamma_1,\gamma_2,\gamma_3$ defined by the horizontal cylinder midths are pinched to zero. We see by Figure~\ref{S-Flaeche}, that the resulting stable curve is irreduzible of geometric genus zero.
During this process, the periods $\smallint_{\gamma_i}\omega$ for $i \in \{ 1,2,3 \}$ remain constant.
Thus $T_{n,r}$ converges to an irreducible stable form $(T,\omega) \in \PcM_3$ with three nodes $n_1,n_2,n_3$. Up to a common scalar multiple, each residuum $\tilde r_i$ at one of the cusps~$c_i$
of~$n_i$ is given by \[ \tilde r_1 = \tilde r_3 = \tfrac{1}{2}(1+\sqrt{2n+1})\quad \mbox{and}\quad \tilde r_2 = 1. \]

By Corollary~\ref{admissible1}, there is some admissible $\II$-weighted boundary stratum $\mathcal{S} \subset \cMM_3(\II)$ with $(T,\omega) \in \pi_\iota(\mathcal{S})$.
But with the aim of doing a crosscheck, we can verify this also directly. Choose
\[ r_1 := (\tfrac{1}{2}(1 \pm \sqrt{2n+1}),1),\quad r_2 := (1,0),\quad r_3 := (\tfrac{1}{2}(1 \pm \sqrt{2n+1}),-1) \] in $F = K \oplus \QQ$
and let $\I := \langle r_1,r_2,r_3 \rangle_\ZZ$. Moreover, let $\varrho$ be the $\I$-weighting of $T$ with weight $r_i$ at each cusp $c_i$
and let $\mathcal{S}$ be the corresponding $\I$-weighted boundary stratum. Then we have $\pi_\iota(T,\varrho) = (T,\omega)$ and compute
\[ Q(r_1) = \begin{pmatrix} 2n(2n+1) \\ 1 \\ \pm 2(2n+1) \end{pmatrix}, Q(r_2) = \begin{pmatrix} -4(2n+1) \\ 0 \\ 0 \end{pmatrix},
Q(r_3) = \begin{pmatrix} 2n(2n+1) \\ -1 \\ \mp 2(2n+1) \end{pmatrix}. \]
In each of the two possibilities for $\iota$ we see that \[ 0 = Q(r_1)+nQ(r_2)+Q(r_3) \] is in the interior of the convex hull of the $Q(r_i)$,
and thus $\mathcal{S}$ is admissible.
\\[1em]
{\bf Cross-ratio equations.}
Now want to determine the cross ratio equations from Theorem~\ref{CRequation}. The dual basis of $(r_1,r_2,r_3)$ is $(s_1,s_2,s_3)$ with
\[ s_1 = \tfrac{1}{2}\left(\pm\tfrac{\sqrt{2n+1}}{2n+1},1\right),\quad s_2 = \tfrac{1}{2}\left(1\mp\tfrac{\sqrt{2n+1}}{2n+1},0\right),\quad s_3 = \tfrac{1}{2}\left(\pm\tfrac{\sqrt{2n+1}}{2n+1},-1\right). \]
By computing the products $s_is_j$ for $s_i \not= s_j$, Equation~\eqref{Ann} becomes
\[ a_1\left(\tfrac{1}{D}(\pm\sqrt{2n+1}-1),0\right) + a_2\left(\tfrac{1}{D},-\tfrac{1}{4}\right) + a_3\left(\tfrac{1}{D}(\pm\sqrt{2n+1}-1),0\right) = 0, \] which is equivalent to $a_1+a_3=a_2=0$.
Therefore, for any given \[ h = \sum_{i,j = 1}^3 b_{ij}(r_i \otimes r_j) \in \Sym_{\QQ}(F) / (\Lambda_1+\Sym_{\ZZ}(\II)), \] the subvariety $\mathcal{S}_r(h) \subset \mathcal{S}_r = \MM_{0,6}$
is cut out by the single equation \[ p_{12}/p_{23} = \exp(-2\pi i (b_{12}-b_{23})). \]
Note that on the right side of this equation, all roots of unity $\exp(\pi i q)$, $q \in \QQ$ appear, as $h$ runs through $ \Sym_{\QQ}(F) / (\Lambda_1+\Sym_{\ZZ}(\II))$.
\clearpage

\bibliographystyle{alpha}
\bibliography{biblio}

\end{document}